\documentclass[UTF8,10pt]{article}

\usepackage[left=1.91cm,right=1.91cm,top=2.54cm,bottom=2.54cm]{geometry}
\usepackage{amsfonts}
\usepackage{amsmath}
\usepackage{amssymb}
\usepackage{pgfplots}
\usepackage{tikz}
\usetikzlibrary{arrows}
\usepackage{titlesec}
\usepackage{bm}
\usepackage{appendix}
\usepackage{tikz-cd}
\usepackage{mathtools}
\usepackage{amsthm}
\usepackage{setspace}
\usepackage{enumitem}
\setenumerate[1]{itemsep=0pt,partopsep=0pt,parsep=\parskip,topsep=5pt}
\setitemize[1]{itemsep=0pt,partopsep=0pt,parsep=\parskip,topsep=5pt}
\setdescription{itemsep=0pt,partopsep=0pt,parsep=\parskip,topsep=5pt}
\usepackage[colorlinks,linkcolor=black,anchorcolor=black,citecolor=black]{hyperref}
\usepackage[numbers]{natbib}
\usepackage{cases}
\usepackage{fancyhdr}
\usepackage[scaled=0.92]{helvet}	
\usepackage{upgreek}
\usepackage{comment}
\usepackage{txfonts}

\theoremstyle{definition}
\newtheorem{thm}{Theorem}[section]

\newtheorem{lem}[thm]{Lemma}
\newtheorem{cor}[thm]{Corollary}
\newtheorem{prop}[thm]{Proposition}
\newtheorem{rmk}[thm]{Remark}
\newtheorem{defn}[thm]{Definition}

\newcommand{\R}{\mathbb{R}}  
\newcommand{\Z}{\mathbb{Z}}

\newcommand{\N}{\mathbb{N}}
\newcommand{\Q}{\mathcal{Q}}
\newcommand{\T}{\mathbb{T}}
\newcommand{\TT}{\mathcal{T}}
\newcommand{\TP}{\overline{\partial}{}}
\newcommand{\TJ}{\langle\TP\rangle}
\newcommand{\TL}{\overline{\Delta}{}}

\newcommand{\curl}{\text{curl }}
\newcommand{\dive}{\text{div }}
\newcommand{\ST}{\text{ ST}}
\newcommand{\RT}{\text{ RT}}

\newcommand{\bd}[1]{\mathbf{#1}}  
\newcommand{\RR}{\mathcal{R}}      
\newcommand{\Om}{\Omega}
\newcommand{\q}{\quad}
\newcommand{\p}{\partial}
\newcommand{\dd}{\mathfrak{D}}
\newcommand{\DD}{\mathcal{D}}

\newcommand{\den}{{\bar\rho}_{0}}

\newcommand{\nab}{\nabla}

\newcommand{\lkk}{\Lambda_{\kk}}
\newcommand{\lap}{\Delta}
\newcommand{\no}{\nonumber}
\newcommand{\di}{\text{div}\,}

\newcommand{\lleq}{\stackrel{L}{=}}
\newcommand{\cnab}{\overline{\nab}}

\newcommand{\dx}{\,\mathrm{d}x}

\newcommand{\dt}{\,\mathrm{d}t}
\newcommand{\dtt}{\,\mathrm{d}\tau}

\newcommand{\deta}{\,\mathrm{d}\eta}

\newcommand{\kk}{\kappa}
\newcommand{\Er}{\mathring{E}}

\newcommand{\Kr}{\mathring{K}}
\newcommand{\Mr}{\mathring{M}}
\newcommand{\ee}{\mathfrak{e}}

\newcommand{\Rr}{\mathring{\mathcal{R}}}
\newcommand{\cc}{\mathfrak{C}}
\newcommand{\ccr}{\mathring{\mathfrak{C}}}
\newcommand{\ddr}{\mathring{\mathfrak{D}}}
\newcommand{\eer}{\mathring{\mathfrak{e}}}
\newcommand{\ssr}{\mathring{\mathcal{S}}}
\newcommand{\Rrr}{\mathring{\mathcal{R}}}
\newcommand{\sss}{\mathcal{S}}
\newcommand{\h}{\mathcal{H}}

\newcommand{\VV}{\mathbf{V}}

\newcommand{\FF}{\mathbf{F}}
\newcommand{\EE}{\mathfrak{E}}
\newcommand{\HE}{\mathcal{E}}
\newcommand{\VVr}{\mathring{\mathbf{V}}}
\newcommand{\QQr}{\mathring{\mathbf{Q}}}

\newcommand{\hhr}{\mathring{\mathfrak{h}}}
\newcommand{\fn}{{\mathfrak{n}}}
\newcommand{\fm}{{\mathfrak{m}}}
\newcommand{\MM}{{\mathfrak{M}}}
\newcommand{\FH}{\mathfrak{H}}
\newcommand{\fr}{\mathring{f}}
\newcommand{\qh}{\underline{q}}

\newcommand{\qc}{\check{q}}
\newcommand{\pc}{\check{p}}

\newcommand{\FFr}{\mathring{\mathbf{F}}}
\newcommand{\Ur}{\mathring{U}}

\newcommand{\Ir}{\mathring{I}}

\newcommand{\PP}{\mathcal{P}}
\newcommand{\LP}{\mathbf{P}}
\newcommand{\uuk}{\mathbf{u}_{\kk, 0}}
\newcommand{\uu}{\textbf{u}_0}
\newcommand{\ww}{\textbf{w}_0}

\newcommand{\idt}{\int_{\mathcal{D}_t}}

\newcommand{\io}{\int_{\Omega}}
\newcommand{\is}{\int_{\Sigma}}
\newcommand{\isb}{\int_{\Sigma_b}}

\newcommand{\ddt}{\frac{\mathrm{d}}{\mathrm{d}t}}
\numberwithin{equation}{section}

\newcommand{\eps}{\varepsilon}

\usepackage{xcolor}

\providecommand{\len}[1]{\langle #1 \rangle }
\providecommand{\ino}[1]{\left\| #1 \right\| }
\providecommand{\bno}[1]{\left| #1 \right| }

\newcommand{\orange}{{}}

\newcommand{\lam}{\lambda}
\newcommand{\Lam}{\Lambda}

\newcommand{\pk}{\widetilde{\varphi}}
\newcommand{\pr}{\mathring{\varphi}}
\newcommand{\pd}{\dot{\varphi}}
\newcommand{\pkd}{\dot{\pk}}
\newcommand{\pkr}{\mathring{\pk}}
\newcommand{\vr}{\mathring{v}}
\newcommand{\rhor}{\mathring{\rho}}
\newcommand{\psk}{\widetilde{\psi}}
\newcommand{\psr}{\mathring{\psi}}
\newcommand{\psd}{\dot{\psi}}
\newcommand{\pskd}{\dot{\psk}}
\newcommand{\pskr}{\mathring{\psk}}
\newcommand{\nnr}{{(n+1)}}
\newcommand{\nnn}{{(n)}}
\newcommand{\nnl}{{(n-1)}}
\newcommand{\nnll}{{(n-2)}}
\newcommand{\npk}{\widetilde{N}}
\newcommand{\npkr}{\mathring{\npk}}

\newcommand{\NN}{\mathbf{N}}
\newcommand{\Npk}{\widetilde{\NN}}
\newcommand{\Npkr}{\mathring{\Npk}}
\newcommand{\Npr}{\mathring{\NN}}

\newcommand{\Npkd}{\dot{\Npk}}
\newcommand{\QQ}{\mathbf{Q}}
\newcommand{\vb}{\overline{v}}

\newcommand{\vbr}{\overline{\vr}}
\newcommand{\ff}{\mathcal{F}}
\newcommand{\ffr}{\mathring{\mathcal{F}}}
\newcommand{\qr}{\mathring{q}}
\newcommand{\ls}{{\lam, \sigma}}
\newcommand{\qi}{\mathfrak{p}}
\newcommand{\nn}{\mathcal{N}}
\newcommand{\vk}{\widetilde{v}}
\newcommand{\pp}{\p^{\varphi}}
\newcommand{\nabp}{\nab^{\varphi}}
\newcommand{\ppk}{\p^{\pk}}
\newcommand{\ppkr}{\p^{\pkr}}
\newcommand{\nabpk}{\nab^{\pk}}
\newcommand{\nabpkr}{\nab^{\pkr}}

\newcommand{\lapp}{\lap^{\varphi}}

\newcommand{\dn}{\mathfrak{N}_\psi}
\newcommand{\Dtp}{D_t^{\varphi}}
\newcommand{\Dtpk}{D_t^{\pk}}
\newcommand{\Dtpkr}{D_t^{\pkr}}

\newcommand{\dvt}{\mathrm{d}\mathcal{V}_t}

\newcommand{\DT}{(\p_t+\vb\cdot \TP)}
\newcommand{\TDt}{\overline{D_t}}

\newcommand{\Tpp}{\overline{\p}^{\varphi}}
\newcommand{\Rbf}{\mathbf{R}}
\newcommand{\Cbf}{\mathbf{C}}

\begin{document}
\title{\bf Compressible Gravity-Capillary Water Waves with Vorticity: \\ Local Well-Posedness, Incompressible\\ and Zero-Surface-Tension Limits}
\date{}
\author{{\sc Chenyun Luo}\thanks{Department of Mathematics, The Chinese University of Hong Kong, Shatin, NT, Hong Kong. CL is supported in part by the Hong Kong RGC grants CUHK--24304621, CUHK--14302922,  CUHK--14304424, and the CUHK Science Faculty Collaborative Research Impact Matching Scheme (CRIMS). 
Email: \texttt{cluo@math.cuhk.edu.hk}}\,\,\,\,\,\, {\sc Junyan Zhang} \thanks{School of Mathematical Sciences, University of Science and Technology of China. 96 Jinzhai Road, Baohe District, Hefei, Anhui 230026, China.
Email: \texttt{yx3x@ustc.edu.cn}} }
\maketitle

\begin{abstract}
We consider the 3D compressible isentropic Euler equations describing the motion of a liquid in an unbounded initial domain with a moving boundary and a fixed flat bottom at finite depth. The liquid is under the influence of gravity and surface tension, and it is not assumed to be irrotational. We prove local well-posedness by combining a carefully designed approximate system and a hyperbolic approach, which allows us to avoid using Nash-Moser iteration. The energy estimates yield no regularity loss and are uniform in both Mach number and surface tension coefficient, provided the Rayleigh-Taylor sign condition is satisfied. We thus simultaneously obtain incompressible and zero surface tension limits. Moreover, we can drop the uniform boundedness (with respect to Mach number) on high-order time derivatives by applying the paradifferential calculus to the analysis of the free-surface evolution.
\end{abstract}



\setcounter{tocdepth}{1}
\tableofcontents

\section{Introduction}

In this paper, we study the motion of  water waves in $\mathbb{R}^3$ described by the compressible Euler equations:
\begin{equation}\label{wweq}
\begin{cases}
\uprho (\p_t+u\cdot\nabla)u =-\nab p-\uprho ge_3,&~~~ \text{in}\,\DD\\
\p_t\uprho + \nab\cdot (\uprho u)=0 &~~~ \text{in}\,\DD\\
p=p(\uprho) &~~~ \text{in}\, \DD \\
\end{cases}
\end{equation} 
where $\DD=\bigcup\limits_{0\leq t\leq T}\{t\}\times \DD_t$ with $\DD_t:=\{(x_1, x_2,x_3)\in\R^3: -b<x_3< \psi(t,x_1,x_2)\}$ with $b>10$ a given constant representing the unbounded domain with finite depth occupied by the fluid at each fixed time $t$, whose boundary $\p\DD_t$ is determined by a moving surface represented via the graph $\Sigma_t:=\{(x_1,x_2, x_3)\in\R^3:x_3=\psi(t,x_1, x_2)\}$ and a flat bottom $\Sigma_b:=\{(x_1,x_2,x_3)\in\R^3:x_3=-b\}$. {We will consider the case when $\Sigma_t\cap \Sigma_b =\emptyset$. This is easy to achieve in a short interval by assuming $\|\psi(0, \cdot)\|_{L^\infty(\R^2)}\leq 1$.}

In the first two equations of \eqref{wweq}, $u,\uprho, p$ represent the fluid's velocity, density, and pressure, respectively. Additionally, we assume that the fluid is under the influence of gravity $\uprho ge_3$, where $g>0$ and $e_3=(0,0,1)^\top$. The third equation of \eqref{wweq} is known to be the equation of states which satisfies 
\begin{align}
p'(\uprho)> 0, \q \text{for}\,\, \uprho\geq \bar{\uprho}_0,
\label{EoS}
\end{align}
where $\bar{\uprho}_0$ is a positive constant (we set $\den=1$ for simplicity), which is in the case of an isentropic \textit{liquid}\footnote{In general, the equation of state is $p=p(\uprho, S)$ where $S$ denotes the entropy of the fluid and satisfies $(\p_t+u\cdot\nabla)S=0$. It is required to have $\p p/\p\uprho>0$. When $S$ is a constant, we say the fluid is isentropic. Also, the assumptions $p'(\uprho)>0$ and $\uprho\geq\bar{\uprho}_0$ ensure the hyperbolicity of \eqref{wweq}.}. The equation of states is required to close the system of compressible Euler equations. We mention here that in the case of a \textit{gas} $\den=0$,  and we shall not discuss this in the paper.

The initial and boundary conditions of the system \eqref{wweq} are
\begin{align}
\DD_0=\{x:(0,x)\in \DD\},\q\text{and}\,\,
u=u_0,\uprho=\uprho_0~~~ \text{on}\,\{t=0\}\times\DD_0,\label{IC}\\
D_t|_{\p\DD}\in T(\p\DD), \q  {u_3|_{\Sigma_b}=0},\q
p|_{\Sigma_t}=\sigma\h,\label{BC}
\end{align}
where $D_t:=\p_t +u\cdot \nab$ is the material derivative, and $T (\p\DD)$ stands for the tangent bundle of $\p\DD$. The first condition in \eqref{BC} is the kinematic boundary condition, which indicates that the free surface boundary moves with the normal component of the velocity (see \eqref{kinematic BC} for an explicit illustration). The second condition is the slip condition imposed on the flat bottom $\Sigma_b$. The last condition in \eqref{BC} shows that the pressure is balanced by surface tension on the moving surface $\Sigma_t$. Here, $\sigma>0$ is called the surface tension constant, and $\h$ denotes the mean curvature of the free boundary of the fluid domain. Note that $\h,~T(\p \DD)$ and $p$ are functions of the unknowns $u,~\uprho$ and $\DD$. These quantities are not known a priori and must therefore be determined alongside a solution to the problem. 
The equations modeling the motion of compressible gravity-capillary water waves read
\begin{equation}\label{CWWST Eulerian}
\begin{cases}
\uprho D_t u =-\nab p-\uprho ge_3,&~~~ \text{in}\,\DD,\\
\p_t\uprho + \nab\cdot(\uprho u)=0, &~~~ \text{in}\,\DD,\\
p=p(\uprho), &~~~ \text{in}\, \DD, \\
(u,\uprho,\DD)|_{t=0}=(u_0,\uprho_0,\DD_0),
\end{cases}
\end{equation}
equipped with the boundary conditions
\begin{equation}\label{CWWST Eulerian BC}
\begin{cases}
p=\sigma\h &~~~\text{on}~\bigcup_{0\leq t\leq T}\{t\}\times\Sigma_t, \\
{u_3}=0&~~~\text{on}~[0,T]\times\Sigma_b,\\
D_t|_{\p\DD}\in T(\p\DD).
\end{cases}
\end{equation}
System \eqref{CWWST Eulerian} together with \eqref{CWWST Eulerian BC} admits a conserved quantity
\[
\mathsf{E}_0(t):=\frac12\idt\uprho|u|^2\dx+\idt\uprho\mathsf{Q}(\uprho)\dx+\idt (\uprho-1)gx_3\dx+\int_{\Sigma_t}g|\psi|^2+\sigma\left(\sqrt{1+|\cnab\psi|^2}-1\right)\dx',
\]where $\mathsf{Q}(\uprho):=\int_1^{\uprho}p(r)r^{-2}\,\mathrm{d}r$ and $\dx':=\dx_1\dx_2$. A direct calculation (cf. \cite[Section 6.1]{WangXin2015good}) shows $\mathsf{E}_0'(t)=0$. Note that we need a \textit{localized} initial data such that $\mathsf{E}_0(0)<+\infty$, which can be achieved similarly as in \cite[Section 7]{Luo2018CWW}.

\subsection{Fixing the fluid domain}
We shall convert \eqref{CWWST Eulerian}-\eqref{CWWST Eulerian BC} into a system of equations defined on the fixed domain $$\Omega = \{(x_1,x_2,x_3) :  -b<x_3< 0\}.$$ One way to achieve this would be to consider the Lagrangian coordinates. Nevertheless, here, we consider a family of diffeomorphisms $\Phi(t, \cdot): \Omega\to \DD_t$ characterized by the moving surface boundary. In particular, let 
\begin{equation}
\Phi(t,x_1,x_2, x_3) = \left(x_1, x_2, \varphi(t,x_1,x_2, x_3)\right), \label{change of variable}
\end{equation} 
where
\begin{align}
 \varphi (t,x_1,x_2, x_3) = x_3+\chi(x_3)\psi (t,x_1,x_2).
\end{align}
Here, $\chi\in C_c^\infty (-b, 0]$ is a smooth cut-off function satisfying the following bound for some small constant $\delta_0>0$:
\begin{equation}
{\|\chi'\|_{L^\infty(-b, 0]}\leq \frac{1}{\|\psi_0\|_{\infty}+1},\quad \sum_{j=1}^5 \|\chi^{(j)}\|_{L^\infty(-b,0]}\leq C, }\quad\chi=1\,\,\,\, \text{on}\,\, (-\delta_0, 0], \label{chi}
\end{equation}
for some generic constant $C>0$.

We will write $x'=(x_1,x_2)$ throughout the rest of this paper. 
It can be seen that 
$$\p_3 \varphi(t, x', x_3) = 1+\chi' (x_3)\psi(t,x')>0,\quad t\in[0,T],
$$
 for some small $T>0$, which ensures that $\Phi(t)$ is a diffeomorphism.  

Let $x=(x',x_3) \in \Omega$. We denote respectively by
\begin{equation}
v(t, x) = u(t, \Phi(t,x)), \quad \rho(t,x) = \uprho(t, \Phi(t,x)),\quad q(t,x) = p(t, \Phi(t,x))
\end{equation}
the velocity, density, and pressure defined on the fixed domain $\Omega$. Also, we introduce the differential operators
\begin{align}
\p_t^\varphi = &~\p_t - \frac{\p_t \varphi}{\p_3 \varphi}\p_3,\label{pt}\\
\nab^\varphi_a= &~\pp_a = \p_a -\frac{\p_a \varphi}{\p_3\varphi}\p_3,\quad a=1,2,\label{nabp 1,2}\\
\nab_3^\varphi = &~\p_3^\varphi = \frac{1}{\p_3 \varphi} \p_3,\label{nabp 3}
\end{align}
and thus the following identities hold:
\begin{align}
\p_\alpha u \circ \Phi = \pp_\alpha v,\quad 
\p_\alpha \uprho \circ \Phi = \pp_\alpha \rho,\quad 
\p_\alpha p \circ\Phi = \pp_\alpha q,\quad \alpha=t,1,2,3. 
\end{align}
Moreover, setting 
\[
\cnab=\TP := (\p_1, \p_2), 
\]
 the boundary condition \eqref{CWWST Eulerian BC} is turned into
\begin{align}
q =-\sigma \cnab \cdot \left( \frac{\cnab \psi}{\sqrt{1+|\cnab\psi|^2}}\right), &\quad \text{on}\,\,[0,T]\times \Sigma,\\
\p_t \psi =v\cdot N, \quad N = (-\p_1\psi, -\p_2\psi, 1)^\top,&\quad  \text{on}\,\,[0,T]\times \Sigma, \label{kinematic BC}\\
v_3=0,&\quad  \text{on}\,\,[0,T]\times \Sigma_b, \label{slip BC}
\end{align} 
respectively, where $\Sigma =\{x_3=0\}$ and $\Sigma_b=\{x_3=-b\}$. Let $D_t^\varphi:=\pp_t + v\cdot \nab^\varphi$. Then the system \eqref{CWWST Eulerian} and \eqref{CWWST Eulerian BC} are converted into
\begin{equation}\label{CWWST0}
\begin{cases}
\rho \Dtp v +\nabp q=-\rho ge_3&~~~ \text{in}~[0,T]\times \Omega,\\
\pp_t\rho+\nabp\cdot (\rho v)=0 &~~~ \text{in}~[0,T]\times \Omega,\\
q=q(\rho) &~~~ \text{in}~[0,T]\times \Omega, \\
q=-\sigma\cnab \cdot \left( \frac{\cnab \psi}{\sqrt{1+|\cnab\psi|^2}}\right) &~~~\text{on}~[0,T]\times\Sigma, \\
\p_t \psi = v\cdot N &~~~\text{on}~[0,T]\times\Sigma,\\
v_3=0&~~~\text{on}~[0,T]\times\Sigma_b,\\
{(v,\rho,\psi)|_{t=0}=(v_0, \rho_0, \psi_0)}.
\end{cases}
\end{equation}
The second equation of \eqref{CWWST0}, i.e., the continuity equation, can be re-expressed as 
\begin{equation}\label{continuity eq}
\Dtp \rho +\rho \nabp\cdot v=0.
\end{equation}
 Let $\ff=\ff(q):= \log \rho(q)$. Since $q'(\rho)>0$ indicates $\ff'(q)>0$,  then \eqref{continuity eq} is equivalent to 
\begin{equation} \label{continuity eq f}
\ff'(q) \Dtp q  +\nabp\cdot v=0. 
\end{equation}

Also, by invoking \eqref{pt}-\eqref{nabp 3}, we can alternatively write the material derivative $\Dtp$ as
\begin{equation}
\Dtp = \p_t + \vb\cdot\cnab+\frac{1}{\p_3 \varphi} (v\cdot \NN-\p_t\varphi)\p_3, \label{Dt alternate}
\end{equation}
where 
$\vb \cdot \cnab = v_1\p_1+v_2\p_2$, and $\NN:= (-\p_1\varphi, -\p_2 \varphi, 1)$. 
This formulation provides a good motivation to define the smoothed material derivative in Section \ref{sect nonlinear approx problem} and the linearized material derivative in Section \ref{sect linear LWP}. 

\subsection{The new formulation with modified pressure}
Since the gravity term $\rho ge_3\notin L^2(\Omega)$, we then use $\pp_i\varphi=\delta_{i3}$ to rewrite the momentum equation as
\[
\rho\Dtp v+\nabp \qc=-(\rho-1)ge_3,
\]
where 
\begin{equation}
\qc:=q+g\varphi, \label{modified q}
\end{equation}
 is the ``modified" pressure balanced by gravity. Under this setting, the fluid pressure gradient $\nabp\qc$ becomes an $L^2(\Omega)$ function and the source term becomes $(\rho-1)ge_3$ which is also in $L^2(\Omega)$ if we assume the initial data $\rho_0-1\in L^2(\Omega)$. We then directly calculate that $\Dtp \varphi=v_3$, so the continuity equation \eqref{continuity eq f} now becomes
\begin{equation} \label{continuity eq qc}
\ff'(q) \Dtp \qc  +\nabp\cdot v=\ff'(q)g\Dtp\varphi=\ff'(q)gv_3 , 
\end{equation}and thus the compressible gravity-capillary water wave system is now reformulated as follows:
\begin{equation}\label{CWWST}
\begin{cases}
\rho \Dtp v +\nabp \qc=-(\rho-1)ge_3&~~~ \text{in}~[0,T]\times \Omega,\\
\ff'(q)\Dtp \qc+\nabp\cdot v=\ff'(q) gv_3&~~~ \text{in}~[0,T]\times \Omega,\\
q=q(\rho),  \qc=q+g\varphi &~~~ \text{in}~[0,T]\times \Omega, \\
\qc=g\psi-\sigma\cnab \cdot \left( \frac{\cnab \psi}{\sqrt{1+|\cnab\psi|^2}}\right) &~~~\text{on}~[0,T]\times\Sigma,\\
\p_t \psi = v\cdot N &~~~\text{on}~[0,T]\times\Sigma,\\
v_3=0&~~~\text{on}~[0,T]\times\Sigma_b,\\
(v,\rho,\psi)|_{t=0}=(v_0, \rho_0, \psi_0).
\end{cases}
\end{equation}

\subsection{The equation of states and sound speed} \label{sect mach number definition}
Part of this paper is devoted to studying the behavior of the solution of \eqref{CWWST} as either the sound speed goes to infinity or the surface tension $\sigma$ coefficient goes to $0$. The former is known to be the incompressible limit, and the latter is known to be the zero surface tension limit. Mathematically, it is convenient to view the sound speed $c_s:=\sqrt{q'(\rho)}$ as a family of parameters.  As in \cite{Disconzi2017limit, DL19limit, Ebin1982limit, LL2018priori, Luo2018CWW}, we consider a family $\{q_{\lam'} (\rho)\}$ parametrized by $\lam' \in (0, \infty)$, where 
\begin{align}
(\lam')^2: = q'_{\lam'} (\rho)|_{\rho=1}.
\end{align}
Here and in the sequel, we slightly abuse the terminology and refer to $\lam'$ as the sound speed. A typical choice of the equation of states $q_{\lam'}(\rho)$ would be the Tait type equation
\begin{align}
q_{\lam'}(\rho) = \gamma^{-1} (\lam')^2 (\rho^\gamma-1), \quad \gamma\geq 1.
\end{align}
When viewing the density as a function of the pressure, this indicates
\begin{align}
\rho_{\lam'} (q) =  \left(\frac{\gamma}{(\lam')^2} q+1\right)^{\frac{1}{\gamma}}, \quad \text{and}\,\,\log \left(\rho_{\lam'} (q)\right) = \gamma^{-1} \log\left(\frac{\gamma}{(\lam')^2} q+1 \right). 
\end{align}
Hence,  we can view $\ff(q)$ as a parametrized family $\{\ff_\lam (q)\}$ as well, where $\lam = \frac{1}{\lam'}$. Indeed, we have
\begin{equation}
\ff_{\lambda} (q) = \gamma^{-1} \log (\lam^2\gamma q +1).
\end{equation}
We again slightly abuse the terminology and call $\lam$ the Mach number\footnote{The Mach number is defined to be $M=u/c_s$. In the paper, the velocity is always of size $O(1)$ (in $L^2(\Omega)$) and thus $M=O(\lambda)$.}. Furthermore, there exists $C>0$ such that
\begin{align}
C^{-1} \lam^2 \leq \ff'_{\lam}(q) \leq C \lam^2. \label{ff' is lambda}
\end{align}
Also, we assume
\begin{align} \label{ff property}
|\ff_{\lam}^{(s)} (q) |\leq C, \quad |\ff_{\lam}^{(s)} (q) | \leq C |\ff'_{\lam} (q)|^s \leq C \ff'_{\lam} (q)
\end{align}
holds for $0\leq s\leq 4$. 

\begin{rmk}[\textbf{Issue with the infinite depth case}]
In this paper, we study the case of water waves with a fixed bottom $\Sigma_b$ at finite depth. On the other hand, it is also interesting to consider the case of infinite depth, that is, $\DD_t=\{(x', x_3): -\infty<x_3<\psi(t,x')\}$. 
Nevertheless, the equation of state should be modified such that the pressure also depends on the depth. For instance, one may assume that $p$ satisfies $\frac{\p p}{\p x_3}|_{\rho=1}=-g$ (cf. Jang-Tice-Wang \cite{JTW}). Otherwise, the Mach number $\lambda$ may also be $x_3$-dependent. We also note that this issue does not arise in the setting of incompressible gravity water waves with infinite depth, where $\qc$ is a Lagrange multiplier thus unrelated to the density. 
\end{rmk}


\subsection{An overview of previous results}\label{sect review}

The study of free-surface inviscid fluids has blossomed over the past two decades or so. Most previous studies have focused on incompressible fluid models, where the fluid velocity satisfies $\dive u = 0$ and thus the density $\rho$ is a constant. In this case, the fluid pressure $p$ is not determined by the equation of states but appears as a Lagrangian multiplier enforcing the divergence-free constraint. For the local well-posedness (LWP) for the free-boundary incompressible Euler equations, the first breakthrough came in Wu \cite{Wu1997LWP, Wu1999LWP} for the irrotational case\footnote{The vorticity $\curl u_0=\mathbf{0}$. This condition is preserved by the evolution.} and Christodoulou-Lindblad \cite{CL2000priori} and Lindblad \cite{Lindblad2002LWP, Lindblad2004LWP} for the case of nonzero vorticity. We also refer to \cite{Nalimov1974LWP,Yosihara1982LWP,Iguchi2001LWP,AM2005ww,Lannes2005LWP,MZ2009} for the irrotational flows and \cite{CS2007LWP,ZZ2008LWP,Lindblad2009priori,SZ2008geometry,SZ2008priori,SZ2011LWP,ABZ2014wwLWP,ABZ2014wwSTLWP,WZZZ2015LWP} for the case of nonzero vorticity. 
In addition to the LWP theory, incompressible and irrotational water waves have attracted significant attention due to their long-term existence. We refer to Wu \cite{Wu2009GWP, Wu2011GWP} for the first breakthrough and numerous related works \cite{GMS2012GWP, GMS2015GWP, Alazard2013GWP, Ionescu2015wwGWP, Deng2017wwSTGWP, HIT2016ww, HIT2017ww, IT2016ww2, IT2016ww3, WangXC2018GWP2D, ZhengF2019GWP} See also \cite{BMSW2017GWP} for the bounded domain case and \cite{IT2016ww1, Su2020GWP} for some special cases when the vorticity is nonzero.

It is well-known that one can reduce the incompressible Euler equations to a system of equations on the moving boundary when the velocity is irrotational. This method cannot be adapted to the study of compressible water waves. The development of free-boundary compressible Euler equations is much less, especially for the case of a liquid as opposed to a gas in a physical vacuum satisfying $\rho|_{\Sigma}=0$. For the gas model, we refer to \cite{Jang2009LWPgas,CLS2010priorigas,CS2012LWPgas,LXZ2014LWPgas,Jang2015LWPgas,IT2020LWPgas} and references therein. For the liquid model, most previous works focus on the case of a bounded domain, and we refer to Lindblad \cite{Lindblad2003LWP, Lindblad2005LWP} and related works \cite{CS2013LWP, LL2018priori, DL19limit, GLL2020LWP}. When the fluid domain is unbounded, as in the compressible gravity water waves problem, the existing literature often neglects the effect of surface tension. Trakhinin \cite{Trakhinin2009LWP} first proved the LWP for the non-isentropic case using Nash-Moser iteration, which results in a loss of regularity from the initial data to the solution. The a priori estimate without loss of regularity is shown in Luo \cite{Luo2018CWW}, but it is still difficult to use the energy constructed there to prove the local existence. Recently, in \cite{LuoZhang2020CWWLWP},  the authors proved the LWP for compressible gravity water waves without using Nash-Moser iteration. 

Regarding the incompressible limit of inviscid fluids, specifically the singular limit as the Mach number approaches $0$, numerous studies have been conducted on the Cauchy problem or fixed-domain problems. We refer to \cite{Klainerman1981limit, Klainerman1982limit, Ebin1982limit, Schochet1986limit, Disconzi2017limit} for ``well-prepared initial data" ($\dive u_0=O(\lam)$ and $\p_t u|_{t=0}=O(1)$, wehre $\lam$ is the Mach number) and \cite{Ukai1986limit, Asano1987limit, Isozaki1987limit, Iguchi1997limit, Metivier2001limit, Alazard2005limit} for ``ill-prepared (general) initial data"  ($\dive u_0=O(1)$ and $\p_t u|_{t=0}=O(\lam^{-1})$). However, much less is known about the incompressible limit of free-surface inviscid fluids: Lindblad and the first author \cite{LL2018priori}, the first author \cite{Luo2018CWW}, and Disconzi and the first author \cite{DL19limit} established incompressible limit results for free-surface Euler equations with zero or nonzero surface tension.

It should be noted that the uniform energy estimates are not consistent with the ones obtained by the local existence result. Moreover, \textit{the uniform boundedness (with respect to Mach number $\lam$) of top-order time derivatives of the velocity is necessary} in \cite{LL2018priori, Luo2018CWW, DL19limit}, which is more restrictive than the commonly-used definition of ``well-prepared initial data". Very recently, the second author \cite{Zhang2021elastoLWP} established LWP and the incompressible limit simultaneously with the same energy functional for compressible elastodynamics, which can be directly applied to the Euler equations without surface tension. Also, only $\p_t^2 u|_{t=0}=O(1)$ is required in \cite{Zhang2021elastoLWP} which is an essential improvement of \cite{LL2018priori, Luo2018CWW, DL19limit} and is also an optimal requirement of well-prepared data for free-surface inviscid fluids without surface tension, as the propagation of Rayleigh-Taylor sign condition already requires the uniform boundedness of $\nab\p_t q\sim \p_t^2 u$. However, the method and observations in \cite{Zhang2021elastoLWP} heavily rely on the vanishing boundary condition for the pressure on the free surface, which cannot be generalized to the case of nonzero surface tension or two-phase vortex-sheet problems.

\subsection{The goal of our paper}
This paper has two main objectives. Since the precise model of water waves studied in this manuscript (i.e., compressible, on an unbounded domain, with surface tension) has not been explicitly considered in the existing well-posedness literature, our first objective is to establish its well-posedness. Meanwhile, 
we seek to develop a \textit{unified framework} that simultaneously addresses the well-posedness for incompressible and compressible free-boundary Euler equations, in cases both with and without surface tension. We achieve this by developing a carefully designed approximate system that is compatible with both the incompressible and zero surface tension limits. In particular, this system is asymptotically consistent with the $\lambda$-uniform estimate and can be further adapted to the $\sigma$-uniform case, provided the Rayleigh--Taylor sign condition holds. 

Second, we aim to improve the incompressible limit by relaxing the assumptions on the initial data. In particular, we do not require boundedness of $\p_t^k \qc$ for $k \geq 2$. To establish a uniform-in-Mach-number energy estimate in this setting, we are motivated by Shatah-Zeng \cite{SZ2008priori, SZ2011LWP} to decompose the pressure into a ``harmonic part" and a ``wave part". We analyze the harmonic part (which is, in fact, contributed by the surface tension) by adapting the paradifferential approach developed by Alazard-Burq-Zuily \cite{ABZ2014wwSTLWP, ABZ2014wwLWP} on the moving boundary, while addressing the wave part in the interior. In other words, we have successfully applied a robust technique commonly used for incompressible water waves to adapt it for use with compressible water waves.


\subsection{The main theorems}
The first theorem concerns the local well-posedness of the motion of compressible gravity-capillary water waves, modeled by \eqref{CWWST}, provided that the initial data satisfy certain compatibility conditions. Particularly, we say the data $(\psi_0, v_0, q_0)$, where $q_0 =q (\rho_0)$, satisfies the $k$-th ($k=0,1,2,3,\cdots$) compatibility conditions if 
\begin{equation} \label{comp cond intro}
\begin{aligned}
&(\Dtp)^k q |_{t=0} =  (\Dtp)^k (\sigma \mathcal{H})|_{t=0},\quad &\text{on}\,\,\Sigma, \\
&\p_t^k v_3|_{t=0} = 0, \quad &\text{on}\,\,\Sigma_b, 
\end{aligned}
\end{equation}
hold. 

\begin{thm}[\textbf{Local well-posedness}]\label{main thm, WP}
Let $b>10$, and $\sigma >0$ be fixed. We define 
\begin{equation} \label{energy main them, WP}
\begin{aligned}
E(t):=&~E_0(t)+E_4(t),\quad \text{where}\\
E_0(t):=&~\|\rho(t)-1\|_0^2+g|\psi|_0^2+\sum_{k=0}^3\|\sqrt{\ff'(q)}\p_t^k\qc(t)\|_{0}^2,\quad \text{and}\\
E_4(t):=&\sum_{k=0}^4\left(\|\p_t^k v(t)\|_{4-k}^2+|\sqrt{\sigma}\cnab\p_t^k\psi(t)|_{4-k}^2\right)+\sum_{k=0}^3\|\p_t^k\p\qc(t)\|_{3-k}^2+\|\sqrt{\ff'(q)}\p_t^4\qc(t)\|_0^2,\\
\end{aligned}
\end{equation}
to be the energy of \eqref{CWWST} expressed in terms of $(\psi, v, \qc)$. 
Let $(\psi_0, v_0, \rho_0-1) \in H^5(\Sigma)\times H^4(\Omega)\times H^4(\Omega)$ be the initial data of \eqref{CWWST} that verifies the compatibility conditions \eqref{comp cond intro} up to the third order,  such that $|\psi_0|_{\infty}\leq 1$ and $E(0)\leq \mathcal{C}$, where $\mathcal{C}$ depends on $\psi_0$, $v_0$ and $\qc_0$. Then, there exists $T>0$, depending only on the initial data, such that \eqref{CWWST} admits a unique solution, expressed in terms of  $(\psi(t), v(t), \qc(t))$, satisfying
\begin{subequations}\label{regularity theorem 1.1 }
\begin{align}
&\p_t^k v(t,\cdot)\in H^{4-k}(\Omega), \quad k=0,\cdots, 4,\\
&\p_t^k \p \qc(t,\cdot)\in H^{3-k}(\Omega),\quad k=0, \cdots ,3, \\
&\lam\p_t^k\qc(t,\cdot)  \in L^2(\Omega), \quad k=0,\cdots, 4,\\
&\psi(t,\cdot)\in L^2(\Sigma), \quad \sqrt{\sigma}\p_t^k \cnab\psi(t,\cdot)\in H^{4-k}(\Sigma), \quad k=0,\cdots, 4. 
\end{align}
\end{subequations}
as well as the energy estimate:
\begin{equation} \label{energy esti intro}
\sup_{0\leq t\leq T}E(t)\leq C(\sigma^{-1}) P(E(0)),
\end{equation}
where $P(\cdot)$ is a generic non-negative continuous function (independent of $\lambda$ and $\sigma$) in its arguments.
In addition to this, we have 
\begin{equation}
\sup_{t\in [0,T]} |\psi (t)|_{\infty} \leq 10.
\end{equation}
Moreover, let $(\bar{\psi}, \bar{v}, \bar{\qc})$ denote the solution of \eqref{CWWST} corresponding to another set of initial data $(\bar{\psi}_0, \bar{v}_0, \bar{\rho}_0-1)\in H^5(\Sigma)\times H^4(\Omega)\times H^4(\Omega)$ that verifies the compatibility conditions \eqref{comp cond intro} up to the third order,  and
$
[E](0) \leq \epsilon, 
$
where 
\begin{align}\label{def [E] intro}
[E](t):=\sum_{k=0}^3&\left(\|\p_t^k(v-\bar{v})(t,\cdot)\|_{3-k}^2+\sigma|\cnab\p_t^{k}(\psi-\bar{\psi})(t,\cdot)|_{3-k}^2+\|\sqrt{\ff'(q)}\p_t^k(\qc-\bar{\qc})(t,\cdot)\|_0^2\right)\notag\\
&+g|(\psi-\bar{\psi})(t,\cdot)|_0^2+\sum_{k=0}^2\|\p \p_t^k(\qc-\bar{\qc})(t,\cdot)\|_{2-k}^2.
\end{align}
Then, there exists some positive constant $\mathcal{C}(\sigma^{-1})$ such that
\begin{equation}
\sup_{0\leq t\leq T}E(t)\leq  \mathcal{C}(\sigma^{-1})\epsilon,
\end{equation}
possibly after choosing $T>0$ smaller. 
\end{thm}

In above and throughout, we use $\|\cdot\|_s$ and $|\cdot|_s$ to represent respectively the interior Sobolev norm $\|\cdot\|_{H^s(\Omega)}$ and the boundary Sobolev norm $\|\cdot\|_{H^s(\Sigma)}$. 
\begin{rmk}
	In Appendix \ref{sect CWWST data}, we show that we can construct smooth initial data $(\psi_0, v_0, \qc_0)$ that satisfy the compatibility conditions up to order 3. These compatibility conditions are required so that we can show $E(0)\leq \mathcal{C}$ by adapting the arguments in \cite[Section 4.3]{DL19limit}.
\end{rmk}
\begin{rmk}
The second line in \eqref{energy main them, WP} is the $L^2$-part of the energy, where $\|\qc(t)\|_0^2$ is $\sqrt{\ff'(q)}$-weighted, which ties to the Mach number. That is why we write $\|\p\qc\|_3$ instead of $\|\qc\|_4$ in the first line. Also, in \eqref{regularity theorem 1.1 }, $\lam\qc(t,\cdot)$ and $\lam\p_t^4\qc(t,\cdot)$ are $\lam$-weighted, since the corresponding quantities in the energy \eqref{energy main them, WP} are weighted by $\sqrt{\ff'(q)}$, where $\sqrt{\ff'(q)}$ is comparable to $\lam$ thanks to \eqref{ff' is lambda}. 
\end{rmk}

Moreover, Theorem \ref{main thm, WP} can be readily generalized to the case of smoother solutions, provided that the initial data is with sufficient regularity assumptions.
\begin{cor}[Persistence of regularity]\label{cor: persistence regularity}
 Let $s\geq 5$ be a fixed integer. 
Let  
 \begin{equation} \label{energy main them, WP'}
\begin{aligned}
E_s(t):=\|\rho(t)-1\|_0^2+g|\psi|_0^2+\sum_{k=0}^s\left(\|\p_t^k v(t)\|_{s-k}^2+\|\sqrt{\ff'(q)}\p_t^k\qc(t)\|_0^2+|\sqrt{\sigma}\cnab\p_t^k\psi(t)|_{s-k}^2\right)+\sum_{k=0}^{s-1}\|\p_t^k\p\qc(t)\|_{s-k-1}^2,\\
\end{aligned}
\end{equation}
Let $(\psi_0, v_0, \rho_0-1) \in H^{s+1}(\Sigma)\times H^s(\Omega)\times H^s(\Omega)$ be the initial data of \eqref{CWWST} that verifies the compatibility conditions \eqref{comp cond intro} up to the $(s-1)$-th order,  such that $|\psi_0|_{\infty}\leq 1$ and $E(0)\leq \mathcal{C}_s(\psi_0, v_0, \qc_0)$. Then, there exists $T_s>0$, depending only on the initial data and $s$, such that \eqref{CWWST} admits a unique solution, expressed in terms of  $(\psi(t), v(t), \qc(t))$, satisfying
\begin{subequations}\label{regularity theorem 1.1' }
\begin{align}
&\p_t^k v(t,\cdot)\in H^{s-k}(\Omega), \quad k=0,\cdots, s,\\
&\p_t^k \p \qc(t,\cdot)\in H^{s-1-k}(\Omega),\quad k=0, \cdots ,s-1, \\
&\lam\p_t^k\qc(t,\cdot)  \in L^2(\Omega) \quad k=0,\cdots, s,\\
&\psi(t,\cdot)\in L^2(\Sigma), \quad \sqrt{\sigma}\p_t^k \cnab\psi(t,\cdot)\in H^{s-k}(\Sigma), \quad k=0,\cdots, s. 
\end{align}
\end{subequations}
as well as the energy estimate:
\begin{equation} \label{energy esti intro'}
\sup_{0\leq t\leq T_s}E_s(t)\leq C(\sigma^{-1}) P(E_s(0)),
\end{equation}
\end{cor}

The next main theorem concerns the incompressible and zero surface tension limits. We consider the Euler equations modeling the motion of incompressible gravity water waves satisfied by $(\xi, w, q_{in})$ with localized initial data $(w_0,\xi_0)$:
\begin{equation} \label{WW intro}
\begin{cases}
\Dtp w +\nabp \qi=0&~~~ \text{in}~[0,T]\times \Omega,\\
\nabp\cdot w=0&~~~ \text{in}~[0,T]\times \Omega,\\
\qi=q_{in}+g\varphi &~~~ \text{in}~[0,T]\times \Omega, \\
\qi=g\xi &~~~\text{on}~[0,T]\times\Sigma,\\
\p_t \xi = w\cdot \nn &~~~\text{on}~[0,T]\times\Sigma,\\
w_3=0&~~~\text{on}~[0,T]\times\Sigma_b,\\
(w,\xi)|_{t=0}=(w_0, \xi_0), 
\end{cases}
\end{equation}
where we slightly abuse the notation by still setting $\varphi(t,x) = x_3+\chi(x_3) \xi(t,x')$ to be the extension of $\xi$ in $\Omega$. Denote by $(\psi^\ls, v^\ls, \rho^\ls)$ the solution of \eqref{CWWST} indexed by $\sigma$ and $\lam$, we prove that $(\psi^\ls, v^\ls, \rho^\ls)$ converges to $(\xi, w,1)$ as $\lam, \sigma \rightarrow 0$ provided the convergence of the initial data in a suitable sense. Note that the convergence of the compressible initial data implies that it is also localized.

\begin{thm}[\textbf{Incompressible and zero surface tension limits}] \label{main thm, double limits}
Let $(\psi_0^\ls, v_0^\ls, \rho_0^\ls-1)$ be the initial data of \eqref{CWWST} for each fixed $(\lam, \sigma)\in \R^+\times \R^+$, verifying:
\begin{enumerate}
\item [a.] The sequence of initial data $(\psi_0^\ls, v_0^\ls, \rho_0^\ls-1) \in H^5(\Sigma)\times H^4(\Omega)\times H^4(\Omega)$ satisfies \eqref{comp cond intro} for $0\leq k\leq 3$, and {$|\psi_0^\ls|_{\infty} \leq 1$}. 
\item [b.] $(\psi_0^\ls, v_0^\ls, \rho_0^\ls-1) \to (\xi_0, w_0, 0)$ in $\orange{H^4(\Sigma)}\times H^4(\Omega)\times H^3(\Omega)$ as $\lam, \sigma\to 0$. 
\item [c.] \textit{Both} incompressible and compressible pressures $q^{\ls}$ and $q_{in}$ satisfy the Rayleigh-Taylor sign condition
\begin{align}
-\p_3 q^{\ls} \geq c_0>0, \quad \text{on}\,\,\{t=0\}\times \Sigma, \label{RT intro comp}\\
-\p_3 q_{in} \geq c_0>0, \quad \text{on}\,\,\{t=0\}\times \Sigma, \label{RT intro incomps}
\end{align}
for some $c_0>0$. 
\end{enumerate} 
Then it holds that 
\begin{align*}
(\psi^\ls, v^\ls, \rho^{\ls}-1)\rightarrow (\xi, w, 0),
\end{align*}
weakly* in $L^\infty([0,T]; \orange{H^4(\Sigma)}\times H^4(\Omega)\times H^3(\Omega))$, 
and strongly in $C^0([0,T]; \orange{H_{\text{loc}}^{4-\delta}(\Sigma)}\times H_{\text{loc}}^{4-\delta}(\Omega)\times H_{\text{loc}}^{3-\delta}(\Omega))$
for any $\delta\in (0,1]$.
\end{thm}

Theorem \ref{main thm, double limits} is a direct consequence of uniform-in-$\ls$ estimates for the compressible gravity-capillary water wave system \eqref{CWWST} and the Aubin-Lions lemma. Indeed, the energy estimate \eqref{energy esti intro} established in Theorem \ref{main thm, WP} is already uniform in Mach number $\lambda$. In addition to this, one can show that \eqref{energy esti intro} is uniform in the surface tension coefficient $\sigma$ provided that the Rayleigh-Taylor sign condition \eqref{RT intro comp} holds initially. 

\begin{rmk}
Although our energy functional $E(t)$ is expressed in terms of $\qc$, the incompressible limit is given in $(\psi^{\ls},v^{\ls},\rho^{\ls})$ which converges to $(\zeta,w,1)$. We do not expect the compressible pressure $q$ to converge to the incompressible pressure $q_{in}$ as $\lam\rightarrow 0$, because the former is the solution to a quasilinear symmetric hyperbolic system, whereas the latter appears as a Lagrangian multiplier. Indeed, as was indicated by \cite{LL2018priori, Luo2018CWW,Zhang2021elastoLWP}, it is the enthalpy $h(\rho):=\int_1^{\rho}q'(r)/r~dr$ of the compressible equations that converges to the incompressible pressure $q_{in}$. On the other hand, the convergence of $\|\rho^{\ls}-1\|_3$ can be easily proved if we write the continuity equation to be $\Dtp(\rho-1)=-\rho(\nabp\cdot v)$ and use Gr\"onwall's inequality for its $H^3$-estimate.
\end{rmk}

It should be noted that the energy \eqref{energy main them, WP} requires that the time derivatives up to at least order $3$ are bounded initially, i.e., $\p_t^k\qc(0)=O(1)$, $0\leq k\leq 3$, while $\p_t^4\qc(0)=O(\lambda^{-1})$, or equivalently the uniform boundedness for the top-order time derivatives of the velocity $\p_t^4 v=O(1)$. \textit{This condition can certainly be weakened}. In fact, the propagation of the Rayleigh-Taylor sign condition only requires the boundedness of $\p_t \p_3 q$, or equivalently $\p_t^2 v=O(1)$, not including higher-order time derivatives. Motivated by this, we prove the following improved estimates.

\begin{thm}[\textbf{Improved uniform estimates in $\ls$}]\label{main thm, well data}
Under the hypothesis of Theorem \ref{main thm, WP}, if we further assume  $(\psi_0, v_0, \rho_0-1) \in H^6(\Sigma)\times H^5(\Omega)\times H^5(\Omega)$ satisfying the compatibility conditions up to the fourth order and assume the Rayleigh-Taylor sign condition \eqref{RT intro comp} holds for the initial data of \eqref{CWWST}, then
\begin{align}
\sup_{0\leq t\leq T}\EE(t) \leq P(\EE(0)), 
\end{align}
holds uniform in both $\lambda$ and $\sigma$,
where
\begin{equation}
\begin{aligned}\label{energyls intro}
\EE(t):=&~E_0(t)+\EE_4(t)+E_5(t),\\
\EE_4(t):=&~\|v\|_4^2+\|\p\qc\|_3^2+|\sqrt{\sigma}\psi|_5^2+|\psi|_{4}^2+\|\p_t v,\p_t\qc\|_3^2+|\sqrt{\sigma}\p_t\psi|_4^2+|\p_t\psi|_{3.5}^2 \\
&+ \|\p_t^2 v,\lam\p_t^2\qc\|_2^2+|\sqrt{\sigma}\p_t^2\psi|_3^2+|\p_t^2\psi|_{2.5}^2+|\p_t^3\psi|_{1.5}^2  \\
&+ \sum_{k=3}^4\|\lam \p_t^k(v,\qc)\|_{4-k}^2+|\sqrt{\sigma}\lam\p_t^k\psi|_{5-k}^2+|\lam \p_t^4\psi|_{0.5}^2 \\
E_5(t):=&\sum_{k=0}^5\ino{\lam^2\p_t^k(v, (\ff'(q))^{\frac{(k-4)_+}{2}}\qc)}_{5-k}^2+\bno{\sqrt{\sigma}\lam^2\p_t^k\psi}_{6-k}^2+\bno{\lam^2\p_t^k\psi}_{5-k}^2,
\end{aligned}
\end{equation}and $(k-4)_+:=\max\{0,k-4\}$.
\end{thm}

\begin{rmk}
We require the solutions to possess greater smoothness than those constructed in Theorem \ref{main thm, WP}. This enhanced regularity is obtained via Corollary \ref{cor: persistence regularity}. 
\end{rmk}

\begin{rmk}
The above estimate only requires $\nab\p_t q(0)\sim \p_t^2 u(0)$ to be bounded (with respect to $\lam$) because we need to control the evolution of the Rayleigh-Taylor sign, namely $\|\p_t\p_3 q\|_{L^\infty(\Sigma)}$, when taking the \textit{incompressible and zero surface tension limits simultaneously}. However, we do not require $\p_t^k v(0)$ to be uniformly bounded for $k>2$. On the other hand, the propagation of the Rayleigh-Taylor sign condition requires the boundedness of $\p_t q$, so we have reached the minimal requirement for the initial data being ``well-prepared". 
\end{rmk}

\subsection*{List of Notations}
\begin{itemize}
\item (Fixed domain and its boundary) $\Omega := \{x\in \R^3: -b<x_3< 0\}$. $x=(x_1,x_2,x_3)$, and $x'=(x_1,x_2)$. $\Sigma :=\{x\in\R^3:x_3=0\},~ \Sigma_b :=\{x\in\R^3:x_3=-b\}$.
\item (Tangential derivatives) $\TT_0= \p_t$, $\TT_1=\TP_1$, $\TT_2 = \TP_2$, $\TT_3=\omega(x_3)\p_3$, where $\omega(x_3)\in C^{\infty}(-b,0)$ is assumed to be bounded, comparable to $|x_3|$ in $[-2,0]$ and vanishing on $\Sigma\cup\Sigma_b$.
\item (Norms) We define $\|\cdot\|_{\infty}:= \|\cdot\|_{L^\infty(\Omega)}$, $|\cdot|_{\infty}:=\|\cdot\|_{L^\infty(\Sigma)}$, $\|\cdot \|_s:= \|\cdot \|_{H^s(\Omega)}$, and $|\cdot|_s:=\|\cdot \|_{H^s(\Sigma)}$. In particular, the identification $\Sigma \cong \mathbb{R}^2$ is understood when applying the norms $|\cdot|_{\infty}$ and $|\cdot|_s$ to functions defined only on $\mathbb{R}^2$. 
\item (Generic continuous functions) We denote by $P(\cdot)$ a generic non-negative continuous function in its arguments, which is independent of $\lambda$ and $\sigma$. Also, we set $\PP_0:=P(E(0)),~\PP^\kk_0:=P(E^{\kk}(0))$. 
\item (Commutators) $[T,f]g=T(fg)-f(Tg)$, $[T,f,g]:=T(fg)-T(f)g-fT(g)$ where $T$ is a differential operator and $f,g$ are functions.
\item (Equality modulo lower order terms) $A\lleq B$ means $A= B$ modulo lower order terms. 
\end{itemize}

\section{An overview of our methodology}\label{sect overview intro}

Before presenting the detailed proofs, we will briefly introduce our methodology for deriving energy estimates that are uniform in both surface tension and Mach number, as well as the construction of solutions to the linearized and nonlinear problems via a carefully designed approximation scheme. 

\subsection{Uniform estimates in Mach number and surface tension}\label{sect apriori intro}
Let us temporarily focus on the a priori energy estimate of the original system \eqref{CWWST} instead of the construction of solutions. Indeed, the strategies on the a priori estimate will illustrate why we need the approximation scheme defined in the next subsection.

\subsubsection{Div-Curl analysis and reduction of pressure}\label{sect divcurl intro}

The first step is to reduce the normal derivatives for \eqref{CWWST}, and we start with the control of $\|v\|_4$. Using the div-curl decomposition, $\|v\|_4$ is bounded by $\|\nabp\times v\|_3$, $\|\nabp\cdot v\|_3$, and $\|\TP^4 v\|_0$, where the curl part can be directly controlled by analyzing its evolution equation. The continuity equation reduces the divergence to $\|\ff'(q)\Dtp q\|_3$, which is a tangential derivative and includes a time derivative. As for the pressure $\qc$, the momentum equation indicates that $-\nabla \qc\sim \Dtp v$, which again converts a normal derivative to a tangential derivative. This reduction can also be applied to the time derivatives of $v$ and $\qc$ up to the third order. As a consequence, the control of the full Sobolev norms of $v$ and $\qc$ (and their time derivatives) is reduced to the control of $\TT^\alpha v$ and $\TT^\alpha \qc$ ($|\alpha|=4$) in $L^2(\Omega)$ with appropriate weights in Mach number where $\TT$ represents any of the tangential derivatives $\p_t,\TP$ or $\omega(x_3)\p_3$ where $\omega(x_3)\in C^{\infty}(-b,0)$ is bounded, comparable to $|x_3|$ in $x_3\in(-2,0)$ and vanishing on $\Sigma\cup\Sigma_b$.


\subsubsection{Tangential estimates: Alinhac good unknowns}\label{sect AGU intro}

We define $\TT^{\alpha}$ to be $\p_t^{\alpha_0}\TP_1^{\alpha_1}\TP_2^{\alpha_2}(\omega\p_3)^{\alpha_3}$ with $|\alpha|:=\alpha_0+\alpha_1+\alpha_2+\alpha_3=4$. In $\TT^\alpha$-tangential estimates, we need to commute $\TT^{\alpha}$ with $\nabp_i$. When $i=t,1,2$, the commutator $[\TT^\alpha,\nabp_i]f$ includes the term $(\p_3\varphi)^{-1}\TT^\alpha\p_i\varphi\p_3 f$, where the $L^2(\Omega)$-norm of $\TT^{\alpha}\p_i\varphi$ is controlled by $|\TT^{\alpha}\p_i\psi|_0$. However, the regularity of $\psi$ obtained in $\TT^{\alpha}$-estimates is $|\sqrt{\sigma}\TT^{\alpha}\cnab\psi|_0$. Thus, the direct control of the aforementioned commutator fails to be uniform in $\sigma$. To overcome this difficulty, we introduce the Alinhac's method which reveals that the ``essential" leading order term in $\TT^\alpha(\nabp f)$ is not $\nabp(\TT^{\alpha} f)$ but the covariant derivative of $\FF$ (i.e., $\nabp \FF$), where $\FF:=\TT^{\alpha}f-\TT^{\alpha}\varphi\pp_3 f$. Here, $\FF$ is the so-called Alinhac good unknown associated with $f$, which satisfies
\begin{align}
\label{AGUDD}  \TT^{\alpha}\nabp_i f=\nabp_i \FF+\cc_i(f),~~~\TT^{\alpha}\Dtp f=\Dtp \FF+\dd(f),
\end{align}where $\|\cc_i(f)\|_0$ and $\|\dd(f)\|_0$ can be directly controlled. In other words, the reformulation in Alinhac good unknowns takes into account the covariance under the change of coordinates, such that we can proceed with the tangential estimates in the same way as the $L^2$-estimate and avoid the additional regularity on the nonlinear coefficients that cannot be controlled in a $\sigma$-uniform fashion. Such a remarkable observation was due to Alinhac \cite{Alinhac1989good} and was first applied (implicitly) to free-surface inviscid fluids by Christodoulou-Lindblad \cite{CL2000priori}. See also \cite{MR2012good, WangXin2015good} for the explicit calculations for the inviscid limit of incompressible free-boundary Navier-Stokes equations.

Let $\VV,\QQ$ be the Alinhac good unknowns of $v,\qc$ associated with $\TT^\alpha$, and then we obtain several major terms from the tangential estimates
\begin{equation}\label{AGUEE}
\begin{aligned}
&\ddt\frac12\left(\io\rho|\VV|^2+\ff'(q)|\QQ|^2\dvt\right)\\
=&\ST+\RT+\is\TT^\alpha \qc [\TT^\alpha, v\cdot, N]\dx'-\io\TT^\alpha\qc [\TT^\alpha, \p_3 v\cdot, \NN]~\dvt+\text{ controllable terms, }
\end{aligned}
\end{equation}where $\dvt:=\p_3\varphi\dx$ and
\begin{align}
\ST:=-\is\TT^{\alpha}(\sigma\h)\p_t\TT^{\alpha}\psi\dx',\quad \RT:=-\is(-\p_3 q)\TT^\alpha\psi \p_t\TT^\alpha\psi\dx.
\end{align}Also note that $\TT^\alpha$ only contains $\p_t$ and $\TP$ on $\Sigma\cup\Sigma_b$ as the weight function $\omega(x_3)$ vanishes on the boundary.

For the term ST, invoking the explicit formula for the mean curvature and integrating $\cnab\cdot$ by parts, we obtain 
\begin{equation}\label{AGUBDRYST}
\begin{aligned}
\ST=-\frac{\sigma}{2}\ddt\is\frac{|\TT^{\alpha}\cnab\psi|^2}{\sqrt{1+|\cnab\psi|^2}}-\frac{|\cnab\psi\cdot\TT^{\alpha}\cnab\psi|^2}{{\sqrt{1+|\cnab\psi|^2}}^3}\dx'+\cdots,
\end{aligned}
\end{equation}which together with the following inequality gives the boundary energy $|\sqrt{\sigma}\TP^{\alpha}\cnab\psi|_0^2$:
\begin{equation}\label{Cauchy}
\forall\mathbf{a}\in\R^2,~~\frac{|\mathbf{a}|^2}{\sqrt{1+|\cnab\psi|^2}}-\frac{|\cnab\psi\cdot\mathbf{a}|^2}{{\sqrt{1+|\cnab\psi|^2}}^3}\geq \frac{|\mathbf{a}|^2}{{\sqrt{1+|\cnab\psi|^2}}^3}.
\end{equation}

For the term RT, it produces the boundary energy without $\sigma$-weight \textit{provided that the Rayleigh-Taylor sign condition\footnote{The Rayleigh-Taylor sign condition is just a constraint for the initial data. One can easily prove its short-time propagation by using the boundedness of $\p_t\p_3 q$. See \cite[Section 3.7]{LuoZhang2020CWWLWP}.} $-\p_3 q_0|_{\Sigma}\geq c_0>0$ holds}. However, the Rayleigh-Taylor sign condition is only assumed when taking the zero surface tension limit, but not in the proof of local well-posedness for each \textit{given} $\sigma>0$. Therefore, we have to use the $\sqrt{\sigma}$-weighted energy to control this term when proving local well-posedness. Indeed, \textit{it is the direct control of $|\TT^\alpha \psi|_0$ and $|\p_t\TT^{\alpha}\psi|_0$ that yields the only possibility that the energy estimate depends on $\sigma^{-1}$}. 

The remaining two terms contribute to a crucial structure for the incompressible limit. When $\TT^\alpha=\p_t^4$ is a full time derivative, we cannot control them individually due to a loss of Mach number weight. Instead, we shall combine them and use the divergence theorem to reduce a time derivative on $\qc$. The leading-order terms are
\begin{align*}
4\is\p_t^4\qc\p_t^3v\cdot\p_t N\dx'-4\io\p_t^4\qc\p_t\NN\cdot\p_3\p_t^3 v\dx=\ddt\io\left(\p_t^3\p_3\qc\p_t\NN+\p_t^3\qc\p_t\p_3\NN\right)\cdot\p_t^3 v\dx+\cdots,
\end{align*}which can be directly controlled under the time integral.

Combining the steps above, we finish the control of Alinhac good unknowns $\VV,\QQ$. Then by using the definition of good unknowns, we know $\|\FF-\TT^{\alpha}f\|_0\leq|\TT^{\alpha}\psi|_0\|\p f\|_{\infty}$ which is already controlled by the boundary energy of $\psi$. Therefore, the a priori estimate for the system \eqref{CWWST} is closed, which is uniform in Mach number and also uniform in $\sigma$ under the Rayleigh-Taylor sign condition.

\subsection{Improved incompressible limit}\label{sect limit intro}
The uniform estimates obtained above require the uniform boundedness of top-order time derivatives of $v$, which is far more restrictive than the usual definition of ``well-prepared initial data" ($\nabp\cdot v|_{t=0}=O(\lam),~\p_t v|_{t=0}=O(1)$). A natural question is whether we can remove such boundedness assumption on high-order time derivatives, which is a necessary step to find a possible way to study the case of ``ill-prepared data" ($\nabp\cdot v|_{t=0}=O(1),~\p_t v|_{t=0}=O(\lam^{-1})$). 

\subsubsection{Difficulties in free-boundary problems}
There have been numerous results for fixed-domain problems or the Cauchy problem \cite{Klainerman1981limit, Klainerman1982limit, Ebin1982limit, Schochet1986limit}, but this is rather nontrivial under the free-surface setting due to the interaction between the free-surface motion and the interior pressure waves. Indeed, when commuting $\TT^\alpha$ with $\nabp$ when $\TT^\alpha$ contains both spatial derivatives and time derivatives, the usage of $\nab q\sim\p_t v$ actually produces an extra time derivative without $\lam$-weight. When $\p_t^k v$ is assigned with a different $\lam$-weight from that of $\p_t^k\qc$ in the energy functional, there exhibits a loss of $\lam$-weight due to the substitution $\nab \qc\sim\p_t v$, which is actually \textit{caused by the free-surface motion}. The second author \cite{Zhang2021elastoLWP} dropped such assumption for the case of zero surface tension, but this result heavily relies on the vanishing boundary value of $q$ as stated at the end of Section \ref{sect review}.

The above analysis indicates that we should avoid the interior tangential estimates. Instead, when treating the time derivatives, we shall use another div-curl inequality
\begin{equation}\label{divcurlNN}
\|X\|_s^2\lesssim C(|\psi|_{s+\frac12},|\cnab\psi|_{W^{1,\infty}})\left(\|X\|_0^2+\|\nabp\cdot X\|_{s-1}^2+\|\nabp\times X\|_{s-1}^2+|X\cdot N|_{s-\frac12}^2\right),\quad\forall s\geq 1,
\end{equation}in order to directly analyze the evolution of the free surface. In view of the new energy $\EE_4(t)$ defined in \eqref{energyls intro}, we shall apply this inequality to $X=\p_t^2 v$ and the kinematic boundary condition indicates us to control $|\p_t^3\psi|_{1.5}$ without any weights of $\lam,\sigma$.

\subsubsection{The evolution equation of the free surface and its paralinearization}
The evolution equation of the free surface is derived by time-differentiating the kinematic boundary condition and invoking the momentum equation, which leads to $\rho\TDt^2\psi=-\p_3 \qc-(\rho-1)g$ with $\TDt:=\Dtp|_{\Sigma}=\p_t+\vb\cdot\cnab$. We shall further differentiate this equation with $\p_t^2$ and \textit{convert the Neumann boundary value of $\qc$ to a Dirichlet-type condition} in order to utilize the boundary condition $\qc=\sigma \h$.  We introduce the Alinhac good unknown $\Q:=\p_t^2\qc-\p_t^2\varphi\pp_3 \qc$ to obtain $$\rho\TDt^2\p_t^2\psi=- N\cdot\nabp\Q+\cdots.$$  The next step is to separate the contribution of $\qc$ on the boundary from that in the interior. We notice that $\Q$ satisfies a wave equation
\[
\rho\lam^2(\Dtp)^2\Q-\lapp\Q=\cdots\text{ in }\Om,\quad \Q|_{\Sigma}=\sigma\p_t^2\h-\p_3 q\p_t^2\psi,\quad \p_3\Q|_{\Sigma_b}=-\p_t^2\rho g.
\]Inspired by Shatah-Zeng \cite{SZ2008priori, SZ2011LWP}, we define $\Q=\Q_h+\Q_w$ where
\begin{align*}
-\lapp\Q_h=0\text{ in }\Om,\quad \Q_h=\Q\text{ on }\Sigma,\quad \p_3\Q_h=0\text{ on }\Sigma_b,\\
-\lapp\Q_w=-\rho\lam^2(\Dtp)^2\Q+\cdots \text{ in }\Om,\quad \Q_w=0\text{ on }\Sigma,\quad \p_3\Q_w=\p_3\Q\text{ on }\Sigma_b.
\end{align*}
Under this setting, we obtain the following evolution equation
\begin{align}
\rho\TDt^2\p_t^2\psi + \sigma\dn(\p_t^2 \h) - \dn(\p_3 q\p_t^2\psi) =- N\cdot\nabp \Q_w +\cdots \quad\text{ on }\Sigma
\end{align}where $\dn$ is the Dirichlet-to-Neumann (DtN) operator associated to $(\Om,\psi)$ and we refer to Definition \ref{defn DtN} for details. Since the DtN operator is a first-order operator with positive principal symbol and the mean curvature operator is a second-order elliptic operator, we formally have 
$$\rho\TDt^2(\p_t^2\psi) +\sigma \underbrace{C_1(\cdots)}_{>0} \TJ^3(\p_t^2\psi) + (-\p_3 q)  \underbrace{C_2(\cdots)}_{>0}\TJ(\p_t^2\psi)=- N\cdot\nabp \Q_w +\cdots \quad\text{ on }\Sigma,
$$where $\TJ:=\sqrt{1-\TL}$ and $\TL=\TP_1^2+\TP_2^2$. Thus, we can adopt the paralinearization used in Alazard-Burq-Zuily \cite{ABZ2014wwSTLWP, ABZ2014wwLWP} to calculate the principal symbol of their composition in order for an explicit uniform-in-$\lam$ energy estimate of $|\p_t^3\psi|_{1.5}$ and $|\sqrt{\sigma}\p_t^2\psi|_{3}$  (and also $|\p_t^2\psi|_{2}$, uniformly in $\sigma$, under the Rayleigh-Taylor sign condition). We refer to Section \ref{sect para DtN}-\ref{sect E5} for detailed computations.

\subsubsection{Necessity of the weighted fifth-order energy}
Note that the new energy $\EE(t)$ defined in \eqref{energyls intro} also includes a $\lam^2$-weighted fifth-order energy. This is actually necessary to control the contribution of the pressure wave, namely the term $|N\cdot\nabp \Q_w|_{1.5}$. Since $\Q_w$ has zero boundary value on $\Sigma$ and its Neumann boundary value on $\Sigma_b$ is easy to control, we can convert it to the control of $\|\lapp \Q_w\|_1$ which further requires the bound for $\|\lam^2\p_t^4\qc\|_1$, which is exactly a $\lam^2$-weighted fifth-order term. 

Note that the control of $E_5(t)$ in \eqref{energyls intro} is completely parallel to that of $E_4(t)$ defined in \eqref{energy main them, WP}, as the structure of these two energy functionals is the same except that each term of $E_5$ is assigned with a $\lam^2$-weight. One can check that the control of all commutators arising from tangential estimates leads to no loss of $\lam$ weight, and we refer to Section \ref{sect E5-2} for details.

\begin{rmk}
The combination of the pressure decomposition and the paralinearization of the free-surface motion allows us to ``separate" the contribution of free-surface motion (in particular, the surface tension) and interior pressure waves, and these two parts are related via the term $N\cdot\nabp \Q_w$, which naturally leads to the fifth-order energy. This method essentially improves upon the previous results \cite{LL2018priori, Luo2018CWW, DL19limit}, where the uniform boundedness of top-order time derivatives of $v$ is necessary. Also, our method no longer relies on the vanishing boundary value of pressure as in \cite{Zhang2021elastoLWP}. Thus, we believe that the approach developed in this paper can be applied to other ``coupled" fluid models or the vortex-sheet problems\footnote{The second author recently applied this method to the incompressible limit for current-vortex sheets in ideal compressible MHD. See \cite{Zhang2023CMHDVS2}. This is, to our knowledge, the first result about the incompressible limit of inviscid vortex sheets.}. Furthermore, our method may open up the possibility of studying the incompressible limit of free-surface fluids with ill-prepared initial data.
\end{rmk}

\subsection{The approximation scheme to prove the existence}\label{sect nonlinearkk intro}

\subsubsection{Motivation to design the approximation}\label{sect kk intro}
For free-surface inviscid fluids, the local existence is not a direct consequence of the a priori estimate. For example, if we try to do Picard iteration for the linearized system whose coefficient $\varphi$ is replaced by a given function $\pr$, then a crucial difference from the nonlinear system is that we may no longer obtain the boundary regularity from the analogue of the ST term as in \eqref{AGUBDRYST}. Specifically, we consider \eqref{AGUBDRYST} with full spatial tangential derivatives:
\begin{equation}\label{AGUSTR}
\begin{aligned}
\ST=\sigma\is\TP^{\alpha}\cnab\cdot\left(\frac{\cnab\psi}{1+|\cnab\psr|^2}\right)\p_t\TP^{\alpha}\psi\dx'=-\frac{\sigma}{2}\ddt\is\frac{|\TP^{\alpha}\cnab\psi|^2}{\sqrt{1+|\cnab{\psr|^2}}}-\frac{(\cnab\psi\cdot\TP^{\alpha}\cnab\psi)(\cnab\psr\cdot\TP^{\alpha}\cnab\psr)}{{\sqrt{1+|\cnab\psr|^2}}^3}\dx'+\cdots,
\end{aligned}
\end{equation}where the second term has no control because inequality \eqref{Cauchy} is not applicable here. Such a linearization yields the loss of a tangential derivative. Besides, the unknowns with full time derivatives only have $L^2(\Om)$ integrability and thus have no boundary regularity. Some crucial cancellations no longer hold after linearization. Therefore, it is natural to regularize the coefficient $\varphi$ in both $t$ and $x'$ variables.

\subsubsection{The approximation system: important steps of its construction}\label{sect mismatch intro}
For each $\kk>0$, we define $\lkk$ to be the standard convolution mollifier on $\R^2$ with parameter $\kk>0$ and then define $\psk:=\lkk^2\psi$ and $\pk(t,x):=x_3+\chi(x_3)\psk(t,x')$ to be the smoothed coefficients. We introduce the following nonlinear system with artificial viscosity whose coefficients are replaced by $\pk,\psk$ that is asymptotically consistent with the original system \eqref{CWWST} as $\kk\to 0_+$. 
\begin{equation}\label{CWWSTkk00}
\begin{cases}
\rho \Dtpk v +\nabpk \qc=-(\rho-1) ge_3,&~~~ \text{in}~[0,T]\times \Omega,\\
\ff'(q) \Dtpk \qc  +\nabpk\cdot v=\ff'(q) gv_3, &~~~ \text{in}~[0,T]\times \Omega,\\
q=q(\rho), \qc=q+g\pk &~~~ \text{in}~[0,T]\times \Omega, \\
\qc=g\psk-\sigma\cnab \cdot \left( \frac{\cnab \psk}{\sqrt{1+|\cnab\psk|^2}}\right)+\kk^2(1-\TL)(v\cdot \npk) &~~~\text{on}~[0,T]\times\Sigma, \\
\p_t \psi = v\cdot \npk &~~~\text{on}~[0,T]\times\Sigma,\\
v_3=0 &~~~\text{on}~[0,T]\times\Sigma_b,\\
(v,\rho,\psi)|_{t=0}=(v_0^\kk, \rho_0^\kk, \psi_0^\kk).
\end{cases}
\end{equation}
Here, 
\begin{align}
\nabpk_i=&~ \ppk_i = \p_i-\frac{\p_i \pk}{\p_3\pk}\p_3,~i=1,2,~~\nabpk_3 =~\ppk_3 = \frac{1}{\p_3 \pk} \p_3,\label{nabpk00}\\
\Dtpk =&~\p_t + \vb\cdot\cnab+\frac{1}{\p_3 \pk} (v\cdot \Npk-\p_t\varphi)\p_3, \label{Dtpk00}
\end{align}
and $\vb:=(v_1,v_2)$, $\cnab:=(\p_1,\p_2)$ are the horizontal velocities and derivatives, $\TL :=\cnab\cdot\cnab= \p_{1}^2+\p_{2}^2$ is the flat tangential Laplacian,  $\npk:= (-\p_1\psk, -\p_2 \psk, 1)^\top$ is the smoothed Eulerian normal vector and $\Npk:= (-\p_1\pk, -\p_2 \pk, 1)^\top$ is the extension of $\npk$ into $\Omega$. 

The tangential smoothing method was first introduced in \cite{CS2007LWP} to study the incompressible Euler equations and subsequently generalized to investigate various free-surface inviscid fluids in Lagrangian coordinates. However, the free surface is now assumed to be a graph, and the construction of a nonlinear approximate system differs significantly from that in Lagrangian coordinates. The following issues are crucial and very technical.
\begin{itemize}
\item \textbf{Design the smoothed material derivative $\Dtpk$.} When restricted on $\Sigma$, the weight function in front of $\p_3$ in $\Dtpk$ should agree with the kinematic boundary condition. Otherwise, there will be a \textit{boundary mismatched term that cannot be controlled} when studying $\frac{d}{dt}E(t)$. Therefore, we cannot mollify $\p_t\varphi$ in $\Dtpk$.
\item \textbf{Introduce the artificial viscosity to control the mismatched terms.} The tangential mollification leads to some mismatched terms that the artificial viscosity should control.
\begin{itemize}
\item [a.] The commutator $\dd(f)$ in \eqref{AGUDD} now involves a new term $\ee(f)=\p_t\TT^\alpha(\pk-\varphi)\ppk_3 f$ which should be bounded by $\kk|\cnab\p_t\TT^{\alpha}\psi|_0$ after using the mollifier property \eqref{lkk33}.
\item [b.] The analysis of the ST term introduces two extra commutators, whose control requires the bound for $\kk|\cnab\p_t\TT^{\alpha}\psi|$.
\end{itemize}
To control the above two crucial mismatched terms, we introduce the artificial viscosity term $-\kk^2(1-\TL)\p_t\psi$ which gives the energy $|\kk\TJ\TT^{\alpha}\p_t\psi|_0$ to enhance the regularity of $\p_t\psi$. Due to technical reasons, it should be noted that the coefficient must be $\kk^2$ instead of any other power of $\kk$ in the artificial viscosity. The details are explained in Section \ref{sect uniformkk} below \eqref{I22}. 
\end{itemize}

It should also be noted that the design of the linearized $\kk$-regularized problem is also crucial and technically complicated, as we must define the``new free surface" in each iteration step and ensure that the boundary conditions remain consistent with the nonlinear problem. We refer to those rather technical constructions to the beginning of Section \ref{sect linear LWP}. 

Now, once the coefficients involving $\varphi$ are regularized in both $t$ and $x'$ variables, the loss of derivatives can be compensated by such regularization for each fixed $\kk>0$. That is, the existence of the nonlinear approximate problem \eqref{CWWSTkk00} is resolved for each fixed $\kk>0$. Based on the strategies introduced in Section \ref{sect apriori intro} and the above analysis of the mismatched terms, we can derive the uniform-in-$\kk$ a priori estimates for the nonlinear approximate system \eqref{CWWSTkk00}. We can also prove the initial data $(v_{0,\kk},\rho_{0,\kk},\psi_{0,\kk})$ of \eqref{CWWSTkk00} converges to the initial data of \eqref{CWWST} as $\kk\to 0$. This completes the proof of the existence of the original system \eqref{CWWST}.

\section{Nonlinear approximate $\kk$-problem}\label{sect nonlinear approx problem}
 The first step in proving local well-posedness is to introduce our approximation scheme. For each $\kk>0$, we construct a suitable approximate problem indexed by $\kk$ which is asymptotically consistent with \eqref{CWWST}. 
\subsection{The tangential mollification}
Let $\zeta=\zeta(x')\in C_c^\infty (\R^2)$, satisfying $0\leq \zeta\leq 1$ and $\int_{\R^2} \zeta\,dx'=1$, be a standard cut-off function supported in the closed unit ball $\overline{B_1(\mathbf{0})}$. For each $\kk>0$, we set
$$
\zeta_\kk (x') = \kk^{-2} \zeta (\kk^{-1}x'),
$$
and for each $f:\R^2\to \R$, we define
\begin{equation}
\lkk f(x') :=\int_{\R^2} \zeta_\kk (x'-z') f(z')\,dz'. 
\end{equation}
Also,  for each $g:\R^3\to \R$, we set
\begin{equation}
\lkk g(x', z) := \int_{\R^2} \zeta_\kk (x'-z')g(z', x_3)\,dz'.
\end{equation}
In other words, when acting on a function of three independent variables, $\lkk$ becomes the smoothing operator in the tangential direction only. The next lemma records the properties that $\lkk$ enjoys. This will be frequently used (sometimes silently) in the rest of this paper.
\begin{lem}[{\cite[Lemma 2.6]{LuoZhang2020CWWLWP}}]\label{tgsmooth} 
Let $f:\R^2\rightarrow \R$ be a smooth function. For each $\kk>0$, we have:
\begin{align}
\label{lkk1 2D} |\lkk f|_s&\lesssim |f|_s,~~\forall s\geq -0.5;\\ 
\label{lkk2} |\TP\lkk f|_0&\lesssim \kk^{-s}|f|_{1-s}, ~~\forall s\in [0,1];\\  
\label{lkk3} |f-\lkk f|_{\infty}&\lesssim \sqrt{\kk}|\TP f|_{0.5}\\
\label{lkk33} |f-\lkk f|_{L^p}&\lesssim \kk|\TP f|_{L^p}.
\end{align}
Also, for a smooth function $g:\R^3\to \R$, then
\begin{equation}
\|\lkk g\|_s\lesssim \|g\|_s,~~\forall s\geq 0. \label{lkk1 3D}
\end{equation}
Moreover, let $h:\R^2\to \R$, and $[\lkk,f]h:=\lkk(fh)-f\lkk(h)$. Then we have:
\begin{align}
\label{lkk4} |[\lkk,f]g|_0 &\lesssim|f|_{L^{\infty}}|g|_0,\\ 
\label{lkk5} |[\lkk,f]\TP g|_0 &\lesssim |f|_{W^{1,\infty}}|g|_0, \\ 
\label{lkk6} |[\lkk,f]\TP g|_{0}&\lesssim \kk|f|_{W^{1,\infty}}|\TP g|_{0}.
\end{align}
\end{lem}

\subsection{Construction of the $\kk$-problem}\label{sect CWWSTkkeq}
Let $\psk:=\lkk^2\psi$, $\pk(t,x)=x_3+\chi(x_3)\psk(t,x')=\lkk^2\varphi(t,x)$, and $\npk:=(-\p_1\psk, -\p_2\psk, 1)^\top$. Then we set the approximate $\kk$-problem of \eqref{CWWST} to be
\begin{equation}\label{CWWSTkk}
\begin{cases}
\rho \Dtpk v +\nabpk \qc=-(\rho-1) ge_3&~~~ \text{in}~[0,T]\times \Omega,\\
\ff'(q) \Dtpk \qc  +\nabpk\cdot v=\ff'(q) gv_3&~~~ \text{in}~[0,T]\times \Omega,\\
q=q(\rho), \qc=q+g\pk &~~~ \text{in}~[0,T]\times \Omega, \\
\qc=g\psk-\sigma\cnab \cdot \left( \frac{\cnab \psk}{\sqrt{1+|\cnab\psk|^2}}\right)+\kk^2(1-\TL)(v\cdot \npk) &~~~\text{on}~[0,T]\times\Sigma, \\
\p_t \psi = v\cdot \npk &~~~\text{on}~[0,T]\times\Sigma,\\
v_3=0 &~~~\text{on}~[0,T]\times\Sigma_b,\\
(v,\rho,\psi)|_{t=0}=(v_{\kk,0}, \rho_{\kk, 0}, \psi_{\kk, 0}).
\end{cases}
\end{equation}
Here, 
\begin{align}
\ppk_t = &~\p_t - \frac{\p_t \varphi}{\p_3 \pk}\p_3,\label{ptk}\\
\nabpk_a=&~ \ppk_a = \p_a -\frac{\p_a \pk}{\p_3\pk}\p_3,\quad a=1,2,\label{nabpk 12}\\
\nabpk_3 =& ~\ppk_3 = \frac{1}{\p_3 \pk} \p_3,\label{nabpk 3}\\
\Dtpk = &~\ppk_t+v\cdot\nabpk \label{Dtpk},
\end{align}
and $\TL = \p_x^2+\p_y^2$ is the flat tangential Laplacian. Thanks to \eqref{ptk}, the smoothed material derivative $\Dtpk$ is equivalent to
\begin{equation}
\Dtpk = \p_t + \vb\cdot\cnab+\frac{1}{\p_3 \pk} (v\cdot \Npk-\p_t\varphi)\p_3, \label{Dtpk alternate}
\end{equation}
where 
$\Npk:= (-\p_1\pk, -\p_2 \pk, 1)^\top$. Note that we do not replace $v\cdot \Npk-\p_t\varphi$ by $v\cdot \Npk-\p_t\pk$ in the last term, as this would generate a severe structural mismatch in the boundary estimates. 

The approximate $\kk$-system \eqref{CWWSTkk} is asymptotically consistent with \eqref{CWWST} as $\kk\to 0$. Furthermore, the artificial viscosity $\kk(1-\TL)(v\cdot \npk)$ in the modified boundary condition 
$$
\qc=g\psk-\sigma\cnab \cdot \left( \frac{\cnab \psk}{\sqrt{1+|\cnab\psk|^2}}\right)+\kk^2(1-\TL)(v\cdot\npk) ~~~\text{ on }\Sigma
$$
is necessary to control the terms generated due to the loss of symmetry in \eqref{CWWSTkk}. 

\section{Uniform energy estimates for the nonlinear $\kk$-problem}\label{sect uniformkk}
For each fixed $\kk>0$, we denote by $(v^{\kk}(t),\rho^{\kk}(t),\qc^{\kk}(t),\psi^{\kk}(t))$ the solution of the nonlinear $\kk$-system \eqref{CWWSTkk}. Let $\sigma>0$ be fixed. We aim to show that $\{v^\kk(t),\qc^\kk(t),\rho^\kk(t),\psi^\kk(t)\}_{\kk>0}$ has a convergent subsequence that approximates the solution to the original system \eqref{CWWST} as $\kk\to0$ in some time interval $[0, T]$ with $T$ being independent of $\kk$. From now on, we drop the superscript $\kk$ when analyzing the nonlinear $\kk$-approximate system for the sake of clean notations. Let
\begin{equation} \label{energykk}
\begin{aligned}
E^{\kk}(t)=&E^{\kk}_{0}(t)+E^{\kk}_{4}(t), \quad \text{where}\\
E^{\kk}_{0}(t)=&\|\rho(t)-1\|_0^2+g|\lkk \psi|_0^2+\sum_{k=0}^3\|\sqrt{\ff'(q)}\p_t^k\qc(t)\|_{0}^2,\\
E^{\kk}_{4}(t)=&\sum_{k=0}^4\|\p_t^kv(t)\|_{4-k}^2+\sigma|\cnab \p_t^k \lkk \psi(t)|^2_{4-k}+\left\|\sqrt{\ff'(q)}\qc(t)\right\|_0^2+\sum_{k=0}^3\|\p_t^k\p\qc(t)\|_{3-k}^2\\
&+\left\|\sqrt{\ff'(q)}\p_t^4\qc(t)\right\|_0^2+{\sum_{k=0}^4\int_0^t\left|\kk\p_t^{k+1}\psi(\tau)\right|_{5-k}^2\mathrm{d}\tau}.
\end{aligned}
\end{equation}

\begin{thm} \label{prop uniform kk energy est}
For each fixed $\sigma>0$, 
there exists some $T_{\sigma}>0$, independent of $\kk$ and $\sqrt{\ff'(q)}$ (and thus $\lam$), such that
\begin{equation}
E^{\kk}(t)\lesssim P(E^{\kk}(0))=:\PP_0^\kk, \quad \text{ for every } 0\leq t\leq T_{\sigma},
\end{equation}
provided that the solutions of \eqref{CWWSTkk} possess the regularity
\begin{subequations}\label{regularity theorem 4.1 }
\begin{align}
&\p_t^k v(t,\cdot)\in L^{\infty}([0, T_{\sigma}];H^{4-k}(\Omega)), \quad k=0,\cdots, 4,\\
&\p_t^k \p \qc(t,\cdot)\in L^{\infty}([0, T_{\sigma}]; H^{3-k}(\Omega)),\quad k=0, \cdots ,3, \\
&\lam\p_t^k\qc(t,\cdot)  \in L^{\infty}([0, T_{\sigma}]; L^2(\Omega)), \quad k=0,\cdots, 4,\\
&\psi(t,\cdot)\in L^{\infty}([0, T_{\sigma}]; L^2(\Sigma)), \quad \sqrt{\sigma}\p_t^k \cnab\psi(t,\cdot)\in L^{\infty}([0, T_{\sigma}]; H^{4-k}(\Sigma)), \quad k=0,\cdots, 4,\\
&\kk\p_t^{k+1}\psi(t,\cdot) \in L^2([0, T_{\sigma}]; H^{5-k}(\Sigma)), \quad k=0,\cdots, 4. \label{regularity from viscosity}
\end{align}
\end{subequations}
\end{thm}
\begin{rmk}
Compared with the regularity of the solutions to the original nonlinear system \eqref{regularity theorem 1.1 }, the additional regularity requirement \eqref{regularity from viscosity} comes from the artificial viscosity. 
\end{rmk}

Thanks to the generalized Gr\"onwall's inequality (Theorem \ref{thm: generalized gronwall}),
the key step of proving Theorem \ref{prop uniform kk energy est} is to show that
\begin{equation} \label{Ek pre Gronwall}
\sup_{0\leq t\leq T}E^{\kk}(t)\leq \PP_0^\kk + \int_0^T P(E^\kk (t))\dt, 
\end{equation}
for some $T>0$ chosen sufficiently small.
The control of $E^\kk (t)$ will be divided into three steps, i.e., the basic $L^2$ estimate, the div-curl analysis, and the interior tangential estimates. We note that the compatibility conditions on $\Sigma$ have changed due to the presence of artificial viscosity. 
The new compatibility conditions, expressed in terms of $\qc$, are
\begin{equation} \label{kk comp con}
(\Dtpk)^k \qc |_{t=0} =  (\Dtpk)^k (-g\psk+\sigma \mathcal{H})|_{t=0} + (\Dtpk)^k \left (\kk^2 (1-\TL) (v\cdot\npk)\right){|_{t=0}},\quad k=0,1,2,3,\quad \text{on}\,\,\Sigma 
\end{equation}
We however are still able to construct initial data satisfying \eqref{kk comp con} in terms of $(\psi_{\kk, 0}, v_{\kk, 0}, \qc_{\kk, 0})$, that is uniformly bounded and converges to $(\psi_0, v_0, \qc_0)$ as $\kk \to 0$. The details can be located in Appendix \ref{sect CWWSTkk data}. 

\subsection{$L^2$-estimate}\label{sect kkL2}
First, we establish $L^2$-energy estimate for \eqref{CWWSTkk}. Invoking Theorem \ref{transport thm nonlinear}, the identity $\nabpk\pk=e_3$, and then integrating by parts, we have:
\begin{equation}\label{kkL2 1}
\begin{aligned}
\frac12\ddt\io\rho|v|^2\p_3\pk\dx=&-\io v\cdot\nabpk \qc\p_3\pk\dx-\io (\rho-1) g v_3\p_3\pk\dx+\frac12\io\rho|v|^2\p_3\p_t(\pk-\varphi)\dx\\
=&\io \qc(\nabpk\cdot v)\p_3\pk\dx+\isb v_3\qc\dx'-\is\p_t\psi q\dx'-\is g\psk \p_t\psi\dx'\\
&-\io (\rho-1) g v_3\p_3\pk\dx+\frac12\io\rho|v|^2\p_3\p_t(\pk-\varphi)\dx.
\end{aligned}
\end{equation}
Plugging the continuity equation into the first integral, we get
\begin{equation}\label{kkL2 q}
\begin{aligned}
\io \qc(\nabpk\cdot v)\p_3\pk\dx=&-\frac12\ddt\io\ff'(q)|\qc|^2\p_3\pk\dx+\frac12\io\rho\Dtpk(\rho^{-1}\ff'(q))|\qc|^2\dx+\frac12\io\ff'(q)|\qc|^2\p_3\p_t(\pk-\varphi)\dx\\
&+\io\qc\ff'(q)gv_3\p_3\pk\dx\\
\lesssim&-\frac12\ddt\left\|\sqrt{\ff'(q)}q\right\|_0^2 +\left\|\sqrt{\ff'(q)}q\right\|_0^2\left(\|\Dtpk\rho\|_{\infty}+\|\p_3\p_t(\pk-\varphi)\|_{\infty}+\left\|\sqrt{\ff'(q)}v_3\right\|_0\right).
\end{aligned}
\end{equation}
Here and in the sequel, we employ the notation $A\lesssim B$ to mean that $A\leq CB$ for a universal constant $C$. 
The boundary integral on $\Sigma_b$ vanishes due to $v_3|_{\Sigma_b}=0$. Then we plug $q=-\sigma\cnab \cdot \left( \frac{\cnab \psk}{\sqrt{1+|\cnab\psk|^2}}\right)+\kk^2(1-\TL)\p_t\psi$ into the first boundary term in \eqref{kkL2 1} and integrate by parts to get:
\begin{equation}\label{kkL2 STkk}
-\is\p_t\psi q\dx'=-\sigma\is\left( \frac{\cnab \psk}{\sqrt{1+|\cnab\psk|^2}}\right) \cdot\cnab\p_t\psi\dx'+\is\left|\kk\TJ\p_t\psi\right|_0^2\dx',
\end{equation}
where $\langle\cdot\rangle$ denotes the Japanese bracket. 
To treat the first term,  we use the self-adjointness of $\lkk$ in $L^2(\Sigma)$ to move one $\lkk$ from $\cnab\psk$ to $\p_t\psi$:
\begin{equation}\label{kkL2 ST}
\begin{aligned}
&-\sigma\is\left( \frac{\cnab \psk}{\sqrt{1+|\cnab\psk|^2}}\right) \cdot\cnab\p_t\psi\dx'=-\sigma\is\frac{\cnab\lkk\psi\cdot\p_t\cnab\lkk\psi}{|\npk|}\dx'-\sigma\is\cnab\lkk\psi\cdot\left([\lkk,|\npk|^{-1}]\cnab\p_t\psi\right)\dx'\\
\lesssim&-\frac12\ddt\left|\sqrt{\sigma}\frac{1}{|\npk|^{\frac{1}{2}}}\cnab\lkk\psi\right|_0^2+\frac12\is\p_t(|\npk|^{-1})\left|\sqrt{\sigma}\cnab\lkk\psi\right|^2+P(|\cnab\psk|_{W^{1,\infty}})\sigma\left|\sqrt{\sigma}\cnab\lkk\psi\right|_0^2+\eps\left|\kk\cnab\p_t\psi\right|_0^2.
\end{aligned}
\end{equation}
Now, we get the non-weighted $L^2$-boundary energy from the second boundary integral in \eqref{kkL2 1}:
\begin{equation}\label{kkL2 psi}
\begin{aligned}
-\is g\p_t\psi\psk\dx'=-\frac12\ddt \io g|\lkk\psi|^2\dx'.
\end{aligned}
\end{equation}
Next, we prove an $L^2$-estimate for $\rho-1$.
We use $\Dtpk\rho=\Dtpk(\rho-1)$ and $\Dtpk\pk=v_3+\p_t(\pk-\varphi)$ to rewrite the continuity equation in terms of $\rho-1$:
\[
\Dtpk(\rho-1)+\rho(\nabpk\cdot v)=-\p_t(\pk-\varphi).
\] Testing this with $\rho-1$ in $L^2(\Omega)$ and using the mollifier property \eqref{lkk33}, we get
\begin{equation}\label{kkL2 rho}
\frac12\ddt\|\rho-1\|_0^2\lesssim \|\rho-1\|_0(\|\p v\|_0+\kk|\TP\p_t\psi|_0).
\end{equation}

Let
\begin{equation} \label{kkE0}
E_0^{\kk}(t)=\|v\|_0^2+\left\|\sqrt{\ff'(q)}\qc\right\|_0^2+\|\rho-1\|_0^2+|\sqrt{g}\lkk\psi|_0^2+\left|\sqrt{\sigma}\cnab\lkk\psi\right|_0^2+\int_0^T\is\left|\kk\TJ\p_t\psi\right|_0^2\dx'\dt.
\end{equation}
Since $1\leq |\npk| = \sqrt{1+(\p_1 \psk)^2+(\p_2 \psk)^2}$,  we combine \eqref{kkL2 1}-\eqref{kkL2 rho} and obtain
\begin{equation}\label{kkL2}
E_0^{\kk}(T)-E_0^{\kk}(0)\lesssim\int_0^T P(|\cnab\psk|_{W^{1,\infty}},\|\p v\|_{\infty},|\kk\TP\p_t\psi|_{0.5})E_0^{\kk}(t)\dt,
\end{equation}
after choosing $\eps>0$ suitably small in \eqref{kkL2 ST}.
Here, we note that, using $\p_3\p_t (\pk-\varphi) = \chi'(x_3)\left(\p_t\psk(t,x')-\p_t\psi (t,x')\right)$
together with \eqref{chi} and \eqref{lkk3} of Lemma \ref{tgsmooth},
we have
\begin{align}\label{AGU TTM1}
\|\p_3 \p_t (\pk-\varphi)\|_{\infty} \leq  |\p_t\psk-\p_t\psi|_{\infty}
\lesssim \sqrt{\kk} |\TP\p_t\psi|_{0.5},
\end{align}where right side is directly controlled by invoking $\p_t\psi=v\cdot \Npk=-(\vb\cdot\cnab)\psk+v_3$ on $\Sigma$ and the Sobolev trace lemma. 

\subsection{Reduction of pressure}\label{sect reduction q}
We show how to reduce the control of the pressure to that of the velocity when there is at least one spatial derivative on $q$. This follows from using the momentum equation 
$\rho\Dtpk v=-\nabpk\qc-(\rho-1)ge_3$. 
In particular, by considering the third component of the momentum equation, we get
\begin{align} \label{mom 3}
-(\p_3\pk)^{-1} \p_3 \qc-(\rho-1) ge_3=\rho\Dtpk v_3.
\end{align}
Since $\p_3\pk$ is bounded from below, i.e., there exists $c_0>0$ such that $\p_3 \pk \geq c_0$,  then
\begin{equation} 
\|\p_3 \qc\|_{0}\lesssim_{g, c_0} \|\rho-1\|_0+\|\rho\|_{\infty} \|\Dtpk v_3\|_0,
\end{equation}
where $\Dtpk v_3 = \p_t v_3 +\vb\cdot \cnab v_3+ \frac{1}{\p_3\pk} (v\cdot \Npk-\p_t\varphi)\p_3 v_3$. This implies that the $L^2$-norm of $\p_3\qc$ is reduced to the $L^2$-norms of $\rho-1$, $\p_t v_3,\TP v_3$ and $\omega(x_3)\p_3 v_3$. Here $\omega(x_3)\in C^{\infty}(-b,0)$ is assumed to be bounded, comparable to $|x_3|$ in $[-2,0]$ and vanishing on $\Sigma$.

Let $\TT=\p_t$ or $\TP$ or $\omega(x_3)\p_3$ and $D=\p$ or $\p_t$.  The above estimate yields the control of $\|D^k \p_3 \qc\|_0$ after taking $D^k$, $k\geq 1$ to \eqref{mom 3}. Specifically, at the leading order, $\|D^k \p_3 \qc\|_0$ is controlled by
\begin{align} 
C(g, c_0)\left(\|\ff'(q)D^k \qc\|_0+\|\ff'(q)D^k \pk\|_0+\|\rho\|_{L^\infty}\|D^k\TT v_3\|_0\right).
\end{align}

In addition, by considering the first two components of the momentum equation, we have:
\begin{equation} \label{reduction q tangential}
-\p_i \qc=-(\p_3\pk)^{-1}\TP_i\pk\p_3 \qc+\rho\Dtpk v_i,~~i=1,2.
\end{equation}
and thus the control of $\TP \qc$ is reduced to $\p_3 \qc$ and $\Dtpk v_i=\p_t v_i+(\vb\cdot\cnab) v_i+(\p_3\pk)^{-1}(v\cdot\Npk-\p_t\varphi)\p_3 v_i$.

Lastly,  using \eqref{mom 3} and \eqref{reduction q tangential}, we obtain
\begin{align}
\|\p_3 \qc\|_{\infty}\lesssim_{g, c_0} \|\rho-1\|_{\infty}+\|\rho\|_{\infty} \|\Dtpk v_3\|_{\infty},\\
\|\TP \qc\|_{\infty} \lesssim_{g, c_0^{-1}} |\TP \psk|_{\infty} \|\p_3 \qc\|_{\infty} +\|\rho\|_{\infty} \|\Dtpk \vb\|_{\infty}. 
\end{align}
Thus, 
\begin{equation}\label{Dq Linf}
\|\p q\|_{\infty}\lesssim_{g, c_0, c_0^{-1}}  P(|\TP \psk|_{\infty}, \|\rho\|_{\infty})\left(\|\rho-1\|_{\infty} + \|\Dtpk v\|_{\infty}\right).
\end{equation}
Invoking the definition of $\Dtpk v$, 
\eqref{Dq Linf} implies that  $\|\p \qc\|_{\infty}$ is reduced to $\p_t v$, $\TP v$ and $\omega(x)\p_3 v$ for some weight function $\omega(x)$ vanishing on $\Gamma$. 

\subsection{Div-Curl analysis}\label{sect div curl}
To estimate the Sobolev norm of $v$, we can use the div-curl analysis to convert one normal derivative to the divergence and curl. First, we record the well-known div-curl decomposition lemma and refer to \cite[Lemma B.2]{GLL2020LWP} for the proof.
\begin{lem}[Hodge-type elliptic estimates]\label{hodgeTT}
For any sufficiently smooth vector field $X$ and $s\geq 1$, one has
\begin{equation}
\|X\|_s^2\lesssim C(|\psk|_s,|\cnab\psk|_{W^{1,\infty}})\left(\|X\|_0^2+\|\nabpk\cdot X\|_{s-1}^2+\|\nabpk\times X\|_{s-1}^2+\|\TP^{\alpha}X\|_0^2\right),
\end{equation} 
for any multi-index $\alpha$ with $|\alpha|=s$. The constant $C(|\psk|_s,|\cnab\psk|_{W^{1,\infty}})>0$ depends linearly on $|\psk|_s^2$.
\end{lem}

We will apply Lemma \ref{hodgeTT} to $\|\p_t^k v\|_{4-k}$ for $0\leq k \leq 3$. Starting from $k=0$, we have:
\begin{align}
\label{vdivcurl}\|v\|_4^2\lesssim&~ C(|\psk|_4,|\cnab\psk|_{W^{1,\infty}})\left(\|v\|_0^2+\|\nabpk\cdot v\|_3^2+\|\nabpk\times v\|_3^2+\|\TP^4 v\|_0^2\right),\\
\label{vtdivcurl} \|\p_t^k v\|_{4-k}^2\lesssim &~C(|\psk|_{4-k},|\cnab\psk|_{W^{1,\infty}})\left(\|\p_t^k v\|_0^2+\|\nabpk\cdot \p_t^kv\|_{3-k}^2+\|\nabpk\times \p_t^kv\|_3^2+\|\TP^{4-k}\p_t^{k} v\|_0^2\right),
\end{align} where the $L^2$-norm has been controlled in \eqref{kkL2} and the tangential derivatives will be studied in the next section by using Alinhac good unknowns. The divergence part is reduced to the estimates of $q$ by using the continuity equation
\begin{equation}\label{divv}
\|\nabpk\cdot v\|_3^2=\left\|\ff'(q)\Dtpk \qc\right\|_3^2+\left\|\ff'(q)gv_3\right\|_3^2,
\end{equation}which will be further reduced to the tangential estimates of $v$ by using the argument in Section \ref{sect reduction q}. Similarly, when $k=1,2,3$, we have
\begin{align*}
\nabpk\cdot\p_t^k v=-\p_t^k(\ff'(q)\Dtpk \qc)+\p_t^k(\ff'(q) gv_3)+[\nabpk\cdot,\p_t^k]v\lleq - \ff'(q) \p_t^k \Dtpk \qc+ \ff'(q) g \p_t^k v_3 + (\p_3\pk)^{-1}\TP\p_t^k\pk\p_3 v,
\end{align*}
where $\lleq$ means equality modulo lower-order terms.

In particular, the lower-order terms consist of $$-[\p_t^k, \ff'(q)]\Dtpk \qc, \quad  g[\p_t^k, \ff'(q)]v_3, \quad  \left[\p_t^k, \frac{1}{\p_3\pk}\right]\p_a \pk \p_3 v^a,
 \quad \text{and}\,\, -\left[\p_t^k, \frac{1}{\p_3\pk}\right]\p_3 v^3,
$$
where $a=1,2$. Note that there are $\leq k-1$ time derivatives on $\Dtpk \qc$, $v_3$, $\p v$. Also, since $\pk(t,x', x_3)=x_3+\chi(x_3)\psk(t,x')$, there are $\leq k-1$ time derivatives on $\TP \psk$, and $\leq k$ time derivatives on $\psk$. From \eqref{ff property} and because $k\leq 3$, these lower-order terms are controlled by $\PP_0^\kk+\int_0^T P(E^\kk(t))\dt$ through a combination of Young's inequality, Sobolev embeddings, and fundamental theorem of calculus.

This implies 
\begin{equation}\label{divvt}
\begin{aligned}
\|\nabpk\cdot\p_t^k v\|_{3-k}^2\leq &~C(c_0, g, \|v\|_{W^{1,\infty}}) \left (\left\|\ff'(q)\p_t^{k}\Dtp \qc\right\|_{3-k}^2+\left|\TP\p_t^{k}\psk\right|_{3-k}^2
+\PP_0^\kk+\int_0^T P(E^\kk(t))\dt\right),
\end{aligned}
\end{equation}
where the last two terms control all lower-order terms generated above. 
Since the material derivative $\Dtp=\p_t+\vb\cdot \cnab$ on $\Sigma$, the term $\ff'(q)\p_t^{k}\Dtp \qc$ involves only tangential derivatives with appropriate $\ff'$-weight. 
By combining this div-curl analysis and the reduction of pressure in Section \ref{sect reduction q}, we eventually only need to control the mixed space-time tangential derivatives of $v$, $\psi$, and $\qc$. We refer to Propositions \ref{TT spatial 4'}, \ref{TT spatial time}, and \ref{TT time 4'} for the details. 


Next, we analyze the vorticity term. We take $\nabpk\times$ in the momentum equation $\rho \Dtpk v=-\nabpk \qc+(\rho-1) ge_3$ to get 
\[
\rho\Dtpk(\nabpk\times v)=-\nabpk\times((\rho-1) ge_3)-(\nabpk\rho)\times\Dtpk v-\rho[\nabpk\times,\Dtpk]v,
\]where the first term on the right side is equal to $(-g\ppk_2\rho,g\ppk_1\rho,0)^\top$ and the second term, using $\Dtpk v=-\rho^{-1}\nabpk q-ge_3$, is equal to
\[
-(\nabpk\rho)\times\Dtpk v=\underbrace{\rho'(q)(\nabpk q)\times(\nabpk q)}_{=\vec{0}}+\nabpk \rho\times ge_3=(g\ppk_2\rho,-g\ppk_1\rho,0)^{\top}
\]which exactly cancels the first term. Using $[\ppk_i,\Dtpk](\cdot)=\ppk_iv^l\ppk_l(\cdot)+\ppk_i\p_t(\pk-\varphi)\ppk_3(\cdot)$, we get the evolution of the smoothed vorticity to be
\begin{equation}\label{curlv eq}
\rho\Dtpk(\nabpk\times v)_i=-\rho\epsilon^{ijk}\ppk_jv^l\ppk_l v_k-\rho\epsilon^{ijk}\ppk_j\p_t(\pk-\varphi)\ppk_3 v_k,
\end{equation}where $\epsilon^{ijk}$ denotes the sign of the permutation $(ijk)\in S_3$.

To control $\|\nabpk\times v\|_3$, we take $\p^3$ in \eqref{curlv eq} to get
\begin{equation}\label{curlv eq3}
\rho\Dtpk\left(\p^3(\nabpk \times v)_i\right)=-\epsilon^{ijk}\p^3(\rho\ppk_jv^l\ppk_l v_k)-\epsilon^{ijk}\p^3(\rho\ppk_j\p_t(\pk-\varphi)\ppk_3 v_k)-[\p^3,\rho\Dtpk](\nabpk \times v)_i.
\end{equation}

It is not necessary to write out the specific form of the right side of \eqref{curlv eq3}, but we just need to know the source terms in \eqref{curlv eq3} contain $\leq 4$ derivatives of $v$ and $\pk$ except the mismatched term involving $\pk-\varphi$. This is easy to see because the only term containing five derivatives is the one on the left side of \eqref{curlv eq3}. Therefore, a straightforward $L^2$ estimate for \eqref{curlv eq3} gives us the energy estimate
\begin{equation}\label{curlv3}
\ddt\frac12\|\nabpk\times v\|_3^2\leq P(\|v\|_4,|\psk|_4,\|\ff'(q)\p q\|_{\infty},\|\ff'(q)\p^2 q\|_1,\kk|\cnab\p_t\psi|_4),
\end{equation}where the mismatched term is controlled by using mollifier property \eqref{lkk6} and $\varphi(t,x)=x_3+\chi(x_3)\psi(t,x')$.

Similarly, we replace $\p^3$ by $\p_t^{k}\p^{3-k}$ for $0\leq k \leq 3$ to get
\begin{equation}
\rho\Dtpk\left(\p^{3-k}\p_t^k(\nabpk \times v)_i\right)=-\epsilon^{ijk}\p_t^{k}\p^{3-k}(\rho\ppk_jv^l\ppk_l v_k)-\epsilon^{ijk}\p_t^k\p^{3-k}(\rho\ppk_j\p_t(\pk-\varphi)\ppk_3 v_k)-[\p_t^{k}\p^{3-k},\rho\Dtpk](\nabpk \times v)_i,
\end{equation}and thus
\begin{equation}\label{curlvt1}
\ddt\frac12\|\p_t^k(\nabpk\times v)\|_{3-k}^2\leq P(E^\kk(t)).
\end{equation}

Then we need to estimate the commutator $\|[\p_t^k,\nabpk\times] v\|_{3-k}$ to get the control of $\|\nabpk\times\p_t^k v\|_{3-k}$. Similarly, as in the control of divergence, we know the highest order term in the commutator should be $\|(-\p_3\pk)^{-1}\TP\p_t^k\pk\p_3 v\|_{3-k}\lesssim \|\p v\|_{3-k}\|\TP\TJ^{3-k}\p_t^k\pk\|_0\leq \|\p v\|_{3-k}|\TP\p_t^k\psk|_{3-k}$. So we have the following conclusion 
\begin{equation}\label{curlvt}
\|\nabpk\times \p_t^k v\|_{3-k}^2\leq |\TP\p_t^k\psk|_{3-k}^2+\PP_0^{\kk}+\int_0^TP(E^{\kk}(t))\dt.
\end{equation}

Combining \eqref{vdivcurl}, \eqref{divv}, \eqref{divvt}, \eqref{curlv3}, \eqref{curlvt} and the argument in Section \ref{sect reduction q}, it remains to control the tangential derivatives of $v$ and full time derivatives of $q$, namely $\|\ff'(q)\p_t^4 \qc\|_0$.

\subsection{The $\TT^\alpha$-differentiated equations}\label{sect tg AGU}
By the div-curl analysis, the crucial step is to study the higher-order tangential energy estimate of \eqref{CWWSTkk}. In particular, we define the following tangential derivatives
\begin{equation}
\TT_0=\p_t,\quad \TT_1 = \p_1, \quad \TT_2=\p_2, \quad \TT_3=\omega(x_3) \p_3,
\end{equation} where $\omega\in C^{\infty}(-b,0)$ is assumed to be bounded, comparable to $|x_3|$ when $-2\leq x_3\leq 0$ and vanishing on $\Sigma$. This requires us to commute $\TT^\alpha$ with \eqref{CWWSTkk}, where $\TT^\alpha:=\TT_0^{\alpha_0}\TT_1^{\alpha_1}\TT_2^{\alpha_2}\TT_3^{\alpha_3}$, and $|\alpha| \leq 4$. 
\begin{rmk}
We need the tangential derivative $\TT_3=\omega(x_3)\p_3$ to control the $(\p_3\varphi)^{-1}(v\cdot\Npk-\p_t\varphi)\p_3$ in the material derivative $\Dtpk$. We do not include it in $E^{\kk}(t)$ as $\omega$ is comparable to 1. However, we still need the estimates of $\TT_3$ in the reduction of $\qc$. 
\end{rmk}

 We will not directly commute $\TT^\alpha$ with $\nabpk$. Instead, for $i=1,2,3$, we observe that
\begin{equation}
\TT^\alpha \ppk_i f = \ppk_i \TT^\alpha f - \ppk_3 f \ppk_i \TT^\alpha \pk + \mathfrak{C}_i'(f), \label{AGU pre}
\end{equation}
where for $i=1,2$, 
\begin{align}\label{AGU comm Ci'}
\mathfrak{C}_i'(f) =
-\left[ \TT^\alpha, \frac{\p_i \pk}{\p_3\pk}, \p_3 f\right]-\p_3 f \left[ \TT^\alpha, \p_i \pk, \frac{1}{\p_3\pk}\right] -\p_i\pk \p_3 f\left[\TT^{\alpha-\gamma}, \frac{1}{(\p_3\pk)^2}\right] \TT^\gamma \p_3 \pk
-\frac{\p_i \pk}{\p_3 \pk} [\TT^\alpha, \p_3] f+ \frac{\p_i \pk}{(\p_3\pk)^2}\p_3f [\TT^\alpha, \p_3]\pk,
\end{align}
with $|\gamma|=1$, and 
\begin{align}\label{AGU comm C3'}
\mathfrak{C}_3'(f) = 
\left[ \TT^\alpha, \frac{1}{\p_3\pk}, \p_3 f\right] + \p_3 f\left[\TT^{\alpha-\gamma}, \frac{1}{(\p_3\pk)^2}\right] \TT^\gamma \p_3 \pk
+\frac{1}{\p_3 \pk} [\TT^\alpha, \p_3] f- \frac{1}{(\p_3\pk)^2}\p_3f [\TT^\alpha, \p_3]\pk.
\end{align}
Since $\ppk_i$ and $\ppk_3$ commute, the identity \eqref{AGU pre} implies
\begin{equation}
\TT^\alpha \ppk_i f = \ppk_i (\TT^\alpha f - \ppk_3 f \TT^\alpha \pk)+\underbrace{\ppk_3\ppk_i f \TT^\alpha \pk + \mathfrak{C}_i'(f)}_{:=\mathfrak{C}_i(f)}. \label{AGU comm 1}
\end{equation}
The quantity $\TT^\alpha f - \ppk_3 f \TT^\alpha \pk$ is the so-called Alinhac good unknown associated with $f$. It was first observed by Alinhac \cite{Alinhac1989good} that the top-order term of $\pk$ does not appear when we use the above good unknown. It is not hard to see that we can obtain the control of $\|\TT^\alpha f\|_0$ from that of $\|\TT^\alpha f - \ppk_3 f \TT^\alpha \pk\|_0$. In particular,
\begin{align}
\|\TT^\alpha f\|_0 \leq \|\TT^\alpha f - \ppk_3 f \TT^\alpha \pk\|_0+ \| \ppk_3 f\|_{\infty} \|\TT^\alpha \pk\|_0. 
\label{AGU L2 bd}
\end{align}

In addition to this, we need to commute $\TT^\alpha$ with 
$$\Dtpk=\p_t + \vb\cdot\cnab+\frac{1}{\p_3 \pk} (v\cdot \Npk-\p_t\varphi)\p_3.$$ 
A direct computation yields:
\begin{align}
\TT^\alpha \Dtpk f &= \TT^\alpha \p_t f + \TT^\alpha (\vb \cdot \TP f) + \TT^\alpha \left(\frac{1}{\p_3\pk} (v\cdot \Npk-\p_t\varphi)\p_3 f\right)\nonumber\\
&=\Dtpk \TT^\alpha f + (v\cdot \TT^\alpha \Npk-\p_t \TT^\alpha \varphi)\ppk_3 f-\ppk_3 \TT^\alpha \pk(v\cdot \Npk-\p_t\varphi) \ppk_3 f+\mathfrak{D}'(f), \label{comm Dt TT pre}
\end{align}
where
\begin{align}
\mathfrak{D}'(f) = [\TT^\alpha, \vb]\cdot \TP f + \left[\TT^\alpha, \frac{1}{\p_3\pk}(v\cdot \Npk-\p_t\varphi), \p_3 f\right]+\left[\TT^\alpha, v\cdot \Npk-\p_t\varphi, \frac{1}{\p_3\pk}\right]\p_3 f+\frac{1}{\p_3\pk} [\TT^\alpha, v]\cdot \Npk \p_3 f\nonumber\\
-(v\cdot \Npk-\p_t\varphi)\p_3 f\left[ \TP^{\alpha-\gamma}, \frac{1}{(\p_3 \pk)^2}\right]\TT^\gamma \p_3 \pk+\frac{1}{\p_3\pk}(v\cdot \Npk-\p_t\varphi) [\TT^\alpha, \p_3] f+ (v\cdot \Npk-\p_t\varphi) \frac{\p_3 f}{(\p_3\pk)^2}[\TT^\alpha, \p_3] \pk, \label{AGU comm D}
\end{align}
with $|\gamma|=1$. 

Since $v\cdot \TT^\alpha \Npk = -v_1\p_1 \TT^\alpha \pk -v_2\p_2 \TT^\alpha \pk$, then we must have
\begin{equation} \label{AGU comm 3}
\begin{aligned}
 &(v\cdot \TT^\alpha \Npk-\p_t \TT^\alpha \varphi)\ppk_3 f-\ppk_3 \TT^\alpha \pk(v\cdot \Npk-\p_t\varphi) \ppk_3 f\\
 =& (v\cdot \TT^\alpha \Npk-\p_t \TT^\alpha \pk)\ppk_3 f-\ppk_3 \TT^\alpha \pk(v\cdot \Npk-\p_t\varphi) \ppk_3 f+ \p_t \TT^\alpha (\pk-\varphi)\ppk_3 f\\
 =& -\ppk_3 f\left (\p_t + \vb\cdot\cnab + (v\cdot \Npk-\p_t\varphi)\ppk_3 \right)\TT^\alpha \pk + \underbrace{\p_t \TT^\alpha (\pk-\varphi)\ppk_3 f}_{:=\mathfrak{E}(f)}\\
 =& -\ppk_3 f \Dtpk \TT^\alpha \pk+ \mathfrak{E}(f).
 \end{aligned}
\end{equation}
Thus, 
\begin{align}
\TT^\alpha \Dtpk f&=\Dtpk \TT^\alpha f  -\ppk_3 f \Dtpk \TT^\alpha \pk +\mathfrak{D}'(f)+\mathfrak{E}(f)\nonumber\\
=& \Dtpk \left(\TT^\alpha f  - \ppk_3 f \TT^\alpha \pk \right) +\mathfrak{D}(f)+\mathfrak{E}(f), \label{AGU comm 2}
\end{align}
where $\mathfrak{D}(f)=(\Dtpk \ppk_3 f) \TT^\alpha \pk +\mathfrak{D}'(f)$.

Let 
\begin{align}
\VV_i : = \TT^\alpha v_i - \ppk_3 v_i \TT^\alpha \pk, \quad \QQ:=\TT^\alpha \qc - \ppk_3 \qc \TT^\alpha \pk
\end{align}
respectively be the Alinhac good unknowns of $v$ and $\qc$. Motivated by \eqref{AGU comm 1} and \eqref{AGU comm 2}, we take $\TT^\alpha$ to the first two equations of \eqref{CWWST Eulerian} to obtain
\begin{align}
\rho \Dtpk \VV_i + \ppk_i \QQ &=\mathcal{R}^1_i,\label{AGU TT mom}\\
\ff'(q) \Dtpk \QQ + \nabpk\cdot \VV&=\mathcal{R}^2- \mathfrak{C}_i(v^i), \label{AGU TT mass}
\end{align}
where
\begin{align}
\label{rr1} \mathcal{R}^1_i:=& -[\TT^\alpha, \rho] \Dtpk v_i-\rho\left(\mathfrak{D}(v_i)+\mathfrak{E}(v_i)\right)-\mathfrak{C}_i(\qc),\\
\label{rr2} \mathcal{R}^2 :=&-[\TT^\alpha, \ff'(q)]\Dtpk \qc- \ff'(q)\left(\mathfrak{D}(\qc)+\mathfrak{E}(\qc)\right)+\TT^{\alpha}(\ff'(q)gv_3).
\end{align}
In addition, since $\TT^\alpha$ reduces to $\TP^\alpha$ on $\Sigma$ and $\TP^\alpha \npk = (-\p_1 \TP^\alpha \psk, -\p_2 \TP^\alpha \psk, 0)^\top$,  the $\TP^\alpha$-differentiated  kinematic boundary condition then reads
\begin{equation}
\p_t \TP^\alpha \psi + (\vb\cdot\cnab) \TP^\alpha\psk - \VV\cdot \npk=\mathcal{S}_1 \quad \text{on}\,\,\Sigma, \quad \text{ and } \VV_3=0\quad \text{on}\,\,\Sigma_b, \label{AGU TT BC kinematic} 
\end{equation}
where
\begin{align}
\mathcal{S}_1:=\p_3 v\cdot \npk \TP^\alpha\psk+[\TP^\alpha, v\cdot, N].
\end{align}
Also, since $\qc=q+g\psk$ and {$\p_3\pk|_{\Sigma}=1$}, we have {$\QQ|_{\Sigma} = (\TP^\alpha\qc - \p_3 \qc \TP^\alpha \psk)|_{\Sigma}=\TP^{\alpha}q+g\TP^{\alpha}\psk-(\p_3q+g)\TP^{\alpha}\psk=\TP^{\alpha}q-\p_3 q\TP^{\alpha}\psk$}, and thus the boundary condition of $\QQ$ on $\Sigma$ reads:
\begin{align}
\QQ =-\sigma \TP^\alpha\cnab \cdot \left( \frac{\cnab \psk}{\sqrt{1+|\cnab\psk|^2}}\right)+\kk^2(1-\TL)\TP^\alpha (v\cdot\npk)-\p_3 q\TP^\alpha \psk. \label{AGU TT BC Q}
\end{align}

\subsection{Tangential energy estimate with full spatial derivatives}\label{sect TT AGU}
In this subsection, we study the spatially-differentiated equations, i.e., the equations obtained by commuting $\TT^\alpha$, $\alpha_0=0$, and $|\alpha|=4$, with \eqref{CWWSTkk}. We aim to prove the following estimate:
\begin{prop}\label{TT spatial 4'}
For $\TT^{\alpha}$ with multi-index $\alpha$ satisfying $\alpha_0=0$ and $|\alpha|=4$, we have
\begin{equation}\label{TT spatial 4}
\|\TT^{\alpha}v(T)\|_0^2+\left\|\sqrt{\ff'(q)}\TT^{\alpha} \qc(T)\right\|_0^2+\left|\sqrt{\sigma}\cnab\TP^{\alpha}\lkk\psi(T)\right|_0^2+ \int_0^T\left|\kk\TP^{\alpha}\p_t\psi(t)\right|_{1}^2\dt\lesssim\PP_0^{\kk}+\int_0^T P(E^{\kk}(t))\dt.
\end{equation}
\end{prop}

 We will not directly consider the $\TT^{\alpha}$-differentiated variables but use Alinhac good unknowns to get rid of higher order terms of $\psk$. Invoking Lemma \ref{int by parts lem} and Theorem \ref{transport thm nonlinear}, testing \eqref{AGU TT mom} with $\VV$ and then integrating over $\Omega$ with respect to the measure $\p_3\pk\dx$, we get
\begin{align}\label{energy estimate AGU spatial}
\frac{1}{2} \frac{d}{dt} \io \rho |\VV|^2\p_3\pk\dx=\frac{1}{2}\io \rho |\VV|^2\p_3\p_t (\pk-\varphi)\dx +\io \QQ (\nabpk\cdot \VV) \p_3\pk\dx-\is \QQ (\VV\cdot \npk)\dx'+ \io \VV\cdot \mathcal{R}^1\p_3\pk\dx,
\end{align}where the boundary integral on $\Sigma_b$ vanishes thanks to $\VV_3|_{\Sigma}=0$. From now on, we will no longer write any boundary integral on $\Sigma_b$ due to the same reason.
Before estimating the integrals in \eqref{energy estimate AGU spatial},
we record some important properties that Alinhac good unknowns enjoy. 
\begin{lem}\label{AGU TT error}
Let $\mathbf{F}:= \TT^\alpha f - \ppk_3 f \TT^\alpha \pk$ with $|\alpha|=4$ and $\alpha_0=0$ be the Alinhac good unknowns associated with the smooth function $f$. 
Suppose that $\p_3 \pk \geq c_0>0$, then
\begin{equation}
\|\TT^\alpha f\|_0 \leq \|\mathbf{F}\|_0+ P\left(c_0^{-1}, |\psk|_4\right)\|\p_3 f\|_{\infty}. \label{AGU TT4 bd}
\end{equation}
 Furthermore, let $\mathfrak{C}(f)$, $\mathfrak{D}(f)$, and $\mathfrak{E}(f)$ be the remainder terms defined respectively in \eqref{AGU comm 1}, \eqref{AGU comm 3}, and \eqref{AGU comm 2}.  Then
\begin{align}
\|\mathfrak{C}_i(f)\|_{0} \leq&~P\left(c_0^{-1}, |\psk|_4\right)\cdot \|f\|_{4},\quad i=1,2,3,\label{C spatial}\\
\|\mathfrak{D}(f)\|_{0} \leq&~P\left(c_0^{-1}, |\psk|_4, |\p_t \psk|_3\right)\cdot \left(\|f\|_{4}+\|\p_t f\|_3\right), \label{D spatial}\\
\|\mathfrak{E}(f)\|_0 \leq&~\kk |\cnab\TT^{\alpha}\p_t \psi|_0 \|\p f\|_{\infty}. \label{E spatial}
\end{align}
\end{lem}
\begin{proof}
Since $\ppk_3 = (\p_3 \pk)^{-1} \p_3$, 
we have
\begin{align}
\| \ppk_3 f\|_{\infty} \|\TT^\alpha \pk\|_0 \leq P\left(c_0^{-1}, |\psk|_4\right)\|\p_3 f\|_{\infty},
\end{align}
and so \eqref{AGU TT4 bd} follows from \eqref{AGU L2 bd}. 
Also, the estimates \eqref{C spatial} and \eqref{D spatial} follow from the definition of $\mathfrak{C}(f)$ and $\mathfrak{D}(f)$, \eqref{chi}, \eqref{lkk1 3D} in Lemma \ref{tgsmooth},  and the Sobolev inequalities. 
To establish \eqref{E spatial}, we notice that
\begin{align*}
\|\mathfrak{E}(f)\|_0 \leq \|\p_t\TT^\alpha (\pk-\varphi)\|_0 \|\ppk_3 f\|_{\infty}+ \|\ppk_3 \TT^\alpha \pk\|_0 \|\p_t (\pk-\varphi)\|_{\infty} \|\ppk_3 f\|_{\infty}. 
\end{align*}
Thus, it suffices to control the leading order terms $\|\p_t\TT^\alpha (\pk-\varphi)\|_0$ and $\|\ppk_3 \TT^\alpha \pk\|_0$. We have
\begin{align*}
\p_t\TT^\alpha (\pk-\varphi) =&~\p_t\TT^\alpha \left(\chi(x_3) \psk-\chi(x_3) \psi\right)\\
\leq&~\chi(x_3) \p_t\TP^\alpha (\psk-\psi) + \left[\TT^\alpha, \chi(x_3)\right] \p_t (\psk-\psi).
\end{align*}
The $L^2$-norm of the second term can be controlled by the RHS of \eqref{E spatial} thanks to \eqref{chi}. 
By \eqref{lkk33} in Lemma \ref{tgsmooth}, we have
\begin{align*}
|\p_t\TP^\alpha (\psk-\psi) |_0 \leq \kk |\p_t \psi|_5. 
\end{align*}
Also, 
\begin{align*}
\ppk_3 \TT^\alpha \pk = \ppk_3 \TT^\alpha \left(\chi(x_3)\psk\right)= \left(\ppk_3 \chi(x_3)\right) \TT^\alpha\psk + \left( \ppk_3 [\TT^\alpha, \chi(x_3)]\right) \psk,
\end{align*}
and so $\|\ppk_3 \TT^\alpha \pk\|_0$ can be controlled by the RHS of \eqref{E spatial}. 
\end{proof}
\begin{rmk}
The appearance of $\mathfrak{E}(f)$ is a consequence of the tangential smoothing. This estimate of $\|\mathfrak{E}(f)\|_0$ yields a top order term $\kk |\p_t\psi|_5$, which can only be controlled by the energy contributed by the artificial viscosity. In other words, the artificial viscosity compensates for the loss of symmetry in the $\kk$-equations.
\end{rmk}

\subsubsection{Control of $\io \rho |\VV|^2\p_3\p_t (\pk-\varphi)\dx$: The integral contains the mismatched term.}  We have
\begin{align}\label{AGU TTM0}
\io \rho |\VV|^2\p_3\p_t (\pk-\varphi)\dx \leq \|\rho\|_{\infty}\|\VV\|_0^2 \|\p_3 \p_t (\pk-\varphi)\|_{\infty}\lesssim\sqrt{\kk}\|\VV\|_0^2 |\TP\p_t\psi|_{0.5}.
\end{align}

\subsubsection{Control of $\io \VV\cdot \mathcal{R}^1\p_3\pk\dx$: Error terms}  We have
\begin{align}\label{AGU TTR0}
\io  \VV\cdot \mathcal{R}^1 \p_3\pk\dx \leq \|\VV\|_0 \|\mathcal{R}^1\|_0 \|\p_3\pk\|_{\infty},
\end{align}
where the $L^2$-norm of $\mathcal{R}^1$ is directly controlled by using \eqref{rr1} and \eqref{C spatial}--\eqref{E spatial}:
\begin{equation}\label{AGU TTR1}
\|\mathcal{R}^1\|_0\leq P(\|\p_3\pk\|_{\infty},|\psk|_4,|\p_t\psk|_3)\left(\kk|\cnab\TT^{\alpha}\p_t \psi|_0\|v\|_4+{\|v\|_4+\|\p_t v\|_3+\|\qc\|_4}\right),
\end{equation}where the term containing $\kk|\cnab\TT^{\alpha}\p_t \psi|_0$ should be controlled under time integral as we will get $L_t^2H_{x'}^1([0,T]\times\Sigma)$ bound for $\kk\p_t \TT^{\alpha}\psi$ later.

\subsubsection{Control of $\io \QQ (\nabpk\cdot \VV) \p_3\pk\,dx$: Tangential energy for $\QQ$} 
Equation \eqref{AGU TT mass} indicates that
\begin{equation}\label{AGU TTQ0}
\io \QQ (\nabpk \cdot \VV)\p_3\pk\dx = -\io \ff'(q) \QQ (\Dtpk \QQ)\p_3\pk\dx+ \io \QQ (\mathcal{R}^2-\cc_i(v^i))\p_3\pk\dx.
\end{equation}
For the second term on the RHS of \eqref{AGU TTQ0},  we invoke the second inequality in \eqref{ff property} and then apply it to the definition of $\mathcal{R}^2$ in \eqref{rr2} to get:
\begin{align}
\io \QQ \mathcal{R}^2\p_3\pk\dx \leq \|\sqrt{\ff'(q)}\QQ\|_0 \|\mathcal{R}^2\|_0 \|\p_3\pk\|_{\infty}.
\end{align}
In other words, we ``borrow" one $\sqrt{\ff'(q)}$ from $\mathcal{R}^2$ and attach it to $\QQ$. Thanks to \eqref{C spatial}-\eqref{E spatial},
we control the $L^2$-norm of the rest of terms in $\mathcal{R}^2$ directly by
\begin{equation}
 P(\|\p_3\pk\|_{\infty},|\psk|_4,|\p_t\psk|_3)\left(\kk|\cnab\TT^{\alpha}\p_t \psi|_0\left\|\sqrt{\ff'(q)}\qc\right\|_4+\left\|\sqrt{\ff'(q)} \qc\right\|_4+\left\|\sqrt{\ff'(q)} \p_t \qc\right\|_3+\left\|\sqrt{\ff'(q)} gv_3\right\|_3\right),
\end{equation}where the term containing $\kk|\cnab\TT^{\alpha}\p_t \psi|_0$ should be controlled under time integral as we will get $L_t^2H_{x'}^1([0,T]\times\Sigma)$ bound for $\kk\p_t \TT^{\alpha}\psi$ later. Then the contribution of $\cc_i(v^i)$ is controlled by
\begin{equation}\label{AGU TT CV}
-\io\QQ(\cc_i(v^i))\p_3\pk\dx\leq P(|\psk|_4,|\cnab\psk|_{W^{1,\infty}})|\TT^{\alpha}\psk|_0\|v\|_{4}\|\QQ\|_0.
\end{equation}
Here, $\|\QQ\|_0$ contributes to $\|\TP^\alpha \qc\|_0$ and $\|\ppk_3 \qc \TP^\alpha \psk\|_0$. The first term $\|\TP^\alpha \qc\|_0$ is not weighted by $\sqrt{\ff'(q)}$ and thus cannot be controlled directly by \eqref{TT spatial 4}. Fortunately, we can overcome this issue by invoking \eqref{reduction q tangential}. Similarly, $\|\ppk_3 \qc \TP^\alpha \psk\|_0 \leq \|\ppk_3 \qc\|_{\infty} \|\TP^\alpha \psk\|_0$, where we use \eqref{Dq Linf} to treat $\|\ppk_3 \qc\|_{\infty}$, and so this can be controlled uniformly as $\ff'(q) \to 0$. 

Furthermore, invoking the integration by parts formula \eqref{transpt nonlinear without rho}, the first integral on the RHS of \eqref{AGU TTQ0} becomes
\begin{equation}\label{AGU TTQ}
\begin{aligned}
\io \ff'(q) \QQ (\Dtpk \QQ)\p_3\pk\dx =&-\frac{1}{2}\ddt \io \ff'(q) |\QQ|^2\p_3\pk\dx+\frac{1}{2}\io  (\Dtpk \ff'(q))|\QQ|^2 \p_3\pk\dx\\
&+\frac{1}{2} \io (\nabpk\cdot v) \ff'(q)|\QQ|^2\p_3\pk\dx+\frac{1}{2}\io \ff'(q) |\QQ|^2\p_3\p_t (\pk-\varphi)\p_3\pk\dx\\
\lesssim& -\frac12\ddt\left\|\sqrt{\ff'(q)}\QQ\right\|_0^2 + \|\p_3\pk\|_{\infty}\|\sqrt{\ff'(q)}\QQ\|_0^2\left(\|\p v\|_{\infty}+\kk|\cnab\p_t\psi|_{0.5}\right).
\end{aligned}
\end{equation}

\subsubsection{Control of $-\is \QQ (\VV\cdot \npk)\dx'$: Boundary energy contributed by surface tension and artificial viscosity}\label{sect AGU TT ST}

Note that $\TT_3=\vec{0}$ on $\Sigma$ implies the corresponding good unknown $\QQ=0$ on $\Sigma$, so it suffices to consider the case $\TT^{\alpha}=\TP^{\alpha}$ when analyzing the boundary integral. Using \eqref{AGU TT BC kinematic}, we have
\begin{align}\label{AGUbdry}
-\is \QQ (\VV\cdot \npk)\dx'= -\is \QQ \left(\p_t \TP^\alpha \psi + (\vb\cdot\cnab) \TP^\alpha\psk -\mathcal{S}_1\right)\dx'.
\end{align}

The first term is expected to contribute to two coercive terms if we invoke the boundary condition \eqref{AGU TT BC Q} of $\QQ$:
\begin{equation}\label{AGUbdry1}
\begin{aligned}
I_1:=-\is \QQ\p_t\TP^{\alpha}\psi\dx'=&~\sigma\is \TP^{\alpha}\cnab\cdot\left(\frac{\cnab\psk}{\sqrt{1+|\cnab\psk|^2}}\right)\p_t\TP^{\alpha}\psi\dx'-\kk^2\is\TP^{\alpha}(1-\TL)\p_t\psi\cdot\TP^{\alpha}\p_t\psi\dx'+\is\p_3 q\TP^{\alpha}\psk\p_t\TP^{\alpha}\psi\dx'\\
=&:\ST_{1}+\ST_{2}+\RT.
\end{aligned}
\end{equation}
Since $1-\TL = \TJ^2$, where $\langle \cdot \rangle$ denotes the Japanese bracket, we find the term$\ST_2$ gives us $\sqrt{\kk}$-weighted enhanced energy after integration by parts :
\begin{equation}\label{ST2}
\ST_2=-\kk^2\is\left|\TP^{\alpha}\TJ\p_t\psi\right|^2\dx'=-\ddt\left|\kk\TP^{\alpha}\p_t\psi\right|_{L_t^2H_{x'}^1}^2.
\end{equation}

In the control of$\ST_1$, we will repeatedly use 
\begin{equation}\label{tpnpk}
\TP\left(\frac{1}{|\npk|}\right)=\frac{\cnab\psk\cdot\TP\cnab\psk}{|\npk|^3},
\end{equation}where $|\npk|=\sqrt{1+|\cnab\psk|^2}$ denotes the length of the smoothed normal vector $\npk=(-\TP_1\psk,-\TP_2\psk,1)^{\top}$. 
Now we integrate $\cnab\cdot$ by parts in$\ST_1$ to get 
\begin{equation}\label{ST10}
\begin{aligned}
\ST_1=&-\sigma\is\frac{\TP^{\alpha}\cnab\psk}{|\npk|}\cdot\p_t\TP^{\alpha}\cnab\psi\dx'+\sigma\is\frac{\cnab\psk\cdot\TP^{\alpha}\cnab\psk}{|\npk|^3}\cnab\psk\cdot\p_t\TP^{\alpha}\cnab\psi\dx'\\
&-\sigma\is\left(\left[\TP^{\alpha-\alpha'},\frac{1}{|\npk|}\right]\TP^{\alpha'}\cnab\psk+\cnab\psk\left[\TP^{\alpha-\alpha'},\frac{1}{|\npk|^3}\right](\cnab\psk\cdot\TP^{\alpha'}\cnab\psk)-\frac{1}{|\npk|^3}\left[\TP^{\alpha-\alpha'},\cnab\psk\right]\TP^{\alpha'}\cnab\psk\right)\cdot\p_t\cnab\TP^{\alpha}\psi\dx'\\
=&:\ST_{11}+\ST_{12}+\ST_{13},
\end{aligned}
\end{equation}where $\alpha'$ is a multi-index with $|\alpha'|=1$.

The first two terms in \eqref{ST10} are expected to produce the energy contributed by the surface tension. Before that, we need to move one mollifier from the top order term of $\psk=\lkk^2\psi$ to the top order term of $\psi$ by using the self-adjointness of $\lkk$ in $L^2(\Sigma)$:
\begin{equation}\label{ST1112}
\begin{aligned}
\ST_{11}+\ST_{12}=&-\sigma\is\dfrac{\TP^{\alpha}\cnab\lkk\psi\cdot\p_t\TP^{\alpha}\cnab\lkk\psi}{|\npk|}-\dfrac{(\cnab\psk\cdot\TP^{\alpha}\cnab\lkk\psi)(\cnab\psk\cdot\p_t\TP^{\alpha}\cnab\lkk\psi)}{|\npk|^3}\dx'\\
&-\sigma\is\TP^{\alpha}\cnab\lkk\psi\cdot\left(\left[\lkk,\frac{1}{|\npk|}\right]\cnab\p_t\TP^{\alpha}\psi\right)\dx'+\sigma\is\TP^{\alpha}\cnab_i\lkk\psi\cdot\left(\left[\lkk,\frac{\cnab_i\psk\cnab_j\psk}{|\npk|^3}\right]\cnab_j\p_t\TP^{\alpha}\psi\right)\dx'\\
=:&\ST_{10}+\ST_{11}^{R}+\ST_{12}^{R}.
\end{aligned}
\end{equation}
Then we find
\begin{align}
\label{ST100}\ST_{10}=&-\frac{\sigma}{2}\ddt\is\frac{|\TP^{\alpha}\cnab\lkk\psi|^2}{\sqrt{1+|\cnab\psk|^2}}-\frac{|\cnab\psk\cdot\TP^{\alpha}\cnab\lkk\psi|^2}{{\sqrt{1+|\cnab\psk|^2}}^3}\dx'\\
&\label{ST10R}+\frac{\sigma}{2}\is\p_t\left(\frac{1}{\sqrt{1+|\cnab\psk|^2}}\right)\left|\TP^{\alpha}\cnab\lkk\psi\right|^2-\p_t\left(\frac{1}{{\sqrt{1+|\cnab\psk|^2}}^3}\right)\left|\cnab\psk\cdot\TP^{\alpha}\cnab\lkk\psi\right|^2\dx'.
\end{align}
To deal with the first term in$\ST_{10}$, we plug $\mathbf{a}=\TP^{\alpha}\cnab\lkk\psi$ into the following inequality, which can be proved by direct calculation:
\begin{equation}\label{STineq}
\frac{|\mathbf{a}|^2}{\sqrt{1+|\cnab\psk|^2}}-\frac{|\cnab\psk\cdot\mathbf{a}|^2}{{\sqrt{1+|\cnab\psk|^2}}^3}\geq \frac{|\mathbf{a}|^2}{{\sqrt{1+|\cnab\psk|^2}}^3},
\end{equation}
in order to get
\begin{equation}\label{ST1}
\int_0^T\ST_{10}\dt+\frac{\sigma}{2}\is\frac{|\TP^{\alpha}\cnab\lkk\psi|^2}{{\sqrt{1+|\cnab\psk|^2}}^3}\dx'\leq P(|\cnab\psk_0|_{L^{\infty}})\left|\sqrt{\sigma}\TP^{\alpha}\cnab\lkk\psi_0\right|_{0}^2 +\int_0^T\eqref{ST10R}\dt,
\end{equation}where the terms in \eqref{ST10R} can be controlled directly:
\begin{equation}\label{ST10RR}
\eqref{ST10R}\leq P(|\cnab\psk|_{L^{\infty}})|\p_t\cnab\psk|_{L^{\infty}}\left|\sqrt{\sigma}\TP^{\alpha}\cnab\lkk\psi\right|_{0}^2.
\end{equation}

To finish the control of$\ST_{1}$, it remains to control$\ST_{13}$ and$\ST_{11}^R,\ST_{12}^R$. The last two terms can be controlled by using the mollifier property \eqref{lkk6} and the $\kk$-weighted energy contributed by the artificial viscosity. For $\ST_{11}^R$, we have
\begin{equation}\label{ST11R}
\begin{aligned}
\int_0^T\ST_{11}^R\lesssim &\int_0^T\left|\sqrt{\sigma}\TP^{\alpha}\lkk\psi\right|_0 P\left(|\cnab\psk|_{\infty}\right)|\cnab\psk|_{W^{1,\infty}}\left|\kk\p_t\TP^{\alpha}\psi\right|_{\dot{H}^1}\dt\\
\lesssim&~\eps\left|\kk\p_t\TP^{\alpha}\psi\right|_{L_t^2H_{x'}^1}^2+\int_0^T P\left(|\cnab\psk|_{\infty}\right)|\cnab\psk|_{W^{1,\infty}}^2\left|\sqrt{\sigma}\TP^{\alpha}\lkk\psi\right|_0^2 \dt.
\end{aligned}
\end{equation}
Additionally, $\ST_{12}^R$ can be controlled similarly. 

As for$\ST_{13}$ in \eqref{ST10}, we find that all three commutators have similar structures and the same leading order terms, so we only show the analysis of the first commutator. Note that the leading order term in $[\TP^{\alpha-\alpha'},|\npk|^{-1}]\TP^{\alpha'}\cnab\psk$ appears when $\TP^{\alpha-\alpha'}$ falls on $|\npk|^{-1}$ or $\TP^{\alpha''}$ falls on $|\npk|^{-1}$ and $\TP^{\alpha-\alpha'-\alpha''}$ falls on $\TP^{\alpha'}\cnab\psk$ for some $|\alpha''|=1$. In either of the two cases, the top-order term contributes to the following integral:
\begin{equation}\label{ST130}
-\sigma\is  |\npk|^{-3} \cnab\psk~\TP^{\alpha-\alpha'}\cnab\psk~\TP^{\alpha'}\cnab\psk\cdot\p_t\cnab\TP^{\alpha}\psi\dx'.
\end{equation}We integrate one $\cnab$ by parts to get
\[
\sigma\is  |\npk|^{-3} \cnab\psk~\TP^{\alpha-\alpha'}\cnab^2\psk~\TP^{\alpha'}\cnab\psk~\p_t\TP^{\alpha}\psi\dx'
\]
modulo lower order terms, 
and then we move one $\lkk$ from $\TP^{\alpha-\alpha'}\cnab^2\psk$ to $\p_t\TP^{\alpha}\psi$ such that the main term is directly controlled as:
\begin{equation}\label{ST131}
\sigma\is  |\npk|^{-3} \cnab\psk~\TP^{\alpha-\alpha'}\cnab^2\lkk\psi~\TP^{\alpha'}\cnab\psk~\p_t\TP^{\alpha}\lkk\psi\dx'\lesssim P(|\cnab\psk|_{\infty})|\cnab\psk|_{W^{1,\infty}}\left|\sqrt{\sigma}\TP^{\alpha}\cnab\lkk\psi\right|_0\left|\sqrt{\sigma}\p_t\TP^{\alpha}\lkk\psi\right|_0,
\end{equation}where the last term will be controlled in $\p_t\TP^3$-estimates. Besides, we have to analyze the commutator involving $\lkk$:
\begin{equation}\label{ST132}
\sigma\is \TP^{\alpha-\alpha'}\cnab^2\lkk\psi\left(\left[\lkk,P(\cnab\psk)\TP^{\alpha'}\cnab\psk\right]\p_t\TP^{\alpha}\psi\right)\dx',
\end{equation}
which is controlled under the time integral:
\begin{equation}\label{ST133}
\begin{aligned}
\int_0^T\eqref{ST132}\dt\lesssim&\sqrt{\sigma}\int_0^T|\sqrt{\sigma}\cnab\TP^{\alpha}\lkk\psi|_0\cdot\kk|\TP\cnab\psk|_{W^{1,\infty}} P(|\cnab\psk|_{W^{1,\infty}})|\p_t\TP^{\alpha}\psi|_0\dt\\
\lesssim&~\eps\left|\kk\p_t\TP^{\alpha}\psi\right|_{L_t^2L_{x'}^2}^2+\int_0^T\left|\sqrt{\sigma}\cnab\TP^{\alpha}\lkk\psi\right|_0^2|\sqrt{\sigma}\cnab\psk|_{3.5}^2P(|\cnab\psk|_{2.5})\dt.
\end{aligned}
\end{equation}

Therefore, 
\begin{equation}\label{STbound}
\int_0^T(\ST_1+\ST_2)\dt+\left|\kk\TP^{\alpha}\p_t\psi\right|_{L_t^2H_{x'}^1}^2+\frac{\sigma}{2}\left|\cnab\TP^{\alpha}\lkk\psi(T)\right|_0^2\lesssim\PP_0^{\kk}+\int_0^T P(E^{\kk}(t))\dt,
\end{equation}where we have chosen $\eps>0$ that appears above to be suitably small such that all $\eps$-terms are absorbed by the $\kk$-weighted energy.

To finish the control of $I_1$ defined in \eqref{AGUbdry1}, it remains to control the term $\RT$. Note that when we drop the mollifier and have the Rayleigh-Taylor sign condition $-\p_3 q\geq \frac{c_0}{2}>0$ assumed on $\Sigma$,$\RT$ should directly give us the non-$\sigma$-weighted boundary energy. But since we are now solving the gravity-capillary water wave system for fixed $\sigma>0$ instead of taking the vanishing surface tension limit,  we cannot assume $-\p_3 q\geq \frac{c_0}{2}>0$ on $\Sigma$. Thus, this term is controlled by the surface tension energy after moving one $\lkk$:
\begin{equation}\label{RTbound}
\begin{aligned}
\int_0^T\RT\dt=&-\int_0^T\is\p_3 q\TP^{\alpha}\lkk\psi\cdot\p_t\TP^{\alpha}\lkk\psi\dx'\dt-\int_0^T\is\TP^{\alpha}\lkk\psi\cdot\left(\left[\lkk,\p_3 q\right]\p_t\TP^{\alpha}\psi\right)\dx'\dt\\
\lesssim&\int_0^T|\p q|_{L^{\infty}}\left|\TP^{\alpha}\lkk\psi\right|_0\left|\p_t\TP^{\alpha}\lkk\psi\right|_0\dt+\eps\left|\kk\p_t\TP^{\alpha}\psi\right|_{L_t^2L_{x'}^2}^2+\int_0^T|\p q|_{W^{1,\infty}}^2\left|\TP^{\alpha}\lkk\psi\right|_0^2\dt\\
\lesssim&~\eps\left|\kk\p_t\TP^{\alpha}\psi\right|_{L_t^2L_{x'}^2}^2+\int_0^T P\left(\|\qc\|_{4},\left|\TP^{\alpha}\lkk\psi\right|_0,\left|\p_t\TP^{\alpha}\lkk\psi\right|_0\right)\dt,
\end{aligned}
\end{equation}where the term $\left|\p_t\TP^{\alpha}\lkk\psi\right|_0$ is the energy term obtained in $\TP^{\alpha-\alpha'}\p_t$-estimates for $|\alpha'|=1$. 
\begin{rmk}
The RHS of \eqref{RTbound} is not uniform in $\sigma$. However, as mentioned earlier, $-\int_0^T\is\p_3 q\TP^{\alpha}\lkk\psi\cdot\p_t\TP^{\alpha}\lkk\psi\dx'\dt$ contributes to a non-$\sigma$-weighted energy term $\is (-\p_3 q) |\TP^\alpha \lkk \psi|^2\dt$ provided the Rayleigh-Taylor sign condition holds. We shall revisit the control of RT in Section \ref{section double lim}, where the zero surface tension limit is considered. 
\end{rmk}
Combining this with \eqref{STbound}, we get the estimate for $I_1$
\begin{equation}\label{AGUbdryI1}
\int_0^T I_1\dt+\left|\kk\TP^{\alpha}\p_t\psi\right|_{L_t^2H_{x'}^1}^2+\frac{\sigma}{2}\left|\cnab\TP^{\alpha}\lkk\psi(T)\right|_0^2\lesssim\PP_0^{\kk}+\int_0^T P(E^{\kk}(t))\dt,
\end{equation}after choosing $\eps>0$ that appears above to be suitably small.

The second term in \eqref{AGUbdry} gives
\begin{equation}\label{AGUbdry2}
\begin{aligned}
I_2:=-\is \QQ(\vb\cdot\cnab)\TP^{\alpha}\psk=&~\sigma\is \TP^{\alpha}\cnab\cdot\left(\frac{\cnab\psk}{\sqrt{1+|\cnab\psk|^2}}\right)(\vb\cdot\cnab)\TP^{\alpha}\psk\dx'-\kk^2\is\TP^{\alpha}(1-\TL)\p_t\psi\cdot(\vb\cdot\cnab)\TP^{\alpha}\psk\dx'\\
&+\is\p_3 q\TP^{\alpha}\psk(\vb\cdot\cnab)\TP^{\alpha}\psk\dx'\\
=&:I_{21}+I_{22}+I_{23},
\end{aligned}
\end{equation}where we find that $I_{22}, I_{23}$ can be directly controlled as follows:
\begin{equation}\label{I22}
\begin{aligned}
\int_0^TI_{22}\dt\overset{\cnab}{=}&-\kk^2\int_0^T\is\TP^{\alpha}\cnab\p_t\psi\cdot\cnab\left((\vb\cdot\cnab)\TP^{\alpha}\psk\right)\dx'\dt-\kk^2\int_0^T\is\TP^{\alpha}\p_t\psi\cdot(\vb\cdot\cnab)\TP^{\alpha}\psk\dx'\dt\\
\lesssim&~\int_0^T\left|\kk\TP^{\alpha}\p_t\psi\right|_0|\cnab \vb|_{\infty}\left|\kk\cnab^2\TP^{\alpha}\psk\right|_0\dt+\kk\int_0^T\left|\kk\TP^{\alpha}\p_t\psi\right|_0|\vb|_{\infty}\left|\cnab\TP^{\alpha}\psk\right|_0\dt\\
\lesssim&~\eps\left|\kk\TP^{\alpha}\p_t\psi\right|_{L_t^2H_{x'}^1}^2+\int_0^T|\vb|_{W^{1,\infty}}^2\left|\cnab\TP^{\alpha}\lkk\psi\right|_0^2\dt\lesssim\eps\left|\kk\TP^{\alpha}\p_t\psi\right|_{L_t^2H_{x'}^1}^2+\int_0^TP(E^{\kk}(t))\dt,
\end{aligned}
\end{equation}where we use the mollifier property \eqref{lkk2} to control $|\kk\cnab^2\TP^{\alpha}\psk|_0\lesssim\kk\cdot\kk^{-1}|\cnab\TP^{\alpha}\lkk\psi|_0$. This step also shows why the power of $\kk$ must be 2 in the artificial viscosity; otherwise, the control of $I_{22}$ is not uniform in $\kk$. For $I_{23}$ we integrate $\vb\cdot\cnab$ by parts to get
\begin{equation}\label{I23}
I_{23}=\frac12 \is\cnab\cdot(\vb~\p_3 q)|\TP^{\alpha}\psk|^2\dx'\lesssim P(E^{\kk}(t)).
\end{equation}

The control of $I_{21}$ is analogous to$\ST_1$. Following \eqref{ST10}, we have
\begin{equation}\label{I210}
\begin{aligned}
I_{21}=&-\sigma\is\left(\frac{\TP^{\alpha}\cnab\psk}{|\npk|}-\frac{\cnab\psk\cdot\TP^{\alpha}\cnab\psk}{|\npk|^3}\cnab\psk\right)\cdot(\vb\cdot\cnab)\TP^{\alpha}\cnab\psk\dx'\\
&-\sigma\is\left(\left[\TP^{\alpha-\alpha'},\frac{1}{|\npk|}\right]\TP^{\alpha'}\cnab\psk+\left[\TP^{\alpha-\alpha'},\frac{1}{|\npk|^3}\right](\cnab\psi\cdot\TP^{\alpha'}\cnab\psk)-\frac{1}{|\npk|^3}\left[\TP^{\alpha-\alpha'},\cnab\psk\right]\TP^{\alpha'}\cnab\psk\right)\cdot(\vb\cdot\cnab)\cnab\TP^{\alpha}\psk\dx'\\
=&:I_{211}+I_{212},
\end{aligned}
\end{equation}where $I_{212}$ can be directly controlled if we integrate $\vb\cdot\cnab$ by parts:
\begin{equation}\label{I212}
I_{212}\lesssim P(|\psk|_4)|\vb|_{W^{1,\infty}}\left|\sqrt{\sigma}\cnab\TP^{\alpha}\psk\right|_0^2\leq P(E^{\kk}(t)).
\end{equation}

For $I_{211}$, we integrate $\vb\cdot\cnab$ by parts and use the symmetric structure to see
\begin{equation}\label{I211}
\begin{aligned}
I_{211}\lleq-\frac{\sigma}{2}\is(\cnab\cdot\vb) \left(\frac{|\TP^{\alpha}\cnab\psk|^2}{|\npk|}-\frac{|\cnab\psk\cdot\TP^{\alpha}\cnab\psk|^2}{|\npk|^{3}}\right)\dx'\lesssim  P(|\cnab\psk|_{\infty})|\vb|_{W^{1,\infty}}\left|\sqrt{\sigma}\cnab\TP^{\alpha}\psk\right|_0^2.
\end{aligned}
\end{equation}
 Therefore, plugging \eqref{I22}-\eqref{I211} into \eqref{AGUbdry2}, we get the estimates for $I_2$:
\begin{equation}\label{AGUbdryI2}
\int_0^T I_2\dt\lesssim\eps\left|\kk\TP^{\alpha}\p_t\psi\right|_{L_t^2H_{x'}^1}^2+\int_0^T P(E^{\kk}(t)).
\end{equation}

It remains to control the term involving $\sss_1$, which reads
\begin{equation}\label{AGUbdry3}
\begin{aligned}
I_3:=&\is\QQ\sss_1\dx'=\is\QQ\left(\p_3v\cdot\npk\TP^{\alpha}\psk+\sum_{\substack{|\beta_1|+|\beta_2|=4\\ |\beta_1|,|\beta_2|>0}} \TP^{\beta_1} v\cdot \TP^{\beta_2}\npk\right)\dx'\\
=&\is\left(\sigma\TP^{\alpha}\h+\kk^2(1-\TL)\TP^\alpha\p_t\psi- \p_3 q \TP^\alpha \psk\right)\left(\p_3v\cdot\npk\TP^{\alpha}\psk+\sum_{|\beta_1|=1,|\beta_2|=3} \TP^{\beta_1} v\cdot \TP^{\beta_2}\npk\right)\dx'\\
&+\sum_{\substack{|\beta_1|+|\beta_2|=4\\ |\beta_1|\geq 1,1\leq|\beta_2|\leq 2}}\is(\TP^{\alpha}q-\TP^{\alpha}\psk\p_3 q)( \TP^{\beta_1} v\cdot \TP^{\beta_2}\npk)\dx'\\
=&:I_{31}+I_{32},
\end{aligned}
\end{equation}where we use the definition of $\QQ$ in $I_{32}$ and invoke the Dirichlet boundary condition \eqref{AGU TT BC Q} for $\QQ$ in $I_{31}$ such that the $L^2(\Sigma)$ bound of $\TP^{\alpha} v$ and non-$\sigma$-weighted $\cnab\TP^{\alpha}\psi$ with $|\alpha|=4$ can be avoided on $\Sigma$.

The term $I_{32}$ can be directly controlled as:
\begin{equation}\label{I32}
I_{32}\lesssim\sum_{\substack{|\beta_1|+|\beta_2|=4\\ |\beta_1|\geq 1,1\leq|\beta_2|\leq 2}}|\TP^{\alpha} q|_{-\frac12}\left|\TP^{\beta_1} \vb\cdot\cnab^{\beta_2}\TP\psk\right|_{\frac12}+\left|\TP^{\alpha}\psk\p_3 q\right|_0\left|\TP^{\beta_1} \vb\cdot\cnab^{\beta_2}\TP\psk\right|_0\lesssim \|q\|_4\|v\|_4|\psk|_{3.5}+|\p q|_{L^{\infty}}\|v\|_{3.5}|\psk|_3|\psk|_4.
\end{equation}

For $I_{31}$, we invoke $\h=-\cnab\cdot(\cnab\psk/|\npk|)$ and then integrate $\cnab\cdot$ by parts in the mean curvature term and integrate one tangential derivative by parts in the viscosity term to get:
\begin{equation}
I_{31}\lesssim P(|\cnab\psk|_{\infty})|\p v|_{\infty}\left(\left|\sqrt{\sigma}\cnab\TP^{\alpha}\psk\right|_4^2|\p v|_{\infty}+\left|\kk\TP^{\alpha}\p_t\psi\right|_1\left|\kk\cnab\TP^{\alpha}\psk\right|_0\right)+|\p q|_{L^{\infty}}|\psk|_4^2|\TP v|_{\infty},
\end{equation}and thus yields
\begin{equation}\label{I31}
\int_0^T I_{31}\dt\lesssim\eps\left|\kk\TP^{\alpha}\p_t\psi\right|_{L_t^2H_{x'}^1}^2+\int_0^T P(E^{\kk}(t))\dt,
\end{equation}which together with \eqref{I32} gives the bound for $I_3$:
\begin{equation}\label{AGUbdryI3}
\int_0^T I_3\dt\leq \eps\left|\kk\TP^{\alpha}\p_t\psi\right|_{L_t^2H_{x'}^1}^2+\int_0^T P(E^{\kk}(t))\dt.
\end{equation}

Combining \eqref{AGUbdry}, \eqref{AGUbdry1}, \eqref{AGUbdryI1}, \eqref{AGUbdry2}, \eqref{AGUbdryI2}, \eqref{AGUbdry3}, \eqref{AGUbdryI3}, we get the estimates for the boundary integral after choosing $\eps>0$ suitably small:
\begin{equation}\label{AGU bdry}
-\int_0^T\is \QQ(\VV\cdot\npk)\dx'+ \left|\kk\TP^{\alpha}\p_t\psi\right|_{L_t^2H_{x'}^1([0,T]\times\Sigma)}^2+\frac{\sigma}{2}\left|\cnab\TP^{\alpha}\lkk\psi(T)\right|_0^2\lesssim\PP_0^{\kk}+\int_0^T P(E^{\kk}(t))\dt.
\end{equation}

Plugging the estimates \eqref{AGU TTM0}-\eqref{AGU TTQ0}, \eqref{AGU TTQ} and \eqref{AGU bdry} into \eqref{energy estimate AGU spatial} and using $\rho\gtrsim 1, \p_3\pk\gtrsim 1$, we get the estimates for the good unknowns:
\begin{equation}\label{AGU TT}
\|\VV(T)\|_0^2+\left\|\sqrt{\ff'(q)}\QQ(T)\right\|_0^2+\left|\sqrt{\sigma}\cnab\TP^{\alpha}\lkk\psi(T)\right|_0^2+ \left|\kk\TP^{\alpha}\p_t\psi\right|_{L_t^2H_{x'}^1([0,T]\times\Sigma)}^2\lesssim\PP_0^{\kk}+\int_0^T P(E^{\kk}(t))\dt.
\end{equation}

Finally, using the definition $\VV=\TT^{\alpha}v-\TT^{\alpha}\pk\ppk_3 v$, we can replace $\|\VV\|_0$ by $\|\TT^{\alpha} v\|_0$ because their difference, namely $\TT^{\alpha}\pk\ppk_3 v$, is bounded by $\PP_0^{\kk}+\int_0^T P(E^{\kk}(t))\dt$. Indeed, using $\pk(t,x)=x_3+\chi(x_3)\psk(t,x')$ we only need to investigate the case $\TT=\TP$ because the weighted derivative $\TT=\omega(x_3)\p_3$ only falls on $\chi(x_3)$ and $x_3$ instead of $\psk$. So we have $\|\TP^{\alpha}\pk\|_0\lesssim|\TP^{\alpha}\psk|_0$ which is already bounded by the surface tension energy and thus by $\PP_0^{\kk}+\int_0^T P(E^{\kk}(t))\dt$ according to \eqref{AGU TT}. Since $\|\ppk_3 v\|_{\infty}\leq\|v\|_3\|\p_3\pk\|_{\infty}\leq \PP_0^{\kk}+\int_0^T P(E^{\kk}(t))\dt$, we have
\begin{equation}\label{TT spatial}
\|\TT^{\alpha}v(T)\|_0^2+\left\|\sqrt{\ff'(q)}\TT^{\alpha} \qc(T)\right\|_0^2+\left|\sqrt{\sigma}\cnab\TP^{\alpha}\lkk\psi(T)\right|_0^2+ \int_0^T\left|\kk\TP^{\alpha}\p_t\psi(t)\right|_{1}^2\dt\lesssim\PP_0^{\kk}+\int_0^T P(E^{\kk}(t))\dt.
\end{equation}

We note here that the same analysis can be employed to prove tangential estimates involving mixed spatial-time derivatives. 
\begin{prop}\label{TT spatial time}
Let $\alpha$ be the multi-index satisfying $1\leq \alpha_0\leq 3$ and $|\alpha|=4$, we have:
\begin{equation}
\|\TT^{\alpha}v(T)\|_0^2+\left\|\sqrt{\ff'(q)}\TT^{\alpha} \qc(T)\right\|_0^2+\left|\sqrt{\sigma}\cnab\TP^{\alpha}\lkk\psi(T)\right|_0^2+ \int_0^T\left|\kk\TT^{\alpha}\p_t\psi(t)\right|_{1}^2\dt\lesssim\PP_0^{\kk}+\int_0^T P(E^{\kk}(t))\dt.
\end{equation}
\end{prop}

\subsection{Tangential energy estimate with time derivatives}\label{sect tt AGU}
In this subsection, we study the time-differentiated equations, i.e., the equations obtained by commuting $\p_t^4$ with \eqref{CWWSTkk}. We aim to prove:
\begin{prop}\label{TT time 4'} We have
\begin{equation}\label{TT time 4}
\|\p_t^4 v(T)\|_0^2+\left\|\sqrt{\ff'(q)}\p_t^4 \qc(T)\right\|_0^2+\left|\sqrt{\sigma}\cnab\p_t^4\lkk\psi(T)\right|_0^2+ \int_0^T\left|\kk\p_t^5\psi(t)\right|_{1}^2\dt\lesssim\PP_0^{\kk}+\int_0^T P(E^{\kk}(t))\dt.
\end{equation}
\end{prop}

Although the proof appears to be similar to what has been done in the previous subsection, it should be mentioned that we only have $L^2(\Omega)$-regularity for the full-time derivatives of $v$ and $q$, and thus we do not have any information about their boundary regularity. When the full-time derivatives of $v$ and $q$ appear on the boundary, we use either the artificial viscosity or the Euler equations to reduce a time derivative to a spatial derivative. 

\subsubsection{Alinhac good unknowns for full-time derivatives}
To begin with, we still introduce the Alinhac good unknowns of $v,q$ with respect to $\p_t^4$. Using the same notation as before, we define
\begin{align}
\VV_i : = \p_t^4 v_i - \ppk_3 v_i \p_t^4 \pk, \quad \QQ:=\p_t^4 \qc - \ppk_3 \qc \p_t^4 \pk.
\end{align}

Parallel to \eqref{AGU comm 1} , we have
\begin{equation}
\p_t^4(\nabpk_i f)=\nabpk_i\mathbf{F}+\cc_i(f),
\end{equation}where $\cc_i(f):=\ppk_3\ppk_if\p_t^4\pk+\cc'_i(f)$ and
\begin{align}
\label{AGU comm Cti'}\mathfrak{C}_i'(f) =&
-\left[\p_t^4, \frac{\p_i \pk}{\p_3\pk}, \p_3 f\right]-\p_3 f \left[ \p_t^4, \p_i \pk, \frac{1}{\p_3\pk}\right] +\p_i\pk \p_3 f\left[ \p_t^3, \frac{1}{(\p_3\pk)^2}\right] \p_t \p_3 \pk,~~i=1,2\\
\label{AGU comm Ct3'}\mathfrak{C}_3'(f) = &
\left[  \p_t^4, \frac{1}{\p_3\pk}, \p_3 f\right] + \p_3 f\left[ \p_t^3, \frac{1}{(\p_3\pk)^2}\right]  \p_t \p_3 \pk.
\end{align} 

Then we take $\p_t^4$ to the first two equations of \eqref{CWWST Eulerian} to obtain
\begin{align}
\rho \Dtpk \VV_i + \nabpk_i \QQ =&~\mathcal{R}^1_i,\label{AGU tt mom}\\
\ff'(q) \Dtpk \QQ + \nabpk\cdot \VV=&~\mathcal{R}^2- \mathfrak{C}^i(v_i), \label{AGU tt mass}
\end{align}
where
\begin{align}
\label{rrt1} \mathcal{R}^1_i:=& -[\p_t^4, \rho] \Dtpk v_i-\rho\left(\mathfrak{D}(v_i)+\mathfrak{E}(v_i)\right)-\mathfrak{C}_i(\qc),\\
\label{rrt2} \mathcal{R}^2 :=&-[\p_t^4, \ff'(q)]\Dtpk \qc- \ff'(q)\left(\mathfrak{D}(\qc)+\mathfrak{E}(\qc)\right)+\p_t^4(\ff'(q)gv_3),
\end{align}and the commutators $\dd(f),\ee(f)$ are defined in the same way as in \eqref{AGU comm D} and \eqref{AGU comm 3} by replacing $\TT^{\alpha}$ with $\p_t^4$ and replacing $\TP$ with $\p_t$. The last two terms in \eqref{AGU comm D} vanish because $\p_t^4$ directly commutes with $\p_3$. Analogous to Lemma \ref{AGU TT error}, we list the estimates for commutators $\cc,\dd,\mathfrak{E}$.
\begin{lem}\label{AGU tt error}
Let $\mathbf{F}:= \p_t^4 f - \ppk_3 f \p_t^4 \pk$ be the Alinhac good unknowns of $f$ with respect to $\p_t^4$. Assuming that $\p_3 \pk \geq c_0>0$, then
\begin{align}
\label{AGU tt4 gap} \|\p_t^4 f\|_0 \leq&~\|\mathbf{F}\|_0+ c_0^{-1}\|\p_3 f\|_{\infty}|\p_t^4\psk|_0, \\
\label{C time} \|\mathfrak{C}_i(f)\|_{0} \leq&~P\left(c_0^{-1}, |\cnab\psk|_{\infty},\sum_{k=1}^3|\cnab\p_t^k\psk|_{3-k}\right)\cdot \left(\|\p f\|_{\infty}+\sum_{k=1}^{3}\|\p_t^k f\|_{4-k}\right),\quad i=1,2,3,\\
 \label{D time} \|\dd(f)\|_0\leq&~P\left(c_0^{-1}, |\cnab\psk|_{\infty},\sum_{k=1}^3|\cnab\p_t^k\psk|_{3-k}\right)\cdot \left(\|\p f\|_{\infty}+\sum_{k=1}^{3}\|\p_t^k f\|_{4-k}\right),\\
\label{E time} \|\mathfrak{E}(f)\|_0 \leq&~\kk |\cnab\p_t^5 \psi|_0  \|\p f\|_{\infty}. 
\end{align}
\end{lem}

The $\p_t^4$-differentiated  kinematic boundary condition now reads
\begin{equation}
\p_t^5 \psi + (\vb\cdot\cnab) \TP^4\psk - \VV\cdot \npk=\mathcal{S}_1^*, \quad \text{on}\,\,\Sigma, \label{AGU tt BC kinematic}
\end{equation}
where
\begin{align}\label{AGU tt S1}
\mathcal{S}_1^*:=\p_3 v\cdot \npk \p_t^4\psk+\sum_{1\leq\beta\leq 3} \binom{4}{\beta}\p_t^{\beta} v\cdot\p_t^{4-\beta}\npk.
\end{align}
Also, since $\QQ|_{\Sigma} = \p_t^4 q - \ppk_3 q \p_t^4 \psk$, the boundary condition of $\QQ$ on $\Sigma$ reads
\begin{align}
\QQ = -\sigma \p_t^4\cnab \cdot \left( \frac{\cnab \psk}{\sqrt{1+|\cnab\psk|^2}}\right)+\kk^2(1-\TL)\p_t^5\psi- \p_3 q \p_t^4 \psk. \label{AGU tt BC Q}
\end{align}

\subsubsection{Energy estimates for the full-time derivatives}\label{sect tt AGUr}

Replacing $\TT^{\alpha}$ by $\p_t^4$ in \eqref{energy estimate AGU spatial}, we have
\begin{align}\label{energy estimate AGU time}
\ddt\frac{1}{2} \io \rho |\VV|^2\p_3\pk\dx=\frac{1}{2}\io \rho |\VV|^2\p_3\p_t (\pk-\varphi)\dx +\io \QQ (\nabpk\cdot \VV) \p_3\pk\dx-\is \QQ (\VV\cdot \npk)\dx'+ \io \VV\cdot \mathcal{R}^1\p_3\pk\dx.,
\end{align}where the first term and the last term are controlled in the same way as \eqref{AGU TTM0}-\eqref{AGU TTR1}, so we omit the details. As for the second term, we follow \eqref{AGU TTQ0}-\eqref{AGU TTQ} to get
\begin{equation}\label{AGU ttQ}
\begin{aligned}
&\io \QQ(\nabpk\cdot\VV)\p_3\pk\dx\\
=&\underbrace{-\io\p_t^4 \qc\cc_i(v^i)\p_3\pk\dx}_{=:I_0^*}+\io\p_t^4\pk\ppk_3 \qc\cc_i(v^i)\p_3\pk\dx-\frac12\ddt\left\|\sqrt{\ff'(q)}\QQ\right\|_0^2\\
&+ \left\|\sqrt{\ff'(q)}\QQ\right\|_0^2(\|\p v\|_{\infty}+\kk|\cnab\p_t\psi|_{0.5})+ \left\|\sqrt{\ff'(q)}\QQ\right\|_0\|\mathcal{R}^2\|_0\\
\lesssim&~ I_0^*-\frac12\ddt\left\|\sqrt{\ff'(q)}\QQ\right\|_0^2+\left\|\sqrt{\ff'(q)}\p_t^4 \qc\right\|_0^2\left(\|\p v\|_{\infty}+\kk|\cnab\p_t\psi|_{0.5}+|\kk\cnab\p_t^5\psi|_0\right)\\
&+P\left(c_0^{-1}, |\cnab\psk|_{\infty},\sum_{k=1}^3|\cnab\p_t^k\psk|_{3-k}\right)|\p_t^4\psk|_0|\left\|\sqrt{\ff'(q)}\p_t^4 \qc\right\|_0\left(\|\p v,\p q\|_{\infty}+\sum_{k=1}^{3}\|\p_t^k \qc,\p_t^k v\|_{4-k}+\|\ff'(q)\p_t^4 v_3\|_0\right).
\end{aligned}
\end{equation}

At this point, we are not able to control $I_0^*:=-\io \p_t^4 q \cc_i(v^i)\p_3\pk\dx$ as in \eqref{AGU TT CV} since this requires the bound for $\|\p_t^4 q\|_0$. We can only obtain the control of $\|\sqrt{\ff'(q)}\p_t^4 q\|_0$ from the energy estimate because we can no longer use the momentum equation to reduce $\p_t^4 q$ due to the lack of spatial derivatives. \textit{Although the method in \eqref{AGU TT CV} is still valid here when we prove the well-posedness, provided that $\ff'(q)$ is bounded from below, we would like to show that our estimate can be adjusted to be uniform in $\ff'(q)$.} To achieve this, we find that the problematic terms in $\cc_i(v^i)$ can be exactly canceled by the boundary error term $\sss_1$ defined in \eqref{AGU tt S1}. Therefore, this term should be controlled together with the boundary integral if we want our energy estimates to be uniform in the Mach number.

Next, we analyze the boundary integral. Most of the steps are parallel to Section \ref{sect AGU TT ST} if we replace $\TP^{\alpha}$ by $\p_t^4$, so we will omit the details of those repeated steps but only list the different steps. Plugging the boundary conditions \eqref{AGU tt BC kinematic} and \eqref{AGU tt BC Q} into $-\is\QQ(\VV\cdot\npk)\dx'$, we get
\begin{align}\label{AGUttbdry}
-\is \QQ (\VV\cdot \npk)\dx'= -\is \QQ \p_t^5 \psi \dx'-\is \QQ (\vb\cdot\cnab) \p_t^4\psk \dx' +\is\QQ\mathcal{S}_1^*\dx'=:I_1^*+I_2^*+I_3^*,
\end{align}and $I_1^*$ is further divided into three parts:
\begin{equation}\label{AGUttbdry1}
\begin{aligned}
I_1^*:=-\is \QQ\p_t^5\psi\dx'=&~\sigma\is  \p_t^4\cnab\cdot\left(\frac{\cnab\psk}{\sqrt{1+|\cnab\psk|^2}}\right)\p_t^5\psi\dx'-\kk^2\is \p_t^4(1-\TL)\p_t\psi\cdot\p_t^5\psi\dx'+\is\p_3 q \p_t^4\psk\p_t^5\psi\dx'\\
=&:\ST_{1}^*+\ST_{2}^*+\RT^*.
\end{aligned}
\end{equation}

Mimicking the steps \eqref{ST2}-\eqref{STbound}, we can get the bounds for$\ST_1^*,\ST_2^*$:
\begin{equation}\label{STttbound}
\int_0^T\ST_1^*+\ST^*_2\dt+\left|\kk\p_t^5\psi\right|_{L_t^2H_{x'}^1}^2+\frac{\sigma}{2}\left|\cnab\p_t^4\lkk\psi(T)\right|_0^2\lesssim\PP_0^{\kk}+\int_0^T P(E^{\kk}(t))\dt.
\end{equation}

\begin{rmk}
Parallel to the remark after \eqref{RTbound}, $-\int_0^T \RT^*\dt$ would contribute to the non-$\sigma$-weighted energy $\is (-\p_3 q) |\p_t^4 \lkk\psi|^2\dt$ if the Rayleigh-Taylor sign condition holds. This will be revisited in Section \ref{section double lim}. 
\end{rmk}

As for$\RT^*$, if we still follow \eqref{RTbound} to get:
\[
\int_0^T\RT^*\dt\lesssim\eps\left|\kk\p_t^5\psi\right|_{L_t^2L_{x'}^2}^2+\int_0^T P\left(\|\qc\|_{4},\left|\p_t^4\lkk\psi\right|_0,\left|\p_t^5\lkk\psi\right|_0\right)\dt,
\] then we find that the term $|\p_t^5\lkk\psi|_0$ is not included in $E^{\kk}(t)$ because there is no spatial derivative here. To overcome this, we invoke the kinematic boundary condition $\p_t\psi=-\vb\cdot\cnab\psk+v_3$ and take $\p_t^4$ to get
\begin{equation}\label{pskt5}
\p_t^5\psi=-(\vb\cdot\cnab)\p_t^4\psk+\p_t^4v_3-[\p_t^4,\vb\cdot]\cnab\psk=-(\vb\cdot\cnab)\p_t^4\psk+\p_t^4v\cdot\npk-[\p_t^4,\vb\cdot,\cnab\psk],
\end{equation}and thus
\begin{align}\label{RTtt0}
\RT^*=&-\is\p_3 q\p_t^4\psk(\vb\cdot\cnab)\p_t^4\psk\dx'+\is\p_3q \p_t^4\psk\p_t^4v\cdot\npk\dx'-\is\p_3 q \p_t^4\psk [\p_t^4,\vb\cdot,\cnab\psk]\dx'\\
=&:\RT_1^*+\RT_2^*+\RT_3^*. \nonumber
\end{align}

Note that we only need to analyze the contribution of$\RT_2^*$ because the contribution of the other two terms will be canceled by part of $I_2^*$ and $I_3^*$. To do this, we need to derive the equation for $\p_t^4\cdot\npk$ on $\Sigma$. Recall that $$\Dtpk\big|_{\Sigma}=\p_t+(\vb\cdot\cnab)+(\p_3\psk)^{-1}\underbrace{(v\cdot\Npk-\p_t\varphi)}_{=0\text{ on }\Sigma}\p_3=\p_t+(\vb\cdot\cnab),$$ we have the following identity by projecting the momentum equation onto the direction of $\npk$ on $\Sigma$:
\[
\rho\p_t v\cdot\npk=-(\rho-1)g-\rho(\vb\cdot\cnab)v\cdot\npk+\cnab\psk\cdot\cnab \qc-(1+|\cnab\psk|^2)\p_3 \qc,
\]and thus
\begin{equation}\label{RTtt20}
\rho\p_t^4 v\cdot\npk\lleq -\p_t^3\rho g -\rho(\vb\cdot\cnab)\p_t^3 v\cdot\npk+\cnab\psk\cdot\cnab\p_t^3 \qc-|\npk|^2\p_3\p_t^3 \qc.
\end{equation}
 The contribution of the first three terms in \eqref{RTtt20} can be directly controlled after integrating $\cnab$ by parts and using the Sobolev trace lemma:
\begin{equation}\label{RTtt2}
\begin{aligned}
&\is\rho^{-1}\p_3 q\p_t^4\psk\left(\rho(\vb\cdot\cnab)\p_t^3v\cdot\npk+\cnab\psk\cdot\cnab\p_t^3 \qc-\p_t^3\rho g\right)\dx'\\
\lleq&-\is \rho^{-1}\cnab\p_t^4\psk\cdot\left(\p_3 q\left(\rho\vb\p_t^3v\cdot\npk+\cnab\psk\p_t^3 \qc\right)\right)\dx'-\is \rho^{-1}\p_3 q\p_t^4\psk \p_t^3\rho g\dx'\\
\lesssim&~\|\p q\|_2\left(\left|\cnab\p_t^4\psk\right|_0 P\left(\|\p_t^3 v\|_1,\|\p_t^3 \qc\|_1,|\cnab\psk|_{\infty}\right)+|\p_t^4\psk|_0\|\ff'(q)\p_t^3 q\|_1|\rho|_{\infty}\right).
\end{aligned}
\end{equation}
\begin{rmk}
Note that the right side of \eqref{RTtt2} involves $|\cnab\p_t^4\psk|_0$ whose control relies on $\sigma^{-1}$. This is due to the lack of the Rayleigh-Taylor sign condition. When taking the zero surface tension limit, the Rayleigh-Taylor sign condition is assumed, and thus the RT term can be directly controlled.
\end{rmk}

Then, for the last term, we need to do the same reduction for $\p_t^4\psi$:
\begin{equation}\label{pskt4}
\p_t^4\psi=-(\vb\cdot\cnab)\p_t^3\psk+\p_t^3v_3-[\p_t^3,\vb\cdot]\cnab\psk=-(\vb\cdot\cnab)\p_t^3\psk+\p_t^3v\cdot\npk-[\p_t^3,\vb\cdot,\cnab\psk].
\end{equation}Using \eqref{pskt4} and Sobolev trace lemma, it is controlled by
\begin{equation}\label{pskt4 L2}
|\p_t^4\psk|_0\lesssim P(|\cnab\psk|_{\infty},|\cnab\p_t\psk|_{\infty})\left(|\cnab\p_t^3\psk|_0+\|\p_t^3 v\|_1+\|\p_t^2 v\|_2+|\cnab\p_t^2\psk|_0\right).
\end{equation}

Now we plug the equality above into the boundary integral $-\is\rho^{-1}\p_3q|\npk|^2\p_t^4\psk\p_3\p_t^3 \qc\dx'$. Note that the unit exterior normal vector to $\Sigma$ is $(0,0,1)^{\top}$ (not the Eulerian normal vector $\npk$ !), we can use the divergence theorem to rewrite the boundary integral into the interior, and integrate by parts in $\p_t$ to get the following estimate:
\begin{equation}\label{RTtt4}
\begin{aligned}
&-\int_0^T\is\rho^{-1}\p_3 q|\npk|^2\p_t^4\psk\p_3\p_t^3\qc\dx'\dt\lleq \int_0^T\is\rho^{-1}\p_3 q|\npk|^2\lkk^2\left((\vb\cdot\cnab)\p_t^3\psk-\p_t^3v\cdot\npk\right)\p_3\p_t^3\qc\dx'\dt\\
=&\int_0^T\io\p_3\left(\rho^{-1}\p_3 q|\Npk|^2\lkk^2\left((\vb\cdot\cnab)\p_t^3\pk-\p_t^3v\cdot\Npk\right)\p_3\p_t^3\qc\right)\dx\dt\\
\lleq&\int_0^T\io\rho^{-1}\p_3 q|\Npk|^2\lkk^2\left((\vb\cdot\cnab)\p_t^3\pk-\p_t^3v\cdot\Npk\right)\cdot\p_3^2\p_t^3 \qc\dx\dt\\
\overset{\p_t}{=}&-\io\rho^{-1}\p_3q|\Npk|^2\lkk^2\left((\vb\cdot\cnab)\p_t^3\pk-\p_t^3v\cdot\Npk\right)\cdot\p_3^2\p_t^2 \qc\dx\\
&+\int_0^T\io\rho^{-1}\p_3q|\Npk|^2\p_t\lkk^2\left((\vb\cdot\cnab)\p_t^3\pk-\p_t^3v\cdot\Npk\right)\cdot\p_3^2\p_t^2 \qc\dx\dt\\
\lesssim&~\eps\|\p_t^2 \p^2 \qc\|_0^2+\PP_0^{\kk}+\int_0^T P\left(\|\p_t^4 v\|_0,\|\p_t^3 v\|_1,\|\p_t v\|_{\infty},\|\p_t^2\qc\|_2,|\cnab\psk|_{\infty},|\cnab\p_t^3\psk|_0,|\cnab\p_t^4\psk|_0\right)\dt.
\end{aligned}
\end{equation}

Combining this with \eqref{STttbound}, \eqref{RTtt0}, \eqref{RTtt2} and \eqref{RTtt4}, we get the estimate for $I_1^*$:
\begin{equation}\label{AGUttbdryI1}
\int_0^T I_1^*\dt+\left|\kk\p_t^5\psi\right|_{L_t^2H_{x'}^1}^2+\frac{\sigma}{2}\left|\cnab\p_t^4\lkk\psi(T)\right|_0^2\lesssim\eps\|\p_t^2\p^2\qc\|_0^2+\int_0^T\RT_1^*+\RT_3^*\dt+\PP_0^{\kk}+\int_0^T P(E^{\kk}(t))\dt,
\end{equation}
after choosing $\eps>0$ that appears above to be suitably small.

Next we expand $I_2^*,I_3^*$ defined in \eqref{AGUttbdry}
\begin{equation}\label{AGUttbdryI2}
\begin{aligned}
I_2^*+I_3^*=&-\is\p_t^4 \qc(\vb\cdot\cnab)\p_t^4\psk\dx'+\is\p_t^4\psk\p_3 q(\vb\cdot\cnab)\p_t^4\psk\dx'\\
&+\is\p_t^4 \qc\sss_1\dx'-\is\p_t^4\psk\p_3 q\p_3 v\cdot\npk\p_t^4\psk\dx'-\is\p_t^4\psk\p_3 q\left(\sum_{1\leq \beta\leq 3}\binom{4}{\beta}\p_t^{\beta}v\cdot\p_t^{4-\beta}\npk\right)\dx'
\end{aligned}
\end{equation}and we find that the second term exactly cancels$\RT_1^*$ and the fifth term exactly cancels$\RT_3^*$ defined in \eqref{RTtt0}. The first term can be controlled in the same way as $I_{21}, I_{22}$ defined in \eqref{AGUbdry2} after replacing $\TP^4$ by $\p_t^4$. The fourth term is directly controlled by $P(E^{\kk}(t))$ by using the Sobolev trace lemma. 

Hence, it suffices to analyze the third term. Using the definition of $\sss_1^*$, we have
\begin{equation}\label{AGUttbdry3}
\is\p_t^4 \qc\sss_1^*\dx'=\is\p_t^4 \qc\left(\p_3v\cdot\npk\p_t^4\psk\right)\dx'-4\is\p_t^4 \qc\p_t^3v\cdot\p_t\npk\dx' +\sum_{1\leq\beta\leq 2}\binom{4}{\beta}\is\p_t^4\qc \p_t^{\beta}v\cdot\cnab\p_t^{4-\beta}\psk\dx',
\end{equation}
where the surface tension energy can control the first term after invoking \eqref{AGU tt BC Q} and integrating $\cnab$ by parts; and the last term can be controlled after integrating by parts in $\p_t$ under the time integral. But for the remaining term
\begin{equation}\label{AGUttI30}
I_{30}^*:=4\is\p_t^4\qc\p_t^3v\cdot\p_t\npk\dx',
\end{equation}
we have neither the $L^2(\Sigma)$-regularity of $\p_t^4 q$ nor the possibility of integrating $\frac{1}{2}$-time derivatives by parts as in the control of \eqref{AGUbdry3}.

Fortunately, we can still control $I_{30}$ \textit{together with the interior term} $I_0^*:=-\io\p_t^4q\cc_i(v^i)\p_3\pk\dx$ defined in \eqref{AGU ttQ}. In fact, invoking \eqref{AGU comm Cti'} and \eqref{AGU comm Ct3'}, we know $\cc_i(v^i)$ includes the following terms involving $\geq 3$ time derivatives of $v^i$ and $\geq 4$ derivatives of $\pk$:
\begin{align}
\label{tt CV 1} \ppk_3\ppk_iv^i\p_t^4\pk&=\cc_i(v^i)-\cc'_i(v^i),~~i=1,2,3,\\
\label{tt CV 2} -4\p_t\left(\frac{\TP_i\pk}{\p_3\pk}\right)\p_t^3\p_3v^i&=4\p_t\Npk_i\ppk_3\p_t^3 v^i+4\frac{\p_3\p_t\pk\TP_i\pk}{\p_3\pk}\ppk_3\p_t^3v^i \text{ from the first commutator in }\cc_i'(v^i)~~i=1,2,\\
\label{tt CV 3} 4\p_t\left(\frac{1}{\p_3\pk}\right)\p_t^3\p_3v^3&=-4\frac{\p_3\p_t\pk}{\p_3\pk}\ppk_3\p_t^3v^3 \text{  from the first commutator in }\cc_3'(v^3),
\end{align}while the terms in $\cc_i'(v^i)$ containing only $\leq 2$ time derivatives of $v^i$ and $\leq 3$ time derivatives of $\pk$ are controlled directly after integrating $\p_t$ by parts under time integral.

The contribution of the above four terms in $I_0^*$ is divided into three parts:
\begin{align}
\label{ttI000} I_{00}^*&:=-4\io\p_t^4 \qc\p_t\Npk_i\p_3\p_t^3v^i\dx\\
\label{ttI010} I_{01}^*&:=-\io\p_t^4 \qc\p_3(\nabpk\cdot v)\p_t^4\pk\dx\\
\label{ttI020} I_{02}^*&:=-4\sum_{i=1}^2\io\p_t^4 \qc\left(\frac{\p_3\p_t\pk\TP_i\pk}{\p_3\pk}\right)\ppk_3\p_t^3v^i \p_3\pk\dx+4\io\p_t^4 \qc\left(\frac{\p_3\p_t\pk}{\p_3\pk}\right)\ppk_3\p_t^3v^3 \p_3\pk\dx.
\end{align}

Integrating $\p_3$ by parts in $I_{00}^*$ and using $N_3=1$, we find the boundary term exactly cancels with $I_{30}^*$, so we have:
\begin{equation}
\begin{aligned}
I_{30}^*+I_{00}^*=&~4\io\left(\p_t^4 \p_3\qc\p_t\Npk+\p_t^4\qc\p_t\p_3\Npk\right)\cdot\p_t^3v\dx\\
=&\ddt\io\left(\p_t^3\p_3\qc\p_t\Npk+\p_t^3 \qc\p_t\p_3\Npk\right)\cdot\p_t^3 v\dx+\io\p_t^3 \p_3\qc\p_t(\p_t\Npk\cdot\p_t^3v)+\p_t^3\qc\p_t(\p_t\p_3\Npk\cdot\p_t^3v)\dx.
\end{aligned}
\end{equation}
Under the time integral, we have the following bounds after using $\eps$-Young's inequality:
\begin{align}\label{AGUttI00'}
\int_0^TI_{00}^*+I_{30}^*\dt\lesssim\eps\|\p_t^3 \p \qc\|_0^2+\PP_0^{\kk}+\int_0^T(\|\p_t^3 \qc(t)\|_0+1)P(E^{\kk}(t))\dt\nonumber\\
\lesssim\eps\|\p_t^3 \p \qc\|_0^2+\eps\int_0^T \|\p_t^3 \qc(t)\|_0^2\dt+\PP_0^{\kk}+\int_0^T P(E^{\kk}(t))\dt.
\end{align}
Here, we still need to study $\int_0^T \|\p_t^3 \qc(t)\|_0^2\dt$, as the reduction scheme does not apply to $\p_t^3\qc(t)$ due to lack of spatial derivatives.  We control this term through the fundamental theorem of calculus: For each $x_3\in (-b,0)$, we write 
$$
\p_t^3 \qc (t, x', x_3) = \p_t^3\qc(t, x', 0) +\int_{0}^{x_3} \p_3 \p_t^3\qc(t, x', z)\,dz,
$$
and so
$$
\left(\p_t^3 \qc(t,x',x_3)\right)^2 \lesssim_b \left(\p_t^3\qc(t,x',0)\right)^2 + \int_0^{x_3} \left(\p_3\p_t^3 \qc(t,x',z)\right)^2\,dz.
$$
Integrating both sides with respect to the spatial variables, we obtain
$$
\|\p_t^3 \qc(t)\|_0^2 \lesssim_b |\p_t^3 \qc(t)|_0^2 + \|\p_3 \p_t^3 \qc(t)\|_0^2.
$$
The second term on the RHS is bounded by $E^\kk(t)$. For the first term, note that
$$\p_t^3\qc = -\sigma \p_t^3 \left(\frac{\cnab\psk}{\sqrt{1+|\cnab\psk|^2}}+g\psk\right)+\kk^2 (1-\TL)\p_t^4\psi,\quad\text{on}\,\,\,\Sigma,
$$ 
where
$$
-\sigma \p_t^3 \left(\frac{\cnab\psk}{\sqrt{1+|\cnab\psk|^2}}\right) \lleq -\sigma \cnab\cdot \left(\frac{\cnab \p_t^3 \psk}{\sqrt{(1+|\cnab \psk|^2})}-\frac{\cnab\psk\cdot \cnab \p_t^3\psk}{{\sqrt{(1+|\cnab \psk|^2})}^3}\cnab\psk\right),
$$
which indicates that $\p_t^3\qc|_{\Sigma}$ consists of non-$\kk$-weighted terms with at most 5 derivatives on $\psk$ with at most 3 times derivatives, and a $\kk$-weighted term $\kk^2(1-\TL)\p_t^4\psi$. Therefore,  
$$
 \int_0^T \|\p_t^3 \qc(t)\|_0^2\dt\lesssim \int_0^T P \left(\sum_{k=0}^3|\sqrt{\sigma}\cnab \p_t^k \lkk \psi(t)|_{4-k}, \sum_{k=0}^3|\p_t^k \lkk \psi(t)|_{4-k}\right)\dt+\int_0^T|\kk \p_t^4\psi(t)|_2^2\dt
$$ 
By combining this with \eqref{AGUttI00'}, we conclude:
\begin{equation}\label{AGUttI00}
\int_0^TI_{00}^*+I_{30}^*\dt\lesssim\eps\|\p_t^3 \p \qc\|_0^2+\eps \int_0^T\|\kk \p_t^4 \psi(t)\|_{2}^2\dt+\PP_0^{\kk}+\int_0^T P(E^{\kk}(t))\dt.
\end{equation}

Next, the term $I_{01}^*$ can be directly controlled if we insert the continuity equation $\nabpk\cdot v=-\ff'(q)\Dtpk q$
\begin{equation}\label{AGUttI01}
I_{01}^*\lesssim\left\|\sqrt{\ff'(q)}\p_t^4 \qc\right\|_0\left\|\sqrt{\ff'(q)}\p_t q,\sqrt{\ff'(q)}\p q\right\|_{W^{1,\infty}}|\p_t^4\psk|_0.
\end{equation}

As for $I_{02}^*$, we note that $-\TP_i\pk\ppk_3\p_t^3v^i=\pp_i\p_t^3v^i-\TP_i\p_t^3v_i$ for $i=1,2$. So it becomes
\begin{equation}
\begin{aligned}
I_{02}^*=&~4\io\p_t^4 \qc \p_3\p_t\pk(\nabpk\cdot\p_t^3 v)\dx-4\sum_{i=1}^2\io\p_t^4 \qc \p_3\p_t\pk\TP_i\p_t^3 v_i\dx\\
\lleq&~4\io\p_t^4\qc\p_t^3(\nabpk\cdot v)\p_3\p_t\pk\dx-4\sum_{i=1}^2\io\p_t^4 \qc \p_3\p_t\pk\TP_i\p_t^3 v_i\dx,
\end{aligned}
\end{equation}where the first term is controlled by 
$$
\left\|\sqrt{\ff'(q)}\p_t^4 \qc\right\|_0\left(\left\|\sqrt{\ff'(q)}\p_t^4 \qc\right\|_{0}+\left\|\sqrt{\ff'(q)}\p_t^3 \p \qc\right\|_{0}+\left\|\sqrt{\ff'(q)}\p_t^3 v_3\right\|_{0}\right)|\p_t\psk|_{\infty}
$$ 
after invoking the continuity equation, and the second term is controlled under a time integral after integrating by parts first in $\p_t$ and then in $\TP_i$. So we have:
\begin{equation}\label{AGUttI02}
\int_0^TI_{02}^*\dt\lesssim~\eps\|\p_t^3 \TP \qc\|_0^2+\PP_0^{\kk}+\int_0^T P(E^{\kk}(t))\dt.
\end{equation}

Summarizing \eqref{energy estimate AGU time}, \eqref{AGU ttQ}, \eqref{STttbound}, \eqref{RTtt0}, \eqref{AGUttbdryI1}-\eqref{AGUttI30}, \eqref{AGUttI00}, \eqref{AGUttI01} and \eqref{AGUttI02}, we finally get the control of the Alinhac good unknowns $\VV$ and $\QQ$ with respect to $\p_t^4$:
\begin{equation}\label{AGU tt}
\|\VV(T)\|_0^2+\left\|\sqrt{\ff'(q)}\QQ(T)\right\|_0^2+\left|\sqrt{\sigma}\cnab\p_t^4\lkk\psi(T)\right|_0^2+\int_0^T\left|\kk\p_t^5\psi\right|_{1}^2\dt\lesssim \eps\left(\|\p_t^2 \p^2\qc\|_0^2+\|\p_t^3 \p \qc\|_0^2\right)+\PP_0^{\kk}+\int_0^T P(E^{\kk}(t))\dt.
\end{equation}

To recover the energy for $\|\p_t^4 v\|_0^2$ and $\|\sqrt{\ff'(q)}\p_t^4 \qc\|_0^2$, it suffices to invoke \eqref{AGU tt4 gap} and use the estimate of $|\p_t^4\psk|_0$ in \eqref{pskt4 L2}. Note that the right side of \eqref{pskt4 L2} has been controlled in $\TP^{4-k}\p_t^k$-estimates for $k\leq 3$, so we already have $|\p_t^4\psk|_0\leq \PP_0^{\kk}+\int_0^T P(E^{\kk}(t))\dt$ and thus
\begin{equation}\label{tt time}
\|\p_t^4 v(T)\|_0^2+\left\|\sqrt{\ff'(q)}\p_t^4 \qc(T)\right\|_0^2+\left|\sqrt{\sigma}\cnab\p_t^4\lkk\psi(T)\right|_0^2+\int_0^T\left|\kk\p_t^5\psi\right|_{1}^2\dt\lesssim \eps\left(\|\p_t^2 \p^2\qc\|_0^2+\|\p_t^3 \p \qc\|_0^2\right)+\PP_0^{\kk}+\int_0^T P(E^{\kk}(t))\dt.
\end{equation}

\subsection{A priori estimates for the nonlinear $\kk$-approximate problem}\label{subsect: 4.7}
The goal of this section is to complete the proof of Theorem \ref{prop uniform kk energy est}. We choose $\eps>0$ suitably small and then combine the tangential estimates \eqref{TT spatial} and \eqref{tt time} with div-curl analysis, reduction of pressure and $L^2$-estimates in Section \ref{sect kkL2}--Section \ref{sect div curl} to get the following energy inequality:
\begin{equation}\label{full kk energy inequality}
\sup_{t\in [0,T]}E^{\kk}(t)\leq P(E^{\kk}(0))+\int_0^T P(E^{\kk}(t))\dt.
\end{equation}
This implies
$$
 \mathcal{E}(t')  \leq P(E^{\kk}(0))+ TP(\mathcal{E}(t')), \quad\text{for all }t'\in [0, T], 
$$
with $\mathcal{E}(t') = \sup_{s\in [0,t']}E^{\kk}(s)$. Here, 
Since the right-hand side of the energy inequality does not depend on $\kk^{-1}$, 
we can use the generalized Gr\"onwall's inequality (Theorem \ref{thm: generalized gronwall}) to show that there exists some $T=T_0>0$, independent of $\kk>0$, such that
\begin{equation}
\sup_{0\leq t\leq T_0}E^{\kk}(t)\leq 2P(E^{\kk}(0)).
\end{equation}

We also note that the above energy estimate does not rely on $\ff'(q)^{-1}$, as a special cancellation structure enjoyed by the Alinhac good unknowns and delicate analysis \eqref{AGUttbdry3}-\eqref{AGUttI02} excludes the only possibility that might make the energy estimates not uniform in Mach number. Therefore, our a priori bound is also uniform in Mach number. 

\section{Well-posedness of the nonlinear $\kk$-approximate system}\label{sect linear LWP}
For the nonlinear $\kk$-approximate problem \eqref{CWWSTkk}, we have established the uniform-in-$\kk$ estimates. Once we prove the well-posedness of \eqref{CWWSTkk} for each \textit{fixed} $\kk>0$, we can take the limit $\kk\rightarrow 0$ to prove the local existence of the original system \eqref{CWWST}. We will use Picard iteration to construct the solution to \eqref{CWWSTkk} for each fixed $\kk>0$.  We start with $(v^{(0)},\rho^{(0)},\psi^{(0)}):=(\mathbf{0},1,0)$ and also define $\psi^{(-1)}:=\psi^{(0)}$. Then we construct the solution by the following iteration scheme: For any $n\geq 0$, given $\{(v^{(k)},\rho^{(k)},\psi^{(k)})\}_{k\leq n}$, we define $(v^{(n+1)},\rho^{(n+1)},\psi^{(n+1)})$ to be the solution to the following linear system whose coefficients depend on $(v^{(n)},\rho^{(n)},\psi^{(n)})$ and $\psi^{(n-1)}$:
\begin{equation}\label{CWWSTlin0}
\begin{cases}
\rho^{(n)}  D_t^{\pk^{(n)}} v^{(n+1)} +\nab^{\pk^{(n)}} \qc^{(n+1)}=-(\rho^{(n)}-1) ge_3&~~~ \text{in}\,\,[0,T]\times \Omega,\\
\ff^{(n)'}(q^{(n)})D_t^{\pk^{(n)}} \qc^{(n+1)}+ \nab^{\pk^{(n)}}\cdot v^{(n+1)}=\ff^{(n)'}(q^{(n)})gv_3^{(n)} &~~~ \text{in}\,[0,T]\times \Omega,\\
q^{(n+1)}=q^{(n+1)}(\rho^{(n+1)}), \qc^{(n+1)}=q^{(n+1)}+g\pk^{(n)}&~~~ \text{in}\, [0,T]\times \Omega, \\
\qc^{(n+1)}=g\psk^{(n)}-\sigma\cnab \cdot \left( \frac{\cnab \psk^{(n)}}{\sqrt{1+|\cnab\psk^{(n)}|^2}}\right)+\kk^2(1-\TL)(v^{(n+1)}\cdot \npk^{(n)}) &~~~\text{on}~[0,T]\times\Sigma, \\
\p_t \psi^{(n+1)} = v^{(n+1)}\cdot \npk^{(n)} &~~~\text{on}~[0,T]\times\Sigma,\\
v^{(n+1)}_3=0 &~~~\text{on}~[0,T]\times\Sigma_b,\\
(v^{(n+1)},\rho^{(n+1)},\psi^{(n+1)})|_{t=0}=(v_{0,\kk}, \rho_{0,\kk}, \psi_{0,\kk}),
\end{cases}
\end{equation}where for any $k\leq n+1$, $\varphi^{(k)}(t,x)$ is the extension of $\psi^{(k)}$ defined by $\varphi^{(k)}(t,x):=x_3+\chi(x_3)\psi^{(k)}$ and $\pk^{(k)}:=x_3+\chi(x_3)\psk^{(k)}$ is the smoothed version of $\varphi^{(k)}$. The linearized material derivative is defined to be the following linear operator:
\begin{align}
\label{Dtpk lin} D_t^{\pk^{(n)}}:=\p_t+\vb^{(n)}\cdot\cnab+\frac{1}{\p_3\pk^{(n)}}(v^{(n)}\cdot\Npk^{(n-1)}-\p_t\varphi^{(n)})\p_3,
\end{align}and the covariant derivatives are defined to be
\begin{align}
\p_t^{\pk^{(n)}} &= \p_t -\frac{\p_t \varphi^{(n)}}{\p_3\pk^{(n)}}\p_3, \label{ptk lin}\\
\nab^{\pk^{(n)}}_a &=\p^{\pk^{(n)}}_a = \p_a -\frac{\p_a \pk^{(n)}}{\p_3\pk^{(n)}}\p_3,\quad a=1,2,\label{nabpk lin 12}\\
\nab^{\pk^{(n)}}_3 &= \p^{\pk^{(n)}}_3 = \frac{1}{\p_3 \pk^{(n)}} \p_3.\label{nabpk lin 3}
\end{align}
\begin{rmk}
Note that the linearized material derivative is no longer equal to $\p_t^{\pk^{(n)}}+v^{(n)}\cdot\nab^{\pk^{(n)}}$. Indeed, one has to set the weight of $\p_3$ to be $v^{(n)}\cdot\Npk^{(n-1)}-\p_t\varphi^{(n)}$ to guarantee both the linearity of this operator and the consistency with the linearized kinematic boundary condition $\p_t \psi^{(n+1)} = v^{(n+1)}\cdot \npk^{(n)}$.
\end{rmk}
\begin{rmk}
Note that the surface tension term in \eqref{CWWSTlin0} is now replaced by a given term instead of being $-\sigma\cnab\cdot(\cnab\psk^{(n+1)}/|\npk^{(n)})$. Under this setting, we can still do energy estimates for $\psi^{(n+1)}$ by using the kinematic boundary condition and the viscosity term.
\end{rmk}

For simplicity of notations, for any $n\geq 0$, we denote $(v^{(n+1)},\rho^{(n+1)},q^{(n+1)},\psi^{(n+1)})$, $(v^{(n)},\rho^{(n)},q^{(n)},\psi^{(n)})$ and $\psi^{(n-1)}$ by $(v,\rho,q,\psi)$, $(\vr,\rhor,\qr,\psr)$ and $\psd$. Hence, we need to solve the following linearized version of system \eqref{CWWSTkk} for each fixed $\kk>0$ and then establish an energy estimate to proceed with the iteration scheme. 
\begin{equation}\label{CWWSTlin}
\begin{cases}
\rhor \Dtpkr v +\nabpkr \qc=-(\rhor-1) ge_3,&~~~ \text{in}\,\,[0,T]\times \Omega,\\
\ffr'(\qr)\Dtpkr \qc+ \nabpkr\cdot v=\ffr'(\qr)g\vr_3, &~~~ \text{in}\,[0,T]\times \Omega,\\
q=q(\rho), \qc=q+g\pkr&~~~ \text{in}\, [0,T]\times \Omega, \\
\qc=g\pskr-\sigma\cnab \cdot \left( \frac{\cnab \pskr}{\sqrt{1+|\cnab\pskr|^2}}\right)+\kk^2(1-\TL)(v\cdot \npkr), &~~~\text{on}~[0,T]\times\Sigma, \\
\p_t \psi = v\cdot \npkr, &~~~\text{on}~[0,T]\times\Sigma,\\
v_3=0 &~~~\text{on}~[0,T]\times\Sigma_b,\\
(v,\rho,\psi)|_{t=0}=(v_0^\kk, \rho_0^\kk, \psi_0^\kk).
\end{cases}
\end{equation}
Here $\ffr:=\log\rhor$. The linearized material derivative now becomes:
\begin{align}  \Dtpkr:=\p_t+\vbr\cdot\cnab+\frac{1}{\p_3\pkr}(\vr\cdot\Npkd-\p_t\pr)\p_3
\label{Dtpkr}
\end{align}and the covariant derivatives with respect to $\pkr$ are defined to be
\begin{align}
\ppkr_t &:= \p_t -\frac{\p_t \pr}{\p_3\pkr}\p_3, \label{ptkr}\\
\nabpkr_a &=\ppkr_a = \p_a -\frac{\p_a \pkr}{\p_3\pkr}\p_3,\quad a=1,2,\label{nabpkr 12}\\
\nabpkr_3 &= \ppkr_3 = \frac{1}{\p_3 \pkr} \p_3,\label{nabpkr 3}
\end{align}
where $\vbr \cdot \cnab := \vr_1 \p_1+\vr_2\p_2$. Note that, by the kinematic boundary condition, the normal component in $\Dtpkr$, namely $(\p_3\pkr)^{-1}(\vr\cdot\Npkd-\p_t\pr)\p_3$ vanishes on $\Sigma$. 

From now on, we assume the following given quantities are bounded in some time interval $t\in [0, T^{\kk}]$. This also works as the induction hypothesis for the uniform-in-$n$ estimates for \eqref{CWWSTlin}:
\begin{equation} \label{linearized assumption}
\begin{aligned}
\|\rhor-1\|_0^2+\sum_{k=0}^4\|\p_t^k \vr\|_{4-k}^2+\|\dot{\ff}'(\dot{q})\qc\|_0^2+\|\p\mathring{\qc}\|_3^2+\sum_{k=1}^3\|\p_t^k \mathring{\qc}\|_{4-k}^2+\|\dot{\ff}'(\dot{q})\p_t^4\mathring{\qc}\|_0^2&\\
+\kk^4|\psr|_{5.5}^2+\sum_{k=0}^3\kk^4|\p_t^{k+1}\psr,\p_t^{k+1}\psd|_{5.5-k}^2+\kk^2\int_0^t|\p_t^5\psr|_1^2\mathrm{d}\tau&<\Kr_0.
\end{aligned}
\end{equation}
Here, the additional $\frac{1}{2}$-regularity for $\p_t^j \psr $ and $\p_t^j\psd$, $j=0,1,2,3,4$ is contributed by the artificial viscosity whenever $\kk>0$ is fixed. 

\subsection{Construction of solution to the linearized approximate system}
We can prove that system \eqref{CWWSTlin} is a symmetric hyperbolic system with characteristic boundary conditions. Therefore, we want to use the duality argument developed by Lax-Phillips \cite{lax1960local} to prove the local existence. Before proceeding, we must ensure that the boundary conditions are homogeneous. 
\subsubsection{The homogeneous linearized approximate system}
We introduce the variable $\hhr$ defined by the harmonic extension
\begin{equation}
\begin{cases}
-\lap\hhr=0~~~&\text{ in }\Omega,\\
\hhr=g\pskr-\sigma\cnab \cdot \left( \frac{\cnab \pskr}{\sqrt{1+|\cnab\pskr|^2}}\right)~~~&\text{ on }\Sigma,\\
\p_3\hhr=0 ~~~&\text{ on }\Sigma_b,\\
\end{cases}
\end{equation}and define $\qh=\qc-\hhr$. Then \eqref{CWWSTlin} becomes the following linear hyperbolic system with \textit{homogeneous} boundary conditions:
\begin{equation}\label{CWWSTlin2}
\begin{cases}
\rhor \Dtpkr v +\nabpkr \qh=-\nabpkr \hhr-(\rhor-1)ge_3,&~~~ \text{in}\,\,[0,T]\times \Omega,\\
\ffr'(\qr)\Dtpkr \qh + \nabpkr\cdot v=\ffr'(\qr)({g}\vr_3-\Dtpkr \hhr), &~~~ \text{in}\,[0,T]\times \Omega,\\
q=q(\rho), ~\qh=q+g\pkr-\hhr &~~~ \text{in}\, [0,T]\times \Omega, \\
\qh=\kk^2(1-\TL)(v\cdot \npkr), &~~~\text{on}~[0,T]\times\Sigma, \\
v_3=0 &~~~\text{on}~[0,T]\times\Sigma_b,\\
{(v,\rho)|_{t=0}=(v_0, \rho_0)}.
\end{cases}
\end{equation} 
Note that the coefficients in \eqref{CWWSTlin2} rely on $\psd, \psr,\vr$, and $\rhor$ only, all of which are already given. The kinematic boundary condition, namely $\p_t \psi = v\cdot\npkr=-(\vb\cdot\cnab)\pskr+v_3$ on $\Sigma$, is used to define $\psi$ after solving $(v,q)$ from \eqref{CWWSTlin2}.

We define $U:=(\qh,v_1,v_2,v_3)^\top$, then \eqref{CWWSTlin2} can be expressed in terms of $U$ by
\begin{equation}
A_0(\Ur)\p_tU+\sum_{i=1}^3 A_i(\Ur)\p_i U=\fr,
\end{equation}where $\fr:=\left(\ffr'(\qr)(g\vr_3-\Dtpkr \hhr),-\ppkr_1\hhr,-\ppkr_2\hhr,-\ppkr_3\hhr-(\rhor-1)g\right)^{\top}$, $A_0(\Ur)=\text{diag}\left[\ffr'(\qr),\rhor,\rhor,\rhor\right]$, and
\begin{align*}
A_i(\Ur)=
\begin{bmatrix}
\ffr'(\qr)\vr_i& e_i^\top\\
e_i&\rhor \vr_i \mathbf{I}_3
\end{bmatrix}\text{ for }i=1,2,~~
A_3(\Ur)=\frac{1}{\p_3\pkr}
\begin{bmatrix}
\ffr'(\qr)(\vr\cdot\Npkd-\p_t\pr)& \Npkr^{\top}\\
\Npkr&\rhor (\vr\cdot\Npkd-\p_t\pr)\mathbf{I}_3
\end{bmatrix}.
\end{align*}

Since $(\p_t\pr-\vr\cdot\Npkd,)|_{\Sigma}=0$ and $e_3=(0,0,1)^{\top}$ is the unit exterior normal vector to $\Sigma$, we know that the boundary matrix, namely the normal projection of the coefficient matrices onto $\Sigma$, is 
\[\sum\limits_{i=1}^3A_i(\Ur)e_{3i}=A_3(\Ur)=\begin{bmatrix}
0& \npkr^{\top}\\
\npkr& \mathbf{0}_3
\end{bmatrix}~~~\text{ on }\Sigma
\]  which is a $4\times 4$ matrix of rank 2 (constant rank but not full rank)  with one negative eigenvalue, one positive eigenvalue, and two zero eigenvalues. This being said, the system \eqref{CWWSTlin2} is a first-order symmetric hyperbolic system with characteristic boundary conditions. The number of boundary conditions should be equal to the number of negative eigenvalues. Therefore, the correct number of boundary conditions for \eqref{CWWSTlin2} is indeed equal to 1, which means \eqref{CWWSTlin2} is solvable. After solving \eqref{CWWSTlin2}, we use the kinematic boundary condition to define $\psi$ for the next step of the iteration.

\subsubsection{Well-posedness in $L^2$ via $\mu$-regularization}
From the duality argument by Lax-Phillips \cite{lax1960local}, we need to prove the following in order to get the well-posedness of \eqref{CWWSTlin2} in some function space $X$:
\begin{itemize}
\item We need to establish a priori estimate (without loss of regularity from the source term) for \eqref{CWWSTlin2} in $X$.
\item We need to establish a priori estimate (without loss of regularity from the source term) for the dual system of \eqref{CWWSTlin2} in $X'$.
\end{itemize}

We choose $X=L^2([0,T];L^2( \Omega))$ whose dual space $X'$ is just itself. We define $W^*=(\qh^*,w_1^*,w_2^*,w_3^*)^\top$ to be the dual variables of $U=(\qh,v_1,v_2,v_3)^\top$. By testing \eqref{CWWSTlin2} with $W^*$ in $L^2([0,T];L^2( \Omega))$, one can derive the system of $W^*$ which reads $$A_0(\Ur)\p_tW^*+\sum\limits_{i=1}^3 A_i(\Ur)\p_i W^*+A_4(\Ur)W^*=\fr^*$$ with boundary condition $\qh^*|_{\Sigma}=-\kk^2(1-\TL)({w^*}\cdot\npkr)$, where $A_4:=-\p_tA_0^\top-\sum\limits_{i=1}^3 \p_iA_i^\top-(\rhor-1) g\mathbf{E}_{44}$ with $\mathbf{E}_{44}=\text{diag}[0,0,0,1]$. Note that we do not have the dual variable for $\psi$ because the original linearized system completely determines $\psi$. That is why we only have one boundary condition for the dual system.

We notice that there is an extra minus sign in the boundary condition for $\qh^*$. So, one cannot close the $L^2$-type a priori estimate for the dual system even if we can derive that $L^2$-type a priori estimate for \eqref{CWWSTlin2}. To avoid this difficulty, we introduce another viscosity term in the boundary for $\qh$ in \eqref{CWWSTlin2}. That is, we alternatively consider the $\mu$-regularized linear problem for $U=(\qh,v_1,v_2,v_3)^\top$, which reads
\begin{equation}\label{CWWSTlinmu}
A_0(\Ur)\p_tU+\sum_{i=1}^3 A_i(\Ur)\p_i U=\fr,
\end{equation} with boundary condition
\begin{align}
\label{CWWSTlinmu BC} \qh=\kk^2(1-\TL)(v\cdot \npkr)+\mu(1-\TL)\p_t(v\cdot \npkr)~~&\text{ on }\Sigma.
\end{align}

Then the dual system of \eqref{CWWSTlinmu}-\eqref{CWWSTlinmu BC} reads
\begin{equation}\label{CWWSTlinmu dual}
A_0(\Ur)\p_tW^*+\sum\limits_{i=1}^3 A_i(\Ur)\p_i W^*+A_4(\Ur)W^*=\fr^*
\end{equation}with boundary condition
\begin{align}
\label{CWWSTlinmu dual BC} \qh^*=-\kk^2(1-\TL)(w^*\cdot \npkr)+\mu(1-\TL)\p_t(w^*\cdot \npkr)~~&\text{ on }\Sigma,
\end{align} where $A_4:=-\p_tA_0^\top-\sum\limits_{i=1}^3 \p_iA_i^\top-(\rhor-1) g\mathbf{E}_{44}$ with $\mathbf{E}_{44}=\text{diag}[0,0,0,1]$. Note that we have to integrate by parts once more in time when deriving the boundary condition for $\qh^*$. {This is the reason that an extra minus sign appears in front of $\kk^2(1-\TL)(w^*\cdot\npkr)$}.

Now we are going to derive the a priori estimates for both \eqref{CWWSTlinmu} and \eqref{CWWSTlinmu dual}. For the linear system \eqref{CWWSTlinmu}, we test it with $U$ in $L^2(\Omega)$ and use the symmetry of the coefficient matrices to get:
\begin{equation}
\io U^\top\cdot A_0(\Ur)U\dx=\io U^\top\cdot \fr-\sum_{i=1}^3\io U^\top\cdot \p_iA_i(\Ur) U\dx-\is U^\top\cdot A_3(\Ur) U\dx,
\end{equation}where the interior integrals are directly controlled by $C(\Kr_0)\|U\|_0^2$ and the boundary integral reads:
\begin{equation}
\begin{aligned}
&-\is U^{\top}\cdot A_3(\Ur)U\dx'=-2\is(v\cdot\npkr)\qh\dx'\\
=&~{-2\kk^2\io\left((1-\TL)(v\cdot\npkr)\right)(v\cdot\npkr)\dx'-2\mu\is \p_t\left((1-\TL)(v\cdot\npkr)\right)(v\cdot\npkr)\dx'}\\
=&-\mu\ddt\is\left|\TJ(v\cdot\npkr)\right|_0^2\dx'-2\kk^2\left|\TJ(v\cdot\npkr)\right|_0^2.
\end{aligned}
\end{equation}

We define 
$$
\Er_0(t):=\|v(t)\|_0^2+\left\|\sqrt{\ffr'(\qr)} \qh(t)\right\|_0^2+\int_0^t|\kk(v\cdot\npkr)(\tau)|_{1}^2\mathrm{d}\tau+|\sqrt{\mu}(v\cdot\npkr)(t)|_1^2,
$$
 then the above analysis shows that
\begin{equation}\label{linmu L2}
\Er_0(T)-\Er_0(0)\leq C(\Kr_0) \int_0^T \Er_0(t)+\sqrt{\Er_0(t)}\|\fr(t)\|_0\dt.
\end{equation} and thus by Gr\"onwall's inequality we finish the $L^2$-estimate for \eqref{CWWSTlinmu}. Note that this a priori bound is also uniform in $\mu$.

Next, we show the $L^2$-estimate for the dual system \eqref{CWWSTlinmu dual}. Note that the matrix $A_4(\Ur)$ is still in $L^{\infty}(\Omega)$, so we test \eqref{CWWSTlinmu dual} by $W^*$ and take $L^2$-inner product to get
\begin{equation}
\io W^{*\top}\cdot A_0(\Ur)W^*\dx=\io W^{*\top}\cdot \fr^*-W^{*\top}\cdot\left(\sum_{i=1}^3\p_i A_i(\Ur)+A_4(\Ur)+\rhor g\mathbf{E}_{44}\right)W^*\dx-\is (W^*)^{\top}\cdot A_3(\Ur)W^*\dx',
\end{equation}where the interior integral is directly controlled by $C(\Kr_0)\|W^*\|_0^2$, but now there is a sign change in the boundary integral, which reads:
\begin{equation}
\begin{aligned}
&-\is (W^*)^{\top}\cdot A_3(\Ur)W^*\dx'=-2\is(w^*\cdot\npkr)\qh^*\dx'\\
=&~2\kk^2\io\left((1-\TL)(w^*\cdot\npkr)\right)(w^*\cdot\npkr)\dx'-2\mu\is \p_t\left((1-\TL)(w^*\cdot\npkr)\right)(w^*\cdot\npkr)\dx'\\
\lesssim&-\mu\ddt\is\left|\TJ(w^*\cdot\npkr)\right|_0^2\dx'+2\kk^2\left|\TJ(w^*\cdot\npkr)\right|_0^2.
\end{aligned}
\end{equation}
 One can see that the new viscosity term involving $\mu$ controls the term $2\kk^2|\TJ(w^*\cdot\npkr)|_0^2$ due to the change of sign.

 So, if we define 
 $$\Er_0^*(t)=\|w^*(t)\|_0^2+\left\|\sqrt{\ffr'(\qr)} \qh^*(t)\right\|_0^2+\mu\left|(w^*\cdot \npkr) (t)\right|_1^2,
 $$ 
 then we have
\begin{equation}\label{linmu dual L2}
\Er_0^*(T)-\Er_0^*(0)\lesssim_{\mu^{-1}}  C(\Kr_0) \int_0^T\Er_0^*(t)+\sqrt{\Er_0^*(t)}\|\fr^*(t)\|_0\dt,
\end{equation} and thus Gr\"onwall's inequality helps us close the $L^2$-estimate. 

Combining \eqref{linmu L2} and \eqref{linmu dual L2}, we close the a priori bounds for both linear systems \eqref{CWWSTlinmu}-\eqref{CWWSTlinmu BC} and its dual system \eqref{CWWSTlinmu dual}-\eqref{CWWSTlinmu dual BC}. Such energy bounds have no regularity loss from their source terms to solutions. Therefore, by the argument in Lax-Phillips \cite{lax1960local}(see also \cite[Theorem 5,9]{Rauch1985}), for each fixed $\mu>0$, system \eqref{CWWSTlinmu dual}-\eqref{CWWSTlinmu dual BC} admits a unique solution $U\in L^2([0,T];L^2(\Omega))$. Since the energy bound \eqref{linmu L2} for \eqref{CWWSTlinmu}-\eqref{CWWSTlinmu BC} is uniform in $\mu$, we can take the limit $\mu\to 0_+$ to obtain a local-in-time solution of the homogeneous linearized problem \eqref{CWWSTlin2}. Finally, the modification $\hhr$ is easily controlled by using the property of the harmonic function
\[
\forall s> -\frac12,~~\|\hhr\|_{s+\frac12}\lesssim|\hhr|_{s}\leq g|\pskr|_s+ P(|\cnab\pskr|_s)|\cnab^2\pskr|_{s},
\]which implies the local existence for ($L^2$-) weak solution to the linearized $\kk$-approximate system \eqref{CWWSTlin}. By the argument in \cite[Theorem 4, 8]{Rauch1985}, the weak solution $U$ is actually a ($L^2$-) strong solution. Here, by a ($L^2$-) strong solution we mean that there exists a sequence of smooth solutions converging to $U$ in $L^2$ (see \cite[Section 2.2.3]{Metivier2004}).

\subsection{Higher-order estimates for the linearized system}\label{sect linear energy}
Now we prove higher-order energy estimates for the linearized system \eqref{CWWSTlin}. 
\begin{prop}
Let
\begin{equation}\label{Erkk}
\begin{aligned}
\Er^\kk(t):=&\|\rho(t)-1\|_0^2+\sum_{k=0}^4\|\p_t^kv(t)\|_{4-k}^2+\kk^2\int_0^t\left|\p_t^{k+1}\psi(\tau)\right|_{5-k}^2\mathrm{d}\tau\\
&+\|\sqrt{\ffr'(\qr)} \qc(t)\|_0^2+\|\p\qc(t)\|_3^2+\sum_{k=1}^3\|\p_t^k \qc(t)\|_{4-k}^2+\|\sqrt{\ffr'(\qr)} \p_t^4\qc(t)\|_0^2.
\end{aligned}
\end{equation}
Then there exists some $T^{\kk}>0$ depending on $\kk$ and a constant $C(\kk^{-1},\Kr_0)>0$, such that
\begin{equation}
\sup_{0\leq t\leq T^{\kk}}\Er^{\kk}(t)\leq C(\kk^{-1},\Kr_0)\Er^{\kk}(0).
\end{equation}
Apart from that, we have
\begin{equation}
|\psi(t)|_{5.5}^2+\sum_{k=0}^3|\p_t^{k+1}\psi(t)|_{5.5-k}^2\leq C(\kk^{-1},\Kr_0)\Er^{\kk}(t), \quad \text{for all}\,\,t\in[0,T^{\kk}].
\end{equation}
\end{prop}

\subsubsection{$L^2$-estimate}\label{sect linear L2}
We define the $L^2$-energy for the linearized system \eqref{CWWSTlin} to be
\begin{equation}
\Er_0^{\kk}(t):=\|\rho(t)-1\|_0^2+\|v(t)\|_0^2+\|\sqrt{\ffr'(\qr)} \qc(t)\|_0^2+\kk^2\int_0^t|\p_t\psi(\tau)|_{1}^2\mathrm{d}\tau.
\end{equation}

The control of $\Er_0$ is identical to the a priori estimate for \eqref{CWWSTlinmu} when $\mu=0$. Note that the control of $\|\rho-1\|_0^2$ follows from testing the linearized continuity equation $\ffr'(\qr)\ff'(q)^{-1}\Dtpkr(\rho-1)+\rho(\nabpkr\cdot v)=0$ by $\rho -1$ in $L^2(\Omega)$. Also one can control the $L^2(\Sigma)$ norm of $\psi$ through $\psi(t)=\psi_{0,\kk}+\int_0^t\p_t\psi(\tau)\mathrm{d}\tau.$

\subsubsection{Div-Curl analysis}\label{sect linear divcurl}
To estimate the Sobolev norms of $v$, we invoke the following Hodge decomposition lemma, which is exactly from \cite[Theorem 1.1]{Shkollerdivcurl}.
\begin{lem}[Hodge elliptic estimates]\label{hodgeLL}
For any sufficiently smooth vector field $X$ and $s\geq 1$, one has
\begin{equation}
\|X\|_s^2\lesssim C(|\pskr|_{s+\frac12},|\cnab\pskr|_{W^{1,\infty}})\left(\|X\|_0^2+\|\nabpkr\cdot X\|_{s-1}^2+\|\nabpkr\times X\|_{s-1}^2+|X\cdot\npkr|_{s-\frac12}^2+|X_3|_{H^{s-\frac12}(\Sigma_b)}^2\right),
\end{equation}where the constant $C(|\pskr|_{s+\frac12},|\cnab\pskr|_{W^{1,\infty}})>0$ depends linearly on $|\pskr|_{s+\frac12}^2$.
\end{lem}
Applying this lemma to $v$ with $s=4$, one has
\begin{equation}\label{vdivcurlr}
\|v\|_4^2\lesssim C(|\pskr|_{4.5},|\cnab\pskr|_{W^{1,\infty}})\left(\|v\|_0^2+\|\nabpkr\cdot v\|_3^2+\|\nabpkr\times v\|_3^2+|v\cdot\npkr|_{3.5}^2\right).
\end{equation}

Now we control the curl term. Taking $\nabpkr\times$ in the first equation of \eqref{CWWSTlin}, we get the evolution equation satisfied by $\nabpkr\times v$:
\begin{equation}
\rhor\Dtpkr(\nabpkr\times v)=\rhor[\nabpkr\times,\Dtpkr] v+\nabpkr\rhor\times(\rhor^{-1}\nabpkr \qc),
\end{equation}and taking three derivatives, we get
\begin{equation}\label{curlrv eq}
\rhor\Dtpkr\p^3(\nabpkr\times v)=\p^3\left(\rhor[\nabpkr\times,\Dtpkr] v)+\nabpkr\rhor\times(\rhor^{-1}\nabpkr \qc)\right)-[\p^3,\rhor\Dtpkr](\nabpkr\times v).
\end{equation}

We expect that the source terms in \eqref{curlrv eq} only contain $\leq 4$ derivatives of $v,\qc$ and quantities marked with a ring, but there still exists a mismatched term in $([\nabpkr\times,\Dtpkr]v)^i=\epsilon^{ijk}\nabpkr_j\vr^l\nabpkr_lv_k+\epsilon^{ijk}\nabpkr_j\p_t(\varphi-\pkd)\ppkr_3 v_k$. The contribution of $\pskd$ is controlled by $\Kr_0$. So, the standard $L^2$-estimate for the $\p^3$-differentiated evolution equation of $\nabpkr\times v$ and Reynold transport formula \eqref{transpt linearized} gives
\begin{equation}\label{curlrv3}
\frac12\ddt\|\nabpkr\times v\|_3^2\leq P(\Kr_0)(\|v\|_4^2+\|\qc\|_4\|v\|_4+|\p_t\psi|_4\|\p v\|_{\infty}).
\end{equation}
Finally, using the linearized continuity equation, we can control the divergence
\begin{equation}\label{divrv3}
\|\nabpkr\cdot v\|_3^2\leq \left\|\ffr'(\qr)\Dtpkr \qc\right\|_3^2+\left\|\ffr'(\qr) g\vr_3\right\|_3^2.
\end{equation}

The div-curl analysis for the time derivatives is treated similarly. First, the div-curl analysis for $\|\p_t^k v\|_{4-k}^2$, $1\leq k\leq 3$ yields
\begin{equation}\label{vtdivcurlr}
\|\p_t^k v\|_{4-k}^2\lesssim C(|\pskr|_{4.5-k},|\cnab\pskr|_{W^{1,\infty}})\left(\|\p_t^k v\|_0^2+\|\nabpkr\cdot \p_t^k v\|_{3-k}^2+\|\nabpkr\times\p_t^k v\|_{3-k}^2+|\p_t^kv\cdot\npkr|_{3.5-k}^2\right).
\end{equation}
We replace $\p^3$ by $\p_t^{k}\p^{3-k}$ for $0\leq k \leq 3$ in \eqref{curlrv eq} to get the evolution equation:
\begin{equation}
\rhor\Dtpkr\left(\p^{3-k}\p_t^k(\nabpk \times v)\right)=\p_t^{k}\p^{3-k}\left(\rhor[\nabpkr\times,\Dtpkr] v)+\nabpkr\rhor\times(\rhor^{-1}\nabpkr \qc)\right)-[\p_t^{k}\p^{3-k},\rhor\Dtpkr](\nabpkr \times v),
\end{equation}and thus
\begin{equation}\label{curlrvt1}
\ddt\frac12\|\p_t^k(\nabpkr\times v)\|_{3-k}^2\leq P(\Er^\kk(t)).
\end{equation}
Now, since the leading order term in the commutator $[\p_t^k,\nabpkr\times] v$ should be $\TP\p_t\pkr\p_t^{k-1}\p_3 v$, we have
\begin{equation}\label{curlrvt}
\|\nabpkr\times\p_t^k v\|_{3-k}^2\leq C(\Kr_0)\left(\Er^{\kk}(0)+\int_0^T \Er^{\kk}(t)\dt\right).
\end{equation}

As for divergence, by taking $\p_t^k$, $1\leq k\leq 3$ in the continuity equation, we get
\[
\nabpkr\cdot\p_t^k v=-\p_t^k(\ffr'(\qr)\Dtpkr\qc+\ffr'(\qr)g\vr_3)+[\nabpkr\cdot,\p_t^k]v\lleq -\ffr'(\qr)(\p_t^{k}\Dtpkr\qc+g\p_t^k\vr_3)+(\p_3\pkr)^{-1}\TP\p_t^k\pkr\p_3 v.
\]
Parallel to the analysis for \eqref{divvt},  since $\|\TP \p_t^k\pkr\|_{3-k} \leq \Kr_0$ thanks to \eqref{linearized assumption},  we have $\|\nabpkr\cdot\p_t^k v\|_{3-k}$ is reduced to the control of $\|\ffr'(\qr)\p_t^{k+1}\qc\|_{3-k}$ and $\|\ffr'(\qr)\p_t^{k}\qc\|_{4-k}$ at the top order. Thus, 
\begin{equation}\label{divrvt}
\|\nabpkr\cdot\p_t^k v\|_{3-k}^2\leq (C(\Kr_0)+1)\left(\left\|\ffr'(q)\p_t^{k+1}\qc\right\|_{3-k}^2+\left\|\ffr'(q)\p_t^{k}\qc\right\|_{4-k}^2\right).
\end{equation}

\subsubsection{Estimates for $\psi$ and normal traces}\label{sect linear bdry}

The normal trace terms in \eqref{vdivcurlr} and \eqref{vtdivcurlr} can be directly controlled by applying boundary elliptic estimates to the linearized viscous surface tension equation $\kk^2(1-\TL)(v\cdot\npkr)=q-\sigma\h(\cnab\pskr,\cnab^2\pskr)$. We start with controlling $|v\cdot \npkr|_{3.5}$:
\begin{equation}
|v\cdot\npkr|_{3.5}^2\leq\kk^{-2}\left(|q|_{1.5}^2+\sigma |\cnab^2\pskr|_{1.5}^2 P(|\cnab\pskr|_{1.5})\right)\leq \kk^{-2} P(\Kr_0)\|\qc\|_2^2.
\end{equation}

Taking time derivatives in the kinematic boundary condition, we obtain:
$$\p_t^k v\cdot\npkr=\p_t^{k+1}\psi-\sum_{j=1}^k\binom{k}{j}\p_t^{k-j}\vb\cdot\p_t^j\cnab\pskr,$$ and thus
\begin{equation}
|\p_tv\cdot\npkr|_{2.5}\leq |\p_t^2\psi|_{2.5}+|\vb\cdot\cnab\p_t\pskr|_{2.5}\leq |\p_t^2\psi|_{2.5}+\|v_{\kk, 0}\|_3^2+P(\Kr_0)\int_0^T\|\p_t \vb(t)\|_3\dt.
\end{equation}
Then, we take a time derivative in the linearized viscous surface tension equation to get
\[
\kk(1-\TL)\p_t^2\psi=\p_tq -\sigma\p_t\h(\cnab\pskr,\cnab^2\pskr),
\]
which implies $|\p_t^2\psi|_{2.5}\leq\|\p_t q\|_1+P(\Kr_0)$. Repeatedly, we can take more time derivatives to obtain
\begin{equation}
|\p_t^kv\cdot\npkr|_{3.5-k}^2\leq|\p_t^{k+1}\psi|_{3.5-k}^2+\PP_0^{\kk}+P(\Kr_0)\int_0^T \Er^{\kk}(t)\dt,
\end{equation}and then $|\p_t^{k+1}\psi|_{3.5-k}$ is controlled via boundary elliptic estimates:
\begin{align}
|\p_t^3 \psi|_{1.5}^2\approx&~|\TJ^{-\frac12}\p_t^3\psi|_{2}^2\leq|\TJ^{-\frac12}\p_t^2\qc|_0^2+P(\Kr_0)\leq\|\p_t^2 \qc\|_{1}^2+P(\Kr_0)\leq \|\p_t^2\qc(0)\|_1^2+P(\Kr_0)+\int_0^T\Er^{\kk}(t)\dt,\\
|\p_t^4 \psi|_{0.5}\approx&~|\TJ^{-\frac32}\p_t^4\psi|_{2}\leq|\TJ^{-\frac32}\p_t^3\qc|_0+P(\Kr_0)\leq\|\p_t^3 \qc\|_{1}+P(\Kr_0),\label{544}
\end{align} 
where the leading order term $\|\p \p_t^3 \qc\|_0$ on the RHS of \eqref{544}  will be further reduced through the reduction scheme shown in the upcoming subsection.

\subsubsection{Reduction of pressure}\label{sect linear reduce q}
 We start with $\|\qc\|_4$. From the linearized momentum equation, we know 
\begin{align*}
-(\p_3\pkr)^{-1}\p_3 q=&~(\rhor-1)g+\rhor\Dtpkr v_3,\\
-\p_i q=&~(\p_3\pkr)^{-1}\TP_i\pkr\p_3 q+\rhor\Dtpkr v_i,~~i=1,2,
\end{align*}and thus we have the following estimates after taking $\p^3$ and using $\Dtpkr=(\p_t+\vbr\cdot\cnab)+(\p_3\pkr)^{-1}(\vr\cdot\Npkr-\p_t\pr)\p_3$ to get
\begin{equation}
\|\qc\|_4\lesssim_{\Kr_0} \|\qc\|_0+\|\TT v\|_3+\|\rhor-1\|_3,
\end{equation}where $\TT$ denotes a tangential derivative, including $\p_t,\TP$ and $\omega(x)\p_3$ for some weight function $\omega$ that vanishes on $\Sigma$ and is approximately equal to $|x_3|$ near $\Sigma$. Replacing $\p^3$ by $\p^{3-k}\p_t^{k}$, we know the estimate of $\|\p_t^k \p^{4-k}\qc\|_0$ is reduced to the estimate of $\|\p_t^k \TT v\|_{3-k}$. Combining this with the div-curl analysis in Subsection \ref{sect linear divcurl} we can reduce the top order mixed norms $\|\p_t^k \p^{4-k} v\|_0$ and $\|\p_t^k \p^{4-k} \qc\|_0$ to $\|\TT^\alpha v\|_0$, $|\alpha|=4$, and $\left\|\sqrt{\ffr'(\qr)} \p_t^4 \qc\right\|_0$, all of which are part of the tangential energy.

\subsubsection{Control of full time derivatives}\label{sect linear tt}
From the reduction procedures for $\qc$ and the div-curl analysis for $v$, we know a spatial derivative of $\qc$ is reduced to a tangential derivative of $v$, and the divergence of $v$ is reduced to $\ffr'(\qr)\p_t\qc$. Repeatedly, it remains to control $\sqrt{\ffr'(\qr)}\p_t^4\qc$ and $\TT^{\alpha} v$ with $|\alpha|=4$ in $L^2(\Omega)$. Here, we only present the proof for the estimate with full-time derivatives, which is parallel to Section \ref{sect tt AGU}. The mixed space-time tangential estimates are easier to obtain. We introduce the Alinhac good unknowns $\VVr,\QQr$ for the $\p_t^4$-differentiated linearized system \eqref{CWWSTlin}:
\begin{equation}
\VVr:=\p_t^4 v-\p_t^4\pkr\ppkr_3 v,~~~\QQr:=\p_t^4 \qc-\p_t^4\pkr\ppkr_3\qc
\end{equation}

Similar to the arguments in Section \ref{sect tt AGU}, when $f=v_i$ and $\qc$, the following identity holds:
\begin{equation}
\p_t^4(\nabpkr_i f)=\nabpkr_i\FFr+\ccr_i(f),
\end{equation}where $\ccr_i(f):=\ppkr_3\ppkr_if\p_t^4\pkr+\ccr'_i(f)$. Also, 
\begin{align}
\label{AGU comm Crti'}\ccr_i'(f) =&
-\left[\p_t^4, \frac{\p_i \pkr}{\p_3\pkr}, \p_3 f\right]-\p_3 f \left[ \p_t^4, \p_i \pkr, \frac{1}{\p_3\pkr}\right] -\p_i\pkr \p_3 f\left[ \p_t^3, \frac{1}{(\p_3\pkr)^2}\right] \p_t \p_3 \pkr,~~i=1,2\\
\label{AGU comm Crt3'}\ccr_3'(f) = &
\left[  \p_t^4, \frac{1}{\p_3\pkr}, \p_3 f\right] + \p_3 f\left[ \p_t^3, \frac{1}{(\p_3\pkr)^2}\right]  \p_t \p_3 \pkr.
\end{align} 
Then we take $\p_t^4$ to the first two equations of \eqref{CWWSTlin} to obtain:
\begin{align}
\rhor \Dtpkr \VVr_i + \nabpkr_i \QQr =&~\Rrr^1_i,\label{AGUr tt mom}\\
\ffr'(\qr) \Dtpkr \QQr + \nabpkr\cdot \VVr=&~\Rrr^2- \ccr_i(v^i), \label{AGUr tt mass}
\end{align}
where
\begin{align}
\label{rrrt1} \Rrr^1_i:=& -[\p_t^4, \rhor] \Dtpkr v_i-\rhor\left(\ddr(v_i)+\eer(v_i)\right)-\ccr_i(\qc)-\p_t^4\rhor g\delta_{3i},\\
\label{rrrt2} \Rrr^2 :=&-[\p_t^4, \ffr'(\qr)]\Dtpkr \qc- \ffr'(\qr)\left(\ddr(\qc)+\eer(\qc)\right)+\p_t^4(\ffr'(\qr)g\vr_3),
\end{align}and the commutators $\ddr(f),\eer(f)$ are defined in the same way as in \eqref{AGU comm D} and \eqref{AGU comm 3} by replacing $\TT^{\alpha}, \TP, \pk$ respectively with $\p_t^4, \p_t, \pkr$. The last two terms in \eqref{AGU comm D} vanish because $\p_t^4$ commutes with $\p_3$. Specifically, we have:
\begin{align}
\p_t^4\Dtpkr f&=\Dtpkr \FFr +\ddr(f)+\eer(f),
\end{align}
where $\ddr(f):=(\Dtpkr\ppkr_3 f)\p_t^4\pkr+\ddr'(f)$, and 
\begin{align}
\ddr'(f) =& [\p_t^4, \vbr]\cdot \TP f + \left[\p_t^4, \frac{1}{\p_3\pkr}(\vr\cdot \Npkd-\p_t\pr), \p_3 f\right]+\left[\p_t^4, \vr\cdot \Npkd-\p_t\pr, \frac{1}{\p_3\pkr}\right]\p_3 f+\frac{1}{\p_3\pkr} [\p_t^4, \vr]\cdot \Npkd \p_3 f\nonumber\\
&\label{AGUr comm D}-4(\vr\cdot \Npkd-\p_t\pr)\p_3 f\left[ \p_t^3, \frac{1}{(\p_3 \pkr)^2}\right]\p_t\p_3 \pkr,
\end{align}
\begin{equation}
\label{AGUr comm E} \eer(f):=\p_t^5(\pkd-\pr)\ppk_3 f.
\end{equation}

Analogous to Lemma \ref{AGU TT error}, the following estimates hold. 
\begin{lem}\label{AGUr tt error}
Let $\FFr:= \p_t^4 f - \ppkr_3 f \p_t^4 \pkr$ be the Alinhac good unknowns associated with the smooth function $f$. Assume $\p_3 \pk \geq c_0>0$, and let $\ccr(f)$, $\ddr(f)$, and $\eer(f)$ be the remainder terms defined as above.  Then
\begin{align}
\label{AGUr tt4 gap} \|\p_t^4 f\|_0 \leq&~\|\FFr\|_0+ c_0^{-1}\|\p_3 f\|_{\infty}|\p_t^4\psk|_0, \\
\label{Cr time} \|\ccr_i(f)\|_{0} \leq&~P\left(c_0^{-1}, |\cnab\pskr|_{\infty},\sum_{k=1}^3|\cnab\p_t^k\pskr|_{3-k}\right)\cdot \left(\|\p f\|_{\infty}+\sum_{k=1}^{3}\|\p_t^k f\|_{4-k}\right),\quad i=1,2,3,\\
 \label{Dr time} \|\ddr(f)\|_0\leq&~P\left(c_0^{-1}, |\cnab\pskr|_{\infty},\sum_{k=1}^3|\cnab\p_t^k\pskr,\cnab\p_t^k\pskd|_{3-k}\right)\cdot \left(\|\p f\|_{\infty}+\sum_{k=1}^{3}\|\p_t^k f\|_{4-k}\right),\\
\label{Er time} \|\eer(f)\|_0 \leq&~(|\p_t^5 \psi|_0+|\p_t^5\pskd|_0) \|\p f\|_{\infty}. 
\end{align}
\end{lem}

Next, we introduce the boundary conditions for $\VVr,\QQr$. The $\p_t^4$-differentiated linearized kinematic boundary condition now reads:
\begin{equation}
\p_t^5 \psi + (\vb\cdot\cnab) \TP^4\pskr - \VVr\cdot \npkr=\ssr_1, \quad \text{on}\,\,\Sigma, \label{AGUr tt BC kinematic}
\end{equation}
where
\begin{align}\label{AGUr tt S1}
\ssr_1:=\p_3 v\cdot \npkr \p_t^4\pskr+\sum_{1\leq j\leq 3} \binom{4}{j}\p_t^j v\cdot\p_t^{4-j}\npkr.
\end{align}
Also, since $\QQr|_{\Sigma} = \p_t^4 \qc - \ppkr_3 \qc \p_t^4 \pskr$, the boundary condition of $\QQr$ on $\Sigma$ reads:
\begin{align}
\QQr = -\sigma \p_t^4\cnab \cdot \left( \frac{\cnab \pskr}{\sqrt{1+|\cnab\pskr|^2}}\right)+\kk^2(1-\TL)\p_t^5\psi- \p_3 \qc \p_t^4 \pskr+g\p_t^4\pskr. \label{AGUr tt BC Q}
\end{align}

Invoking \eqref{transpt linearized}, we have
\begin{equation}\label{energy estimate AGUr time}
\begin{aligned}
\ddt\frac{1}{2} \io \rhor |\VVr|^2\p_3\pkr\dx=\frac{1}{2}\io |\VVr|^2\left((\Dtpkr\rhor+\rhor\nabpkr\cdot\vr)\p_3\pkr+\rhor\Mr\right)\dx \\
+\io \QQr (\nabpkr\cdot \VVr) \p_3\pkr\dx-\is \QQr (\VVr\cdot \npkr)\dx'+ \io \VVr\cdot \Rr^1\p_3\pk\dx,
\end{aligned}
\end{equation}where $\Mr:=\p_t\p_3(\pkd-\pr)+\p_3(\p_t+\vbr\cdot\cnab)(\pkr-\pkd)$ represents the mismatched terms involving tangential smoothing in \eqref{transpt linearized}. The first integral on the RHS can be directly controlled by $P(\Kr_0)$ because all these quantities are already given. Moreover, the last integral is directly controlled by $P(\Kr_0)\|\VVr\|_0\sqrt{\Er^{\kk}(t)}$. For the second term in \eqref{energy estimate AGUr time}, we invoke \eqref{AGUr tt mass} to get the estimates parallel to \eqref{AGU ttQ}:
\begin{equation}\label{AGUr ttQ}
\begin{aligned}
&\io \QQr(\nabpkr\cdot\VVr)\p_3\pkr\dx\\
=&\underbrace{-\io\p_t^4 \qc\ccr_i(v^i)\p_3\pkr\dx}_{=:\Ir_0}+\io\p_t^4\pkr\ppkr_3 \qc\ccr_i(v^i)\p_3\pkr\dx-\io\ffr'(\qr)\Dtpkr\QQr~\QQr\p_3\pkr\dx+\io\Rr^2\QQr\p_3\pkr\dx\\
\lesssim&~ \Ir_0-\frac12\ddt\left\|\sqrt{\ffr'(\qr)}\QQ\right\|_0^2+\left\|\sqrt{\ffr'(\qr)}\p_t^4 \qc\right\|_0^2(\|\nabpkr\cdot\vr\|_{\infty}+\|\Mr\|_{\infty})\\
&+\|\ccr_i(v^i)\|_0\|\p \qc\|_{\infty}|\p_t^4\pskr|_0+\left\|\sqrt{\ffr'(\qr)}\QQ\right\|_0\left\|\sqrt{\ffr'(\qr)}^{-1}\Rr^2\right\|_0\\
\lesssim&~\Ir_0-\frac12\ddt\left\|\sqrt{\ffr'(\qr)}\QQ\right\|_0^2+P(\Kr_0)\Er^{\kk}(t),
\end{aligned}
\end{equation}where we note that all terms in $\Rr^2$ come with $\ffr'(\qr)$ and thus the control of $\sqrt{\ffr'(\qr)}^{-1}\Rr^2$ is still uniform in $\ffr'(\qr)$.

Now it remains to control the boundary integral. Compared with the nonlinear system, the estimate for the linearized system is easier, as the surface tension term now becomes a given term. Plugging \eqref{AGUr tt BC kinematic} and \eqref{AGUr tt BC Q} into the boundary integral, we get
\begin{equation}\label{AGUrttbdry0}
\begin{aligned}
-\is\QQr(\VVr\cdot\npkr)\dx'=&-\is\p_t^4\cnab\cdot(\cnab\pskr/|\npkr|)\p_t^5\psi\dx'-\kk^2\is\p_t^4(1-\TL)\p_t\psi\cdot\p_t^5\psi\dx'\\
&-\is g\p_t^4\pskr\p_t^5\psi\dx'+\is\p_3\qc\p_t^4\pskr\p_t^5\psi\dx'\\
&-\is\QQ(\vbr\cdot\cnab)\p_t^4\pskr\dx'+\is\QQr\ssr_1\dx',
\end{aligned}
\end{equation}where the second term gives us the boundary energy
\begin{equation}\label{AGUrST1}
-\kk^2\is\p_t^4(1-\TL)\p_t\psi\cdot\p_t^5\psi\dx'=\kk^2\is\left|\TJ\p_t^5\psi\right|^2\dx'.
\end{equation}
We note that the first, the third, and the fourth terms in \eqref{AGUrttbdry0} can all be directly controlled under the time integral, i.e., 
\begin{align}
-\int_0^T\is\p_t^4\cnab\cdot(\cnab\pskr/|\npkr|)\p_t^5\psi\dx'\dt\lesssim&~\eps|\p_t^5\psi|_{L_t^2H_{x'}^1}^2+P(|\cnab\pskr|_{\infty})|\cnab\p_t^4\pskr|_0^2\leq_{\kk^{-1}}\eps\Er^{\kk}(T)+P(\Kr_0)\\
-\int_0^T\is (g-\p_3\qc)\p_t^4\pskr\p_t^5\psi\dx'\dt\leq&~ \eps|\p_t^5\psi|_{L_t^2L_{x'}^2}^2+|\p_t^4\pskr|_0^2(1+\|\p\qc\|_{L_t^2L_x^{\infty}}^2)\leq_{\kk^{-1}}\eps\Er^{\kk}(T)+P(\Kr_0)\int_0^T \Er^{\kk}(t)\dt.
\end{align}
Further, the fifth term is controlled directly by invoking \eqref{lkk3}:
\begin{equation}\label{AGUrST2}
\begin{aligned}
-\int_0^T\is\QQ(\vbr\cdot\cnab)\p_t^4\pskr\dx'\dt=&-\int_0^T\sigma\is\p_t^4\cnab\cdot(\cnab\pskr/|\npkr|)(\vbr\cdot\cnab)\p_t^4\pskr\dx'\dt+\kk^2\int_0^T\is(1-\TL)\p_t^5\psi(\vbr\cdot\cnab)\p_t^4\pskr\dx'\dt\\
&+\int_0^T\is(g-\p_3\qc)\p_t^4\pskr(\vbr\cdot\cnab)\p_t^4\pskr\dx'\dt\\
\lesssim_{\kk^{-1}}&~ \eps|\p_t^5\psi|_{L_t^2H_{x'}^1}^2+P(\Kr_0)\int_0^T\Er^{\kk}(t)\dt.
\end{aligned}
\end{equation}

It remains to analyze the last integral in \eqref{AGUrttbdry0}, which will be canceled with $\Ir_0$ defined in \eqref{AGUr ttQ}. Following the analysis in \eqref{AGUttbdry3}--\eqref{AGUttI00}, we have
\begin{align}
\label{AGUrST3} \is\QQr\ssr_1\dx'=&~4\is\p_t^4\qc\p_t^3v\cdot\p_t\npkr\dx'+\text{ controllable terms},\\
\label{AGUrI0} \Ir_0=&-4\io\p_t^4\qc\p_t\Npkr_i\p_3\p_t^3v^i\dx+\text{ controllable terms},
\end{align}and then we add them together and use the divergence theorem to get
\begin{equation}
\begin{aligned}
&4\is\p_t^4\qc\p_t^3v\cdot\p_t\npkr\dx'-4\io\p_t^4\qc\p_t\Npkr_i\p_3\p_t^3v^i\dx\\
=&\ddt\io(\p_t^3\p_3\qc\p_t\Npkr+\p_t^3\qc\p_t\p_3\Npkr)\cdot\p_t^3v\dx+\io\p_t^3\p_3\qc\p_t(\p_t\Npkr\cdot\p_t^3 v)+\p_t^3\qc\p_t(\p_t\p_3\Npkr\cdot\p_3 v)\dx,
\end{aligned}
\end{equation}whose time integral can be easily bounded by $\eps\|\p_t^3\p_3\qc\|_0^2+\Er^{\kk}(0)+P(\Kr_0)\int_0^T \Er^{\kk}(t)\dt$. Hence, we get control of the boundary integral
\begin{equation}\label{AGUrttbdry}
-\int_0^T\is\QQr(\VVr\cdot\npkr)\dx'\dt+\kk^2\int_0^T\is\left|\TJ\p_t^5\psk\right|_0^2\dt\leq \eps\|\p_t^3\p_3\qc\|_0^2+\Er^{\kk}(0)+P(\Kr_0)\int_0^T \Er^{\kk}(t)\dt.
\end{equation}

Combining this with \eqref{energy estimate AGUr time}, \eqref{AGUr ttQ}, and the definition of Alinhac good unknowns, we get the estimates for the full-time derivatives
\begin{equation}\label{ttr time}
\|\p_t^4v(t)\|_0^2+\left\|\sqrt{\ffr'(\qr)}\p_t^4\qc\right\|_0^2+\kk^2\int_0^t\is\left|\TJ\p_t^5\psk\right|_0^2\mathrm{d}\tau\leq \eps\|\p_t^3\p_3\qc\|_0^2+\Er^{\kk}(0)+P(\Kr_0)\int_0^t\Er^{\kk}(t)\mathrm{d}\tau.
\end{equation}
This, together with div-curl analysis, gives us the energy inequality of $\Er^{\kk}(t)$ after choosing $\eps>0$ suitably small:
\begin{equation}
\Er^{\kk}(t)\leq_{\kk^{-1}} \Er^{\kk}(0)+P(\Kr_0)\int_0^t \Er^{\kk}(\tau)\mathrm{d}\tau,
\end{equation}which implies that there exists some $T^{\kk}>0$ such that 
\[
\sup_{0\leq t\leq T^{\kk}}\Er^{\kk}(t)\leq C(\kk^{-1},\Kr_0)\Er^{\kk}(0).
\]
Therefore, the uniform-in-$n$ estimates for \eqref{CWWSTlin} are proven by induction.
\subsubsection{Regularity of $\psi$ and its time derivatives}
The regularity of $\p_t^{k+1}\psi~(0\leq k\leq 3)$ can be enhanced to $H^{5.5-k}$ by the boundary elliptic estimates once we close the energy estimates for $\Er^{\kk}(t)$. Note that the boundary condition gives
\[
\kk^2 (1-\TL) \p_t\psi=\qc- g\pskr+\sigma \mathcal{H} (\cnab \pskr, \cnab^2 \pskr),
\]
thus, by \eqref{linearized assumption} and the elliptic estimate, it holds that
\begin{equation}
|\p_t^{k+1}\psi|_{5.5-k}\leq \kk^{-2} \left (\sigma P(|\cnab\pskr|_{\infty}) |\p_t^k\cnab^2\pskr|_{3.5-k}+|\p_t^k q|_{3.5-k}+ P(\Kr_0)\right) \leq C(\kk^{-1}, \Kr_0)\Er^{\kk}. 
\end{equation}
Moreover,  $|\psi|_{5.5}$ is controlled by 
\begin{equation}
|\psi|_{5.5}\leq |\psi_{0,\kk}|_{5.5}+\int_0^T|\p_t\psi(t)|_{5.5}\dt.
\end{equation}

\subsection{Picard iteration}
So far, we have established the local existence and the uniform-in-$n$ estimates for the linearized system \eqref{CWWSTlin0} for each fixed $\kk>0$, namely
\begin{equation}\label{CWWSTlin00}
\begin{cases}
\rho^{(n)}  D_t^{\pk^{(n)}} v^{(n+1)} +\nab^{\pk^{(n)}} \qc^{(n+1)}=-(\rho^{(n)}-1) ge_3 &~~~ \text{in }[0,T]\times \Omega,\\
\ff^{(n)'}(q^{(n)})D_t^{\pk^{(n)}} \qc^{(n+1)}+ \nab^{\pk^{(n)}}\cdot v^{(n+1)}=\ff^{(n)'}(q^{(n)})gv_3^{(n)} &~~~ \text{in }[0,T]\times \Omega,\\
q^{(n+1)}=q^{(n+1)}(\rho^{(n+1)}), \qc^{(n+1)}=q^{(n+1)}+g\pk^{(n)}&~~~ \text{in }[0,T]\times \Omega, \\
\qc^{(n+1)}=g\psk^{(n)}-\sigma\cnab \cdot \left( \frac{\cnab \psk^{(n)}}{\sqrt{1+|\cnab\psk^{(n)}|^2}}\right)+\kk^2(1-\TL)(v^{(n+1)}\cdot \npk^{(n)}) &~~~\text{on }[0,T]\times\Sigma, \\
\p_t \psi^{(n+1)} = v^{(n+1)}\cdot \npk^{(n)} &~~~\text{on }[0,T]\times\Sigma,\\
v_3^\nnr=0 &~~~\text{on}~[0,T]\times\Sigma_b,\\
(v^{(n+1)},\rho^{(n+1)},\psi^{(n+1)})|_{t=0}=(v_0^\kk, \rho_0^\kk, \psi_0^\kk),
\end{cases}
\end{equation}where $\psi^\nnn,\varphi^\nnn,D_t^{\pk^\nnn},\nab^{\pk^\nnn}$ are defined in \eqref{Dtpk lin}-\eqref{nabpk lin 3}. Now it suffices to prove that, for each fixed $\kk>0$, the sequence $\{(v^\nnn,\qc^\nnn,\psi^\nnn)\}_{n\in\N^*}$ has a strongly convergent subsequence. Once we prove this, the limit of that subsequence becomes the solution to the nonlinear $\kk$-approximate system \eqref{CWWSTkk} for this chosen $\kk$. 

For a function sequence $\{f^\nnn\}$ we define $[f]^\nnn:=f^\nnr-f^\nnn$ and then we find that $\{([v]^\nnn,[\qc]^\nnn,[\psi]^\nnn)\}$ satisfies the following linear system
\begin{equation}\label{CWWSTlingap}
\begin{cases}
\rho^{(n)}  D_t^{\pk^{(n)}} [v]^{(n)} +\nab^{\pk^{(n)}} [\qc]^{(n)}=-\fr_v^\nnn &~~~ \text{in }[0,T]\times \Omega,\\
\ff^{(n)'}(q^{(n)})D_t^{\pk^{(n)}} [\qc]^{(n)}+ \nab^{\pk^{(n)}}\cdot [v]^{(n)}=-\fr_q^\nnn &~~~ \text{in }[0,T]\times \Omega,\\
[\qc]^{(n)}=[q]^{(n)}+g[\pk]^{(n-1)}&~~~ \text{in }[0,T]\times \Omega, \\
[\qc]^{(n)}=g[\psk]^{(n-1)}-\sigma[\h]^{\nnl}+\kk^2(1-\TL)([v]^{(n)}\cdot \npk^{(n)})+\kk^2(1-\TL)(v^\nnn\cdot[\npk]^\nnl), &~~~\text{on }[0,T]\times\Sigma, \\
\p_t [\psi]^{(n)} = [v]^{(n)}\cdot \npk^{(n)}+(v^\nnn\cdot[\npk]^\nnl), &~~~\text{on }[0,T]\times\Sigma,\\
[v_3^\nnn]=0 &~~~\text{on}~[0,T]\times\Sigma_b,\\
([v]^{(n)},[\rho]^{(n)},[\psi]^{(n)})|_{t=0}=(\mathbf{0},0,0),
\end{cases}
\end{equation}where $\fr_v^\nnn$ and $\fr_q^\nnn$ are defined by
\begin{align}
\fr_v^\nnn:=&~[\rho]^\nnl\p_t v^\nnn+[\rho\vb]^\nnl\cdot\cnab v^\nnn+[\rho V_{\Npr}]^\nnl\p_3 v^\nnn+[\rho]^\nnl ge_3+\p_3\qc^\nnn[A_{i3}]^\nnl,\\
\fr_q^\nnn:=&~[\ff'(q)]^\nnl(\p_t\qc^\nnn-gv_3^\nnl)+[\ff'(q)\vb]^\nnl\cdot\cnab\qc^\nnn+[\ff'(q)V_{\Npr}]^\nnl\p_3\qc^\nnn\\
&-\ff^{(n)'}(q^\nnn)g[v_3]^\nnl+\p_3v_i^\nnn[A_{i3}]^\nnl,\nonumber
\end{align}and
\begin{align*}
&V_{\Npr}^\nnn:=~\frac{1}{\p_3\pk^\nnn}(v^\nnn\cdot\Npk^\nnl-\p_t\varphi^\nnn),~A^\nnn_{13}:=-\frac{\p_1\pk^\nnn}{\p_3\pk^\nnn},~A^\nnn_{23}:=-\frac{\p_2\pk^\nnn}{\p_3\pk^\nnn},~A^\nnn_{33}:=\frac{1}{\p_3\pk^\nnn},\\
&[\h]^\nnl:=\h(\cnab\psk^\nnn)-\h(\cnab\psk^\nnl),~\h(\cnab\psk):=-\cnab\cdot\left(\frac{\cnab\psk}{1+|\cnab\psk|^2}\right).
\end{align*}

For $n\geq 1$, we define the energy of \eqref{CWWSTlingap} $[E]^\nnn$ to be the following quantity
\begin{equation}
\begin{aligned}
[E]^\nnn(t):=\sum_{k=0}^3\|\p_t^k[v]^\nnn(t)\|_{3-k}^2+\|\p_t^k[\qc]^\nnn(t)\|_{3-k}^2+\int_0^t\left|\p_t^{k+1}[\psi]^\nnn(\tau)\right|_{4-k}^2\mathrm{d}\tau+|[\psi]^\nnn(t)|_4^2
\end{aligned}
\end{equation}
It suffices to control $[E]^\nnn(t)$ and use $([v]^{(n)},[\rho]^{(n)},[\psi]^{(n)})|_{t=0}=(\mathbf{0},0,0)$ to show that $[E]^\nnn(t)\leq \frac14 ([E]^\nnl(t)+[E]^\nnll(t))$ in some time interval $[0,T^\kk_1]$. The estimates for $[E]^\nnn(t)$ are parallel to Section \ref{sect linear energy}, so we will not go into every detail but only list the sketch of the proof.

\subsubsection{Div-curl analysis for $[v]^\nnn$}

By Lemma \ref{hodgeLL}, we have the following inequalities for $k=0,1,2$:
\begin{align}
\|\p_t^k[v]^\nnn\|_{3-k}^2\leq&~ C(\Kr_0)\left(\|\p_t^k[v]^\nnn\|_0^2+\|\nab^{\pk^\nnn}\times\p_t^k[v]^\nnn\|_{2-k}^2+\|\nab^{\pk^\nnn}\cdot\p_t^k[v]^\nnn\|_{2-k}^2+|\p_t^k[v]^\nnn\cdot\npk^\nnn|_{2.5-k}^2\right).
\end{align}

The estimates for $L^2(\Omega)$ norms follow in the same manner as in Section \ref{sect linear L2}, so we do not repeat them here. For the curl part, we take $\nab^{\pk^\nnn}\times$ in the first equation of \eqref{CWWSTlingap} to get
\begin{equation}\label{curllv eq0}
\rho^\nnn D_t^{\pk^\nnn}(\nab^{\pk^\nnn}\times [v]^\nnn)=-\nab^{\pk^\nnn}\times \fr_v^\nnn-\nab^{\pk^\nnn}\rho^\nnn\times D_t^{\pk^\nnn}[v]^\nnn+\rho^\nnn[\nab^{\pk^\nnn}\times,D_t^{\pk^\nnn}][v]^\nnn,
\end{equation}where $([\nab^{\pk^\nnn}\times,D_t^{\pk^\nnn}][v]^\nnn)^i=\epsilon^{ijk}\nab_j^{\pk^\nnn}v_l^\nnn\nab_l^{\pk^\nnn}[v]_k^\nnn+\epsilon^{ijk}\nab_j^{\pk^\nnn}\p_t(\pk^\nnn-\pk^\nnl)\p_3[v]_k^\nnn$ and $\nab^{\pk^\nnn}\times \fr_v^\nnn$ contains at most two derivatives of $v^\nnn,\varphi^\nnn,\varphi^\nnl,\varphi^\nnll$. Taking $\p^2$, we have
\begin{equation}\label{curllv eq}
\rho^\nnn D_t^{\pk^\nnn}(\p^2\nab^{\pk^\nnn}\times [v]^\nnn)=\p^2(\text{RHS of }\eqref{curllv eq0})-[\p^2,\rho^\nnn D_t^{\pk^\nnn}](\nab^{\pk^\nnn}\times [v]^\nnn).
\end{equation}

Based on the analysis above, we find that the leading-order terms of $[v]^\nnn,[v]^\nnl$ must be linear in $[v]^\nnn,[v]^\nnl$, respectively. Using the Reynold transport formula \eqref{transpt linearized} for the linearized system, the curl part can be directly controlled as in \eqref{curlrv3}:
\begin{align}
\|\nab^{\pk^\nnn}\times [v]^\nnn(T)\|_2^2\leq&~ C(\Kr_0)\left(\|\underbrace{\nab^{\pk^\nnn}\times [v]^\nnn(0)}_{=0}\|_2^2+\int_0^T P(\Er^\nnn,\Er^\nnl,\Er^\nnll)[E]^\nnn(t)\dt\right)\\
\leq&~C(\Kr_0)\int_0^T [E]^\nnn(t)+[\Er]^\nnl(t)+[\Er]^\nnll(t)\dt. \nonumber
\end{align}

Similarly, replacing $\p^2$ by $\p^{2-k}\p_t^{k}$ for $k=1,2$, we get
\begin{align}
\|\nab^{\pk^\nnn}\times \p_t^k[v]^\nnn(T)\|_{2-k}^2\leq~C(\Kr_0)\int_0^T [E]^\nnn(t)+[\Er]^\nnl(t)+[\Er]^\nnll(t)\dt. 
\end{align}

As for the divergence,  the second equation in \eqref{CWWSTlingap} gives
\begin{equation}
\|\nab^{\pk^\nnn}\cdot [v]^\nnn\|_2^2\leq \|\ff^{(n)'}(q^{(n)})D_t^{\pk^{(n)}} [\qc]^{(n)}\|_2^2+\|\fr_q\|_2^2\leq P(\Kr_0)\|\ff^{(n)'}(q^{(n)})\TT [\qc]^{(n)}\|_2^2,
\end{equation}where $\TT=\p_t$ or $\TP$ or $\omega\p_3$ for a bounded weight function $\omega$ vanishing on $\Sigma$. Therefore, the divergence is then reduced to the tangential derivatives of $[\qc]$. Similarly, the divergence of $\p_t^k[v]^\nnn$ is reduced to $\p_t^k\TT\qc$.

Next, the normal traces are still controlled by using boundary elliptic estimates. Note that the Dirichlet boundary condition for $[\qc]^\nnn$ can be written as
\begin{equation}
-\kk^2(1-\TL)([v]^{(n)}\cdot \npk^{(n)})=-[q]^{(n)}-\sigma\left(\h(\cnab\psk^\nnn)-\h(\cnab\psk^\nnl)\right)+\kk^2(1-\TL)(v^\nnn\cdot[\npk]^\nnl),
\end{equation}and thus
\begin{equation}\label{vngap}
|[v]^{(n)}\cdot \npk^{(n)}|_{2.5}^2\leq_{\kk^{-1}} \|[q]^\nnn\|_{1}^2+P(\Kr_0)+|\vb^\nnn\cdot\cnab\psk^\nnl|_{2.5}^2+|v_3^\nnn|_{2.5}^2\leq \|[q]^\nnn\|_{1}^2+P(\Kr_0).
\end{equation}
Similarly, we have for $k=1,2$
\begin{equation}\label{vtngap}
|\p_t^k[v]^{(n)}\cdot \npk^{(n)}|_{2.5-k}^2\leq_{\kk^{-1}} \|\p_t^k[q]^\nnn\|_{1}^2+P(\Kr_0).
\end{equation}

\subsubsection{Reduction of pressure $[\qc]^\nnn$}
This is similar to the arguments in Section \ref{sect linear reduce q}. We first consider the third component of the first equation in \eqref{CWWSTlingap}:
\begin{equation}
(\p_3\pk^\nnn)^{-1}\p_3[\qc]^\nnn=-\rho^\nnn D_t^{\pk^{(n)}} [v]^{(n)} +\fr_v^\nnn,
\end{equation}which means the control of $\p_3 [\qc]^\nnn$ is reduced to $\TT [v]^\nnn$. Then by considering the first and second components, we can further reduce the control of $\TP_i \qc~(i=1,2)$ to $\p_3\qc$ and $\TT v$ since $\nabpkr_i=\TP_i-\TP_i\pkr\ppkr_3$. Therefore, combining the div-curl analysis and reduction procedures for $[\qc]^\nnn$, it suffices to control $\p_t^2\TP[\qc]^\nnn$ and $\p_t^3[\qc]^\nnn$.

\subsubsection{Tangential estimates for full-time derivatives}\label{sect AGUr gap}
Again we only show the control of $\p_t^3[\qc]^\nnn$ by introducing the Alinhac good unknowns:
\begin{align}
[\VV]^\nnn:=\VVr^\nnr-\VVr^\nnn=\p_t^3[v]^\nnn-\p_t^3 \pk^\nnn \p_3^{\pk^\nnn}[v]^\nnn -\p_t^3\pk^\nnn\p_3^{[\pk]^\nnl}v^\nnn-\p_t^3[\pk]^\nnl\p_3^{\pk^\nnl} v^\nnn,\\
[\QQ]^\nnn:=\QQr^\nnr-\QQr^\nnn=\p_t^3[\qc]^\nnn-\p_t^3 \pk^\nnn \p_3^{\pk^\nnn}[\qc]^\nnn -\p_t^3\pk^\nnn\p_3^{[\pk]^\nnl}\qc^\nnn-\p_t^3[\pk]^\nnl\p_3^{\pk^\nnl} \qc^\nnn.
\end{align}For a function $f$ and its associated Alinhac good unknown $\FF$, we have
\begin{align*}
\p_t^3(\p_i^{\pk^\nnn}[f]^\nnn+\p_i^{[\pk]^\nnl}f^\nnn)=&~\p_i^{\pk^\nnn}[\FF]^\nnn +[\cc]_i^\nnn(f),\\ \p_t^3(D_t^{\pk^\nnn}[f]^\nnn+D_t^{[\pk]^\nnl}f^\nnn)=&~D_t^{\pk^\nnn}[\FF]^\nnn+[\dd]^\nnn(f)+[\ee]^\nnn(f)
\end{align*}with
\begin{align*}
[\cc]_i^\nnn(f)=\cc_i^\nnn(f^\nnr)-\cc_i^\nnl(f^\nnn) +\text{lower-order controllable terms},\\
[\dd]_i^\nnn(f)=\dd^\nnn(f^\nnr)-\dd^\nnl(f^\nnn) +\text{lower-order controllable terms},\\
[\ee]_i^\nnn(f)=\ee^\nnn(f^\nnr)-\ee^\nnl(f^\nnn) +\text{lower-order controllable terms},
\end{align*}where $\cc_i^\nnn(f^{(m)}), \dd^\nnn(f^{(m)})$ and $\ee^\nnn(f^{(m)})$ are defined by replacing $\p_t^3, \p_t^4,\p_t^5$ by $\p_t^2, \p_t^3,\p_t^4$, respectively, in \eqref{AGU comm Crti'}-\eqref{AGUr comm E}. Also, we need to  replace the coefficient $4$ in $\ddr$ by $3$ and setting $\pr=\varphi^\nnn$, $\pd=\varphi^\nnl$, $f^{(n+1)}=f$ and $f^{(n)}=\mathring{f}$. 
Here, by ``lower-order controllable terms" we mean terms with $\leq 2$ time derivatives on both $[f]^{(n)}$ and $[\pk]^{(n-1)}$. These terms are directly controlled by $C(\Kr_0)\left( [\Er]^\nnn+[\Er]^\nnl+[\Er]^\nnll\right)$ through a combination of Young's inequality and Sobolev embeddings.

The Alinhac good unknowns $[\VV]^\nnn,[\QQ]^\nnn$ satisfy the following linear system:
\begin{align}
\rho^\nnn D_t^{\pk^\nnn}[\VV]^\nnn+\nab^{\pk^\nnn}[\QQ]^\nnn&=-\cc^\nnn(\qc^\nnr)+\cc^\nnl(\qc^\nnn)+[\Rr]_v,\\
\ff^{\nnn '}(q^\nnn) D_t^{\pk^\nnn} [\QQ]^\nnn+\nab^{\pk^\nnn}\cdot\VV^\nnn&=-\cc_i^\nnn(v_i^\nnr)+\cc_i^\nnl(v_i^\nnn)+[\Rr]_q,
\end{align}where $[R]$ terms consist of $\p_t^3 \fr$ terms in \eqref{CWWSTlingap} and the omitted commutator terms in the definition of Alinhac good unknowns $[\VV],[\QQ]$, and they are controllable in $L^2(\Om)$.
\begin{align}
\|[\Rr]\|_0^2\leq C(\Kr_0)([E]^\nnn(t)+[E]^\nnl(t)+[E]^\nnll(t)).
\end{align}
 The boundary conditions now become:
\begin{align}
[\QQ]^\nnn&=g\p_t^3[\psk]^\nnl+\sigma\p_t^3(\h(\cnab\psk^\nnn)-\h^(\cnab\psk^\nnl))-\kk^2(1-\TL)\p_t^4[\psi]^\nnn\nonumber\\
&+\p_t^3\psk^\nnn\p_3[\qc]^\nnn+\p_t^3[\psk]^\nnl\p_3\qc^\nnn\\
[\VV]^\nnn\cdot\npk^\nnn&=\p_t^4[\psi]^\nnn+[\vb]^\nnn\cdot\cnab\p_t^3\psk^\nnn+(\vb\cdot\cnab)\p_t^3[\psk]^\nnl+\p_t^3\vb^\nnn\cdot\cnab[\psk]^\nnl\nonumber\\
&-(\p_3[v]^\nnn\cdot\npk^\nnn)\p_t^3\psk^\nnn+(\p_3v^\nnn\cdot\npk^\nnn)\p_t^3[\psk]^\nnl+[\p_t^3,\npk^\nnn\cdot,v^\nnr]-[\p_t^3,\npk^\nnl\cdot,v^\nnn].
\end{align}

Following the analysis in Section \ref{sect linear tt}, we have
\begin{equation}
\begin{aligned}
&\frac12\ddt\left(\io\rho^\nnn|[\VV]^\nnn|^2\p_3\pk^\nnn\dx+\io\ff^{\nnn '}(q^\nnn)|[\QQ]^\nnn|^2\p_3\pk^\nnn\dx\right)+\kk^2\int_0^T\left|\p_t^4[\psi]^\nnn\right|_1^2\dt\\
\leq&~ C(\Kr_0)\left([\Er]^\nnn(0)+\int_0^T [\Er]^\nnn(t)+[\Er]^\nnl(t)+[\Er]^\nnll(t)\dt\right)\\
&- \is [\QQr]^\nnn[\p_t^3,\Npk^\cdot,v^\nnl]\dx'+\io[\QQr]^\nnn\cc_i^\nnl(v_i^\nnn)\dvt^\nnn
\end{aligned}
\end{equation}
where the last line is analyzed in the same way as in \eqref{AGUrI0} (by using divergence theorem and integration by parts in time). Here we only list the highest-order terms. We have
\begin{equation}
\begin{aligned}
&- \is [\QQr]^\nnn[\p_t^3,\Npk^\cdot,v^\nnl]\dx'+\io[\QQr]^\nnn\cc_i^\nnl(v_i^\nnn)\dvt^\nnn
\lleq \io \p_3^{\pk^\nnn} [\QQr]^\nnn[\p_t^3,\Npk^\cdot,v^\nnl]\dvt^\nnn
\end{aligned}
\end{equation}and thus it can be controlled under the time integral:
\begin{equation}
\begin{aligned}
&\int_0^T\io \p_3^{\pk^\nnn} [\QQr]^\nnn[\p_t^3,\Npk^\cdot,v^\nnl]\dvt^\nnn\dt
\lesssim&~\eps\|\p_t^2[\qc]^\nnn\|_1^2+C(\Kr_0)\left([\Er]^\nnn(0)+\int_0^T[\Er]^\nnn(t)+[\Er]^\nnl(t)\dt\right).
\end{aligned}
\end{equation}

Combining the above analysis and using the definition of Alinhac good unknowns, we get
\begin{equation}\label{AGUttgap}
\begin{aligned}
&\|\p_t^3 [v]^\nnn(t)\|_0^2+\|\sqrt{\ff^{\nnn '}}(q^\nnn)\p_t^3[\qc]^\nnn(t)\|_0^2+\kk^2\int_0^t|\p_t^4\psi(\tau)|_1^2\mathrm{d}\tau\\
\leq&~\eps\|\p_t^2[\qc]^\nnn\|_1^2+C(\Kr_0,\kk^{-1})\left([\Er]^\nnn(0)+\int_0^T[\Er]^\nnn(t)+[\Er]^\nnl(t)+[\Er]^\nnll(t)\dt\right).
\end{aligned}
\end{equation}

\subsection{Well-posedness of the nonlinear $\kk$-approximate problem}
Combining the div-curl analysis, the control of the normal traces, the reduction of $[\qc]$ and the analysis of full-time derivatives for the linear system \eqref{CWWSTlingap} for $[v]^\nnn,[\qc]^\nnn,[\psi]^\nnn$, we arrive at the energy estimate:
\begin{equation}
[\Er]^\nnn(t)\leq C(\Kr_0,\kk^{-1})\left([\Er]^\nnn(0)+\int_0^T[\Er]^\nnn(t)+[\Er]^\nnl(t)+[\Er]^\nnll(t)\dt\right).
\end{equation}

Since $[v]^\nnn,[\qc]^\nnn,[\psi]^\nnn$ have zero initial data, one can repeatedly use \eqref{CWWSTlingap} to show that their time derivatives also vanish on $\{t=0\}$, as one can observe that every term in the first two equations of \eqref{CWWSTlingap} contains exactly one term involving $[f]^\nnn$ or $[f]^\nnl$ whose initial value is zero. This implies $[\Er]^\nnn(0)=0$, and thus there exists some $T_1^{\kk}>0$ independent of $n$, such that
\begin{equation}
\sup_{0\leq t\leq T_1^{\kk}}[\Er]^{\nnn}(t)\leq\frac{1}{4}\left(\sup_{0\leq t\leq T_1^{\kk}}[\Er]^{\nnl}(t)+\sup_{0\leq t\leq T_1^{\kk}}[\Er]^{\nnll}(t)\right),
\end{equation}and thus we know by induction that
\begin{equation}
\sup_{0\leq t\leq T_1^{\kk}}[\Er]^{\nnn}(t)\leq C(\Kr_0,\kk^{-1})/2^{n-1}\to 0 \text{ as }n\to+\infty.
\end{equation}
Hence, for any fixed $\kk>0$, the sequence of approximate solutions $\{(v^\nnn,\qc^\nnn,\rho^\nnn,\psi^\nnn)\}_{n\in\N^*}$ is Cauchy, whose limit $(v^\kk,\qc^\kk,\rho^\kk,\psi^\kk)$ is exactly the solution to the nonlinear $\kk$-problem \eqref{CWWSTkk}. The uniqueness follows from a parallel argument.

\section{Well-posedness and incompressible limit of the gravity(-capillary) water wave system}
 We are ready to prove the local existence of the original water wave system \eqref{CWWST} for each fixed $\sigma>0$. In Section \ref{sect linear LWP}, we prove the local well-posedness and higher-order energy estimates of the linearized system \eqref{CWWSTlin} for each \textit{fixed} $\kk>0$ and use Picard iteration to construct the unique strong solution to the nonlinear $\kk$-approximate problem \eqref{CWWSTkk} defined in Section \ref{sect CWWSTkkeq}. To pass the limit $\kk\to 0_+$ to the original system \eqref{CWWST}, we prove the uniform-in-$\kk$ estimates for \eqref{CWWSTkk} in Section \ref{sect uniformkk}. Therefore, we prove the local-in-time existence for the stronger solution to the compressible gravity-capillary water wave system \eqref{CWWST}, that is, given initial data $(v_0,\rho_0,\psi_0)$, there exists $T'>0$ only depending on the initial data, such that the original system \eqref{CWWST} has a solution $(v,\rho,\psi)$ satisfying the energy estimates
\begin{equation}
\sup_{0\leq t\leq T'}E(t)\leq P(E(0)).
\end{equation}

The detailed procedure is as follows. First, the uniform-in-$\kk$ boundedness of $(v^{\kk}, \qc^{\kk}, \psi^{\kk})$ and their time derivatives $(\p_t v^{\kk}, \p_t \qc^{\kk}, \p_t \psi^{\kk})$ guarantees that there exist subsequences $(v^{\kk_j}, \qc^{\kk_j})$ converges weakly* to $(v, \qc)$ in $L^{\infty}([0,T']; H^4(\Omega)^2)$, and $\psi^{\kk_j}$ converges to $\psi$ weakly* in $L^{\infty}([0,T']; H^5(\Sigma))$, via the Banach-Alaoglu Theorem. Second, the Aubin-Lions Lemma ensures that $(v^{\kk_j}, \qc^{\kk_j})$ converges strongly to $(v, \qc)$ in $C^0([0,T']; H_{loc}^{4-\delta}(\Omega)^2)$, and $\psi^{\kk_j}$ converges to $\psi$ strongly in $C^0([0,T']; H_{loc}^{5-\delta}(\Sigma))$, for any fixed $\delta>0$. Third, since the approximate $\kk$-equations are asymptotically consistent with the compressible equations \eqref{CWWST}, the (local) strong convergence in Sobolev spaces that are algebras implies that $(v, \qc, \psi)$ is, in fact, a solution of \eqref{CWWST}. Moreover, the regularities of $(v, \qc, \psi)$ and their time derivatives were obtained from their approximating sequences by lower semi-continuity. 
Therefore, 
\begin{subequations}
\begin{align}
&\p_t^k v(t,\cdot)\in L^{\infty}([0, T'];H^{4-k}(\Omega)), \quad k=0,\cdots, 4,\\
&\p_t^k \p \qc(t,\cdot)\in L^{\infty}([0, T']; H^{3-k}(\Omega)),\quad k=0, \cdots ,3, \\
&\lam\p_t^k\qc(t,\cdot)  \in L^{\infty}([0, T']; L^2(\Omega)), \quad k=0,\cdots, 4,\\
&\psi(t,\cdot)\in L^{\infty}([0, T']; L^2(\Sigma)), \quad \sqrt{\sigma}\p_t^k \cnab\psi(t,\cdot)\in L^{\infty}([0, T']; H^{4-k}(\Sigma)), \quad k=0,\cdots, 4.
\end{align}
\end{subequations}

\subsection{Uniqueness and continuous dependence on data}\label{sect uniqueness}
To prove the well-posedness, it suffices to prove the uniqueness of the solution to \eqref{CWWST}. We assume $$\{(v^\nnn,\qc^\nnn,\rho^\nnn,\psi^\nnn)\}_{n=1,2}$$ to be two solutions to \eqref{CWWST} and define $[f]=f^{(1)}-f^{(2)}$ for any function $f$. Then it suffices to prove $([v],[\qc],[\rho],[\psi])=(\mathbf{0},0,0,0)$. We find that $([v],[\qc],[\rho],[\psi])=(\mathbf{0},0,0,0)$ satisfies the following system:
\begin{equation}\label{CWWSTgap}
\begin{cases}
\rho^{(1)}D_t^{\varphi^{(1)}}[v]+\nabla^{\varphi^{(1)}}[\qc]=-f_v &~~~\text{ in }[0,T]\times\Omega,\\
\ff'(q^{(1)})D_t^{\varphi^{(1)}}[\qc]+\nabla^{\varphi^{(1)}}\cdot[v]=-f_q&~~~\text{ in }[0,T]\times\Omega,\\
[\qc]=[q]+g[\varphi]&~~~\text{ in }[0,T]\times\Omega,\\
[\qc]=g[\psi]-\sigma\left(\h(\cnab\psi^{(1)})-\h(\cnab\psi^{(2)})\right)&~~~\text{ on }[0,T]\times\Sigma,\\
\p_t[\psi]=[v]\cdot N^{(1)}+v^{(2)}\cdot[N]&~~~\text{ on }[0,T]\times\Sigma,\\
[v_3]=0 &~~~\text{ on }[0,T]\times\Sigma_b,\\
([v],[\qc],[\psi])|_{t=0}=(\mathbf{0},0,0),
\end{cases}
\end{equation}where the functions $f_v,f_q$ are defined by 
\begin{align}
f_v&:=[\rho]\p_tv^{(2)}+[\rho \vb]\cdot\cnab v^{(2)}+[\rho V_\NN]\p_3v^{(2)}+\rho^{(2)}ge_3+\p_3\qc^{(2)}[A_{i3}]\\
f_q&:=[\ff'(q)](\p_t\qc^{(2)}-gv_3^{(2)})+[\ff'(q)\vb]\cdot\cnab \qc^{(2)}+[\ff'(q)V_\NN]\p_3 q^{(2)}\\
&-\ff'(q^{(2)})g[v_3]+\p_3v_i^{(2)}[A_{i3}],  \nonumber
\end{align}and 
\begin{align*}
V_\NN&:=\frac{1}{\p_3\varphi}(v\cdot\NN-\p_t\varphi),~A_{13}:=-\frac{\p_1\varphi}{\p_3\varphi},~A_{23}:=-\frac{\p_2\varphi}{\p_3\varphi},~A_{33}:=\frac{1}{\p_3\varphi},\\
\h(\cnab\psi)&:=\cnab\cdot\left(\frac{\cnab\psi}{|N|}\right),~~\h(\cnab\psi^{(1)})-\h(\cnab\psi^{(2)})=\cnab\cdot\left(\frac{\cnab[\psi]}{|N^{(1)}|}-\left(\frac{1}{|N^{(1)}|}-\frac{1}{|N^{(2)}|}\right)\cnab\psi^{(2)}\right).
\end{align*}

Let
\begin{equation}
[E](t):=\sum_{k=0}^3\left(\|\p_t^k[v]\|_{3-k}^2+\sigma|\cnab\p_t^{k}[\psi]|_{3-k}^2+\|\sqrt{\ff'(q^{(1)})}\p_t^k[\qc]\|_0^2\right)+g|[\psi]|_0^2+\sum_{k=0}^2\|\p \p_t^k[\qc]\|_{2-k}^2.
\end{equation}
We can then follow the proof for the uniform-in-$\kk$ estimates (setting $\kk=0$) in Section \ref{sect uniformkk} to show that $[E](0)=0$ and $[E](t)$ satisfies the following energy inequality
\begin{equation}\label{Gronwall [E]}
[E](T)\leq \int_0^TP(E(t))[E](t)\dt.
\end{equation}
Here, compared to the Picard iteration process, the only difference is that the boundary integral produces some extra terms, which are controlled using mollification beforehand. Consequently, we must use the surface tension instead of the artificial viscosity term to achieve boundary regularity. Following the analysis in Section \ref{sect AGUr gap}, the main contribution of the boundary integral arising from $\TP^3$-tangential estimates is
\begin{equation}
-\is[\QQ][\VV]\cdot N^{(1)}\dx'\lleq-\is\TP^3[q]\p_t\TP^3[\psi]\dx'+\is\TP^3[q]\TP^3(v^{(2)}\cdot[N])\dx',
\end{equation}
where $[\QQ],[\VV]$ are the Alinhac good unknowns of $[\qc],[v]$ with respect to $\TP^3$ and $\varphi^{(1)}$, that is, $[\FF]:=\FF^{(1)}-\FF^{(2)}$.
For the first integral, we have
\begin{equation}
\begin{aligned}
-\is\TP^3[q]\p_t\TP^3[\psi]\dx'\lleq&-\frac{\sigma}{2}\ddt\is|N^{(1)}|^{-1}\left|\TP^3\cnab[\psi]\right|_0^2\dx'-\sigma{\is\frac{\TP^3\cnab[\psi]\cdot\cnab(\psi^{(1)}+\psi^{(2)})}{|N^{(1)}||N^{(2)}|(|N^{(1)}|+|N^{(2)}|)}\cnab\psi^{(2)}\cdot\p_t\cnab\TP^3[\psi]\dx',}
\end{aligned}
\end{equation}
where the first term gives the boundary energy in $[E](t)$, and the second term appears when $\TP^3$ falls on $$|N^{(1)}|^{-1}-|N^{(2)}|^{-1}=\frac{|N^{(2)}|^2-|N^{(1)}|^2}{|N^{(1)}||N^{(2)}|(|N^{(1)}+|N^{(2)}|)}.$$ 
This term is controlled by
\begin{align*}
&-\sigma\is\frac{\TP^3\cnab[\psi]\cdot\cnab(\psi^{(1)}+\psi^{(2)})}{|N^{(1)}||N^{(2)}|(|N^{(1)}|+|N^{(2)}|)}\cnab\psi^{(2)}\cdot\p_t\cnab\TP^3[\psi]\dx'\\
\leq&~ P(|\cnab\psi^{(1)},\cnab\psi^{(2)}|_\infty)|\sqrt{\sigma}\cnab\TP^3[\psi]|_0(|\sqrt{\sigma}\p_t\psi^{(1)}|_4+|\sqrt{\sigma}\p_t\psi^{(2)}|_4)\\
\leq&~\eps|\sqrt{\sigma}\cnab\TP^3[\psi]|_0^2+P(|\cnab\psi^{(1)},\cnab\psi^{(2)}|_\infty)E(t)\leq\eps[E](t)+P(E(t)).
\end{align*}

The energy inequality for $[E](t)$ together with Gr\"onwall's inequality and the energy bounds for $E(t)$ implies that there exists some $T\in[0,T']$ only depending on the initial data of \eqref{CWWST}, such that $\sup\limits_{0\leq t\leq T}[E](t)\leq 2C [E](0)=0$, for some constant $C>0$. Therefore, the solution to \eqref{CWWSTgap} must be zero. The uniqueness is proven, and the continuous dependence on initial data in $H^3(\Omega)$ for $v,\qc$ and in $H^4(\Sigma)$ for $\psi$ is shown in a similar manner. 

In particular, 
let $(\bar{\psi}_0, \bar{v}_0, \bar{\rho}_0-1)\in H^5(\Sigma)\times H^4(\Omega)\times H^4(\Omega)$ be another set of initial data, and we denote by $(\bar{\psi}, \bar{v}, \bar{\qc})$ the  solution of \eqref{CWWST} corresponding to $(\bar{\psi}_0, \bar{v}_0, \bar{\rho}_0-1)$. Under this setting, the last line of \eqref{CWWSTgap} becomes
$$
([v],[\qc],[\psi])|_{t=0}=(v_0-\bar{v_0}, \qc_0-\bar{\qc}_0, \psi_0-\bar{\psi}_0).
$$
Consequently, the inequality \eqref{Gronwall [E]} becomes
$$
[E](T)\leq [E](0)+\int_0^TP(E(t))[E](t)\dt.
$$
Now, suppose that $[E](0) \leq \epsilon$. Then, from the standard Gr\"onwall's inequality, we conclude that $\sup\limits_{0\leq t\leq T}[E](t)\leq 2C\epsilon$, provided that $T>0$ is sufficiently small. We emphasize that this constant $C>0$ depends on $\sigma^{-1}$ if the initial data does not satisfy the Rayleigh–Taylor sign condition.

\subsection{Incompressible and zero-surface-tension limits}\label{sect double limits}
This section is devoted to showing that we can pass the solution of \eqref{CWWST} to the incompressible and zero surface tension double limits.  In other words, we study the behavior of the solution of \eqref{CWWST} as both the Mach number $\lam$ and surface tension coefficient $\sigma$ tend to $0$. Recall that the Mach number $\lam$ is defined in Section \ref{sect mach number definition}. 

We study the incompressible Euler equations modeling the motion of incompressible gravity water waves without surface tension, satisfied by $(\xi, w, q_{in})$ with initial data $(w_0,\xi_0)$ and $w_{0}^3|_{\Sigma_b}=0$:
\begin{equation} \label{WW}
\begin{cases}
\Dtp w +\nabp \qi=0&~~~ \text{in}~[0,T]\times \Omega,\\
\nabp\cdot w=0&~~~ \text{in}~[0,T]\times \Omega,\\
\qi=q_{in}+g\varphi &~~~ \text{in}~[0,T]\times \Omega, \\
\qi=g\xi &~~~\text{on}~[0,T]\times\Sigma,\\
\p_t \xi = w\cdot \nn &~~~\text{on}~[0,T]\times\Sigma,\\
w_3=0 &~~~\text{on}~[0,T]\times\Sigma_b,\\
(w,\xi)|_{t=0}=(w_0, \xi_0),
\end{cases}
\end{equation}
where we define $\varphi(t,x) = x_3+\chi(x_3) \xi(t,x')$ to be the extension of $\xi$ in $\Omega$ after slightly abuse of notations. 
Denote by
$
(\psi^\ls, v^\ls, \rho^\ls) 
$ 
the solution of \eqref{CWWST}
indexed by $\lam$ and $\sigma$, our goal is to show:
\begin{equation}
(\psi^\ls, v^\ls, \rho^\ls)  \to (\xi, w, 1)\quad \text{in}\,\,C^0([0,T]; \orange{H_{\text{loc}}^{4-\delta}(\Sigma)}\times H_{\text{loc}}^{4-\delta}(\Omega)\times H_{\text{loc}}^{3-\delta}(\Omega)), \quad\text{for any}\,\,\delta\in (0,1], 
\end{equation}
provided that:
\begin{enumerate}
\item The sequence of initial data $(\psi_0^\ls, v_0^\ls, \rho_0^\ls-1)\in H^5(\Sigma)\times H^4(\Omega)\times H^4(\Omega)$ satisfies the compatibility conditions up to order $3$, $|\psi_0^\ls|_{\infty}\leq 1$,  and $v_{0}^{3;\ls}|_{\Sigma_b}=0$. The compatibility condition of order $k$ ($k\geq 0$), expressed in terms of the modified pressure, reads
\begin{align} 
(\Dtp)^k \qc^{\ls}|_{\{t=0\}\times \Sigma}  = \sigma (\Dtp)^k \left( \mathcal{H}^{\ls} + g\psi^{\ls}\right)|_{\{t=0\}\times \Sigma}\label{comp cond pre}. 
\end{align}
Since $\Dtp = \p_t + \vb^{\ls} \cdot \TP$ on $\Sigma$, we can rewrite \eqref{comp cond pre} as:
\begin{align} \label{comp cond}
(\p_t + \vb^{\ls} \cdot \TP)^k \qc^{\ls}|_{t=0} = \sigma (\p_t +\vb^{\ls}\cdot \TP)^k \left(\mathcal{H}^{\ls}+g\psi^{\ls}\right)|_{t=0}\quad \text{on}\,\,\,\Sigma. 
\end{align}
Apart from this, we require 
\begin{equation}
\p_t^k v^{3; \ls}|_{\{t=0\}\times \Sigma_b} = 0, \quad k=0,1,2,3,
\end{equation}
The existence of such data is discussed in Appendix \ref{sect CWWST data}.

\item $(\psi_0^\ls, v_0^\ls, \rho_0^\ls)  \to (\xi, w, 1)$ in $\orange{H^{4}(\Sigma)}\times H^{4}(\Omega)\times H^{3}(\Omega)$ as $\ls\to 0$.
\item The compressible pressure $q^{\ls}$ and the incompressible pressure $q_{in}$ satisfy the Rayleigh-Taylor sign condition:
\begin{align}
-\p_3 q^{\ls} \geq c_0>0, \quad \text{on}\,\,\{t=0\}\times \Sigma, \label{Taylor0}\\
-\p_3 q_{in} \geq c_0>0, \quad \text{on}\,\,\{t=0\}\times \Sigma. \label{incompressible Taylor0}
\end{align}
\end{enumerate}
The key step of showing the $\ls$-double limits is to prove an energy estimate of \eqref{CWWST} that is uniform in both $\lam$ and $\sigma$. The analysis in Section \ref{sect uniformkk} indicates that the energy estimate for \eqref{energykk} is already uniform in $\lambda$. In particular, one can see that the tangential energy estimates in Sections \ref{sect TT AGU}-\ref{sect tt AGU} are uniform in $\ff'_{\lam}$, which is of size  $O(\lam^2)$ by \eqref{ff' is lambda}. 

The energy bound that we obtained from the local existence implies the boundedness of $\|\p_t^kv^\ls (t)\|_{4-k}^2 + |\p_t^k\psi^\ls(t)|_{4-k}^2~(k\leq 4)$ uniformly in both $\lam$ and $\sigma$ within the time interval $[0,T]$. Thus, 
\begin{align}
(v^{\ls},\psi^{\ls})\rightarrow (w,\xi),\quad \text{ as } \ls\to 0,
\end{align}
weakly-* in $L^\infty([0,T];H^4(\Omega)\times \orange{H^4(\Sigma)})$, and strongly in $C^0([0,T]; H_{\text{loc}}^{4-\delta}(\Omega)\times \orange{H_{\text{loc}}^{4-\delta}(\Sigma)})$ for any $0<\delta\leq 1$. Here, the strong convergence is a direct consequence of the Aubin-Lions lemma, and the uniqueness of the limit function implies the convergence without squeezing a subsequence. 

Moreover, as $\Dtp = \p_t + (\vb \cdot \cnab) + (\p_3 \varphi)^{-1} (v\cdot \mathbf{N} - \p_t\varphi)\p_3$, 
invoking the continuity equation 
$$
\ff'_{\lam} (q)\Dtp \qc^\ls+\nabp\cdot v^\ls=g\ff'_{\lam}(q)\Dtp v_3^\ls,
$$
and because $\|\Dtp \qc^\ls(t)\|_3, \|\Dtp v^\ls(t)\|_3$ are uniformly bounded in $[0,T]$, we have
\begin{align}
\nabp\cdot v^\ls \rightarrow \nabp\cdot w= 0,  
\end{align}
weakly-* in $L^\infty([0,T];H^3(\Omega))$, and strongly in $C^0([0,T];H_{\text{loc}}^{3-\delta}(\Omega))$. Once again, the strong convergence is obtained thanks to the Aubin-Lions lemma.

Finally, since the continuity equation can be expressed as
$$
\Dtp (\rho^\ls-1) +\rho^\ls (\nabp\cdot v^\ls) = 0,
$$
we can derive the energy estimate for $\rho^\ls-1$ in $H^3(\Omega)$ as:
\begin{align} \label{rho-1 H3}
\ddt\frac{1}{2}\|\rho^\ls-1\|_3^2 \leq \|\rho^\ls-1\|_0\Big(\|v^\ls\|_{4} + |\TP \psi^{\ls}|_3 \Big),
\end{align}
where $\|v^\ls\|_{4}$, $|\TP \psi^{\ls}|_3$ are bounded by $E^\ls$.  Similarly, we can prove the uniform bound also for $\|\p_t(\rho^{\ls}-1)\|_2^2$. 
Therefore,  $\rho^{\ls}\rightarrow1$ weakly* in $L^\infty([0,T];H^3(\Omega)$, and strongly in $C^0([0,T]; H_{\text{loc}}^{3-\delta}(\Omega))$.

\section{Improved incompressible limit for well-prepared initial data} \label{section double lim}

We recall that the uniform boundedness (with respect to the Mach number) of top-order time derivatives is required to establish the uniform-in-$(\ls)$ estimates in Theorem \ref{main thm, double limits}. However, only the boundedness of first-order time derivatives is required, namely $\dive v=O(\lam)$ and $\p_t v=O(1)$ if the initial data is well-prepared. In this section, we aim to drop the boundedness assumption for high-order time derivatives. Since we also need to guarantee the propagation of the Rayleigh-Taylor sign condition, the uniform boundedness of $\p_t\p_3 q\sim\p_t^2 v$ is still required. 

It should be noted that there is a new difficulty in the control of the ``weaker" energy $\EE(t)$: there is a loss of $\lam$-weights in $\TP^2\p_t^2$-tangential estimates when analyzing $\EE_4(t)$. In particular, we have to control the following quantity in the cancellation structure used at the end of Section \ref{sect tt AGU}:
\[
\io (\TP\p_3\p_t^2 v_i)(\TP\NN_i)(\TP\p_3\p_t^2 q)\dx,
\]in which $\p_t^2 q$ has to be uniformly bounded with respect to Mach number. However, now we only have $\p_t^2 q=O(1/\lam)$, which leads to a loss of $\lam$-weight. Besides, similar difficulty also appears in the control of $-\io \VV^\pm\cdot\cc(q^\pm)\dvt$. Indeed, such loss of $\lam$-weight necessarily happens in  $\TP^2\p_t^2$-tangential estimates because of the following two reasons
\begin{enumerate}
\item $\TP^2\p_t^2 q$ needs one more $\lam$-weight than $\TP^2\p_t^2 v$, and
\item The (extension of) normal vector $\NN$, which arises from the commutator $[\TP^2\p_t^2,\NN_i/\p_3\varphi, \p_3 f]$ in $\cc_i(f)$, may NOT absorb a time derivative. 
\end{enumerate}
\textit{This loss of $\lambda$-weight is entirely due to the presence of the moving surface boundary, as the commutator $\cc(f)$ is unnecessary in the fixed-domain setting.} In the second author's previous work \cite{Zhang2021elastoLWP}, considering compressible inviscid fluids \textit{without surface tension}, such essential difficulty can be avoided thanks to the vanishing Dirichlet boundary condition $q|_{\Sigma}=0$, but that framework is no longer applicable here. To overcome the loss of $\lam$-weights, we have to find a new way to control $\p_t^2 v$. We also need to introduce a new energy functional:
\begin{align}
\label{weak EE}\EE(t):=&~\EE_4(t)+E_5(t),\\
\EE_4(t):=&~\|v\|_4^2+\|\qc\|_4^2+|\sqrt{\sigma}\psi|_5^2+|\psi|_{4}^2+\|\p_t v,\p_t\qc\|_3^2+|\sqrt{\sigma}\p_t\psi|_4^2+|\p_t\psi|_{3.5}^2\no \\
\label{weak EE4}&+ \|\p_t^2 v,\lam\p_t^2\qc\|_2^2+|\sqrt{\sigma}\p_t^2\psi|_3^2+|\p_t^2\psi|_{2.5}^2+|\p_t^3\psi|_{1.5}^2  \\
&+ \sum_{k=3}^4\|\lam \p_t^k(v,\qc)\|_{4-k}^2+|\sqrt{\sigma}\lam\p_t^k\psi|_{5-k}^2+|\lam \p_t^4\psi|_{0.5}^2\no \\
\label{weak E5}E_5(t):=&\sum_{k=0}^5\ino{\lam^2\p_t^k(v, \lam^{(k-4)_+}\qc)}_{5-k}^2+\bno{\sqrt{\sigma}\lam^2\p_t^k\psi}_{6-k}^2+\bno{\lam^2\p_t^k\psi}_{5-k}^2
\end{align}

We now introduce the following div-curl inequality:
\begin{lem}[Hodge-type elliptic estimates]\label{hodgeNN}
For any sufficiently smooth vector field $X$ and $s\geq 1$, one has
\begin{equation}
\|X\|_s^2\lesssim C(|\psi|_{s+\frac12},|\cnab\psi|_{W^{1,\infty}})\left(\|X\|_0^2+\|\nabp\cdot X\|_{s-1}^2+\|\nabp\times X\|_{s-1}^2+|X\cdot N|_{s-\frac12}^2\right),
\end{equation} 
where the constant $C(|\psi|_{s+\frac12},|\cnab\psi|_{W^{1,\infty}})>0$ depends linearly on $|\psi|_{s+\frac12}^2$.
\end{lem}
Applying this inequality to $X=\p_t^2 v$ and $s=2$, we obtain that
\begin{align}
\|\p_t^2 v\|_2^2\lesssim C(|\psi|_{2.5},|\cnab\psi|_{W^{1,\infty}})\left(\|\p_t^2 v\|_0^2+\|\nabp\cdot \p_t^2 v\|_{1}^2+\|\nabp\times \p_t^2 v\|_{1}^2+|\p_t^2 v\cdot N|_{s-\frac12}^2\right).
\end{align} The divergence and vorticity are controlled in the same way as in Section \ref{sect div curl}. As for the boundary term, we have 
\[
\p_t^2 v\cdot N=\p_t^3\psi+\vb_j\TP_j\p_t^2\psi,
\]so we shall turn to control $|\p_t^3\psi|_{1.5}$ and $|\p_t^2\psi|_{2.5}$ without any weights of $\ls$. 

\subsection{Time-differentiated evolution equation of the free surface}\label{sect surface eq}
We derive the evolution equation of the free surface by further differentiating the kinematic boundary condition with respect to the time variable. 
\subsubsection{Time-differentiated kinematic boundary condition}
 Let $\TDt:=\Dtp|_{\Sigma} = \p_t +\vb\cdot \cnab$. The kinematic boundary condition then implies
\begin{equation}\label{Dt KBC}
\TDt \psi= v_3, \quad \text{ on }\Sigma.  
\end{equation}
Taking one more $\TDt$ to \eqref{Dt KBC}, we infer from the momentum equation that
\begin{equation}\label{Dt2 KBC}
\rho \TDt^2 \psi = -\p_3 \qc - (\rho-1)g, \quad \text{ on }\Sigma. 
\end{equation}
Since $[\p_t, \TDt]f = \p_t\vb_j \TP_j f$, we obtain
\begin{align*}
[\p_t^2, \TDt]f = \p_t^2\vb_j \TP_j f +2\p_t \vb_j \,\p_t\TP_j f.  
\end{align*}
Furthermore, 
\begin{align*}
[\p_t^2, \TDt^2 ] f=& ~\TDt\left(\p_t^2\vb_j \TP_j f+2\p_t \vb_j \p_t\TP_j f\right)+\p_t^2\vb_j\,\TP_j\TDt f+2\p_t\vb_j\,\TP_j\p_t\TDt f\\
=&~ \p_t^2\TDt\vb_j\,\TP_j f-2\p_t^2\vb_j\,\TP_j\vb_k\,\TP_k f-2\p_t\vb_j\,\p_t\TP_j\vb_k\,\TP_k f\\
&+2\p_t^2\vb_j\,\TP_j\TDt f+2\p_t\TDt\vb_j\,\TP_j\p_t f -4\p_t\vb_j\,\TP_j\vb_k\,\TP_k\p_tf\\
&+4\p_t\vb_j\,\TP_j\p_t\TDt f-2\p_t\vb_j\,\TP_j(\p_t\vb_k\,\TP_kf). 
\end{align*}
Therefore, we have
\begin{equation}
\begin{aligned}
\p_t^2\TDt^2 \psi =& \TDt^2\p_t^2\psi  + \p_t^2\TDt\vb_j\,\TP_j \psi-2\p_t^2\vb_j\,\TP_j\vb_k\,\TP_k \psi-2\p_t\vb_j\,\p_t\TP_j\vb_k\,\TP_k \psi\\
&+2\p_t^2\vb_j\,\TP_j\TDt \psi+2\p_t\TDt\vb_j\,\TP_j\p_t \psi -4\p_t\vb_j\,\TP_j\vb_k\,\TP_k\p_t\psi\\
&+4\p_t\vb_j\,\TP_j\p_t\TDt \psi-2\p_t\vb_j\,\TP_j(\p_t\vb_k\,\TP_k\psi). 
\end{aligned}
\end{equation}

Combining this with \eqref{Dt2 KBC} yields
\begin{align}
\begin{aligned}
\TDt^2\p_t^2\psi= - \frac{1}{\rho}\p_t^2\p_3 \qc-\p_t^2\TDt\vb\cdot\cnab\psi + \mathbf{R}_{\psi}+\mathbf{R}_{\rho}, \quad \text{ on }\Sigma,
\end{aligned}
\end{align}
where
\begin{align}\label{Rbf psi}
\begin{aligned}
-\mathbf{R}_{\psi} =&-2\p_t^2\vb_j\,\TP_j\vb_k\,\TP_k \psi-2\p_t\vb_j\,\p_t\TP_j\vb_k\,\TP_k \psi+2\p_t^2\vb_j\,\TP_j\TDt \psi+2\p_t\TDt\vb_j\,\TP_j\p_t \psi \\
&-4\p_t\vb_j\,\TP_j\vb_k\,\TP_k\p_t\psi+4\p_t\vb_j\,\TP_j\p_t\TDt \psi-2\p_t\vb_j\,\TP_j(\p_t\vb_k\,\TP_k\psi)
\end{aligned}
\end{align}
and
\begin{align}\label{Rbf rho}
\mathbf{R}_{\rho} = -\p_t^2 \left(\frac{(\rho-1)g}{\rho}\right)- \left[\p_t^2, \frac{1}{\rho}\right]\p_3\qc.
\end{align}

For $i=1,2$, since $\rho\TDt v_i = -\Tpp_i \qc$ with $\Tpp_i :=\pp_i|_{\Sigma}=\p_i-\p_i\psi\p_3$,  we have
\begin{align}\label{Rbf psi rho}
-\p_t^2\TDt\vb\cdot \cnab \psi = \frac{1}{\rho}\p_t^2 \Tpp \qc\cdot\cnab\psi+ \underbrace{\left[\p_t^2,\frac{1}{\rho}\right] \Tpp \qc\cdot\cnab\psi}_{=\mathbf{R}_{\psi,\rho}}\quad\text{ on }\Sigma.
\end{align}
Also, since $\Tpp_3:=\pp_3|_{\Sigma} = \p_3$, then
$$
- \frac{1}{\rho}\p_t^2\p_3 \qc+\frac{1}{\rho}\p_t^2 \Tpp \qc\cdot\cnab\psi =-\frac{1}{\rho} N\cdot \p_t^2\Tpp \qc\quad\text{ on }\Sigma. 
$$
This leads to the following evolution equation of the moving interface:
\begin{equation}\label{interface eq'}
\TDt^2\p_t^2\psi = -\frac{1}{\rho} N\cdot \p_t^2\Tpp \qc+\Rbf_{\psi}+\Rbf_{\rho}+\Rbf_{\psi,\rho}, \quad \text{ on }\Sigma,
\end{equation}
where $\Rbf_{\psi}, \Rbf_{\rho}$, and $\Rbf_{\psi,\rho}$ are given respectively in \eqref{Rbf psi}, \eqref{Rbf rho} and \eqref{Rbf psi rho}. 

\subsubsection{The reformulation in Alinhac good unknowns}
In the next, we introduce $\Q$ to be the Alinhac's good unknown of $\qc$ associated with $\p_t^2$:
\begin{equation}
\Q :=\p_t^2 \qc - \p_t^2\varphi\pp_3 \qc,\quad \text{ in }\Omega. 
\end{equation}
For $j=1,2,3$, similar to \eqref{AGU comm 1}, we have
\begin{equation}\label{pt2 AGU interior}
\p_t^2 \nabp_j \qc = \nabp_{j} \Q + C_{j}(\qc),\quad \text{ in }\Omega.
\end{equation}
Here, for a generic function $f$, we define
\begin{align}\label{def C}
C_{i}(f) = \pp_i\pp_3 f \p_t^2\varphi+ C'_i(f),\quad i=1,2, \text{ and }\quad 
C_3(f) = (\pp_3)^2 f \p_t^2\varphi + C'_3(f),
\end{align}
where
\begin{align*}
C_i'(f) =
-\left[\p_t^2, \frac{\p_i \varphi}{\p_3\varphi}, \p_3 f\right]-\p_3 f \left[ \p_t^2, \p_i \varphi, \frac{1}{\p_3\varphi}\right] -\p_i\varphi \p_3 f\p_t\left( \frac{1}{(\p_3\varphi)^2}\right)\p_t \p_3 \varphi,
\end{align*}
 and 
\begin{align*}
C_3'(f) = 
\left[ \p_t^2, \frac{1}{\p_3\varphi}, \p_3 f\right] + \p_3 f\p_t\left(\frac{1}{(\p_3\varphi)^2}\right) \p_t \p_3 \varphi.
\end{align*}
Note that $\p_3\varphi|_{\Sigma}=1$, \eqref{pt2 AGU interior} then yields
\begin{align}
\p_t^2 \Tpp_j \qc = \Tpp_j\Q+\Cbf_j(\qc), \quad \text{ on }\Sigma, 
\end{align}
where $\Cbf_i(\qc) = \Tpp_i \p_3\qc \p_t^2\psi-[\p_t^2,\p_i\psi, \p_3\qc]$ when $i=1,2$, and $\Cbf_3(\qc) = \p_3^2\qc \p_t^2\psi$. 
Therefore, the equation \eqref{interface eq'} turns into
\begin{equation}\label{interface eq}
\TDt^2\p_t^2\psi = -\frac{1}{\rho} N\cdot \nabp \Q-\frac{1}{\rho}N\cdot \Cbf(\qc)+\Rbf_{\psi}+\Rbf_{\rho}+\Rbf_{\psi,\rho}, \quad \text{ on }\Sigma.
\end{equation}

Parallel to $\Q$, we define $\mathcal{V}$ to be the Alinhac's good unknown of $v$ associated with $\p_t^2$:
\begin{equation}
\mathcal{V}:= \p_t^2v - \p_t^2\varphi\pp_3 v,\quad \text{ in }\Omega. 
\end{equation}
Then, similar to \eqref{AGU TT mom}--\eqref{AGU TT mass}, $(\mathcal{V},\Q)$ verifies 
\begin{equation}\label{AGU pt^2}
\begin{aligned}
&\rho \Dtp \mathcal{V} + \nabp \Q  = G^1,\quad &\text{ in }\Omega, \\
&\lam^2\Dtp \Q +\nabp\cdot \mathcal{V} =G^2-C_i(v^i),\quad &\text{ in }\Omega, 
\end{aligned}
\end{equation}
where we write $\ff'(q)=\lam^2$ for simplicity of notations\footnote{This is reasonable when discussing the incompressible limit, as explained in Section \ref{sect mach number definition}.}, with
\begin{align*}
&G^1_i= -[\p_t^2,\rho]\Dtp v_i-\rho D(v_i) - C_i(\qc)- (\p_t^2\rho) g \delta_{i3},\quad i=1,2,3,\\
&G^2 = -\lam^2D(\qc) +\lam^2 g \p_t^2v_3.
\end{align*}
Here, for a generic function $f$, we define
$$
D(f) =( \Dtp \pp_3 f) (\p_t^2\varphi)+D'(f), 
$$
with
\begin{align*}
D'(f) = [\p_t^2,\vb]\cdot\TP f+\left[\p_t^2,\frac{1}{\p_3\varphi}(v\cdot\NN-\p_t\varphi),\p_3 f\right]+\left[\p_t^2, v\cdot\NN-\p_t\varphi,\frac{1}{\p_3\varphi}\right] \p_3 f\\
+\frac{1}{\p_3\varphi}[\p_t^2,v]\cdot\NN\p_3f -(v\cdot\NN-\p_t\varphi)\p_3 f \p_t \left(\frac{1}{(\p_3\varphi)^{2}}\right) \p_t\p_3\varphi. 
\end{align*}

In the next, we commute the divergence operator $\nabp\cdot$ to the first equation of \eqref{AGU pt^2} to obtain:
\begin{align}\label{Wave}
\begin{aligned}
\rho \lam^2 (\Dtp)^2 \Q  - \lap^{\varphi} \Q =\rho\pp_i v^k \pp_k \mathcal{V}^i +\nabp\rho\cdot \Dtp\mathcal{V}+\rho\Dtp \left(G^2-C_i(v^i)\right)-\nabp\cdot G^1. 
\end{aligned}
\end{align}

\subsubsection{Decomposition of the pressure: Dirichlet-to-Neumann operator}
Since 
$$\Q= \p_t^2\qc -\p_t^2\psi \p_3 \qc= \sigma \p_t^2 \mathcal{H} -\p_3 q\p_t^2\psi,\quad \text{ on }\Sigma, $$
we define $\Q= \Q_h+\Q_w$, where $\Q_h$ solves the elliptic equation
\begin{align}
\begin{aligned}
&-\lap^{\varphi} \Q_h =0,\quad&\text{ in }\Omega,\\
&\Q_h = \sigma \p_t^2 \mathcal{H}-\p_3 q\p_t^2\psi, \quad&\text{ on }\Sigma,\\
&\p_3\Q_h = 0,\quad&\text{ on }\Sigma_b,
\end{aligned}
\end{align}
and $\Q_w$ satisfies
\begin{align}\label{Qw}
\begin{aligned}
 & - \lap^{\varphi} \Q_w =-\rho \lam^2(\Dtp)^2 \Q+ \rho\pp_i v^k \pp_k \mathcal{V}^i +\nabp\rho\cdot \Dtp\mathcal{V} +\rho\Dtp \left(G^2-C_i(v^i)\right)-\nabp\cdot G^1,\quad &\text{ in }\Omega,\\
&\Q_w=0, \quad&\text{ on }\Sigma,\\
&\p_3 \Q_w=\p_3\Q=-\p_t^2 \rho g, \quad&\text{ on }\Sigma_b,
\end{aligned}
\end{align}
where $\p_3\Q|_{\Sigma}$ is computed by restricting the third component of the first equation in \eqref{AGU pt^2} on $\Sigma_b$. 

With this decomposition, we can further reduce the evolution equation of the free surface \eqref{interface eq} by introducing the Dirichlet-to-Neumann operator.
\begin{defn}[Dirichlet-to-Neumann (DtN) operator] \label{defn DtN}
For a function $f:\Sigma\to\R$, we define the Dirichlet-to-Neumann (DtN) operator associated with $(\Om,\psi)$ by
\begin{align}
\dn f:= N\cdot\nabp(\HE_\psi f),
\end{align}where $\HE_\psi f$ is defined to be the harmonic extension of $f$ into $\Om$, namely 
\begin{align}
-\lapp(\HE_\psi f)=0\quad \text{ in }\Om,\quad \HE_\psi f=f\quad \text{ on }\Sigma,\quad \p_3(\HE_\psi f)=0\quad\text{ on }\Sigma_b.
\end{align}
\end{defn}
With this definition, we can rewrite
\begin{align*}
N\cdot\nabp \Q =&~ N\cdot\nabp\Q_h+ N\cdot\nabp \Q_w =\dn(\sigma\p_t^2 \h -\p_3 q \p_t^2\psi) +  N\cdot\nabp \Q_w\\
=&~\sigma\dn(\p_t^2 \h) - \dn(\p_3 q\p_t^2\psi) + N\cdot\nabp \Q_w,
\end{align*}and thus the evolution equation \eqref{interface eq} becomes
\begin{align}\label{psi equ}
\rho\TDt^2\p_t^2\psi + \sigma\dn(\p_t^2 \h) - \dn(\p_3 q\p_t^2\psi) =- N\cdot\nabp \Q_w - N\cdot\mathbf{C}(\qc)+\rho(\Rbf_{\psi}+\Rbf_{\rho}+\Rbf_{\psi,\rho}) \quad\text{ on }\Sigma.
\end{align}

\subsection{Preliminaries on pardifferential calculus}\label{sect para pre}

In the equation \eqref{psi equ}, the term involving DtN operators is fully nonlinear, so we shall find out their concrete forms in order for an explicit energy estimate. In the remaining part of this paper, we will introduce several preliminary lemmas about paradifferential calculus that have been proven in Alazard-Burq-Zuily \cite{ABZ2014wwSTLWP}. Following the notations in M\'etivier \cite{MetivierPara}, we first introduce the basic definition of a paradifferential operator. Note that the dimension $d$ in this section is actually the Hausdorff dimension of the free surface.
\begin{defn}[Symbols]
Given $r\geq 0,~m\in\R$, we denote $\Gamma_r^m(\R^d)$ to be the space of locally bounded functions $a(x',\xi)$ on $\R^d\times(\R^d\backslash\{0\})$, which are $C^\infty$ with respect to $\xi (\xi\neq \bd{0})$, such that for any $\alpha\in\N^d,\xi\neq \bd{0}$, the function $x'\mapsto \p_\xi^\alpha a(x',\xi)$ belongs to $W^{r,\infty}(\R^d)$ and there exists a constant $C_\alpha$ such that
\[
\bno{\p_\xi^\alpha a(\cdot,\xi)}_{W^{r,\infty}(\R^d)}\leq C_\alpha (1+|\xi|)^{m-|\alpha|},~~\forall|\xi|\geq 1/2.
\]
\end{defn} 

\begin{defn}[Paradifferential operator]
Given a symbol $a$, we shall define the \textbf{paradifferential operator }$T_a$ by
\begin{align}\label{para FT}
\widehat{T_a u}(\xi):=(2\pi)^{-d}\int_{\R^d} \tilde{\chi}(\xi-\eta,\eta)\hat{a}(\xi-\eta,\eta)\phi(\eta)\hat{u}(\eta)\deta
\end{align}where $\hat{a}(\theta,\xi)=\int_{\R^d}\exp(-ix'\cdot\theta) a(x',\xi)\dx'$ is the Fourier transform of $a$ in variable $x'\in\R^d$. Here $\tilde{\chi}$ and $\phi$ are two given cut-off functions such that
\[
\phi(\eta)=0~~\text{for }|\eta|\leq 1,\quad \phi(\eta)=1~~\text{for }|\eta|\geq 2,
\]and $ \tilde{\chi}(\theta,\eta)$ is homogeneous of degree 0 and satisfies that for $0<\eps_1<\eps_2\ll 1$, $ \tilde{\chi}(\theta,\eta)=1$ if $|\theta|\leq \eps_1|\eta|$ and $ \tilde{\chi}(\theta,\eta)=0$ if $|\theta|\geq \eps_2|\eta|$. We also introduce the semi-norm
\begin{align}
M_r^{\alpha}(a):=\sup_{|\alpha|\leq \frac{d}{2}+1+r}\sup_{|\xi|\geq 1/2}\bno{(1+|\xi|)^{|\alpha|-m}\p_\xi^\alpha a(\cdot,\xi)}_{W^{r,\infty}(\R^d)}.
\end{align}
\end{defn}

For $m\in\R$, we say $T$ is of order $m$ if for all $s\in\R$, $T$ is bounded from $H^s$ to $H^{s-m}$.
\begin{prop}
Let $m\in\R$. If $a\in \Gamma_0^m(\R^d)$, then $T_a$ is of order $m$. Moreover, for any $s\in\R$, there exists a constant $K$ such that $\|T_a\|_{H^s\to H^{s-m}}\leq K M_0^m(a)$.
\end{prop}

\begin{prop}[Composition, {\cite[Theorem 3.7]{ABZ2014wwSTLWP}}]\label{prop para composition}
Let $m\in\R$ and $r>0$. If $a\in\Gamma_r^m(\R^d),~b\in\Gamma_r^{m'}(\R^d)$, then $T_aT_b-T_{a\# b}$ is of order $m+m'-r$ where \[
a\# b:= \sum_{|\alpha|<r}\frac{1}{i^{|\alpha|}\alpha!}\p_\xi^\alpha a \p_{x'}^\alpha b
\]and $\p_{x'}=(\TP_{x_1},\TP_{x_2})$. Moreover, for all $s\in\R$, there exists a constant $K$ such that
\begin{align}
\|T_aT_b - T_{a\# b}\|_{H^s\to H^{s-m-m'+r}}\leq K M_r^m(a) M_r^{m'}(b).
\end{align}
\end{prop}

\begin{prop}[Adjoint, {\cite[Theorem 3.10]{ABZ2014wwSTLWP}}]\label{prop para adjoint}
Let $m\in\R$, $r>0$ and $a\in\Gamma_r^m(\R^d)$. We denote by $(T_a)^*$ the adjoint operator of $T_a$. Then $(T_a)^*-T_{a^*}$ is of order $m-r$ where $$a^*:=\sum_{|\alpha|<r}\frac{1}{i^{\alpha}\alpha!}\p_\xi^\alpha \p_{x'}^\alpha \bar{a}.$$ Moreover, for any $s\in\R$, there exists a constant $K$ such that $\|(T_a)^*-T_{a^*}\|_{H^s\to H^{s-m+r}}\leq KM_r^m(a)$.
\end{prop}

Here and thereafter in this section, $\psi\in C([0,T];H^{s+\frac12}(\R^d))$ is a given function with $s>2+\frac{d}{2}$. The symbolic calculus is not defined for $C^\infty$ symbols, so we need to introduce the following symbols. 
\begin{defn}
Given $m\in\R$, we denote $\Sigma^m$ to be the class of symbols $a$ of the form $a=a^{(m)}+a^{(m-1)}$ with
\[
a^{(m)}(t,x',\xi)=F(\p_{x'}\psi(t,x'),\xi),\quad a^{(m-1)}(t,x',\xi)=\sum_{|\alpha|_2}G_\alpha(\p_{x'}\psi(t,x'),\xi)\p_{x'}^\alpha\psi(t,x')
\]such that 
\begin{itemize}
\item [i.] $T_a$ maps real-valued functions to real-valued functions;
\item [ii.] $F$ is a $C^\infty$ real-valued functions of $(\zeta,\xi)\in\R^d\times(\R^d\backslash\{0\})$, homogeneous of degree $m$ in $\xi$, such that there exists a continuous function $K=K(\zeta)>0$ such that $F(\zeta,\xi)\geq K(\zeta)|\xi|^m$ for all $(\zeta,\xi)\in\R^d\times(\R^d\backslash\{0\})$;
\item [iii.] $G_\alpha$ is a $C^\infty$ complex-valued function of $(\zeta,\xi)\in\R^d\times(\R^d\backslash\{0\})$, homogeneous of degree $m-1$ in $\xi$.
\end{itemize}
\end{defn} 

\begin{defn}[Equivalence of paradifferential operators]
Given $m\in\R$ and consider two families of operators of order $m$: $\{A(t):t\in[0,T]\}$ and  $\{B(t):t\in[0,T]\}$, we say $A\sim B$ if $A-B$ has order $m-1.5$ and satisfies the estimate: for all $r\in\R$ there exists a continuous function $C(\cdot)$ such that
\[
\forall t\in[0,T], ~~\|A(t)-B(t)\|_{H^r\to H^{r-(m-1.5)}}\leq C(|\psi(t)|_{s+\frac12}).
\]
\end{defn}  From now on, we use the notation $|\cdot|_{s_1\to s_2}$ to represent the operator norm $\|\cdot\|_{H^{s_1}\to H^{s_2}}$, and use the notation $|\cdot|_{s}$ to represent $\|\cdot\|_{H^s(\R^d)}$. We have the following theorem for the composition
\begin{prop}[{\cite[Prop. 4.3]{ABZ2014wwSTLWP}}]
Let $m,m'\in\R$. Then
\begin{enumerate}
\item If $a\in\Sigma^m,~b\in\Sigma^{m'}$, then $T_aT_b\sim T_{a\# b}$ where $a\# b$ is given by 
\[
a\# b = a^{(m)}b^{(m')} + a^{(m-1)} b^{(m')} + a^{(m)}b^{(m'-1)} + \frac{1}{i} \p_\xi a^{(m)}\cdot\p_{x'} b^{(m')}.
\]
\item  If $a\in\Sigma^m$, then $(T_a)^*\sim T_b$ where $b\in \Sigma^m$ is given by
\[
b=a^{(m)}+\overline{a^{(m-1)}}+\frac{1}{i}(\p_{x'}\cdot\p_{\xi})a^{(m)}.
\]
\end{enumerate}
\end{prop}We denote $\Re z$ and $\Im z$ to be the real part and the imaginary part of a complex number $z$, respectively. As a corollary, we have
\begin{cor}[{\cite[Prop. 4.3(2)]{ABZ2014wwSTLWP}}]
If $a\in \Sigma^m$ satisfies $\Im a^{(m-1)}=-\frac12 (\p_\xi \cdot\p_{x'})a^{(m)}$, then $(T_a)^*\sim T_a$.
\end{cor}

The next proposition is significant for estimates in Sobolev norms via paradifferential calculus. 
\begin{prop}[{\cite[Prop. 4.4 and 4.6]{ABZ2014wwSTLWP}}]\label{prop para Hs}
Let $m\in\R,~r\in\R$. Then for all symbols $a\in\Sigma^m$ and $t\in[0,T]$, the following estimate holds.
\begin{align}
|T_{a(t)} u|_{r-m} \leq &~C(|\psi(t)|_{s-1})|u|_{r},\\
|u|_{r+m}\leq &~ C(|\psi(t)|_{s-1})\left(|T_{a(t)} u|_{r} + |u|_0\right).
\end{align}
\end{prop}

\subsection{Paralinearization of evolution equation of the free surface}\label{sect para DtN}
Now we can start to paralinearize the term involving $\dn$ and $\h$ in \eqref{psi equ}.
\begin{lem}[Paralinearization of the DtN operator, {\cite[Sect. 4.4]{AMDtN}}]\label{prop para DtN}
For $f,\psi\in H^{s+\frac12}(\R^d)$, we have
\begin{align}\label{DtNsymbol}
\dn f = T_{\Lam} f + R_{\Lam}^\psi(f) ,
\end{align}with the symbols $\Lam=\Lam^{(1)}+\Lam^{(0)}$ give by
\begin{align}
\Lam^{(1)}=&\sqrt{(1+|\cnab_{x'}\psi|^2)|\xi|^2 - (\cnab_{x'}\psi\cdot\xi)^2},\\
\Lam^{(0)}=&\frac{1+|\cnab_{x'}\psi|^2|}{2\Lam^{(1)}}\left(\cnab_{x'}\cdot(\alpha^{(1)}\cnab_{x'}\psi)+i\p_\xi \Lam^{(1)}\cdot\p_{x'}\alpha^{(1)}\right),
\end{align}and $\alpha^{(1)}:=(\Lam^{(1)}+i\cnab_{x'}\psi\cdot \xi)/(1+|\cnab_{x'}\psi|^2)$. The remainder terms satisfy the following estimates
\begin{align}
|R_\Lam^\psi(f)|_{r}\leq C(|\psi|_{s+\frac12})|f|_{r}\quad \forall \frac12\leq r\leq s-\frac12,~s>2+\frac{d}{2}.
\end{align}
\end{lem}

Next, we paralinearize the mean curvature term. Let $\h(\psi)=-\cnab\cdot\left(\frac{\cnab\psi}{\sqrt{1+|\cnab\psi|^2}}\right)$. We have
\begin{lem}[Paralinearization of the mean curvature, {\cite[Lemma 3.25]{ABZ2014wwSTLWP}}]\label{prop para ST}
There holds $\h(\psi)=T_\FH \psi+R_\FH$ where $\FH=\FH^{(2)}+\FH^{(1)}$ is defined by
\begin{align}
\FH^{(2)}=&\frac{1}{\sqrt{1+|\cnab_{x'}\psi|^2}}\left(|\xi|^2-\frac{(\cnab_{x'}\psi\cdot\xi)^2}{1+|\cnab_{x'}\psi|^2}\right),\\
\FH^{(1)}=&-\frac{i}{2}(\p_{x'}\cdot\p_\xi)\FH^{(2)},
\end{align}and the remainder term $R_\FH$ satisfies
\begin{align}
|R_\FH|_{2s-3}\leq C(|\psi|_{s+\frac12}).
\end{align}
\end{lem}

Now, we can treat the nonlinear terms on the left side of \eqref{psi equ}. The term involving surface tension is treated as follows
\begin{align}
\sigma\dn(\p_t^2\h(\psi))=&~\sigma \dn(\p_t^2 T_\FH\psi)+ \sigma \dn(\p_t^2 R_\FH) \notag \\
\label{psi equ ST}=&~\sigma T_{\Lam}T_\FH(\p_t^2 \psi) +\sigma \dn([\p_t^2,T_\FH]\psi + \p_t^2 R_\FH) + \sigma R_{\Lam}^\psi(T_\FH\p_t^2(\psi))\\
=:&~\sigma T_{\Lam}T_\FH(\p_t^2 \psi) + \Rbf^{ST}_\psi \notag
\end{align} The term involving the Rayleigh-Taylor sign is treated as follows
\begin{align}
\label{psi equ RT} \dn(\p_3 q\p_t^2\psi)=&~(\p_3 q)\dn^{\frac12}\dn^{\frac12}(\p_t^2\psi) + [\dn,\p_3 q]\p_t^2\psi.\\
=:&~(\p_3 q)\dn^{\frac12}\dn^{\frac12}(\p_t^2\psi)+ \Rbf^{RT}_\psi \notag
\end{align}
Now, the evolution equation \eqref{psi equ} becomes
\begin{align}\label{psi equ para}
\rho\TDt^2\p_t^2\psi + \sigma T_\Lam T_\FH (\p_t^2\psi)  +(-\p_3 q)\dn^{\frac12}\dn^{\frac12}(\p_t^2\psi) = &-N\cdot\nabp \Q_w  +\Rbf^\sigma_\psi+\Rbf^{RT}_\psi\\
&-N\cdot\mathbf{C}(\qc)+\rho(\Rbf_{\psi}+\Rbf_{\rho}+\Rbf_{\psi,\rho}) \quad\text{ on }\Sigma. \notag
\end{align}

\subsection{Uniform estimates for the free surface}
To obtain an explicit energy estimate via \eqref{psi equ para}, we will symmetrize the third-order paradifferential operator $T_{\Lam}T_\FH$. That is, find suitable symbols $\fm \in \Sigma^{1.5}$ and $\fn\in \Sigma^0$ such that $T_\fn T_\Lam T_\FH \sim T_\fm T_\fm T_\fn$ and $T_\fm \sim (T_\fm)^*$.

\begin{prop}[Symmetrization, {\cite[Prop. 4.8]{ABZ2014wwSTLWP}}]\label{prop para symm}
Let $\fn\in\Sigma^0$ and $\fm\in \Sigma^{1.5}$ be defined by 
\begin{align}
\label{symbol n} \fn:=&~\frac{1}{\sqrt[4]{1+|\cnab\psi|^2}}=|N|^{-\frac12},\\
\label{symbol m} \fm :=&~ \underbrace{\sqrt{\FH^{(2)}\Lam^{(1)}}}_{=:\fm^{(1.5)}}+\underbrace{\frac{1}{2i}(\p_\xi\cdot\p_{x'})\sqrt{\FH^{(2)}\Lam^{(1)}}}_{=:\fm^{(0.5)}}.
\end{align} Then $T_\fn T_\Lambda T_\FH \sim T_\fm T_\fm T_\fn$ and $T_\fm \sim (T_\fm)^*$ are both fulfilled.
\end{prop}

Recall that we need the uniform bounds for $|\p_t^3\psi|_{1.5}^2$, so we shall take the 1.5-th order derivative in \eqref{psi equ para}. Since the symbol $\fm$ also belongs to $\Sigma^{1.5}$, we alternatively consider the $T_\fm$-differentiated evolution equation thanks to the symmetrization result. We introduce the following energy functional
\begin{align}
\MM(t):=\frac12\is \rho\bno{T_\fm T_\fn\TDt \p_t^2\psi}^2 +\sigma \bno{T_\fm T_\fm T_\fn \p_t^2\psi}^2 + \frac{c_0}{2} \bno{\dn^{\frac12} T_\fm T_\fn \p_t^2\psi}^2\dx'.
\end{align}In view of Proposition \ref{prop para Hs} and Lemma \ref{lem DtNHs}, we have the comparison between $\MM(t)$ and standard Sobolev norms
\begin{align}
\label{para compare 1}\MM(t)\lesssim& \bno{\TDt\p_t^2\psi}_{1.5}^2+\sigma|\p_t^2\psi|_3^2+\frac{c_0}{4}|\p_t^2\psi|_2^2;\\
\label{para compare 2} \bno{\TDt\p_t^2\psi}_{1.5}^2\lesssim& \bno{T_\fm T_\fn\TDt \p_t^2\psi}_0^2 + \bno{\TDt \p_t^2\psi}_0^2,\\
\label{para compare 3}\sigma|\p_t^2\psi|_3^2\lesssim&~\sigma \bno{T_\fm T_\fm T_\fn \p_t^2\psi}_0^2+\sigma|\p_t^2\psi|_0^2,\quad  |\p_t^2\psi|_2^2\lesssim \bno{\dn^{\frac12} T_\fm T_\fn \p_t^2\psi}_0^2+|\p_t^2\psi|_0^2.
\end{align} For those $L^2(\Sigma)$ norms, we invoke the kinematic boundary condition, trace lemma, and Young's inequality to get
\[
|\p_t^2\psi|_0^2\lesssim\|v_t\|_1^2|\cnab\psi|_{L^{\infty}}^2+\|v\|_1^2|\cnab\p_t\psi|_{L^{\infty}}^2
\]and
\begin{align*}
|\p_t^3\psi|_0^2\lesssim\eps\|\p_t^2 v\|_2^2+\|\p_t^2 v\|_0^2|\cnab\psi|_{L^{\infty}}^4 + \|v_t\|_1^2|\cnab\p_t\psi|_{L^{\infty}}^2+\|v\|_2^2|\cnab\p_t^2\psi|_{0}^2.
\end{align*}The $\eps$-term contributes to $\eps\EE_4(t)$. The term $\|\p_t^2 v\|_0^2$ can be controlled via $\p_t^2$-estimates of \eqref{CWWST0} in which there is no loss of $\lam$-weight in the corresponding commutators of Alinhac good unknowns, as we take a full-time derivative $\p_t^2$. The other terms contain at most one time derivative and thus can be controlled directly. Thus, we have 
\begin{align}\label{para compare}
\MM(t)\lesssim\EE_4(t)\text{  and  }  \bno{\p_t^3\psi}_{1.5}^2+\sigma|\p_t^2\psi|_3^2+|\p_t^2\psi|_2^2\lesssim \MM(t)+\eps \EE_4(t) + \text{controllable terms}.
\end{align}So, it suffices to control $\MM(t)$ in order to establish the bound for $\bno{\p_t^3\psi}_{1.5}^2+\sigma|\p_t^2\psi|_3^2+|\p_t^2\psi|_2^2$ in $\EE_4(t)$.

We are ready to estimate $\MM(t)$. By time differentiating the first term of $\MM(t)$, we obtain
\begin{equation}\label{para energy 1}
\begin{aligned}
&\ddt\frac12\is \rho\bno{T_\fm T_\fn\TDt \p_t^2\psi}^2\dx'\\
=&\is T_\fm T_\fn(\rho \TDt^2 \p_t^2\psi)(T_\fm T_\fn \TDt \p_t^2\psi)\dx'+\is [\rho-1, T_\fm T_\fn]\TDt^2 \p_t^2\psi\,(T_\fm T_\fn \TDt \p_t^2\psi)\dx'\\
&+\frac12\is (\p_t \rho + \cnab\cdot(\rho\vb))\bno{T_\fm T_\fn \TDt \p_t^2\psi}^2\dx' + \is \rho [\TDt, T_\fm T_\fn](\TDt \p_t^2\psi)\,(T_\fm T_\fn\TDt \p_t^2\psi)\dx'\\
=:&\is T_\fm T_\fn (\rho\TDt^2 \p_t^2\psi)(T_\fm T_\fn\TDt \p_t^2\psi)\dx' + R_1^M
\end{aligned}
\end{equation}
Next, by plugging the paralinearized equation \eqref{psi equ para} into the first term in the last line above, we obtain
\begin{equation}\label{para energy 2}
\begin{aligned}
&\is T_\fm T_\fn(\rho \TDt^2 \p_t^2\psi)(T_\fm T_\fn \TDt \p_t^2\psi)\dx' \\
\lleq&-\sigma\is  T_\fm T_\fm T_\fn(\p_t^2\psi) \,  T_\fm(T_\fm T_\fn \TDt \p_t^2\psi) \dx' -\is (-\p_3 q)  \dn^{\frac12}(T_\fm T_\fn \p_t^2\psi)\,  \dn^{\frac12}(T_\fm T_\fn \TDt \p_t^2\psi)\dx'\\
&-\is T_\fm T_\fn(N\cdot\nabp \Q_w)\,(T_\fm T_\fn \TDt \p_t^2\psi)\dx'\\
&+\io [T_\fm T_\fn, \p_3 q] \dn(\p_t^2\psi)\,(T_\fm  T_\fn \TDt \p_t^2\psi) \dx'- \io (-\p_3 q) [T_\fm T_\fn,\dn]\p_t^2\psi\,(T_\fm T_\fn \TDt \p_t^2\psi)\dx'\\
&-\is \dn^{\frac12}(T_\fm T_\fn \p_t^2\psi)\, [\dn^{\frac12},\p_3 q](T_\fm T_\fn \TDt \p_t^2\psi) \dx'\\
& + \is T_\fm T_\fn( \Rbf^{ST}_\psi+\Rbf^{RT}_\psi-N\cdot\mathbf{C}(\qc)+\rho(\Rbf_{\psi}+\Rbf_{\rho}+\Rbf_{\psi,\rho}))\,(T_\fm T_\fn \TDt \p_t^2\psi) \dx' \\
=:&~M^{ST}+M^{RT} +M^W +R_2^M+R_3^M+R_4^M+R_5^M.
\end{aligned}
\end{equation}
Here we use $T_\fm T_\fm T_\fn\sim T_\fn T_\Lam T_\FH$ and $T_\fm\sim T_\fm^*$ to derive $M^{ST}$ and omit the low-order error terms in this equivalence. We also use the self-adjointness of the DtN operator in $L^2(\Sigma)$ to derive $M^{RT}$. Note that we may not use $T_\Lam$ to replace $\dn$ in the term involving the Rayleigh-Taylor sign, as we do not have $T_\Lam\sim T_\Lam^*$. The major terms are $M^{ST}, M^{RT}$ and $M^W$, among which the first two terms contribute to the boundary regularity with or without $\sigma$-weight, while the term $M^W$ contributes to a fifth-order term that motivates us to involve $E_5(t)$ in the energy functional $\EE(t)$. The control of  $R_2^M,\cdots, R_5^M$ and other commutators generated by $M^{ST}, M^{RT}$ and $M^W$ will be postponed to the end of this section.

The term $M^{ST}$ contributes to $\sqrt{\sigma}$-weighted boundary regularity. We have
\begin{equation}\label{para energy ST}
\begin{aligned}
M^{ST}=&-\sigma\is  T_\fm T_\fm T_\fn(\p_t^2\psi) \,  T_\fm(T_\fm T_\fn \TDt \p_t^2\psi) \dx' \\
=&-\frac{\sigma}{2}\ddt\is \bno{T_\fm T_\fm T_\fn \p_t^2\psi}^2\dx' - \sigma\is  T_\fm T_\fm T_\fn(\p_t^2\psi) \,  [T_\fm T_\fm T_\fn, \TDt] \p_t^2\psi \dx' -\frac{\sigma}{2}( \cnab\cdot \vb)\bno{T_\fm T_\fm T_\fn \p_t^2\psi}^2\dx' \\
=:&-\frac{\sigma}{2}\ddt\is \bno{T_\fm T_\fm T_\fn \p_t^2\psi}^2\dx'+R_6^M.
\end{aligned}
\end{equation}

The term $M^{RT}$ contributes to non-weighted boundary regularity with the help of the Rayleigh-Taylor sign condition $-\p_3 q\geq c_0/2>0$. We have
\begin{equation}\label{para energy RT}
\begin{aligned}
M^{RT}=&-\is (-\p_3 q)  \dn^{\frac12}(T_\fm T_\fn \p_t^2\psi)\,  \dn^{\frac12}(T_\fm T_\fn \TDt \p_t^2\psi)\dx'\\
=&-\frac12\ddt\io (-\p_3 q)\bno{ \dn^{\frac12}(T_\fm T_\fn \p_t^2\psi)}^2\dx' \\
&+\is \p_3 q\dn^{\frac12}(T_\fm T_\fn \p_t^2\psi)\,  [\dn^{\frac12}T_\fm T_\fn, \TDt]\p_t^2\psi \dx'-\frac12\io \p_t\p_3 q + \cnab\cdot(\p_3 q \vb)\bno{ \dn^{\frac12}(T_\fm T_\fn \p_t^2\psi)}^2\dx' \\
=:&-\frac12\ddt\io (-\p_3 q)\bno{ \dn^{\frac12}(T_\fm T_\fn \p_t^2\psi)}^2\dx' +R_7^M.
\end{aligned}
\end{equation}

Currently, we have arrived at the following energy inequality
\begin{equation}\label{EE4 energy 1}
\MM(t)\leq \MM(0)+\int_0^t M^W(\tau)\dtt+\sum_{j=1}^7\int_0^t R_j^M(\tau)\dtt.
\end{equation}

\subsection{Weighted fifth-order energy}\label{sect E5}
\subsubsection{Necessity of fifth-order energy}\label{sect E5-1}
Recall that $M^W=-\is T_\fm T_\fn(N\cdot\nabp \Q_w)\,(T_\fm T_\fn \TDt \p_t^2\psi)\dx'$ and $T_\fm T_\fn$ is a 1.5-th order paradifferential operator, so it remains to control $|N\cdot \nabp \Q_w|_{1.5}$ in order for the control of $M^W$. Using trace theorem, we have $|N\cdot \nabp \Q_w|_{1.5}\leq |\cnab\psi|_{1.5}\|\nabp \Q_w\|_{2}$. Then, we use the following div-curl inequality
\begin{align}
\|\nabp \Q_w\|_{2}^2\leq C(|\cnab\psi|_{W^{1,\infty}},|\psi|_{2.5})\left(\|\nabp \Q_w\|_{0}^2+\|\lapp \Q_w\|_{1}^2+\|\nabp\times\nabp \Q_w\|_{1}^2+|N\times \nabp \Q_w|_{1.5}^2+|\p_3\Q_w|_{H^{1.5}(\Sigma_b)}^2\right),
\end{align}where the third and the fourth terms are all zero because $\nabp\times\nabp f=\bf{0}$ and $\Q_w$ has zero boundary value on $\Sigma$. The fifth term is easy to control; we have
\begin{align}
|\p_3\Q_w|_{H^{1.5}(\Sigma_b)}^2=|\p_t^2 \rho g|_{H^{1.5}(\Sigma_b)}^2\lesssim \lam^2\|\p_t^2 q\|_{2}^2.
\end{align} The first term is of lower order, and we omit the treatment. For the second term, invoking \eqref{Qw}, we have
\begin{align}
\|\lapp \Q_w\|_{1}\lesssim&~C\left(\sum_{k=2}^4 |\lam^{(k-2)_+}\p_t^k\psi|_{4-k}\right)\left(\sum_{k=2}^4\|\lam^2 \p_t^k \qc\|_{5-k} + \|\p_t^2 v\|_2\|\p v\|_2 + \lam^2 \|\p q\|_{2} \|\p_t^3 v\|_1 +P(\EE_4(t))\right),
\end{align} where the first term requires the control of $E_5(t)$. It should be noted that the $\lam^2$ in the third term is generated from $\nabp\rho$ such that the term $\Dtp \mathcal{V}$ can be controlled without loss of $\lam$-weight. The last two terms on the right side of \eqref{Qw} can be directly controlled by $P(\EE_4(t))$, as the number of derivatives does not exceed four and the number of time derivatives does not exceed 2. Therefore, the energy inequality \eqref{EE4 energy 1} becomes 
\begin{equation}\label{EE4 energy 1'}
\MM(t)\leq \MM(0)+\int_0^t P(\EE_4(\tau))E_5(\tau)\dtt+\sum_{j=1}^7\int_0^t R_j^M(\tau)\dtt,
\end{equation} and it remains to control 
$$E_5(t):=\sum_{k=0}^5\ino{\lam^2\p_t^k(v, \lam^{(k-4)_+}\qc)}_{5-k}^2+\bno{\sqrt{\sigma}\lam^2\p_t^k\psi}_{6-k}^2+\bno{\lam^2\p_t^k\psi}_{5-k}^2$$ 
uniformly in $\ls$. 

\subsubsection{Control of $E_5(t)$ and the remaining terms in $\EE_4(t)$}\label{sect E5-2}
Notice that $E_5(t)$ has exactly the same structure as the energy $E(t)$ used to prove the local existence in Theorem \ref{main thm, WP} if we remove the weight $\lam^2$: all quantities except the top-order time derivative of $\qc$ share the same weight of Mach number. This indicates us to use the div-curl inequality in Lemma \ref{hodgeTT} and follow the same strategies as in Section \ref{sect uniformkk} instead of using the one in Lemma \ref{hodgeNN} to establish the following energy inequality
\begin{align}\label{E5 energy 0}
E_5(t)\leq P(\EE(0))+P(\EE(t))\int_0^t P(\EE(\tau))\dtt.
\end{align}
Also, notice that $\p_t^3v, \p_t^3 \qc$ and $\p_t^4 v,\p_t^4\qc$ in $\EE_4(t)$ also share the same weight of Mach number, so we can still control them by following the same strategies as in Section \ref{sect uniformkk}. What's different is that the high-order time derivatives in both $\EE_4(t)$ and $E_5(t)$ may need more weights in order for the uniform boundedness. Thus, it remains to carefully check if there is any loss of $\lam$-weight in the control of commutators $\cc'(\qc),\cc_i'(v_i)$ and $\dd'(\qc),\dd'(v_i)$ in the tangential estimates. Let us recall the concrete forms of these commutators when $\TT^\alpha$ has the form $\TP^k\p_t^l$.
\begin{align}
\mathfrak{C}_i'(f) =\left[ \TT^\alpha, \frac{\NN_i}{\p_3\varphi}, \p_3 f\right]+\p_3 f \left[ \TT^\alpha, \NN_i, \frac{1}{\p_3\varphi}\right] +\NN_i\p_3 f\left[\TT^{\alpha-\gamma}, \frac{1}{(\p_3\varphi)^2}\right] \TT^\gamma \p_3 \varphi
\end{align}
with $|\gamma|=1$, and 
\begin{align}
\mathfrak{D}'(f) =&~ [\TT^\alpha, \vb]\cdot \TP f + \left[\TT^\alpha, \frac{1}{\p_3\varphi}(v\cdot \NN-\p_t\varphi), \p_3 f\right]+\left[\TT^\alpha, v\cdot \NN-\p_t\varphi, \frac{1}{\p_3\varphi}\right]\p_3 f+\frac{1}{\p_3\varphi} [\TT^\alpha, v]\cdot \NN \p_3 f\nonumber\\
&-(v\cdot \NN-\p_t\varphi)\p_3 f\left[ \TT^{\alpha-\gamma}, \frac{1}{(\p_3 \varphi)^2}\right]\TT^\gamma \p_3 \varphi.
\end{align}

Here, we only check the most difficult cases and omit the other easier ones: $\cc(\qc)$ and $\dd(\qc)$ in $\lam\p_t^3\TP$-estimates, $\lam\p_t^4$-estimates (for $\EE_4(t)$), $\lam^2\p_t^4\TP$-estimates and $\lam^2\p_t^5$-estimates (for $E_5(t)$). Note that there is no need to check the same commutators for $v_i$ because the power of $\lam$ weight that $\p_t^kv_i$ needs never exceeds that for $\p_t^k \qc$.

\paragraph*{$\lam\p_t^3\TP$-estimates for $\EE_4(t)$.} We shall control $\|\lam\cc_i(\qc)\|_0$ uniformly in $\lam$ when $\TT^\alpha=\p_t^3\TP$ and $f=\qc$. The worst case is that all time derivatives fall on $\qc$ and such terms have the following forms
\[
\lam\TP(\NN_i/\p_3\varphi) \p_t^3\p_3 \qc,\quad \lam\TP\left(\frac{1}{\p_3\varphi}(v\cdot \NN-\p_t\varphi)\right) \p_t^3\p_3  \qc,
\]whose $L^2$ norms are bounded by $P(|\cnab\psi,\p_t\psi|_{W^{1,\infty}},\|\TP v\|_{L^{\infty}})\|\lam\p_t^3 \qc\|_1$.

\paragraph*{$\lam\p_t^4$-estimates for $\EE_4(t)$.} We shall control $\|\lam\cc_i(\qc)\|_0$ uniformly in $\lam$ when $\TT^\alpha=\p_t^4$ and $f=\qc$. The worst case is that three out of the four time derivatives fall on $\qc$, and such terms have the following forms
\[
\lam\p_t(\NN_i/\p_3\varphi) \p_t^3\p_3 \qc,\quad \lam\p_t\left(\frac{1}{\p_3\varphi}(v\cdot \NN-\p_t\varphi)\right) \p_t^3\p_3 \qc,
\]whose $L^2$ norms are bounded by $P(|\cnab\psi,\p_t\psi|_{W^{1,\infty}},\|\TP v\|_{L^{\infty}},|\p_t^2\psi|_{L^{\infty}})\|\lam\p_t^3 \qc\|_1$.

\paragraph*{$\lam^2\p_t^4\TP$-estimates for $E_5(t)$.} We shall control $\|\lam^2\cc_i(\qc)\|_0$ uniformly in $\lam$ when $\TT^\alpha=\p_t^4\TP$ and $f=\qc$. Although every quantity in $E_5(t)$ needs $\lam^2$-weight, the terms involving $\qc$ become a lower order term and contain at most one time derivative if there are five derivatives falling on $\NN$ or $v\cdot\NN$. Since $\p_t\nab \qc$ is uniformly bounded in $L^{\infty}(\Om)$, there is no need to put extra effort on such terms. The worst case is still that all time derivatives fall on $\qc$ and such terms have the following forms:
\[
\lam^2\TP(\NN_i/\p_3\varphi) \p_t^4\p_3 \qc,\quad \lam^2\TP\left(\frac{1}{\p_3\varphi}(v\cdot \NN-\p_t\varphi)\right) \p_t^4 \p_3 \qc
\]whose $L^2$ norms are bounded by $P(|\cnab\psi,\p_t\psi|_{W^{1,\infty}},\|\TP v\|_{L^{\infty}})\|\lam^2\p_t^4\qc\|_1$. As for the intermediate terms, we check the case that $\p_t^3$ falls on $\p_3\qc$ and $\TP\p_t$ falls on $(v\cdot \NN-\p_t\varphi)$ because neither of these two terms can be uniformly bounded in $L^\infty$. We have
\begin{align*}
&\left\|\lam^2\p_t\TP((\p_3\varphi)^{-1}(v\cdot\NN-\p_t\varphi)) \p_t^3\p_3\qc\right\|_0\lesssim \|\lam^2 \p_t^3\p_3\qc\|_{L^6}\|\p_t\TP((\p_3\varphi)^{-1}(v\cdot\NN-\p_t\varphi))\|_{L^3}\\
\lesssim&~\|\lam^2 \p_t^3\qc\|_{2}(|\psi_t|_{2.5}+|\psi_{tt}|_{1.5})\|v_t\|_3 P(|\cnab\psi,\p_t\psi|_{W^{1,\infty}})\leq \sqrt{E_5(t)} P(\EE_4(t)).
\end{align*}

\paragraph*{$\lam^2\p_t^5$-estimates for $E_5(t)$.} We shall control $\|\lam^2\cc_i(\qc)\|_0$ uniformly in $\lam$ when $\TT^\alpha=\p_t^5$ and $f=\qc$. Again, the worst case is still that all time derivatives fall on $\qc$ and such terms have the following forms
\[
\lam^2\p_t(\NN_i/\p_3\varphi) \p_t^4\p_3 \qc,\quad \lam^2\p_t\left(\frac{1}{\p_3\varphi}(v\cdot \NN-\p_t\varphi)\right) \p_t^4 \p_3 \qc
\]whose $L^2$ norms are bounded by $P(|\cnab\psi,\p_t\psi|_{W^{1,\infty}},\|\TP v\|_{L^{\infty}},|\p_t^2\psi|_{L^{\infty}})\|\lam^2\p_t^4\qc\|_1$. As for the intermediate terms, we check the case that $\p_t^3$ falls on $\p_3\qc$ and $\p_t^2$ falls on $(v\cdot \NN-\p_t\varphi)$. We have that
\begin{align*}
&\left\|\lam^2\p_t^2((\p_3\varphi)^{-1}(v\cdot\NN-\p_t\varphi)) \p_t^3\p_3\qc\right\|_0\lesssim\|\lam \p_t^3\qc\|_{0}(|\psi_t|_{2.5}+|\psi_{tt}|_{2.5}+|\lam\psi_{ttt}|_{1.5})\|v_t\|_3 P(|\cnab\psi,\p_t\psi|_{W^{1,\infty}})\leq P(\EE_4(t)).
\end{align*} The omitted terms can be controlled similarly or more easily. Thus, we conclude that inequality \eqref{E5 energy 0} holds.

\subsubsection{Uniform estimates for $\EE(t)$ and the incompressible limit}\label{sect control E5}
So far, we have obtained the following energy inequalities.
\begin{enumerate}
\item The terms involving less than 2 time derivatives in $\EE_4(t)$:
\begin{align}
\sum_{k=0}^1\|\p_t^k v\|_{4-k}^2+\sigma|\p_t^k\psi|_{5-k}^2+|\psi|_4^2+|\p_t\psi|_{3.5}^2+\|\qc\|_4^2\lesssim P(\EE_4(0))+\int_0^t P(\EE_4(\tau))\dtt,\\
\|\p_t\qc\|_3^2\lesssim \|\p_t^2 v\|_2^2 + P(\EE_4(0))+\int_0^t P(\EE_4(\tau))\dtt.
\end{align}These two inequalities are proved in the same way as in Section \ref{sect uniformkk}.
\item  The terms involving 3 and 4 time derivatives in $\EE_4(t)$:
\begin{align}
\sum_{k=3}^4\|\lam \p_t^k v,\lam^{1+(k-3)_+}\p_t^k \qc \|_{4-k}^2+\sigma|\lam\p_t^k\psi|_{5-k}^2+|\lam\p_t^3\psi|_{1.5}^2+|\lam\p_t^4\psi|_{0.5}^2\lesssim P(\EE_4(0))+\int_0^t P(\EE_4(\tau))\dtt,
\end{align}which is obtained by following the same strategy as in Section \ref{sect uniformkk} and the analysis of commutator in Section \ref{sect E5-2}.
\item Control of $E_5(t)$:
\begin{align}
E_5(t)\leq P(\EE(0))+P(\EE(t))\int_0^t P(\EE(\tau))\dtt.
\end{align}
\item Control of $\|v_{tt}\|_2^2$ in $\EE_4(t)$ via paradifferential calculus:
\begin{align}
\|v_{tt}\|_2^2\lesssim&~ \eps \EE_4(t)+ \sum_{j=1}^7\int_0^t R_j^M(\tau)\dtt +  P(\EE_4(0))+\int_0^t P(\EE_4(\tau))\dtt,\\
\|\lam\p_t^2 \qc\|_2^2\lesssim&~ \|\lam \p_t^3 v\|_1^2+ P(\EE_4(0))+\int_0^t P(\EE_4(\tau))\dtt.
\end{align}
\end{enumerate}
Summing up the above estimates, we can prove the Gr\"onwall-type inequality for the energy functional $\EE(t)$
\begin{align}
\EE(t)\leq P(\EE(0))+P(\EE(t))\int_0^t P(\EE(\tau))\dtt,\quad \text{uniformly in }\ls,
\end{align}provided that we have the bounds for the remainders $|R_j^M(t)|_{L^2}\leq P(\EE(t))$ for $1\leq j\leq 7$.
From a similar argument as in Section \ref{subsect: 4.7}, we conclude from \eqref{generalized gronwall} that there exists some $t=T_1>0$, independent of $\lam$ and $\sigma$, such that 
$$
\sup_{0\leq t\leq T_1}\EE(t) \leq 2 P(\EE(0)). 
$$

 Since the first-order time derivatives in $\EE_4(t)$ remain uniformly bounded, we can obtain the same convergence result as in Section \ref{sect double limits}, and we no longer repeat the statement here.

\subsection{Control of commutators involving paradifferential operators}\label{sect para remainder}
At the end of this paper, it remains to prove that $|R_j^M(t)|_{L^2}\leq P(\EE(t))$ for $1\leq j\leq 7$. It should be noted that many time derivatives are involved in these remainders, so the commutator estimates presented in \cite{ABZ2014wwSTLWP} may not be directly applicable. In particular, we only show the control of the most difficult ones:
\begin{itemize}
\item Two commutators in $R_1^M$: $[T_\fm T_\fn,\rho-1]\TDt^2\p_t^2\psi$ and $[T_\fm T_\fn, \TDt]\TDt\p_t^2\psi$.
\item The commutator in $R_4^M$: $[\dn^{\frac12},\p_3 q]T_\fm T_\fn\p_t^2\psi$.
\item The concrete forms of $\Rbf_\psi^{ST}$ and $\Rbf_\psi^{RT}$.
\end{itemize}The control of the other terms in these remainders will be omitted. In fact, $R_2^M, R_6^M, R_7^M$ can be controlled in the same way as $R_1^M$. The control of $R_3^M$ is easier than that of $R_2^M$ as the function contains fewer time derivatives. The other terms in $R_5^M$, apart from $\Rbf_\psi^{ST}$ and $\Rbf_\psi^{RT}$,  can be controlled in $L^2(\Sigma)$ uniformly in $\ls$ by directly counting the number of derivatives.

We start with the first one. Writing $f=\TDt^2\p_t^2\psi$ and $a=\rho-1$ for convenience, we have
\begin{align*}
[T_\fm T_\fn,a] f= T_\fm([T_\fn, a]f) + [T_\fm,a](T_\fn f),
\end{align*} where the two terms share similar structures, and we only show the control of the first one. Using Bony's paraproduct decomposition in Appendix \ref{sect app para defn}, we rewrite this commutator as
\begin{align}\label{para RM 1}
[T_\fn, a] f =&~T_\fn T_{a} f+T_\fn T_f a +T_\fn (R(a, f)) - T_{a}T_\fn f - T_{T_\fn f}a- R(a, T_\fn f) \notag \\
=&~[T_\fn, T_a]f+T_\fn T_f a - T_{T_\fn f}a +T_\fn (R(a, f)) - R(a, T_\fn f) 
\end{align}Here we must $a:=\rho-1$ instead of $\rho$ because $\rho\gtrsim 1$ does not belong to $L^2$. This also avoids the loss of $\lam$-weight in $f=\TDt^2\p_t^2 \psi$, as $\rho-1=O(\lam^2)$. The last two terms on the right side of \eqref{para RM 1} are controlled by using Lemma \ref{lem para norm}
\begin{align}
|T_\fn (R(a, f))|_{1.5}\lesssim&~ |R(a, f)|_{1.5}\lesssim |\rho-1|_{2.5}|f|_0\lesssim |\lam^2\TDt^2\p_t^2\psi|_0,\\
|R(a, T_\fn f)|_{1.5}\lesssim&~|a|_{1.5}|T_\fn f|_{0.5}\lesssim|\lam^2\TDt^2\p_t^2\psi|_{0.5}.
\end{align} 

Next, we control the commutator $[T_\fn, T_a]f$. Since $\fn, a$ are both function depending on $x'\in\R^2$, not a symbol depending on both $x'$ and the frequency variable $\xi\in\R^2$, we have $a\# \fn=\fn\# a=a\fn$ and thus
\[
\|T_aT_\fn-T_\fn T_a\|_{0.5\to 1.5}\leq \|T_aT_\fn-T_{\fn a}\|_{0.5\to 1.5}+\|T_\fn T_a-T_{\fn a}\|_{0.5\to 1.5}\lesssim M_1^0(\fn)M_1^0(a)\lesssim C(|\psi|_{C^2})|a|_{W^{1,\infty}},
\]which leads to
\begin{align}
|[T_\fn, T_a]f|_{1.5}\lesssim C(|\psi|_{C^2})|a|_{W^{1,\infty}}|f|_{0.5}\lesssim C(|\psi|_{C^2})|\lam^2\TDt^2\p_t^2\psi|_{0.5}.
\end{align} The other two terms in \eqref{para RM 1} are controlled in the same way, and we only show the control of $T_\fn T_f a$. Since $\fn\in\Sigma^0$, it suffices to control $|T_f a|_{1.5}$. Using Plancherel's identity and the definition \eqref{para FT} of paradifferential operators, we have
\begin{align}
|T_f a|_{1.5}=|\len{\xi}^{1.5}\widehat{T_f a}(\xi)|_{L_\xi^2(\R^2)}=(2\pi)^{-2}\bno{\int_{\R^2}\len{\xi}^{1.5} \tilde{\chi}(\xi-\eta,\eta)\hat{f}(\xi-\eta)\phi(\eta)\hat{a}(\eta)\deta}_{L_\xi^2}.
\end{align}
By definition of $\tilde{\chi}$ and $\phi$ (see Appendix \ref{sect app para defn}), we know that the integrand is nonzero only if $|\eta|> 1$ and $|\xi-\eta|<\eps_2|\eta|$ for some $0<\eps_2\ll 1$, which means $\len{\xi}$ and $\len{\eta}$ are comparable: $(1-\eps_2)|\eta|\leq |\xi|\leq (1+\eps_2)|\eta|$. Then, using this, $|\tilde{\chi}|\leq1, |\phi|\leq 1$ and Minkowski's inequality for integrals, we have
\begin{align}
|T_f a|_{1.5}\lesssim&~\bno{\int_{\R^2} \tilde{\chi}(\xi-\eta,\eta)\hat{f}(\xi-\eta)\len{\eta}^{1.5}\phi(\eta)\hat{a}(\eta)\deta}_{L_\xi^2}\notag\\
(\forall 0<\delta<1)\quad\quad\lesssim&~|\hat{f}|_{L^2(\R^2)}\left|\len{\eta}^{-1-\delta}\len{\eta}^{2.5+\delta}\hat{a}(\eta)\right|_{L_\eta^1}\notag\\
\lesssim&~|\hat{f}|_{L^2(\R^2)}|\len{\eta}^{-1-\delta}|_{L_\eta^2(\R^2)}|\len{\eta}^{2.5+\delta}\hat{a}(\eta)|_{L_\eta^2(\R^2)}\lesssim |f|_0|a|_3\leq|\lam^2\TDt^2\p_t^2\psi|_0|q|_3.
\end{align}

Next, we analyze the commutator $[T_\fm T_\fn, \TDt]f$ for $f=\TDt\p_t^2\psi$. Since $\TDt=\p_t+\vb\cdot\cnab$ and $\p_t$ is a time derivative, we only show the details for the control of $[T_\fm T_\fn, \p_t]f$. Expanding this commutator, we have
\[
[T_\fm T_\fn, \p_t]f= T_\fm([T_\fn ,\p_t ]f) + [T_\fm,\p_t]T_\fn f.
\]Again, these two terms have similar structures, so we only focus on the first one, that is, the control of $|[T_\fn, \p_t ]f|_{1.5}$. We have that $[T_\fn ,\p_t ]f=- T_{\p_t \fn} f$, so using Lemma \ref{lem para norm}, we have
\begin{align}
|T_{\p_t \fn} f|_{1.5}\lesssim |\p_t\fn|_{L^{\infty}}|f|_{1.5}\lesssim C(|\cnab\psi|_{L^{\infty}})|\cnab\psi_t|_{L^{\infty}}|\TDt\p_t^2\psi|_{1.5}\leq P(\EE_4(t)).
\end{align}

Next, we analyze the commutator $[\dn^{\frac12},a]f$ with $a:=\p_3 q$ and $f:=T_\fm T_\fn\p_t^2\psi$. Using Lemma \ref{lem DtN1/2}, we have
\begin{align}
\forall s>3,\quad |[\dn^{\frac12},a]f|_0\lesssim C(|\psi|_s)|a|_{1.5}|f|_{0.5}\lesssim C(|\psi|_s)\|q\|_3|\p_t^2\psi|_2\leq P(\EE_4(t)).
\end{align} Also, the term $\Rbf_\psi^{RT}:=[\dn,\p_3 q]\p_t^2\psi$ is controlled in the same way
\begin{align}
\forall s>3,\quad |[\dn,\p_3 q]\p_t^2\psi|_0\lesssim C(|\psi|_s)|\p_3 q|_{1.5}|\p_t^2\psi|_{1}\leq P(\EE_4(t)).
\end{align} 

Finally, we need to establish the $L^2(\Sigma)$ of $\Rbf_{\psi}^{ST}:=\sigma \dn([\p_t^2,T_\FH]\psi + \p_t^2 R_\FH) + \sigma R_{\Lam}^\psi(T_\FH\p_t^2(\psi)).$ The difficulty is that this term simultaneously contains the commutators between a paradifferential operator and $\p_t^2$, the time derivatives of $R_\FH$, which is not explicitly calculated in previous works about incompressible fluids \cite{ABZ2014wwLWP, ABZ2014wwSTLWP, SWZ2015MHDLWP}, and the control of remainders for the DtN operator. Among the three terms in $\Rbf_{\psi}^{ST}$, the last one is directly controlled by using Lemma \ref{lem DtNR} and Proposition \ref{prop para Hs}:
\begin{align}
|\sigma R_{\Lam}^\psi(T_\FH\p_t^2(\psi))|_0\lesssim \sigma|T_\FH\p_t^2(\psi)|_0\lesssim C(|\psi|_{W^{1,\infty}})|\psi_{tt}|_2\leq P(\EE_4(t)).
\end{align}
Next, we control the first term in $\Rbf_{\psi}^{ST}$. In view of Lemma \ref{lem DtNHs}, it remains to control $|\sigma[\p_t^2, T_\FH]\psi|_1$. Expanding the commutators, we have
\[
[\p_t^2, T_\FH]\psi= T_{\p_t^2\FH}\psi + 2 T_{\p_t\FH}\p_t\psi.
\] We only analyze the first one as the symbol contains second-order time derivative and $\p_t^2\FH\notin C^2(\Sigma)$ and the second one is directly controlled with the help of Proposition \ref{prop para Hs}. Again, using the definition \eqref{para FT}, Plancherel's identity and Minkowski's inequality, we have
\begin{align}
|T_{\p_t^2\FH}\psi|_1=&~\bno{\len{\xi}\int_{\R^2}\chi(\xi-\eta)\widehat{\p_t^2\FH}(\xi-\eta)\phi(\eta)\hat{\psi}(\eta)\deta}_{L_\xi^2(\R^2)}\no\\
(\forall 0<\delta<1)\quad\lesssim&~|\widehat{\p_t^2\FH}|_{L^2}|\len{\eta}^{-1-\delta}|_{L_\eta^2(\R^2)}|\len{\eta}^{2+\delta}\hat{\psi}(\eta)|_{L_\eta^2(\R^2)}\no\\
\lesssim&~C(|\cnab\psi,\cnab\psi_t|_{L^{\infty}})|\cnab\psi_{tt}|_{0}|\psi|_{2+\delta}\lesssim P(\EE_4(t)).
\end{align}
The last step is to control $\dn(\p_t^2R_\FH)$ and it suffices to control $|\p_t^2 R_\FH|_1$. This step is actually a refinement of \cite[Lemma 3.25]{ABZ2014wwSTLWP}. Recall that the mean curvature is given by $\h=-\cnab\cdot F(\cnab\psi)$ with $F(x):=\frac{x}{\sqrt{1+x^2}}$ and $F(0)=0$. We expand $F$ into 2nd-order term to get
\[
F(a)=0+F'(a)a+\frac{F''(\zeta)}{2}a^2=T_{F'(a)}a+T_a (F'(a)) +R(a, F'(a))+\frac{F''(\zeta)}{2}a^2.
\]Let $a=\cnab\eta$ and then $T_{F'(a)}a $ is exactly the term $T_\FH\psi$ defined in Lemma \ref{prop para ST} and
\begin{align*}
R_\FH=&T_{\cnab\psi}\left(\frac{Id}{\sqrt{1+|\cnab\psi|^2}}-\frac{\cnab\psi\otimes\cnab\psi}{(\sqrt{1+|\cnab\psi|^2})^3}\right) + R(\cnab\psi, F'(\cnab\psi))+(\cnab\psi)^\top\cdot \mathbf{M}(\cnab\psi) \cdot(\cnab\psi)
\end{align*} where $\mathbf{M}(\cnab\psi)$ is a $2\times 2$ matrix depending on $\cnab\psi$. Thus, the leading-order part in the last two terms of $\p_t^2 R_\FH$ must have the form $(\cnab\psi_{tt}+\cnab\psi_{t}\cdot\cnab\psi_t)C'(\cnab\psi)$ for some continuous function $C'(\cdot)$, while the first term is controlled by either using Lemma \ref{lem para norm} or following the same way as in the control of $[\p_t^2, T_\FH]\psi$. We conclude the result as follows
\begin{align}
\sigma|\p_t^2 R_\FH|_0\lesssim C(|\cnab\psi,\cnab\psi_t|_{L^{\infty}})(|\sigma\cnab\psi_{tt}|_{L^2}+|\sqrt{\sigma}\cnab\psi_t|_{L^{\infty}}|\sqrt{\sigma}\cnab\psi_t|_0)\leq P(\EE_4(t)).
\end{align}
Now, we have finished all estimates and the proof of improved incompressible limit ends here.

\paragraph*{Acknowledgment.} We would like to thank Tao Wang and Kai Zhou for reading and commenting on the first draft of this paper. Additionally, CL would like to thank his Ph.D. student, Hetong Wang, for reviewing the most recent version of this paper.  Finally, we thank the anonymous referee for carefully reading our manuscript and helping us improve it.
\subsection*{Ethic Declarations.}
\paragraph*{Conflict of interest.} The authors declare that they have no conflict of interest.

\paragraph*{Data avaliability.} Data sharing is not applicable as no datasets were generated or analyzed in this article.

\begin{appendix}
\section{Useful auxiliary results}
\subsection{The Reynold's transport theorems}
Below, the formulas involving $\pk,\psk$ are used for the nonlinear $\kk$-problem \eqref{CWWSTkk} and the formulas involving $\pkr,\pskr$ are used for the linearized $\kk$-problem \eqref{CWWSTlin}.
\begin{lem}\label{time deriv transport pre}
Let $f,g$ be smooth functions defined on $[0,T]\times \Omega$. Then:
\begin{align}
\frac{d}{dt} \int_\Omega fg \p_3\pk\,dx= \int_\Omega (\ppk_t f)g\p_3 \pk\,dx +\int_\Omega f(\ppk_t g)\p_3\pk\,dx+\int_{x_3=0}fg\p_t \psi\,dx'+\io fg \p_3\p_t (\pk-\varphi)\,dx,\label{time deriv transport pre tilde}\\
\frac{d}{dt} \int_\Omega fg \p_3\pkr\,dx= \int_\Omega (\ppkr_t f)g\p_3 \pkr\,dx +\int_\Omega f(\ppkr_t g)\p_3\pkr\,dx+\int_{x_3=0}fg\p_t \psr\,dx'+\io fg \p_3\p_t (\pkr-\pr)\,dx.\label{time deriv transport pre tildering}
\end{align}
\end{lem}
\begin{proof} 
In view of \eqref{ptk}, 
\begin{align*}
\text{LHS of\,\,}\eqref{time deriv transport pre tilde} &= \io (\p_t f) g \p_3 \pk\dx+\io f(\p_t g) \p_3\pk\dx+\io fg \p_3\p_t \pk\dx\\
&=\io fg \p_3\p_t \pk\dx+\io (\ppk_t f)g\p_3\pk\dx +\io f (\ppk_t g)\p_3\pk \dx+\overbrace{\io \p_t \varphi(\p_3 f)g\dx}^{i}+\overbrace{\io \p_t \varphi(\p_3 g)f\dx}^{ii}.
\end{align*}
Integrating $\p_3$ in $ii$ by parts, we have
\begin{align*}
ii = \int_{x_3=0} fg \p_t \psi \,dx'-\int_{x_3=-b}fg\underbrace{\p_t\varphi}_{=\p_t(-b)=0}\dx'-\io fg \p_3\p_t \varphi\dx-i.
\end{align*}
This concludes the proof of \eqref{time deriv transport pre tilde}.  Moreover, in light of \eqref{ptkr}, 
\begin{align*}
\text{LHS of\,\,}\eqref{time deriv transport pre tildering} &= \io (\p_t f) g \p_3 \pkr\dx+\io f(\p_t g) \p_3\pkr\dx+\io fg \p_3\p_t \pkr\dx\\
&=\io fg \p_3\p_t \pkr\dx+\io (\ppkr_t f)g\p_3\pkr\dx +\io f (\ppkr_t g)\p_3\pkr \dx+\overbrace{\io \p_t \pr(\p_3 f)g\dx}^{\mathring{i}}+\overbrace{\io \p_t \pr(\p_3 g)f\dx}^{\mathring{ii}}.
\end{align*}
Integrating $\p_3$ in $\mathring{ii}$ by parts, we have
\begin{align*}
\mathring{ii} = \int_{x_3=0} fg \p_t \psr \,dx'-\io fg \p_3\p_t \pr\dx-\mathring{i},
\end{align*}
and thus \eqref{time deriv transport pre tildering} follows. 
\end{proof}
\begin{lem}[\textbf{Integration by parts for covariant derivatives}] \label{int by parts lem}
Let $f, g$ be defined as in Lemma \ref{time deriv transport pre}. Then:
\begin{align}
\int_\Omega (\ppk_i f)g \p_3 \pk \dx= -\int_\Omega f(\ppk_i g)\p_3 \pk\dx+ \int_{x_3=0} fg \npk_i\,dx',\label{int by parts tilde}\\ 
\int_\Omega (\ppkr_i f)g \p_3 \pkr \dx= -\int_\Omega f(\ppkr_i g)\p_3 \pkr\dx+ \int_{x_3=0} fg \npkr_i\,dx'.\label{int by parts tildering}
\end{align}
\end{lem}
\begin{proof}
\eqref{int by parts tilde} follows from the fact that $\ppk_i$ is the covariant spatial derivative and $\p_3\psk\dx$ is the associated volume element. \eqref{int by parts tildering} follows from a parallel argument. 
\end{proof}

Let $\Dtpk$ be the smoothed material derivative defined in \eqref{Dtpk}. Then the following theorem holds.
\begin{thm}[\textbf{Reynold's transport theorem for nonlinear $\kk$-problem}]\label{transport thm nonlinear}
Let $f$ be a smooth function defined on $[0, T]\times\Omega$. Then:
\begin{align}
\frac{1}{2}\frac{d}{dt}  \io \rho |f|^2 \p_3\pk\dx = \io \rho (\Dtpk f)f\p_3 \pk\dx+\frac{1}{2} \io\rho|f|^2 \p_3\p_t (\pk-\varphi)\dx.  \label{transpt nonlinear}
\end{align}
\end{thm}
\begin{proof}
First, we express
\begin{equation*}
\io \rho (\Dtpk f)f\p_3 \pk\dx = \io \rho (\ppk_t f)f\p_3\pk\dx+\io \rho (v\cdot \nabpk f)f\p_3\pk\dx. 
\end{equation*}
Invoking \eqref{time deriv transport pre tilde}, we have
$$
\io \rho (\ppk_t f)f\p_3\pk\dx = \p_t \io \rho |f|^2\p_3\pk\dx-\io \ppk_t (\rho f) f\p_3\pk\dx-\int_{x_3=0} \rho|f|^2 \p_t\psi\,dx'-\io \rho |f|^2 \p_3\p_t (\pk-\varphi)\dx,
$$
and this indicates that
\begin{align}
\io \rho (\ppk_t f)f\p_3\pk\dx =\frac{1}{2}\frac{d}{dt}\io \rho|f|^2 \p_3\pk\dx \overbrace{-\frac{1}{2}\io (\ppk_t \rho)|f|^2 \p_3\pk\dx}^{A}\overbrace{-\frac{1}{2}\int_{x_3=0}\rho|f|^2\p_t\psi \,dx'}^{C}-\frac{1}{2}\io \rho |f|^2 \p_3\p_t (\pk-\varphi)\dx. \label{A6}
\end{align}
Furthermore, invoking \eqref{int by parts tilde}, we have
\begin{align*}
\io \rho (v\cdot \nabpk f)f\p_3\pk\dx = \io \nabpk\cdot (\rho v f)f\p_3\pk\dx- \io \nabpk\cdot (\rho v)|f|^2\p_3\pk \dx\\
=-\io \rho f(v\cdot \nabpk f)\p_3\pk\dx +\int_{x_3=0}\rho|f|^2v\cdot \npk\,dx'- \io \nabpk\cdot (\rho v)|f|^2\p_3\pk \dx,
\end{align*}
and thus
\begin{align}
\io \rho (v\cdot \nabpk f)f\p_3\pk\dx =\overbrace{\frac{1}{2}\int_{x_3=0} \rho|f|^2 v\cdot \npk\,dx'}^{D} \overbrace{-\frac{1}{2} \io \nabpk\cdot (\rho v)|f|^2\p_3\pk \dx}^{B}. \label{A7}
\end{align}
We have $A+B=C+D=0$ thanks to the second and fifth equations of \eqref{CWWSTkk}, respectively.  Hence,  \eqref{transpt nonlinear} follows after adding \eqref{A6} and \eqref{A7} up. 
\end{proof}
Theorem \ref{transport thm nonlinear} leads to the following two corollaries. The first one records the integration by parts formula for $\Dtpk$. 
\begin{cor}[\textbf{Reynold's transport theorem for nonlinear $\kk$-problem}] \label{transport thm without rho}
It holds that
\begin{align}
\frac{d}{dt}  \io fg \p_3\pk\dx =  \io (\Dtpk f)g \p_3\pk\dx+\io f(\Dtpk g)\p_3\pk\dx+\io (\nabpk\cdot v) fg\p_3\pk\dx+\io fg \p_3\p_t (\pk-\varphi)\dx. \label{transpt nonlinear without rho}
\end{align}
\end{cor}
\begin{proof}
Given \eqref{time deriv transport pre tilde},  we have
$$
\io  (\ppk_t f)g\p_3\pk\dx = \frac{d}{dt} \io  fg\p_3\pk\dx-\io   f(\ppk_t g)\p_3\pk\dx-\int_{x_3=0} fg \p_t\psi\,dx'-\io fg \p_3\p_t (\pk-\varphi)\dx,
$$
Also, \eqref{int by parts tilde} yields 
\begin{align*}
\io  (v\cdot \nabpk f)g\p_3\pk\dx = \io \nabpk\cdot ( v f)g\p_3\pk\dx- \io (\nabpk\cdot v) fg\p_3\pk \dx\\
=-\io  f(v\cdot \nabpk g)\p_3\pk\dx +\int_{x_3=0}fg (v\cdot \npk)\,dx'- \io (\nabpk\cdot  v)fg\p_3\pk \dx.
\end{align*}
Then we obtain \eqref{transpt nonlinear without rho} by adding these up. 
\end{proof}
 The second corollary concerns the transport theorem as well as the integration by parts formula for the linearized material derivative $\Dtpkr$, defined in \eqref{Dtpkr}. 

\begin{cor}[\textbf{Reynold's transport theorem for linearized $\kk$-problem}] \label{transport thm linearized}
Let $\Dtpkr:=\p_t+(\vbr\cdot\cnab)+\frac{1}{\p_3\pkr}(\vr\cdot\Npkd-\p_t\pr)\p_3$ be the linearized material derivative defined in \eqref{Dtpkr}. Then:
\begin{align}
\frac{1}{2}\frac{d}{dt}  \io \rhor |f|^2 \p_3\pkr\dx =& \io \rhor (\Dtpkr f)f\p_3 \pkr\dx
+\frac{1}{2}\io \left( \Dtpkr \rhor +\rhor\nabpkr\cdot \vr\right) |f|^2\p_3\pkr\,dx \label{transpt linearized} \\
&+\frac{1}{2}\io \rhor |f|^2 \left(\p_3\p_t (\pkd-\pr)+\p_3(\p_t+\vbr\cdot\cnab)(\pkr-\pkd)\right)\dx.  \nonumber
\end{align}
\begin{align}
\frac{1}{2}\frac{d}{dt}  \io |f|^2 \p_3\pkr\dx = & \io (\Dtpkr f)f \p_3\pkr\dx+\frac{1}{2}\io \nabpkr\cdot \vr |f|^2\p_3\pkr\dx \label{transpt linearized without rho}\\
&+\frac{1}{2}\io |f|^2 \left(\p_3\p_t (\pkd-\pr)+\p_3(\p_t+\vbr\cdot\cnab)(\pkr-\pkd)\right)\dx.\nonumber
\end{align}
\end{cor}
\begin{proof}
It suffices to show \eqref{transpt linearized} only since the proof of \eqref{transpt linearized without rho}  follows by setting $\rhor=1$. We write the first term on the RHS of \eqref{transpt linearized} as
\begin{align}
\io \rhor (\Dtpkr f)f\p_3 \pkr\dx=&\io\rhor(\p_tf)f\p_3\pkr\dx+\io \rhor \left((\vbr\cdot\cnab)f\right)f\p_3\pkr\dx +\io\rhor\left((\vr\cdot \Npkd-\p_t\pr)\p_3 f\right) f\dx,
\end{align}
and then integrate $\p_t$, $\cnab$ and $\p_3$ by parts respectively in these terms to get:
\begin{equation}\label{thm A11}
\begin{aligned}
\io \rhor (\Dtpkr f)f\p_3 \pkr\dx
=&\ddt\frac12\io\rhor|f|^2\p_3\pkr\dx-\frac12\io\left(\p_t\rhor+\vbr\cdot\cnab\rhor+\frac{1}{\p_3\pkr}(\vr\cdot\Npkd-\p_t\pr)\p_3\rhor\right) |f|^2\p_3\pkr\dx\\
&-\frac12\io\rhor(\cnab\cdot\vbr)|f|^2\p_3\pkr\dx-\frac12\io\rhor|f|^2(\p_t+\vbr\cdot\cnab)\p_3\pkr\dx\\
&-\frac12\io\rhor\p_3(-(\vbr\cdot\cnab)\pkd+\vr_3-\p_t\pr)|f|^2\dx,
\end{aligned}
\end{equation}
where we used $\vr \cdot \Npkd = -(\vbr \cdot \cnab) \pkd+\vr_3$ in the last line.
We find that the second integral in the first line is $\int_{\Omega}\Dtpkr\rhor|f|^2\p_3\pkr\dx$. Also, the term in the last line can be written as
\begin{equation}\label{thm A12}
\begin{aligned}
&-\frac12\io\rhor\p_3(-(\vbr\cdot\cnab)\pkd+\vr_3-\p_t\pr)|f|^2\dx\\
=&-\frac12\io\rhor|f|^2\left(\frac{1}{\p_3\pkr}\p_3v_3-\frac{\TP_1\pkr}{\p_3\pkr}\p_3\vr_1-\frac{\TP_2\pkr}{\p_3\pkr}\p_3\vr_2\right)\p_3\pkr\dx\\
&+\frac12\io\rhor |f|^2\p_3\vbr\cdot\cnab(\pkd-\pkr)\dx+\frac12\io\rhor |f|^2(\p_t\p_3\pr+(\vbr\cdot\cnab)\p_3\pkd)\dx.
\end{aligned}
\end{equation}
The first term on the RHS together with the third term in \eqref{thm A11} 
contributes to
\begin{align*}
\frac{1}{2}\io \rhor (\nabpkr\cdot \vr)|f|^2\p_3\pkr\dx
\end{align*}
in
 \eqref{transpt linearized}. Meanwhile, the terms in the last line of  \eqref{thm A12} together with the fourth term in \eqref{thm A11} give the terms in \eqref{transpt linearized} with mismatches.
\end{proof}
\subsection{Generalized Gr\"onwall's inequality}
The following inequality shall be used frequently when closing the estimates for nonlinear PDEs, which can be found in Tao \cite{Tao}. 
\begin{thm}\label{thm: generalized gronwall}
Let $I\subset \mathbb{R}$ be a (possibly infinite) time interval, and $\mathcal{E}\in C_{\text{loc}}^0(I \rightarrow \mathbb{R}^+)$ be a function obeying the inequality
\begin{align}\label{generalized gronwall}
\mathcal{E}(t) \leq A +\epsilon F(\mathcal{E}(t)), 
\end{align}
for some $A, \epsilon>0$ and some function $F:\mathbb{R}^+\to \mathbb{R}^+$ which is locally bounded. Suppose also that $u(t_0)\leq 2A$ for some $t_0\in I$. If $\epsilon$ is sufficiently small, then $u(t)\leq 2A$ for all $t\in I$. 
\end{thm}
\begin{proof}
We proceed by the standard bootstrap argument. It suffices to show that we can improve the bootstrap assumption $u(t) \leq 2A$ after choosing $\epsilon$ suitably small. 
Since $F$ is locally bounded, we have $F(y)\leq M$, $\forall y\in B(0,2A)$. Now, we choose $\epsilon\leq \frac{A}{2M}$, then 
        $$
        u(t) \leq A + \frac{A}{2M}\cdot M = \frac{3}{2}A.
        $$
\end{proof}

\section{Construction of initial data for the original system}\label{sect CWWST data}
This section aims to construct the initial data for Theorem \ref{main thm, double limits} and Theorem \ref{main thm, well data} satisfying the compatibility conditions
\begin{equation*}
(\Dtp)^j q |_{\{t=0\}\times \Sigma} =  (\Dtp)^j (\sigma \mathcal{H})|_{\{t=0\}\times \Sigma },\quad \p_t^j v^3|_{\{t=0\}\times \Sigma_b}=0,\quad  j=0,1,2,3.
\end{equation*}
Since $\Dtp|_{\Sigma} = \p_t + \vb \cdot \TP$ and $\mathcal{H} =-\cnab \cdot \left (\frac{\cnab\psi}{\sqrt{1+|\cnab \psi|^2}}\right)$,  we rewrite the compatibility conditions in terms of $\qc$ as
\begin{equation} \label{ccond full}
(\p_t + \vb\cdot \TP)^j \qc |_{\{t=0\}\times \Sigma} =  (\p_t +\vb\cdot \TP)^j  \left ({-\sigma}\cnab\cdot \frac{\cnab\psi}{\sqrt{1+|\cnab \psi|^2}}+g\psi\right) \Bigg |_{\{t=0\}\times \Sigma}, \quad j=0,1,2,3. 
\end{equation} 
Here, we use the modified pressure $\qc$ since we want $\p \qc_0 \in L^2(\Omega)$ for the sake of convenience. 
Such compatibility conditions are required to show that $E(t)$ (defined as \eqref{energy main them, WP}), and $E^\ls (t)$ (defined as \eqref{energyls intro}) are bounded at $t=0$ by adapting the arguments in \cite[Section 4.3]{DL19limit}.
\subsection{Formal construction} \label{subsect. B1}
We shall adapt the method developed in \cite{DL19limit} to construct smooth data $(\psi_0, v_0, \qc_0)$ that satisfies \eqref{ccond full}. 
We first describe the method formally, which serves as a good guideline. The key difference, however, is that in \cite{DL19limit} we constructed the initial data in Lagrangian coordinates, where \eqref{ccond full} has a different formulation. 

By identifying $\ff'_{\lam}(q) = \lam^2$ without loss of generality,  and since $\p_1 \varphi|_{\Sigma} = \p_1 \psi$, $\p_2 \varphi |_{\Sigma}= \p_2 \psi$, $\p_3 \varphi|_{\Sigma} =1$,
the momentum and continuity equations reduce, respectively, to
\begin{align}
\rho (\p_t+\vb\cdot \TP) v + \nabp \qc = -g(\rho-1)e_3, &~~ \text{on}\,\,\Sigma \label{mom bdy}\\
\lam^2 (\p_t+ \vb\cdot \TP) \qc + \di v=\p_1 \psi \p_3 v^1+ \p_2\psi \p_3 v^2+\lam^2 g v^3, &~~\text{on}\,\,\Sigma, \label{cont bdy}
\end{align} 
where $\nabp q = (\p_1 q- \p_1\psi \p_3q, \p_2 q - \p_2\psi\p_3 q, \p_3 q)^\top $ and $\di v= \p\cdot v$.  By ignoring the terms contributed by the denominator, we have 
$\mathcal{H} \sim - \TL \psi$. Invoking the kinematic boundary condition $\p_t\psi =v\cdot N$, we have
$$
\DT\psi = v^3, \quad \text{on }\,\Sigma,
$$
 we obtain from the zeroth compatibility condition $ \qc \sim  -\sigma \TL \psi$ that
 \begin{align}
  (\p_t + \vb\cdot \TP) \qc   \sim  -\sigma \TL v^3,\quad \text{on }\,\Sigma, \label{q_t bdy}
\end{align}
which is the first compatibility condition. 
Since the continuity equation \eqref{cont bdy} implies $\lam^2 (\p_t+ \vb\cdot \TP) \qc \sim -\di v$, we can deduce from \eqref{q_t bdy} that:
\begin{equation}
\di v \sim \sigma \lam^2  \TL v^3,\quad \text{on}\,\,\Sigma.  \label{cc 1}
\end{equation}

Furthermore, the momentum equation \eqref{mom bdy} implies $(\p_t + \vb\cdot \TP) v^3 \sim -  \p_3 \qc$, and thus the second compatibility condition on $\qc$ becomes:
\begin{equation}
(\p_t+\vb\cdot \TP)^2 \qc \sim -\sigma \DT \TL v^3 \sim \sigma \TL \p_3 \qc,\quad \text{on}\,\,\Sigma.  \label{q_tt bdy}
\end{equation}
Taking $\p_t+\vb\cdot \TP$ to the continuity equation to obtain $\lam^2 \DT^2 \qc \sim -\di \DT v  \sim  \lap \qc $, and this gives 
\begin{equation}
\p_3^2 \qc \sim \sigma  \lam^2 \TL \p_3 \qc-\TL \qc,\quad \text{on}\,\,\Sigma.  \label{cc 2}
\end{equation}

Finally, we derive from the third compatibility condition on $\qc$ that
\begin{align}
\DT^3 \qc \sim \sigma \TL \p_3 \DT \qc \sim \sigma \lam^{-2} \TL \p_3 \di v,\quad \text{on}\,\,\Sigma, 
\end{align}
 together with the relation $\lam^2 \DT^3 \qc \sim \lap \DT \qc \sim \lam^{-2} \lap \di v$ obtained by taking $\DT^2$ to the continuity equation that
 \begin{align}
 \lap \di v \sim \sigma \lam^2 \TL \p_3 \di v,\quad \text{on}\,\,\Sigma. 
 \end{align}
 In other words, 
 \begin{align}
\p_3^3 v \sim \sigma \lam^2 \TL \p_3 \di v - \lap \p_1 v -\lap \p_2 v -\TL \p_3 v,\quad \text{on}\,\,\Sigma.  \label{cc 3}
 \end{align}
 Therefore, the first order compatibility condition on $\qc$ yields an ``identity in terms of $v$" \eqref{cc 1}, the second order compatibility condition on $\qc$ yields an ``identity in terms of $q$" \eqref{cc 2}, and lastly, the third order compatibility condition on $\qc$ yields an ``identity in terms of $v$" again \eqref{cc 3}. 
 
We construct our data by the following iterative procedure. To begin with, let $(\xi_0, \ww, p_0)$ be the generic smooth \textit{localized} incompressible data that verifies the zeroth order compatibility condition $\pc_0 = -\sigma \cnab \cdot \frac{\cnab \xi_0}{\sqrt{1+|\cnab \xi_0|^2}}+g\xi_0$ on $\Sigma$.  In the first step, we fixed a smooth function $\psi_0$ that represents the moving interface, and constructed the data satisfying the first compatibility condition.  Given \eqref{cc 1}, we shall need to construct the appropriate velocity vector field denoted by $\uu = (\uu^1, \uu^2, \uu^3)$.  We achieve this by setting $\uu^1=\ww^1$, $\uu^2 = \ww^2$, and construct $\uu^3$ by solving a poly-harmonic equation of order 2:  
\begin{align} \label{equation cc 1}
\begin{cases}
\lap^2 \uu^3 = \lap^2 \ww^3,\quad &\text{in}\,\,\Omega,\\
\uu^3 = \ww^3,\quad \p_3 \uu^3 \sim -\p_1 \ww^1-\p_2 \ww^2+ \sigma \lam^2 \TL \ww^3, \quad&\text{on}\,\,\Sigma,\\
\uu^3=\ww^3,\quad \p_3\uu^3=\p_3\ww^3\quad&\text{on}\,\,\Sigma_b.
\end{cases}
\end{align}
In particular, the boundary condition  $\p_3 \uu^3 \sim -\p_1 \ww^1-\p_2 \ww^2+ \sigma \lam^2 \TL \ww^3$ is derived from \eqref{cc 1}. 

In the second step, we construct the data verifying the second compatibility condition. We shall construct $\qc_0$ here because of \eqref{cc 2}. This is achieved by solving a poly-harmonic equation of order 3:
\begin{align} \label{equation cc 2}
\begin{cases}
\lap^3 \qc_0 = \lap^3 \pc_0,\quad& \text{in}\,\,\Omega,\\
\qc_0=\pc_0,\quad \p_3\qc_0=\p_3 \pc_0, \quad&\text{on}\,\,\Sigma,\\
\p_3^2 \qc_0 \sim \sigma \lam^2 \TL \p_3 \pc_0 - \TL \pc_0, \quad&\text{on}\,\,\Sigma,\\
\p_3^{j}\qc_0=0~~(0\leq j\leq 2), \quad&\text{on}\,\,\Sigma_b.
\end{cases}
\end{align}
It can be seen that the boundary condition $\p_3^2 \qc_0 \sim \sigma \lam^2 \TL \p_3 \pc_0$ is a consequence of \eqref{cc 2}.

In the third (and final) step, we construct the data that verifies the compatibility conditions up to order $3$ with a fixed smooth function representing the moving interface, still denoted by $\psi_0$. Since $q_0$ has been constructed, we need only to construct $v_0 = (v_0^1, v_0^2, v_0^3)$ by setting $\ww^1=v_0^1, \ww^2=v_0^2$, and solving the following order 4 poly-harmonic equation for $v_0^3$:
\begin{align} \label{equation cc 3}
\begin{cases}
\lap^4 v_0^3 = \lap^4 \uu^3,\quad& \text{in}\,\,\Omega,\\
v_0^3 = \uu^3,\quad \p_3 v_0^3 \sim -\p_1 \uu^1-\p_2 \uu^2+ \sigma \lam^2 \TL \uu^3 \quad&\text{on}\,\,\Sigma,\\
\p_3^2 v_0^3 \sim -\p_3\p_1 \uu^1-\p_3\p_2 \uu^2+ \sigma \lam^2 \TL \p_3\uu^3, \quad&\text{on}\,\,\Sigma,\\
 \p_3^3 v_0^3=-\lap\p_1 \uu^1-\lap\p_2\uu^2 + \sigma \lam^2 \TL \p_3 \di \uu-\TL\p_3 \uu^3, \quad&\text{on}\,\,\Sigma,\\
\p_3^{j}v_0^3=\p_3^j\uu^3~~(0\leq j\leq 3) \quad&\text{on}\,\,\Sigma_b.
\end{cases}
\end{align}
The second and third boundary conditions arise from \eqref{cc 1}, whereas the fourth boundary condition is derived from \eqref{cc 3}.

\subsection{The full construction procedure}
We shall repeat the method introduced in Subsection \ref{subsect. B1} with detailed boundary conditions generated by the compatibility conditions. We will use $\PP, \Q$ to denote generic non-negative continuous functions. Apart from this, we will set
$$
0\leq k'\leq 1, \quad 0\leq k \leq 2, \quad 0\leq l\leq 3,
$$ 
throughout. 

By invoking the commutator 
\begin{align}
[\TP^s, \p_t+\vb\cdot \TP] = [\TP^s, \vb]\cdot \TP, \label{comm DT}
\end{align}
and since it holds on $\Sigma$ that
$$
\DT\psi = v^3, \quad \qc= -\sigma \left (\frac{\TL \psi} {|N|} - \frac{\TP\psi\cdot \TP \cnab \psi}{|N|^3} \right)+g\psi, \quad |N| = \sqrt{1+|\cnab \psi|^2},
$$
the first compatibility condition on $\qc$ reads:
\begin{align}  \label{14}
(\p_t+\vb\cdot \TP)\qc = \sigma \PP (\frac{1}{|N|}, \TP^k \psi, \TP^k \vb, \TP^k v^3), \quad \text{on}\,\,\Sigma.
\end{align}
 In addition, the continuity equation \eqref{cont bdy} gives
\begin{align} \label{15}
\lam^2 (\p_t+\vb\cdot \TP)\qc = -\di v +\TP\psi  \cdot \p_3 \vb+\lam^2gv^3 , \quad\text{on}\,\,\Sigma.
\end{align}
Hence, we combine \eqref{14} and \eqref{15} to get
\begin{equation}\label{ccond 1}
\di v = \sigma \lam^2  \PP (|N|^{-1}, \TP^k \psi,  \TP^k \vb, \TP^k v^3, \p_3 \vb), \quad \text{on}\,\,\Sigma. 
\end{equation}
and the equation used to determine $\uu^3$ is 
\begin{align} \label{equation ccond 1}
\begin{cases}
\lap^2 \uu^3 = \lap^2 \ww^3,\quad &\text{in}\,\,\Omega,\\
\uu^3 = \ww^3,\quad&\text{on}\,\,\Sigma\cup\Sigma_b,\\
 \p_3 \uu^3 = -\p_1 \ww^1-\p_2 \ww^2+ \sigma \lam^2 \PP (|N_0|^{-1}, \TP^k \psi_0, \TP^k \overline{\mathbf{w}}_0, \TP^k \ww^3, \p_3 \overline{\mathbf{w}}_0), \quad&\text{on}\,\,\Sigma,\\
\p_3\uu^3=\p_3\ww^3 \quad&\text{on}\,\,\Sigma_b.
\end{cases}
\end{align}
whose rough version is given by \eqref{equation cc 1}. Let $s_0\geq 8$. The poly-harmonic estimate yields
\begin{align} \label{est ccond 1}
\|\uu^3 -\ww^3\|_{s_0} \lesssim \underbrace{\|\lap^2 (\uu^3-\ww^3)\|_{s_0-4}}_{=0}+ \underbrace{|\uu^3-\ww^3|_{s_0-0.5}}_{=0}+|\p_3 (\uu^3-\ww^3) |_{s_0-1.5}\leq \lam^2 C(|\psi_0|_s, \|\ww\|_s),
\end{align}
for some $s>s_0$, and hence $\|\uu^3 -\ww^3\|_{s_0} \to 0$ as $\lam\to 0$.

We construct $\qc_0$ using the second-order compatibility condition in the next stage.  Owing to \eqref{mom bdy}, the identities 
\begin{align}
\rho  \DT \vb + \TP \qc =\TP \psi \p_3 \qc, \quad \text{and}\,\,
\rho  \DT v_3 + \p_3 \qc = -g(\rho-1),
\end{align}
hold on $\Sigma$, and we view $\rho =\rho(\qc)$ here and throughout. Taking $\p_t + \vb\cdot \TP$ to \eqref{14} and invoking \eqref{comm DT}, we have
\begin{align}  \label{cond2}
(\p_t+\vb\cdot \TP)^2\qc = \sigma \PP (\rho^{-1}, |N|^{-1}, \TP^l \psi, \TP^k \vb, \TP^k v^3, \TP^l \qc, \TP^k \p_3 \qc), \quad \text{on}\,\,\Sigma.
\end{align}
Moreover, 
by taking $\p_t+\vb\cdot \TP$ to the continuity equation \eqref{cont bdy}, we get
\begin{align}
\lam^2 (\p_t+ \vb\cdot \TP)^2 \qc = - \di \DT v+ [\di, \DT] v+ \DT(\TP\psi \cdot \p_3 \vb+\lam^2gv^3),
\end{align}
where $[\di, \DT] v = \p_i \vb \cdot \TP v^i$, 
\begin{align}
- \di \DT \vb = \p^{\tau} (\rho^{-1} \p_{\tau} \qc)-\p^{\tau} (\rho^{-1} \p_\tau \psi \p_3 \qc) +\underbrace{\p_3 (\rho^{-1} \p_3 \qc)}_{=\rho^{-1} \p_3^2 \qc + \p_3 \rho^{-1} \p_3 \qc} + g\p_3 \Big( \rho^{-1} (\rho-1)\Big) ,~~~~~\tau=1,2,
\end{align}
and 
\begin{align} \label{p_3^2 q appear}
\DT(\TP\psi \cdot \p_3 \vb+\lam^2 g v^3)= \TP v^3 \cdot \p_3 \vb+ \TP \psi \cdot \p_3 (-\rho^{-1} \TP \qc +\rho^{-1}\TP\psi \p_3 \qc)\nonumber\\
 - \TP\vb \cdot \TP \psi \cdot \p_3\vb -\TP\psi \cdot \p_3 \vb \cdot \p_3 \vb + \lam^2 g (-\rho^{-1}\p_3 \qc - g\rho^{-1}(\rho-1)).
\end{align}
Since the third term on the RHS of \eqref{p_3^2 q appear} contributes to $\rho^{-1} |\TP \psi|^2 \p_3^2 \qc$, it holds that
\begin{align}\label{cond2'}
\lam^2 (\p_t+ \vb\cdot \TP)^2 \qc =\rho^{-1} (1+|\TP \psi|^2) \p_3^2 \qc+ \Q( \rho^{-1}, |N|^{-1}, \TP^k \psi, \TP^{k'} \p_3 v, \TP^{k'}\p_3 \qc), \quad \text{on}\,\,\,\Sigma.
\end{align}
Therefore, we combine \eqref{cond2} and \eqref{cond2'} to get
\begin{align} \label{ccond 2}
\rho^{-1} (1+|\TP \psi|^2) \p_3^2 q = \sigma  \lam^2 \PP (\rho^{-1}, |N|^{-1}, \TP^l \psi, \TP^k \vb, \TP^k v^3, \TP^l q, \TP^k \p_3 q)+\Q(\rho^{-1}, |N|^{-1}, \TP^k \psi, \TP^{k'} \p_3 v, \TP^{k'}\p_3 q), \quad \text{on}\,\,\Sigma,
\end{align}
and we set $\qc_0$ by solving 
\begin{align} \label{equation ccond 2}
\begin{cases}
\lap^3 \qc_0 = \lap^3 \pc_0,\quad& \text{in}\,\,\Omega,\\
\qc_0=\pc_0, \quad \p_3\qc_0=\p_3\pc_0, \quad&\text{on}\,\,\Sigma,\\
\p_3^2 \qc_0 = \rho_0 (1+|\TP\psi_0|^2)^{-1} \bigg (\sigma  \lam^2 \PP (\rho_0^{-1}, |N_0|^{-1}, \TP^l \psi_0, \TP^k \overline{\mathbf{u}}_0, \TP^k \uu^3, \TP^l \pc_0, \TP^k \p_3 \pc_0)\\
+\Q(\rho_0^{-1}, |N_0|^{-1}, \TP^k \psi_0, \TP^{k'} \p_3 \uu, \TP^{k'}\p_3 \pc_0)\bigg),\quad&\text{on}\,\,\Sigma, \\
\p_3^j\qc_0=0~~(0\leq j\leq 2)\quad &\text{on}\,\,\Sigma_b.
\end{cases}
\end{align}
whose rough version is \eqref{equation cc 2}. Also, the poly-harmonic estimate implies
\begin{equation}\label{est ccond 2}
\begin{aligned}
\|\qc_0\|_{s_0} \lesssim \|\lap^3\pc_0\|_{s_0-6}+ |\pc_0|_{s_0-0.5}+|\p_3 \pc_0|_{s_0-1.5} + |\p_3^2 \qc_0|_{s-2.5} 
\leq \lam^2 C_1(|\psi_0|_s, \|\uu\|_s, \|\pc_0\|_s) + C_2(|\psi_0|_s, \|\uu\|_s, \|\pc_0\|_s),
\end{aligned}
\end{equation}
for some $s>s_0$. 

Finally, we construct $v_0^3$ using the third-order compatibility condition in the last stage. We obtain 
\begin{align} \label{cond3}
\DT^3 \qc = \sigma \PP (\rho^{-1}, |N|^{-1}, \TP^l \psi, \TP^l \vb, \TP^l v^3, \p^l \qc) \left (\lam^{-2}\TP^4\psi + \lam^{-2} \TP^4 v+ \lam^{-2} \TP^l\p_3 v+\lam^{-2} \TP^k \p_3^2 v\right ), \quad\text{on}\,\,\Sigma,
\end{align}
by taking $\DT$ to \eqref{cond2}. Further, taking $\DT$ to \eqref{cond2'} to get
\begin{align}\label{cond3'}
\lam^2 (\p_t+ \vb\cdot \TP)^3 \qc =-\lam^{-2}\rho^{-1} (1+|\TP \psi|^2) \p_3^2 \di v+ \Q(\rho^{-1}, |N|^{-1}, \TP^l \psi, \TP^{l} v, \TP^k \p_3 v, \TP^{k'} \p_3^2 v, \TP^{l} \qc),\quad \text{on}\,\,\,\Sigma.
\end{align}
Therefore, we combine \eqref{cond3} and \eqref{cond3'} to obtain
\begin{equation} \label{ccond 3}
\begin{aligned}
\rho^{-1} (1+|\TP \psi|^2) \p_3^2\di v=
 \sigma \lam^2 \PP (\rho^{-1}, |N|^{-1}, \TP^l \psi, \TP^l \vb, \TP^l v^3, \p^l \qc) \left (\TP^4\psi 
+  \TP^4 v+  \TP^l\p_3 v+ \TP^k \p_3^2 v\right )\\
 + \lam^2 \Q(\rho^{-1}, |N|^{-1}, \TP^l \psi, \TP^l \psi, \TP^{l} v, \TP^k \p_3 v, \TP^{k'} \p_3^2 v, \TP^{l} \qc), \quad \text{on}\,\,\,\Sigma,
\end{aligned}
\end{equation}
and we set $v_0^3$ by solving
\begin{equation} \label{equation ccond 3}
\begin{cases}
\lap^4 v_0^3 = \lap^4 \uu^3,\quad& \text{in}\,\,\Omega,\\
v_0^3 = \uu^3,\quad&\text{on}\,\,\Sigma,\\
 \p_3 v_0^3 =-\p_1 \uu^1-\p_2 \uu^2+ \sigma \lam^2 \PP (|N_0|^{-1}, \TP^k \psi_0, \TP^k \uu, \TP^k \uu^3, \p_3\overline{\mathbf{u}}_0), \quad 0\leq k\leq 2, \quad&\text{on}\,\,\Sigma,\\
\p_3^2 v_0^3 = -\p_1 \p_3\uu^1-\p_2 \p_3\uu^2+ \sigma \lam^2 \p_3 \PP (|N_0|^{-1}, \TP^k \psi_0, \TP^k \uu, \TP^k \uu^3, \p_3\overline{\mathbf{u}}_0), \quad&\text{on}\,\,\Sigma,\\
 \p_3^3 v_0^3= \rho_0 (1+|\TP \psi_0|^2)^{-1}\bigg(\sigma \lam^2 \PP (\rho_0^{-1}, |N_0|^{-1}, \TP^l \psi_0, \TP^l \overline{\mathbf{u}}_0, \TP^l \uu^3, \p^l \qc_0) \left (\TP^4\psi_0 
+  \TP^4 \uu+  \TP^l\p_3 \uu+ \TP^k \p_3^2 \uu\right )\\
 + \lam^2 \Q(\rho_0^{-1}, |N_0|^{-1}, \TP^l \psi_0, \TP^l \psi, \TP^{l} \uu, \TP^k \p_3 \uu, \TP^{k'} \p_3^2 \uu, \TP^{l} \qc_0)\bigg) - \rho_0^{-1} (1+|\TP\psi_0|^2) \p_3^2 (\p_1 \uu^1+\p_2\uu^2) ,  \quad&\text{on}\,\,\Sigma,\\
\p_3^jv_0^3=\p_3^j\uu^3~~(0\leq j\leq 3)\quad&\text{on}\,\,\Sigma_b.
\end{cases}
\end{equation}
whose rough version is \eqref{equation cc 3}. By the poly-harmonic estimate, we have
\begin{align}
\|v_0^3-\uu^3\|_{s_0} \lesssim \|\lap^4 (v_0^3-\uu^3)\|_{s_0-8} + |v_0^3-\uu^3|_{s_0-0.5} + |\p^3 (v_0^3-\uu^3)|_{s_0-1.5}+|\p_3^2 (v_0^3-\uu^3)|_{s_0-2.5} + |\p^3 (v_0^3-\uu^3)|_{s_0-3.5}.
\end{align}
The first two terms on the RHS are $0$. Invoking \eqref{est ccond 1}, \eqref{est ccond 2},  we have, for some $s, s'$ satisfying $s>s'>s_0$, that
$$
|v_0^3-\uu^3|_{s_0-0.5} + |\p^3 (v_0^3-\uu^3)|_{s_0-1.5} \leq \lam^2 C(|\psi_0|_{s'}, \|\uu\|_{s'}) \leq \lam^2 C(|\psi_0|_{s}, \|\ww\|_{s}) ,
$$ 
and 
$$
|\p_3^2 (v_0^3-\uu^3)|_{s_0-2.5} \leq \lam^2 C(|\psi_0|_{s'}, \|\uu\|_{s'}, \|\qc_0\|_{s'}) \leq \lam^2 C(|\psi_0|_{s}, \|\ww\|_{s}, \|\pc_0\|_{s}).
$$
Thus, 
\begin{align}\label{est ccond 3}
\|v_0^3-\uu^3\|_{s_0} \leq \lam^2 C(|\psi_0|_{s}, \|\ww\|_{s}, \|\pc_0\|_{s}). 
\end{align}
In particular, since we have set $ \ww^\tau = \uu^\tau =v_0^\tau$, $\tau=1,2$,  we deduce from \eqref{est ccond 1} and \eqref{est ccond 3} that
\begin{align}\label{est data convergence}
\|v_0 - \ww\|_{s_0} \leq \|v_0^3 - \uu^3\|_{s_0} + \|\uu^3-\ww^3\|_{s_0} = O(\lam^2).
\end{align}
In addition, we deduce from $\nabp\cdot \ww=0$ and \eqref{est data convergence} that
\begin{align}
\|\nabp\cdot v_0\|_{C^1}= O(\lam^2). 
\end{align}
Apart from these, it can be seen from \eqref{equation ccond 2} and \eqref{equation ccond 3} that $\|v_0\|_{s_0}$ and $\|\qc_0\|_{s_0}$ are uniform in both $\sigma$ and $\lam$. This allows us to take the zero surface tension and incompressible limits simultaneously.

\section{Construction of initial data for the nonlinear $\kk$-approximate system}\label{sect CWWSTkk data}
The construction of smooth initial data for the $\kk$-problem \eqref{CWWSTkk} is parallel to that in the previous section, and thus we shall only sketch the details. We will set
$$
0\leq k'\leq 1, \quad 0\leq k \leq 2, \quad 0\leq l\leq 3, \quad 0\leq m\leq 4,  \quad 0\leq n\leq 5
$$ 
in the sequel. 

Let $(\psi_0, v_0, q_0)$ be the smooth initial data constructed in the previous section. Our goal is to construct $(\psi_{\kk,0}, v_{\kk, 0}, q_{\kk, 0})$ that satisfies the $\kk$-compatibility conditions up to the third order:
\begin{align} 
\DT^j q |_{\{t=0\}\times \Sigma} = \sigma \DT^j \mathcal{H} |_{\{t=0\}\times \Sigma}& + \kk^2 \DT^j \left ((1-\TL) (-\p_1 \psk v^1-\p_2\psk v^2+v^3)\right) |_{\{t=0\}\times \Omega}, \quad &j=0,1,2,3, \label{compat cond kk}\\
&\p_t^j v^3|_{\{t=0\}\times \Sigma_b}=0,\quad  &j=0,1,2,3.
\end{align}
Setting $\psi_{\kk, 0} = \psi_0$, we need only to compute the last term on the RHS to formulate the poly-harmonic equations for $q_{\kk, 0}$ and $v_{\kk, 0}$. 
Since 
\[
[(1-\TL), \p_t +\vb\cdot \TP] = - [\TL, \vb]\cdot \TP,
\]
we have, when $j=1$:
\begin{align}
\DT \Big((1-\TL) (-\TP\psk \cdot \vb+v^3)\Big) = \RR(\TP^l \psi, \TP^l v, \TP^l \psk, \TP^l \vk^3, \TP^l \qc, \TP^k \p_3 \qc ), ~~\text{on}\,\,\Sigma. 
\end{align}
This implies that the equation used to determine $\uuk^3$ is
\begin{align} \label{equation ccondkk 1}
\begin{cases}
\lap^2 \uuk^3 = \lap^2 v_0^3,\quad &\text{in}\,\,\Omega,\\
\uuk^3 = v_0^3,\quad&\text{on}\,\,\Sigma,\\
 \p_3 \uuk^3 = -\p_1 v_0^1-\p_2 v_0^2+ \sigma \lam^2 \PP (|N_0|^{-1}, \TP^k \psi_0, \TP^k v_0, \TP^k  v_0^3, \p_3 \vb_0)\\
 +\kk^2 \lam^2 \RR(\TP^l \psi_0, \TP^l v_0, \TP^l \psk_0, \TP^l \vk_0^3, \TP^l \qc_{0}, \TP^k \p_3 \qc_{0} ),  \quad&\text{on}\,\,\Sigma,\\
\p_3^jv_{\kk,0}^3=\p_3^j\uu^3~~(0\leq j\leq 1)\quad&\text{on}\,\,\Sigma_b.
\end{cases}
\end{align}
which is parallel to \eqref{equation ccond 1}. 

Then,  when $j=2$, we have
\begin{align}
\DT^2 \Big ( (1-\TL) (-\TP\psk \cdot \vb+ v^3)\Big) = \DT \RR(\TP^l \psi, \TP^l v, \TP^l \psk, \TP^l \vk^3, \TP^l \qc, \TP^k \p_3 \qc )\nonumber\\
=\RR(\TP^l \psi, \TP^l \psk, \TP^l v^3, \TP^l \vk^3, \TP^m \qc, \TP^l \p_3 \qc, \lam^{-2} \TP^4 v, \lam^{-2} \TP^l \p_3 v, \lam^{-2}\TP^k\p_3^2 v, \lam^{-2} \TP^4 \psi), \quad \text{on}\,\,\Sigma,
\end{align}
where the power of $\lam^{-1}$ does not exceed $2$. 
Thus, we determine $q_{\kk, 0}$ by solving 
\begin{align} \label{equation ccondkk 2}
\begin{cases}
\lap^3 \qc_{\kk,0} = \lap^3 \qc_0,\quad& \text{in}\,\,\Omega,\\
\qc_{\kk,0}=\qc_0,\quad \p_3\qc_{\kk,0}=\p_3 \qc_0, \quad&\text{on}\,\,\Sigma,\\
\p_3^2 \qc_{\kk, 0} = \rho_0 (1+|\TP\psi_0|^2)^{-1} \bigg (\sigma  \lam^2 \PP (\rho_0^{-1}, |N_0|^{-1}, \TP^l \psi_0, \TP^k \overline{\mathbf{u}}_{\kk,0}, \TP^k \uuk^3, \TP^l \qc_0, \TP^k \p_3 \qc_0)\\
+\Q(\rho_0^{-1}, |N_0|^{-1}, \TP^k \psi_0, \TP^{k'}\p_3 \uuk, \TP^{k'}\p_3 \qc_0)\\
+\kk^2\lam^2\RR(\TP^l \psi_0, \TP^l \psk_0, \TP^l \uuk^3, \TP^l \widetilde{\mathbf{u}}_{\kk,0}^3, \TP^m \qc_0, \TP^l \p_3 \qc_0, \TP^4 \uuk, \TP^l \p_3 \uuk, \TP^k\p_3^2 \uuk, \TP^4 \psi_0)\bigg), \quad&\text{on}\,\,\Sigma,\\
\p_3^j\qc_0^3=0~~(0\leq j\leq 2)\quad&\text{on}\,\,\Sigma_b.
\end{cases}
\end{align}

Finally, when $j=3$, we have
\begin{align}
\DT^3 \Big ( (1-\TL) (-\TP\psk \cdot \vb+ v^3)\Big) = \RR(\TP^m \psi, \TP^m \psk, \TP^m v^3, \TP^m \vk^3, \TP^n \qc, \TP^m \p_3 \qc, 
\lam^{-2} \TP^5 v, \lam^{-2} \TP^m \p_3 v, \lam^{-2}\TP^l\p_3^2 v, \lam^{-2} \TP^5 \psi),\quad \text{on}\,\,\Sigma,
\end{align}
where the power of $\lam^{-1}$ does not exceed $4$. Therefore, we construct $v_{\kk, 0}^3$ by solving
{\small\begin{align}\label{equation ccondkk 3}
\begin{cases}
\lap^4 v_{\kk,0}^3 = \lap^4 \uuk^3, \quad&\text{in}\,\,\Omega,\\
v_{\kk, 0}^3 = \uuk^3,\quad&\text{on}\,\,\Sigma, \\
\p_3 v_{\kk,0}^3 = -\p_1 \uuk^1-\p_2 \uuk^2
+ \sigma \lam^2 \PP (|N_0|^{-1}, \TP^k \psi_0, \TP^k \uuk, \TP^k  \uuk^3, \p_3 \overline{\mathbf{u}}_{\kk,0})\\
 +\kk^2 \lam^2 \RR(\TP^l \psi_0, \TP^l \uuk, \TP^l \psk_0, \TP^l \widetilde{\mathbf{u}}_{\kk,0}^3, \TP^l \qc_{\kk,0}, \TP^k \p_3 \qc_{\kk,0} ),\quad&\text{on}\,\,\Sigma,\\
 \p_3^2 v_{\kk, 0}^3 = -\p_1\p_3 \uuk^1-\p_2\p_3 \uuk^2
+ \sigma \lam^2 \p_3\PP (|N_0|^{-1}, \TP^k \psi_0, \TP^k \uuk, \TP^k  \uuk^3, \p_3 \overline{\mathbf{u}}_{\kk,0})\\
 +\kk^2 \lam^2 \p_3\RR(\TP^l \psi_0, \TP^l \uuk, \TP^l \psk_0, \TP^l \widetilde{\mathbf{u}}_{\kk,0}^3, \TP^l \qc_{\kk,0}, \TP^k \p_3 \qc_{\kk,0} ),\quad&\text{on}\,\,\Sigma,\\
\p_3^3 v_{\kk,0}^3= \rho_0 (1+|\TP \psi_0|^2)^{-1}\bigg(\sigma \lam^2 \PP (\rho_0^{-1}, |N_0|^{-1}, \TP^l \psi_0, \TP^l \overline{\mathbf{u}}_{\kk,0}, \TP^l \uuk^3, \p^l \qc_0) \left (\TP^4\psi_0 
+  \TP^4 \uuk+  \TP^3\p_3 \uuk+ \TP^2 \p_3^2 \uuk\right )\\
 + \lam^2 \Q(\rho_0^{-1}, |N_0|^{-1}, \TP^l \psi_0, \TP^{l} \uu, \TP^k \p_3 \uu, \TP^{k'}\p_3^2 \uu,  \TP^{l} \qc_{\kk,0})\\
 +\RR(\TP^m \psi_0, \TP^m \psk_0, \TP^m \uuk^3, \TP^m \widetilde{\mathbf{u}}_{\kk,0}^3, \TP^n \qc_{\kk,0}, \TP^m \p_3 \qc_{\kk,0}, 
\TP^5 \uuk, \TP^m \p_3 \uuk, \TP^l\p_3^2 \uuk,  \TP^5 \psi_0)\bigg) \\
- \rho_0^{-1} (1+|\TP\psi_0|^2) \p_3^2 (\p_1 \uuk^1+\p_2\uuk^2),\quad&\text{on}\,\,\Sigma,\\
\p_3^jv_{\kk,0}^3=\p_3^j\uu^3~~(0\leq j\leq 3)\quad&\text{on}\,\,\Sigma_b.
\end{cases}
\end{align}}
Let $\lam>0$ be fixed. Invoking the poly-harmonic estimate subsequently to \eqref{equation ccondkk 1}, \eqref{equation ccondkk 2}, and \eqref{equation ccondkk 3}, we obtain that $\|v_{\kk, 0}\|_{s_0}$ and $\|\qc_{\kk, 0}\|_{s_0}$ are bounded for some $s_0\geq 8$. 
Thus, the energy $E^{\kk}(t)$ (defined as \eqref{energykk}) is bounded at $t=0$. 
In addition, 
\begin{align*}
\|v_{\kk, 0} -v_0\|_{s_0}, ~\text{and}~ \|\qc_{\kk,0} - \qc\|_{s_0} \rightarrow 0,\quad  \text{as}\,\,\,\kk\to 0. 
\end{align*}

\section{Paraproducts and the Dirichlet-to-Neumann operator}\label{sect app para}
\subsection{Bony's paraproduct decomposition}\label{sect app para defn}
We already introduced the paradifferential operator in Section \ref{sect para pre}. Here we present the relations between paradifferential operators and paraproducts. The cutoff function $\tilde{\chi}(\xi,\eta)$ in the definition of $T_a u$ is 
\[
\tilde{\chi}(\xi,\eta)=\sum_{k=0}^{\infty}\Theta_{k-3}(\xi)\vartheta(\eta),
\]where $\Theta(\xi)=1$ when $|\xi|\leq 1$ and $\Theta(\xi)=0$ when $|\xi|\geq 2$ and 
\[
\Theta_k(\xi):=\Theta(\frac{\xi}{2}),~~k\in\Z,\quad \vartheta_0=\Theta,\quad \vartheta_k:=\Theta_k-\Theta_{k-1},~~k\geq 1.
\] Based on this, we can introduce the Littlewood-Paley projections $\LP_k$ and $\LP_{\leq k}$ as follows
\begin{align*}
\widehat{\LP_k u}(\xi):=\vartheta_k(\xi)\hat{u}(\xi),~~\forall k\geq 0,\quad \LP_k u:=0~~\forall k<0,\quad \LP_{\leq k} u:=\sum_{l\leq k}\LP_l u.
\end{align*} When the symbol $a(x,\xi)$ (in the paradifferential operator $T_a$) does not depend on $\xi$, we can take $\psi(\eta)\equiv 1$ and then we have $$T_a u=\sum_k \LP_{\leq k-3} a(\LP_k u)$$ which is the usual Bony's paraproduct. In general, the well-known Bony's paraproduct decomposition is
\[
au=T_a u+T_u a+R(u,a),\quad R(u,a)=\sum_{|k-l|\leq 2}(\LP_k a) (\LP_l u).
\]

We have the following estimates for the remainder $R(u,a)$
\begin{lem}[{\cite[Section 2.3]{ABZ2014wwSTLWP}}] \label{lem para norm}
For $s\in\R,~r<d/2,~\delta>0$, we have \[
|T_a u|_{H^s}\lesssim \min\{|a|_{L^{\infty}}|u|_{H^s},|a|_{H^r}|u|_{H^{s+\frac{d}{2}-r}},|a|_{H^{\frac{d}{2}}}|u|_{H^{s+\delta}}\}
\]and for any $s>0,s_1,s_2\in\R$ satisfying $s_1+s_2=s+\frac{d}{2}$, we have
\[
|R(u,a)|_{H^s}\lesssim |a|_{H^{s_1}}|u|_{H^{s_2}}.
\]
\end{lem}
\subsection{Basic properties of the Dirichlet-to-Neumann operator}
Let the space dimension $d=3$ for simplicity. Given a function $f:\Sigma=\T^2\to\R$, we define the Dirichlet-to-Neumann (DtN) operator (with respect to $\psi$ and region $\Om^\pm$) by 
\[
\dn f:= \mp N\cdot\nabp (\HE_\psi ^\pm f)|_{\Sigma},\quad -\lapp (\HE_\psi^\pm f)=0\text{ in }\Om^\pm,~~\HE_\psi^\pm f|_{\Sigma}=f,~~\p_3(\HE_\psi^\pm f)|_{\Sigma^\pm}=0.
\]Here the Laplacian operator is defined by $\lapp:=\nabp\cdot\nabp = \p_i(\bm{E}^{ij}\p_j)$ with 
\[\bm{E}=\frac{1}{\p_3 \varphi}
\begin{bmatrix}
\p_3 \varphi&0&-\TP_1\varphi\\
0&\p_3\varphi&-\TP_2\varphi\\
-\TP_1\varphi&-\TP_2\varphi&\frac{1+|\cnab\varphi|^2}{\p_3\varphi}
\end{bmatrix} = \frac{1}{\p_3\varphi}\bm{P}\bm{P}^\top,~~
\bm{P}:=\begin{bmatrix}
\p_3 \varphi&0&0\\
0&\p_3\varphi&0\\
-\TP_1\varphi&-\TP_2\varphi&1
\end{bmatrix},
\] and $\varphi(t,x):=x_3+\chi(x_3)\psi(t,x')$ is defined as the extension of $\psi$ into $\Om^\pm$. The choice of $\chi(x_3)$ is slightly different from \cite{ABZ2014wwSTLWP, ABZ2014wwLWP, AMDtN}, but it does not introduce any substantial difference because the expression of $\lapp$ is still written to be $\lapp:=\nabp\cdot\nabp = \p_i(\bm{E}^{ij}\p_j)$ and we have $\lapp \varphi=0$ in $\Om^\pm$. The DtN operators satisfy the following estimates, and we refer to \cite[Appendix A.4]{SWZ2015MHDLWP} for the proof.

\begin{lem}[Sobolev estimates for DtN operators]\label{lem DtNHs}
For $s>2+\frac{d}{2},~-\frac12\leq r\leq s-1$ and $\psi\in H^{s}(\R^d)$, we have $$|\dn f|_r  \leq C(|\psi|_s)|f|_{r+1}.$$
\end{lem}

\begin{lem}[Remainder estimates for DtN operators]\label{lem DtNR}
For $s>2+\frac{d}{2}$ and $\psi\in H^{s}(\R^d)$, we have $$\dn f = T_{\Lam}f + R_\Lam^\psi (f)$$ with $\Lam$ defined in Proposition \ref{prop para DtN}. The remainder $ R_\Lam^\psi (f)$ satisfies
$$|R_\Lam^\psi(f)|_{r}\leq C(|\psi|_{s+\frac12})|f|_{r}.$$
\end{lem}

\begin{lem}[Sobolev estimates for the inverse of the DtN operator]\label{lem DtN-1}
For $s>2+\frac{d}{2},~-\frac12\leq r\leq s-1$ and $\psi\in H^{s}(\R^d)$, we have $$|(\dn)^{-1} f|_{r+1}  \leq C(|\psi|_s)|f|_{r}.$$
\end{lem}

\begin{lem}[Commutator estimate for the DtN operator and its square root]\label{lem DtN1/2}
For $s>2+\frac{d+1}{2}$ and $\psi\in H^{s}(\R^d)$, we have $$|[\dn,a]f|_{r-1}  \leq C(|\psi|_s)|a|_{r+1}|f|_{r}\quad\forall 0<r\leq s-\frac{1}{2},$$ and $$|[(\dn)^{\frac12},a]f|_{r-\frac12}  \leq C(|\psi|_s)|a|_{r+1}|f|_{r}\quad\forall -\frac12<r\leq s-1.$$
\end{lem}

\end{appendix}


\begin{thebibliography}{99}
\addcontentsline{toc}{section}{References}
\setlength{\itemsep}{0.5ex}
\begin{spacing}{0.9}
\bibitem{Alazard2005limit}
T. Alazard.
\newblock Incompressible limit of the nonisentropic Euler equations with the
  solid wall boundary conditions.
\newblock {\em Adv. Differ. Equ.}, 10(1):19--44, 2005.

\bibitem{ABZ2014wwSTLWP}
T. Alazard, N. Burq, C. Zuily.
\newblock On the water-wave equations with surface tension.
\newblock {\em Duke Math. J.}, 158(3):413--499, 2011.

\bibitem{ABZ2014wwLWP}
T. Alazard, N. Burq, C. Zuily.
\newblock On the Cauchy problem for gravity water waves.
\newblock {\em Invent. Math.}, 198(1):71--163, 2014.

\bibitem{Alazard2013GWP}
T. Alazard, J.-M. Delort.
\newblock Global solutions and asymptotic behavior for two-dimensional gravity
  water waves.
\newblock {\em Ann. Sci. {\'E}c. Norm. Sup{\'e}r.(4)}, 48(5):1149--1238, 2015.

\bibitem{AMDtN}
T. Alazard, G. M\'etivier,
Paralinearization of the Dirichlet to Neumann operator, and regularity of three-dimensional water waves.
\newblock {\em Commun. Partial Differ. Equ.}, 34(12), 1632-1704, 2009.

\bibitem{AM2005ww}
D. Ambrose, N. Masmoudi.
\newblock The zero surface tension limit of two-dimensional water waves.
\newblock  {\em Commun. Pure Appl. Math.}, 58(10), 1287-1315, 2005.

\bibitem{Alinhac1989good}
S. Alinhac.
\newblock Existence d'ondes de rarefaction pour des systèmes quasi-lineaires hyperboliques multidimensionnels.
\newblock {\em Commun. Partial differ. Equ.}, 14(2):173--230, 1989.

\bibitem{Asano1987limit}
K. Asano.
\newblock On the incompressible limit of the compressible Euler equation.
\newblock {\em Japan J. Appl. Math.}, 4(3):455--488, 1987.

\bibitem{BMSW2017GWP}
L. Bieri, S. Miao, S. Shahshahani, S. Wu.
\newblock On the motion of a self-gravitating incompressible fluid with free
  boundary.
\newblock {\em Commun. Math. Phys.}, 355(1):161--243, 2017.

\bibitem{Shkollerdivcurl}
C.-H. A. Cheng, S. Shkoller.
\newblock Solvability and regularity for an elliptic system prescribing the
  curl, divergence, and a partial trace of a vector field on Sobolev-class
  domains.
\newblock {\em J. Math. Fluid Mech.}, 19(3):375--422, 2017.

\bibitem{CL2000priori}
D. Christodoulou, H. Lindblad.
\newblock On the motion of the free surface of a liquid.
\newblock {\em Commun. Pure. Appl. Math.}, 53(12):1536--1602, 2000.

\bibitem{CS2013LWP}
D. Coutand, J. Hole, S. Shkoller.
\newblock Well-posedness of the free-boundary compressible 3-D {E}uler
  equations with surface tension and the zero surface tension limit.
\newblock {\em SIAM J. Math. Anal.}, 45(6):3690--3767, 2013.

\bibitem{CLS2010priorigas}
D. Coutand, H. Lindblad, S. Shkoller.
\newblock A priori estimates for the free-boundary 3D compressible {E}uler
  equations in a physical vacuum.
\newblock {\em Commun. Math. Phys.}, 296(2):559--587, 2010.

\bibitem{CS2007LWP}
D. Coutand, S. Shkoller.
\newblock Well-posedness of the free-surface incompressible {E}uler equations
  with or without surface tension.
\newblock {\em J. Amer. Math. Soc.}, 20(3):829--930, 2007.

\bibitem{CS2012LWPgas}
D. Coutand, S. Shkoller.
\newblock Well-posedness in smooth function spaces for the moving boundary
  three-dimensional compressible {E}uler equations in the physical vacuum.
\newblock {\em Arch. Rational Mech. Anal.}, 206(2):515--616, 2012.

\bibitem{Deng2017wwSTGWP}
Y. Deng, A. D. Ionescu, B. Pausader, F. Pusateri.
\newblock Global solutions of the gravity-capillary water-wave system in three
  dimensions.
\newblock {\em Acta Math.}, 219(2):213--402, 2017.

\bibitem{Disconzi2017limit}
M. M. Disconzi, D. G. Ebin.
\newblock Motion of slightly compressible fluids in a bounded domain. II.
\newblock {\em Commun. Contemp. Math.}, 19(04):1650054, 2017.

\bibitem{DL19limit}
M. M. Disconzi, C. Luo.
\newblock On the incompressible limit for the compressible free-boundary
  {E}uler equations with surface tension in the case of a liquid.
\newblock {\em Arch. Rational Mech. Anal.}, 237(2), 829-897, 2020.

\bibitem{Ebin1982limit}
D. G. Ebin.
\newblock Motion of slightly compressible fluids in a bounded domain. I.
\newblock {\em Commun. Pure. Appl. Math.}, 35(4):451--485, 1982.

\bibitem{GMS2012GWP}
P. Germain, N. Masmoudi, J. Shatah.
\newblock Global solutions for the gravity water waves equation in dimension 3.
\newblock {\em Ann. Math.}, 175(2), 691--754, 2012.

\bibitem{GMS2015GWP}
P. Germain, N. Masmoudi, J. Shatah.
\newblock Global existence for capillary water waves.
\newblock {\em Commun. Pure. Appl. Math.}, 68(4):625--687, 2015.

\bibitem{GLL2020LWP}
D. Ginsberg, H. Lindblad, C. Luo.
\newblock Local well-posedness for the motion of a compressible,
  self-gravitating liquid with free surface boundary.
\newblock {\em Arch. Rational Mech. Anal.}, 236(2):603--733, 2020.

\bibitem{HIT2017ww}
B. Harrop-Griffiths, M.Ifrim, D.Tataru.
\newblock Finite depth gravity water waves in holomorphic coordinates.
\newblock {\em Ann. PDE}, 3(1):4, 2017.

\bibitem{HIT2016ww}
J. K. Hunter, M. Ifrim, D. Tataru.
\newblock Two-dimensional water waves in holomorphic coordinates.
\newblock {\em Commun. Math. Phys.}, 346(2):483--552, 2016.

\bibitem{IT2016ww2}
M. Ifrim, D. Tataru.
\newblock Two-dimensional water waves in holomorphic coordinates II: global
  solutions.
\newblock {\em Bulletin de la Soci\'et\'e math\'ematique de France},
  144(2):366--394, 2016.

\bibitem{IT2016ww3}
M. Ifrim, D. Tataru.
\newblock The lifespan of small data solutions in two-dimensional capillary
  water waves.
\newblock {\em Arch. Rational Mech. Anal.}, 225(3):1279--1346, 2017.

\bibitem{IT2016ww1}
M. Ifrim, D. Tataru.
\newblock Two-dimensional gravity water waves with constant vorticity: I. cubic
  lifespan.
\newblock {\em Anal. \& PDE}, 12(4):903--967, 2018.

\bibitem{IT2020LWPgas}
M. Ifrim, D.Tataru.
\newblock The compressible Euler equations in a physical vacuum: a
  comprehensive Eulerian approach.
\newblock {\em Ann. Inst. H. Poincar\'e (C) Anal. Non Lin\'eaire}, 41 (2024), 405–495, 2024.

\bibitem{Iguchi1997limit}
T. Iguchi.
\newblock The incompressible limit and the initial layer of the compressible
  Euler equation in $\mathbb{R}^{n}_+$.
\newblock {\em Math. Methods Appl. Sci.}, 20(11):945--958, 1997.

\bibitem{Iguchi2001LWP}
T. Iguchi.
\newblock Well-posedness of the initial value problem for capillary-gravity
  waves.
\newblock {\em Funkcial. Ekvac.}, 44(2):219--242, 2001.

\bibitem{Ionescu2015wwGWP}
A. Ionescu, F. Pusateri.
\newblock Global solutions for the gravity water waves system in 2D.
\newblock {\em Invent. Math.}, 199(3):653--804, 2015.

\bibitem{Isozaki1987limit}
H. Isozaki.
\newblock{Singular limits for the compressible Euler equations in an exterior domain}
\newblock {\em J. Reine Angew. Math.}, 381:1-36, 1987.

\bibitem{Jang2009LWPgas}
J. Jang, N.Masmoudi.
\newblock Well-posedness for compressible {E}uler equations with physical
  vacuum singularity.
\newblock {\em Commun. Pure. Appl. Math.}, 62(10):1327--1385, 2009.

\bibitem{Jang2015LWPgas}
J. Jang, N. Masmoudi.
\newblock Well-posedness of compressible {E}uler equations in a physical
  vacuum.
\newblock {\em Commun. Pure. Appl. Math.}, 68(1):61--111, 2015.

\bibitem{Klainerman1981limit}
S. Klainerman, A. Majda.
\newblock Singular limits of quasilinear hyperbolic systems with large
  parameters and the incompressible limit of compressible fluids.
\newblock {\em Commun. Pure. Appl. Math.}, 34(4):481--524, 1981.

\bibitem{Klainerman1982limit}
S. Klainerman, A. Majda.
\newblock Compressible and incompressible fluids.
\newblock {\em Commun. Pure. Appl. Math.}, 35(5):629--651, 1982.

\bibitem{KTV2017LWP}
I. Kukavica, A. Tuffaha, V. Vicol.
\newblock On the local existence and uniqueness for the 3D {E}uler equation
  with a free interface.
\newblock {\em Appl. Math. \& Optim.}, 76(3):535--563, 2017.

\bibitem{JTW}
J. Jang, I. Tice, Y. Wang.
\newblock The Compressible Viscous Surface-Internal Wave Problem: Stability and Vanishing Surface Tension Limit.
\newblock {\em Commun. Math. Phys.}, 343(3): 1039-1113, 2016.

\bibitem{Lannes2005LWP}
D. Lannes.
\newblock Well-posedness of the water-waves equations.
\newblock {\em J. Amer. Math. Soc.}, 18(3):605--654, 2005.

\bibitem{lax1960local}
P. D. Lax, R. S. Phillips.
\newblock Local boundary conditions for dissipative symmetric linear
  differential operators.
\newblock {\em Commun. Pure. Appl. Math.}, 13(3):427--455, 1960.

\bibitem{Lindblad2002LWP}
H. Lindblad.
\newblock Well-posedness for the linearized motion of an incompressible liquid
  with free surface boundary.
\newblock {\em Commun. Pure. Appl. Math.}, 56(02):153--197, 2002.

\bibitem{Lindblad2003LWP}
H. Lindblad.
\newblock Well-posedness for the linearized motion of a compressible liquid
  with free surface boundary.
\newblock {\em Commun. Math. Phys.}, 236(2):281--310, 2003.

\bibitem{Lindblad2005LWP}
H. Lindblad.
\newblock Well-posedness for the motion of a compressible liquid with free surface boundary.
\newblock {\em Commun. Math. Phys.}, 260(2):319--392, 2005.

\bibitem{Lindblad2004LWP}
H. Lindblad.
\newblock Well-posedness for the motion of an incompressible liquid with free
  surface boundary.
\newblock {\em Ann. Math.}, 162(1), 109--194, 2005.

\bibitem{LL2018priori}
H. Lindblad, C. Luo.
\newblock A priori estimates for the compressible {E}uler equations for a
  liquid with free surface boundary and the incompressible limit.
\newblock {\em Commun. Pure. Appl. Math.}, 2018.

\bibitem{Lindblad2009priori}
H. Lindblad, K.H. Nordgren.
\newblock A priori estimates for the motion of a self-gravitating
  incompressible liquid with free surface boundary.
\newblock {\em J. Hyperbolic Differ. Equ.}, 6(02):407--432,
  2009.

\bibitem{Luo2018CWW}
C. Luo.
\newblock On the motion of a compressible gravity water wave with vorticity.
\newblock {\em Ann. PDE}, 4(2):1--71, 2018.

\bibitem{LuoZhang2020CWWLWP}
C. Luo, J. Zhang.
\newblock Local well-posedness for the motion of a compressible gravity water
  wave with vorticity.
\newblock {\em J. Differ. Equ.}, 332:333--403, 2022.

\bibitem{LXZ2014LWPgas}
T. Luo, Z. Xin, H. Zeng.
\newblock Well-posedness for the motion of physical vacuum of the
  three-dimensional compressible {E}uler equations with or without
  self-gravitation.
\newblock {\em Arch. Rational Mech. Anal.}, 213(3):763--831, 2014.

\bibitem{MR2012good}
N. Masmoud, F. Rouss\'et.
\newblock Uniform regularity and vanishing viscosity limit for the free surface
  {N}avier-{S}stokes equations.
\newblock {\em Arch. Rational Mech. Anal}, 223(1):301--417, 2017.

\bibitem{Metivier2004}
G. M\'etivier.
\newblock {\em Small viscosity and boundary layer methods: Theory, stability analysis, and applications.}
\newblock Springer Science \& Business Media, 2004.

\bibitem{MetivierPara}
G. M\'etivier
\newblock {\em Para-differential calculus and applications to the Cauchy problem for nonlinear systems.}
\newblock Edizioni della Normale, Pisa, (5) (2008).

\bibitem{Metivier2001limit}
G. M\'etivier, S. Schochet.
\newblock The incompressible limit of the non-isentropic {E}uler equations.
\newblock {\em Arch. Rational Mech. Anal.}, 158(1):61--90, 2001.

\bibitem{MZ2009}
M. Ming, Z. Zhang.
\newblock Well-posedness of the water-wave problem with surface tension.
\newblock {\em J. Math. Pures Appl.}, 92(5), 429-455, 2009.

\bibitem{Nalimov1974LWP}
V. Nalimov.
\newblock {The Cauchy-Poisson problem}.
\newblock {\em Dinamika Splo{\v{s}}n. Sredy,(Vyp. 18 Dinamika Zidkost. so
  Svobod. Granicami)}, 254:104--210, 1974.

\bibitem{Rauch1985}
J. Rauch.
\newblock Symmetric Positive Systems with Boundary Characteristic of Constant Multiplicity
\newblock {\em Trans. Amer. Math. Soc.}, 291(1), 167-187, 1985.

\bibitem{Schochet1986limit}
S. Schochet.
\newblock The compressible {E}uler equations in a bounded domain: Existence of
  solutions and the incompressible limit.
\newblock {\em Commun. Math. Phys.}, 104(1):49--75, 1986.

\bibitem{SZ2008geometry}
J. Shatah, C. Zeng.
\newblock Geometry and a priori estimates for free boundary problems of the
  {E}uler's equation.
\newblock {\em Commun. Pure. Appl. Math.}, 61(5):698--744, 2008.

\bibitem{SZ2008priori}
J. Shatah, C. Zeng.
\newblock A priori estimates for fluid interface problems.
\newblock {\em Commun. Pure. Appl. Math.}, 61(6):848--876, 2008.

\bibitem{SZ2011LWP}
J. Shatah, C. Zeng.
\newblock Local well-posedness for fluid interface problems.
\newblock {\em Arch. Rational Mech. Anal.}, 199(2):653--705, 2011.

\bibitem{Stevens2016CVS}
B. Stevens
\newblock Short-time structural stability of compressible vortex sheets with surface tension.
\newblock {\em Arch. Rational Mech. Anal.}, 222(2), 603-730, 2016.

\bibitem{Su2020GWP}
Q. Su.
\newblock Long time behavior of 2D water waves with point vortices.
\newblock {\em Commun. Math. Phys.}, 380(3):1173--1266, 2020.

\bibitem{SWZ2015MHDLWP}
Y. Sun, W. Wang, Z. Zhang.
\newblock Nonlinear Stability of the Current-Vortex Sheet to the Incompressible MHD Equations.
\newblock  {\em Commun. Pure Appl. Math.}, 71(2), 356-403, 2018.

\bibitem{Tao}
T. Tao. 
\newblock Nonlinear dispersive equations: local and global analysis
\newblock American Mathematical Soc., 2006.

\bibitem{Trakhinin2009LWP}
Y. Trakhinin.
\newblock Local existence for the free boundary problem for nonrelativistic and
  relativistic compressible {E}uler equations with a vacuum boundary condition.
\newblock {\em Commun. Pure. Appl. Math.}, 62(11):1551--1594, 2009.

\bibitem{TrakhininWangCMHDLWP}
Y. Trakhinin, T. Wang.
\newblock Well-posedness of free boundary problem in non-relativistic and
  relativistic ideal compressible magnetohydrodynamics.
\newblock {\em Arch. Rational Mech. Anal.}, 239(2):1131--1176, 2021.

\bibitem{TrakhininWangCMHDSTLWP}
Y. Trakhinin, T. Wang.
\newblock Well-posedness for the free-boundary ideal compressible
  magnetohydrodynamic equations with surface tension.
\newblock {\em Math. Ann.}, 383(1):761--808, 2022.

\bibitem{Ukai1986limit}
S. Ukai.
\newblock The incompressible limit and the initial layer of the compressible
  Euler equation.
\newblock {\em J. Math. Kyoto Univ.}, 26(2):323--331, 1986.

\bibitem{WZZZ2015LWP}
C. Wang, Z. Zhang, W. Zhao, Y. Zheng.
\newblock Local well-posedness and break-down criterion of the
  incompressible {E}uler equations with free boundary,
\newblock {\em Memoirs Amer. Math. Soc.}, vol. 270, 2021.

\bibitem{WangXC2018GWP2D}
X. Wang.
\newblock Global infinite energy solutions for the 2D gravity water waves
  system.
\newblock {\em Commun. Pure. Appl. Math.}, 71(1):90--162, 2018.

\bibitem{WangXin2015good}
Y. Wang, Z. Xin.
\newblock Vanishing viscosity and surface tension limits of incompressible
  viscous surface waves.
\newblock {\em SIAM J. Math. Anal.}, 53(1):574--648, 2021.

\bibitem{Wu1997LWP}
S. Wu.
\newblock Well-posedness in {S}obolev spaces of the full water wave problem in
  2-{D}.
\newblock {\em Invent. Math.}, 130(1):39--72, 1997.

\bibitem{Wu1999LWP}
S. Wu.
\newblock Well-posedness in {S}obolev spaces of the full water wave problem in
  3-{D}.
\newblock {\em J. Amer. Math. Soc.}, 12(2):445--495, 1999.

\bibitem{Wu2009GWP}
S. Wu.
\newblock Almost global wellposedness of the 2-D full water wave problem.
\newblock {\em Invent. Math.}, 177(1):45, 2009.

\bibitem{Wu2011GWP}
S. Wu.
\newblock Global wellposedness of the 3-D full water wave problem.
\newblock {\em Invent. Math.}, 184(1):125--220, 2011.

\bibitem{Yosihara1982LWP}
H. Yosihara.
\newblock Gravity waves on the free surface of an incompressible perfect fluid
  of finite depth.
\newblock {\em Publ. Res. Inst. Math. Sci.}, 18(1):49--96, 1982.

\bibitem{Zhang2021elastoLWP}
J. Zhang.
\newblock Local well-posedness and incompressible limit of the free-boundary
  problem in compressible elastodynamics.
\newblock {\em Arch. Rational Mech. Anal.}, 244(3):599--697, 2022.

\bibitem{Zhang2023CMHDVS2}
J. Zhang.
\newblock On the incompressible limit of current-vortex sheets with or without surface tension.
\newblock arXiv:2405.00421, preprint, 2024.

\bibitem{ZZ2008LWP}
P. Zhang, Z. Zhang.
\newblock On the free boundary problem of three-dimensional incompressible
  {E}uler equations.
\newblock {\em Commun. Pure. Appl. Math.}, 61(7):877--940, 2008.

\bibitem{ZhengF2019GWP}
F. Zheng.
\newblock Long-term regularity of 3D gravity water waves.
\newblock {\em Commun. Pure Appl. Math.}, 75(5):1074--1180, 2022.
\end{spacing}
\end{thebibliography}
\end{document}